\documentclass[11pt,letterpaper]{amsart}

\usepackage[a4paper,includeheadfoot,
left=32mm,right=32mm,
top=32mm,bottom=35mm]{geometry}

\usepackage[T1]{fontenc}
\usepackage[utf8]{inputenc}
\usepackage[english]{babel}

\usepackage{amsmath,amssymb,amsfonts,amsthm}
\usepackage{mathrsfs}
\usepackage{mathtools}
\usepackage[mathscr]{euscript}

\usepackage{xcolor}

\usepackage{microtype}

\usepackage[hidelinks]{hyperref}
\hypersetup{colorlinks={true},linkcolor={blue},citecolor=blue}
\numberwithin{equation}{section}

\theoremstyle{plain}
\newtheorem{theorem}{Theorem}[section]
\newtheorem{proposition}[theorem]{Proposition}
\newtheorem{lemma}[theorem]{Lemma}
\newtheorem{corollary}[theorem]{Corollary}

\theoremstyle{definition}
\newtheorem{definition}[theorem]{Definition}

\theoremstyle{remark}
\newtheorem{remark}[theorem]{Remark}

\renewcommand{\epsilon}{\varepsilon}

\newcommand{\noi}{\noindent}

\newcommand{\ov}{\overline}

\def\Im{\textrm{Im}}

\newcommand{\R}{\mathbb{R}}
\newcommand{\C}{\mathbb{C}}
\newcommand{\Z}{\mathbb{Z}}
\newcommand{\N}{\mathbb{N}}

\newcommand{\la}{\langle}
\newcommand{\ra}{\rangle}

\newcommand{\dotcirc}{\mathbin{\odot}}

\newcommand{\e}{\mathrm{e}}
\newcommand{\dd}{\mathrm{d}}

\makeatletter
\@namedef{subjclassname@2020}{\textup{2020} Mathematics Subject Classification}
\makeatother

\begin{document}
	
	\title[
Beyond the probabilistic scaling for NLS on the ball
	]{Beyond the probabilistic scaling for NLS on the ball }
	
	\author{Nicolas Burq}
	\address{Universit\'e Paris-Saclay, Laboratoire de Math\'ematique d’Orsay, UMR CNRS 8628, Orsay,
		France, and Institut universitaire de France }
	\email{nicolas.burq@universite-paris-saclay.fr}
	
	\author{Nicolas Camps}
	\address{Univ Rennes, IRMAR - UMR CNRS 6625, F-35000 Rennes, France}
	\email{nicolas.camps@univ-rennes.fr}
	
	\author{Chenmin Sun}
	\address{CNRS, Universit\'e Paris-Est Cr\'eteil,  Laboratoire d'Analyse et de Math\'ematiques Appliqu\'ees, UMR CNRS 8050, Cr\'eteil, France}
	\email{chenmin.sun@cnrs.fr}
	
	\author{Nikolay Tzvetkov}
	\address{Institut Universitaire de France (IUF) and Ecole Normale Sup\'erieure de Lyon, Unit\'e de Math\'ematiques Pures et Appliqu\'es,  UMR CNRS 5669, Lyon, France}
	\email{nikolay.tzvetkov@ens-lyon.fr}

	
	
	\date{\today}

	\begin{abstract}
	We construct probabilistic strong solutions to the cubic Schr\"odinger equation on the three-dimensional ball
	with radial initial data, which is a significant improvement of a result by Bourgain--Bulut.
	These solutions lie in the supercritical regime with respect to the probabilistic scaling introduced by Deng--Nahmod--Yue.
	We achieve this result through gauge transformations, without renormalization, combined with a refined modulation analysis using random averaging operators.
\end{abstract}

	\ \vskip -1cm  \hrule \vskip 1cm \vspace{-8pt}
	\maketitle 
	{ \textwidth=4cm \hrule}

	\maketitle
	\setcounter{tocdepth}{1}
	\tableofcontents
	
	\section{Introduction}
Consider the defocusing cubic Schrödinger equation
\begin{equation}
	\label{eq:NLS}
	\tag{NLS}
	(i\partial_{t}+ \Delta)u = |u|^{2}u, \qquad (t,x)\in \R\times \mathbf{B}_{d},
\end{equation}
posed on the unit ball~$\mathbf{B}_{d}$ in~$\R^{d}$:
\[
\mathbf{B}_{d}:=\{x\in\R^{d}:\ |x|\leq 1\}, \qquad d=2,3.
\]
We restrict to radial solutions with Dirichlet boundary conditions
$
u|_{\R\times \partial\mathbf{B}_{d}}=0.
$
In two dimensions, the Gibbs measure and the associated global dynamics were constructed in~\cite{Tzv08}.
In three dimensions, the situation is significantly more challenging due to the increased singularity of the initial data in the support of the Gibbs measure, and the lack of nonlinear smoothing in the evolution. In~\cite{BB14}, Bourgain and Bulut proved almost sure global existence and uniqueness of solutions that preserve the Gibbs measure. Our work can be seen as a natural continuation of~\cite{Tzv08,BB14}.

In three dimensions, the critical Sobolev regularity for~\eqref{eq:NLS} is~$s=1/2$, while typical initial data distributed according to the Gibbs measure belong to the radial Sobolev space~$H_{\mathrm{rad}}^{s}(\mathbf{B}_{3})$ for every~$s<1/2$. As discussed later in this introduction, the threshold~$s=1/2$ also coincides with the critical regularity for the probabilistic scaling introduced in~\cite{DNY24,DNY22}.

The purpose of the present work is to address the three-dimensional problem and to substantially improve the result of~\cite{BB14} through a refined analysis of the nonlinearity combined with appropriate gauge transformations.

\subsection{Statement of the main result}
For notational convenience, we rescale the problem and assume throughout that the ball~$\mathbf{B}:=\mathbf{B}_{3}$ has radius~$\pi$. An orthonormal basis of radial eigenfunctions of the Dirichlet Laplacian on~$\mathbf{B}$ is given by
\begin{equation}
	\label{eq:en}
	\mathbf{e}_{n}(x)
	= \frac{1}{\sqrt{2}\pi}\frac{\sin(n|x|)}{|x|}\mathbf{1}_{|x|\leq \pi},
	\qquad n\in\N\setminus\{0\},
\end{equation}
associated with the eigenvalues~$n^{2}$.
Denoting by~$\langle\cdot,\cdot\rangle$ the $L^{2}(\mathbf{B})$ inner product, any radial function
$u\in L^{2}(\mathbf{B})$ admits the expansion
\[
u=\sum_{n\geq1}\langle u,\mathbf{e}_{n}\rangle\,\mathbf{e}_{n}.
\]
The random initial data are defined by
\begin{equation}\label{eq:initial_data}
	u|_{t=0}
	=
	\phi_{\alpha}(\omega,x):=\sum_{n\geq1}\frac{g_{n}(\omega)}{n^{\alpha}}\,\mathbf{e}_{n}(x),
	\qquad \alpha\leq 1,
\end{equation}
where~$(g_{n})_{n\geq1}$ is a sequence of independent standard complex Gaussian variables.
The map~$\omega\mapsto \phi_\alpha(\omega,\cdot)$ induces the Gaussian measure~$\mu_\alpha$, supported on
\[
H_\mathrm{rad}^{(\alpha-\frac{1}{2})-}(\mathbf{B}):=\bigcap_{s<\alpha-\frac{1}{2}}H_\mathrm{rad}^{s}(\mathbf{B})\,.
\]
For~$N\geq1$, we denote by~$\Pi_{N}$ the spectral projector onto frequencies~$\leq N$,
\[
\Pi_{N}u
= \sum_{1\leq n\leq N}\langle u,\mathbf{e}_{n}\rangle\,\mathbf{e}_{n},
\]
and by~$\mathbf{P}_{N}$ the Littlewood--Paley projector onto frequencies of size comparable to~$N$,
\[
\mathbf{P}_{N}:=\Pi_{N}-\Pi_{N/2}.
\]
Bourgain and Bulut studied in~\cite{BB14} the following finite-dimensional approximation of~\eqref{eq:NLS}:
\begin{equation}
	\label{eq:BB}
	\begin{cases}
		(i\partial_{t}+\Delta)u_{N}
		= \Pi_{N}(|u_{N}|^{2}u_{N}),\\
		u_{N}|_{t=0}
		= \Pi_{N}\phi.
	\end{cases}
\end{equation}
Since the dynamics evolve in a finite-dimensional space, the solution
$u_{N}=\Pi_{N}u_{N}$ is globally defined by standard Cauchy--Lipschitz theory, together with mass conservation. They proved the following theorem.
\begin{theorem}[\cite{BB14}]
	\label{thm:BB}
	There exists a measurable set~$\Sigma_{1}$ with~$\mu_{1}(\Sigma_{1})=1$ such that
	for every~$\phi\in\Sigma_{1}$ and every~$T>0$, the sequence~$(u_{N})_{N\geq1}$ converges in
	$L^{\infty}([-T,T]; H_\mathrm{rad}^{\frac12-}(\mathbf{B}_{3}))$ to a unique limit
	$
	u\in C(\R;H_\mathrm{rad}^{\frac12-}(\mathbf{B}_{3}))
	$
	solving~\eqref{eq:NLS}.
\end{theorem}
The analysis of~\cite{BB14} also shows that the Gibbs measure is invariant under the solution map~$\phi\mapsto u(t,\cdot)$. Recall that the Gibbs measure is absolutely continuous with respect to~$\mu_1$, with Radon--Nikodym derivative
$
\exp\big(-\frac12\int_{\mathbf{B}}|u(x)|^{4}\,\mathrm{d}x\big).
$

Our goal is to substantially improve the result of~\cite{BB14}.
We establish almost sure local well-posedness for Gaussian initial data given by~\eqref{eq:initial_data} at Sobolev regularities strictly below the typical regularity of functions in the support of the Gibbs measure, for some~$\alpha < 1$, hence in a supercritical regime with respect to the probabilistic scaling introduced by Deng--Nahmod--Yue in~\cite{DNY24,DNY22}. See Appendix \ref{sec:scaling}.

To this end, we still consider frequency-truncated initial data but evolve them according to the full
cubic Schrödinger equation:
\begin{equation}
	\label{eq:NLS-N}
	\begin{cases}
		(i\partial_{t}+\Delta)u_{N}
		= |u_{N}|^{2}u_{N},\\
		u_{N}|_{t=0}
		= \Pi_{N}\phi.
	\end{cases}
\end{equation}
Since the initial data of~\eqref{eq:NLS-N} are smooth almost surely,
using the methods of~\cite{Tzv08} one can show that~\eqref{eq:NLS-N} has a unique global solution
in
$
C(\R;H_\mathrm{rad}^2(\mathbf{B})\cap H^1_0(\mathbf{B})),
$
for every~$N\geq 1$. However, the limit~$N\to\infty$ of~$u_N$ is highly nontrivial because of the low regularity of the limit of
$\Pi_{N}\phi_\alpha$.
Our main result proves the almost sure local-in-time convergence of~$u_N$.
\begin{theorem}
	\label{thm:main1}
	Let~$\alpha>17/18$. There exists a measurable set~$\Sigma_{\alpha}$ with~$\mu_{\alpha}(\Sigma_{\alpha})=1$ such that, for every~$\phi\in\Sigma_{\alpha}$, there exists~$T>0$ such that the sequence~$(u_{N})_{N\geq1}$ of solutions to~\eqref{eq:NLS-N}
	converges in
	$L^{\infty}([-T,T];H_\mathrm{rad}^{\alpha-\frac12-}(\mathbf{B}))$
	to a limit~$u\in C([-T,T];H_\mathrm{rad}^{\alpha-\frac12-}(\mathbf{B}))$ solving~\eqref{eq:NLS}.
\end{theorem}
Our proof of Theorem~\ref{thm:main1} applies equally well to both approximations~\eqref{eq:BB} and~\eqref{eq:NLS-N}, and naturally produces the same limit. By contrast, the approach of Bourgain--Bulut crucially exploits the presence of the projector~$\Pi_{N}$ on the right-hand side of~\eqref{eq:BB}.

A quantitative statement of Theorem~\ref{thm:main1}, describing the precise local-in-time structure of~$u_{N}$, is given in Theorem~\ref{thm:quan} below. The key point is to prove the convergence of the dyadic subsequence~$(u_N)_{N\in 2^{\mathbb{N}}}$, and we focus only on dyadic integers~$N$ in the sequel.
The convergence for the full sequence~$(u_{N})$ can then be deduced as in~\cite{DNY24,BCST24}.

The constraint~$\alpha>15/16$, which is discussed in Remark~\ref{rem:alpha}, is not optimal. Our only objective here is to go beyond the critical threshold~$\alpha=1$.

We now compare our approach with that of Bourgain--Bulut.
For~$N<M$ and a norm~$\|\cdot\|$, we set
\begin{align}\label{def:xNM}
x_{N,M}(t):=\|u_{N}(t,\cdot)-u_{M}(t,\cdot)\|.
\end{align}
To prove statements like Theorem~\ref{thm:BB} or Theorem~\ref{thm:main1}, one has to show that for~$t\in [-T,T]$ and a suitable choice of the norm~$\|\cdot\|$,
\[
\lim_{N,M\to\infty}
x_{N,M}(t)=0.
\]
Initially, $x_{N,M}(0)=0$. In~\cite{BB14}, it is shown that for a suitable choice of norm in \eqref{def:xNM} there exist constants~$C>0$ and~$\delta>0$ such that, for all~$N<M$ and for short times~$t\geq0$,
\begin{equation}
	\label{eq:int}
	x_{N,M}(t)\leq CN^{-\delta}+C\log(N)\int_{0}^{t}x_{N,M}(\tau)\,\mathrm{d}\tau.
\end{equation}
Applying Gronwall's inequality yields
\[
x_{N,M}(t)\leq CN^{-\delta+Ct},
\]
which implies convergence of~$(u_{N})_{N\geq1}$ on time intervals~$[-T,T]$ with~$CT<\delta$.

\subsubsection{Probabilistic scaling}
The logarithmic divergence in~\eqref{eq:int} arises from a purely high-frequency interaction, isolated below in~\eqref{res:hhh}.
This contribution cannot be handled by the random averaging operators nor by the random tensor techniques from~\cite{DNY24,DNY22}.

In the terminology of~\cite[Paragraph~2.2]{DNYvietnam}, the problem is critical with respect to the probabilistic scaling when~$\alpha=1$, and supercritical when~$\alpha<1$.
Further discussion and quantitative statements are deferred to Appendix~\ref{sec:scaling}.

Let us compare this situation with the case of~$\mathbb{T}^{d}$.
On the torus, as recognized by the physicists in the 60', non-smoothing contributions also arise, such as
\[
\|u_{N}(t)\|_{L^{2}(\mathbb{T}^{d})}^{2}u_{N}(t)
\]
in the cubic case, or more generally
\[
\|u_{N}(t)\|_{L^{2r}(\mathbb{T}^{d})}^{2r}u_{N}(t)
\]
for power-type nonlinearities\footnote{More precisely, these arise in their Wick-ordered formulations which  were necessary to define the Gibbs measures.}.
Crucially, these terms can be completely eliminated by a global gauge transformation leading to a renormalisation of the (wick ordered) NLS. See also the works  ~\cite{Bou96} in the cubic case or~\cite{NS15,DNY24} for NLS in the general case.

In the present setting, however, no global gauge transformation can simultaneously remove all singular contributions while preserving the remaining structure of the nonlinearity.
Indeed, we shall see that each Fourier mode~$\widehat{u_{N}}(n)$ satisfies
\begin{equation}
	\label{eq:Cn}
	(i\partial_{t}-n^{2})\widehat{u_{N}}(n)
	= C_{n}(u_{N})\,\widehat{u_{N}}(n)+\text{(other terms)},
\end{equation}
where the divergent real-valued coefficient~$C_{n}(u_{N})$, made explicit in~\eqref{res:hhh}, depends on the frequency~$n$ and has to be gauged out in order to achieve well posedness.

Let us mention that our work illustrates the fact that the probabilistic scaling threeshold reflects a property of a multilinear map which does not take into account the fact that in practice, this map acts on a solution of the PDE, i.e. it satisfies partial differential relations.

\subsubsection{Gauge transformations}
\label{sec:gauge}
A key observation of the present work is that the divergent coefficient~$C_{n}(u_{N})$ possesses a favorable algebraic structure.
Exploiting this structure, we introduce refined gauge transformations adapted to~\eqref{eq:Cn}, through which the problematic resonant contribution is converted into non-resonant quintic terms~\eqref{eq:N4}. These new contributions can then be treated perturbatively.

Gauge transformations of this type are inspired by Tao's work on wave maps and the Benjamin--Ono equation~\cite{Tao01,Tao04}, as well as by~\cite{OTW20} in the probabilistic setting.

Combining this transform with the random averaging ansatz of~\cite{DNY24} and a refined modulation analysis in the spirit of~\cite{BCST24}, we obtain the analogue of~\eqref{eq:int} without the logarithmic divergence:~$x_{N,M}$ satisfies, for short times~$t\geq0$,
\begin{equation}
	\label{eq:int2}
	x_{N,M}(t)\leq CN^{-\delta}+C\int_{0}^{t}x_{N,M}(\tau)\,\mathrm{d}\tau.
\end{equation}
In contrast, the approach of~\cite{BB14} would yield in this regime an estimate of the form
\begin{equation}
	\label{eq:int3}
	x_{N,M}(t)\leq CN^{-\delta}
	+CN^{\beta(\alpha)}\int_{0}^{t}x_{N,M}(\tau)\,\mathrm{d}\tau,
\end{equation}
with~$\beta(\alpha)>0$ and~$\beta(\alpha)\to 0$ as~$\alpha\to 1^{-}$. The divergence in~$N^{\beta(\alpha)}$ prevents any Gronwall-type argument.
	
	\subsection{Overview of the main ingredients}
We outline below the main ingredients in the proof of Theorem~\ref{thm:main1}.

We begin by decomposing the cubic nonlinearity to isolate the singular resonant contribution producing the logarithmic divergence identified in~\cite{BB14}.
This decomposition motivates the gauge transformations discussed in Paragraph~\ref{gauge}.

Our strategy builds on the framework developed for the nonlinear Schrödinger equation on the two-dimensional sphere in our previous work~\cite{BCST24}.
In particular, we introduce refined random averaging operators acting on a very small portion of the nonlinearity.
In the radial setting, these reduce to multiplication by complex phases --- one-dimensional analogues of the unitary operators in~\cite{BCST24}.

Finally, unlike~\cite{BCST24}, the analysis is carried out entirely in frequency space, reflecting the one-dimensional nature of the radial problem.
The quantitative part of the argument relies on a refined modulation analysis in Fourier--Lebesgue spaces and requires new Strichartz-type estimates, which constitute one of the main technical contributions.

\subsubsection{Structure of the nonlinearity}
We now decompose the cubic nonlinearity to identify the interaction responsible for the logarithmic divergence in~\cite{BB14}.

Fix $N\in 2^{\N}$. The approximate solution $u_{N}$ satisfies
\begin{equation}
	\label{eq:nlsN}
	(i\partial_{t}+\Delta)u_{N} = |u_{N}|^{2}u_{N}\,,\quad
	u_{N}(0) = \Pi_{N}\phi_{\alpha}\,.
\end{equation}
For $n\in\N\setminus\{0\}$ and $t\in\R$, we denote the $n$-th mode of $u_{N}$ by
\[
u_{n}(t):=\la u_{N}(t)\,|\,\mathbf{e}_{n}\ra \in\C\,.
\]
In this subsection, for conciseness, we drop the $N$-dependence.
The coefficients $(u_{n})$ are solutions to
\[
(i\partial_{t}-n^2)u_{n} =
\sum_{n_{1},n_{2},n_{3}}
\gamma_{nn_{1}n_{2}n_{3}}\cdot
u_{n_1}\overline{u}_{n_2}u_{n_3}\,,
\]
where, given $(n,n_{1},n_{2},n_{3})\in\N^{4}$, the correlation function $\gamma$, which is symmetric under all permutations of its indices (since $\mathbf{e}_n$ is real-valued), is defined by
\begin{equation}
	\label{eq:gamma}
	\gamma_{nn_{1}n_{2}n_{3}}:=\int_{\mathbf{B}}\mathbf{e}_{n}(x)\mathbf{e}_{n_{1}}(x)\mathbf{e}_{n_{2}}(x)\mathbf{e}_{n_{3}}(x)\mathrm{d}x\,.
\end{equation}
With these notations we may associate an $L^{2}$-function $u$ with the sequence of its coefficients $u_{n}:=\la u\,|\,\mathbf{e}_{n}\ra$, for $n\in\N\setminus\{0\}$, and we expand the nonlinearity
\begin{equation}
	\label{eq:non}
	|u|^{2}u
	=
	\sum_{n,n_{1},n_{2},n_{3}}
	\gamma_{nn_{1}n_{2}n_{3}}\cdot
	u_{n_{1}}\overline{u_{n_{2}}}u_{n_{3}} \mathbf{e}_{n}\,.
\end{equation}

For fixed $n\in\N\setminus\{0\}$, the $n$-th mode of the nonlinearity can be decomposed as follows:
\begin{align}
	(i\partial_{t}-n^{2})u_{n}
	&=
	\sum_{n_{1},n_{2},n_{3}}
	\mathbf{1}_{n\notin\{n_{1},n_{3}\}}
	\gamma_{nn_{1}n_{2}n_{3}}u_{n_{1}}
	\overline{u_{n_{2}}}u_{n_{3}} - \gamma_{nnnn}|u_{n}|^{2}u_{n}
	\label{eq:111}
	\\
	&+
	2
	\big(
	\sum_{n_{2},n_{3}}\mathbf{1}_{n_{2}\neq n_{3}}
	\gamma_{nnn_{2}n_{3}}
	\overline{u_{n_{2}}}u_{n_{3}}
	\big)
	u_{n}
	\label{eq:112}
	\\
	&+
	2\big(
	\sum_{ m}\gamma_{nnmm}|u_{m}|^{2}
	\big)u_{n}\,.
	\label{res:hhh}
\end{align}
As in~\cite{BCLST25,BCST24,BCST24-2}, the high$\times$low$\times$low interactions with a partial pairing are not regularizing, so that the affine ansatz introduced by Bourgain~\cite{Bou96}
\[
u_N(t) = \e^{it\Delta} \Pi_N \phi_\alpha + \textit{smoother remainder}
\]
which relies on dispersive nonlinear smoothing effects is not applicable to the present model. In addition, when $\alpha \leq 1$, the first Picard iterate associated with resonant high$\times$high$\times$high interactions in \eqref{res:hhh} also fails to exhibit any smoothing.

\subsubsection{The gauge transform}
\label{gauge}
We identify the term~\eqref{res:hhh} as the critical one, producing a logarithmic divergence when $\alpha=1$ even in the high$\times$high$\times$high regime. The correction needed to handle this term is a priori nonlinear. Note, however, that one can benefit from cancellations in $|u_{n}(0)|^{2}-1$ in the probabilistic framework. Exploiting these cancellations allows us to avoid a nonlinear transformation of the equation and instead perform a correction at the level of the linear part, in sharp contrast with~\cite{OTW20}.

For each $n \in \mathbb{N}\setminus\{0\}$, define
\begin{equation}
	\label{eq:mu}
	\mu_n^2 := 2 \sum_{m\geq1} \gamma_{nnmm} \, \mathbb{E}[|u_m(0)|^2]
	= 2 \sum_{m\geq1} \frac{\gamma_{nnmm}}{m^{2\alpha}}.
\end{equation}
For $\alpha > 1/2$, the growth of $\mu_n^2$ is sublinear. More precisely, Lemma~\ref{lem:gamma} gives, as $n\to\infty$,
\begin{equation}
	\label{eq:mu-est}
	\mu_{n}^{2}\sim n^{2(1-\alpha)}\ \text{when}\ \alpha<1\,,\quad \mu_{n}^{2} \sim \log(n)\ \text{when}\ \alpha=1\,.
\end{equation}

We then set the modified dispersion relation
\[
\lambda_n^2 := n^2 + \mu_n^2,
\]
and introduce the associated twisted Laplace operator $\widetilde{\Delta}$ via its action on the eigenfunctions: for each $n \in \mathbb{N}$,
\[
\widetilde{\Delta} \mathbf{e}_n := \lambda_n^2 \mathbf{e}_n.
\]
Subtracting $\mu_n^2 u_n$ from both sides and inserting $\pm\, 2(\sum_m \gamma_{nnmm} |u_m(0)|^2) u_n$, we obtain:
\begin{align*}
	(i\partial_{t}-n^2-\mu_{n}^{2})u_{n}
	&=
	\sum_{n_{1},n_{2},n_{3}}
	\mathbf{1}_{n\notin \{n_{1},n_{3}\}}
	\gamma_{nn_{1}n_{2}n_{3}}
	u_{n_{1}}
	\overline{u_{n_{2}}}u_{n_{3}}
	- \gamma_{nnnn}|u_{n}|^{2}u_{n}
	\\
	&+2
	\sum_{n_{2},n_{3}}\mathbf{1}_{n_{2}\neq n_{3}}
	\gamma_{nnn_{2}n_{3}}
	\overline{u_{n_{2}}}u_{n_{3}}
	u_{n} \\
	&+
	2
	\sum_{m}\gamma_{nnmm}
	\big(
	|u_{m}|^{2}-|u_{m}(0)|^{2}
	\big)
	u_{n} \\
	&+
	2
	\sum_{m}\gamma_{nnmm}
	\big(
	|u_{m}(0)|^{2}-\mathbb{E}|u_{m}(0)|^{2}
	\big)
	u_{n}.
\end{align*}
Note that
\begin{align*}
	|u_{m}(t)|^{2} - |u_{m}(0)|^{2}
	&=
	\int_{0}^{t}\frac{\mathrm{d}}{\mathrm{d}t'}|u_{m}(t')|^{2}\mathrm{d}t'\\
	&\hspace{-20pt}=2\Im
	\int_{0}^{t}\sum_{m_{1},m_{2},m_{3}}
	\gamma_{mm_{1}m_{2}m_{3}}
	u_{m_{1}}(t')\overline{u_{m_{2}}}(t')u_{m_{3}}(t')\overline{u_{m}}(t')
	\mathrm{d}t',
\end{align*}
the resonant interaction on the right hand side is eliminated by the observation that for fixed $m$ (see Lemma \ref{lem:cancellation}), and we obtain that
\begin{align}\label{eq:cancellation}
	|u_{m}(t)|^{2} - |u_{m}(0)|^{2}
=
&2\Im
\int_{0}^{t}\sum_{\substack{m_{1},m_{2},m_{3}\\
\{m_1,m_3 \}\cap\{m_2,m\}=\emptyset
 } }
\gamma_{mm_{1}m_{2}m_{3}}
u_{m_{1}}(t')\overline{u_{m_{2}}}(t')u_{m_{3}}(t')\overline{u_{m}}(t')
\mathrm{d}t' \notag \\
+&4\Im\int_0^t\sum_{\substack{m_1,m_2,m_3\\
m_3\neq m, m_1=m_2	
} }\gamma_{mm_1m_2m_3}|u_{m_1}(t')|^2u_{m_3}(t')\ov{u_m}(t')\dd t' \notag \\
-&2\Im\int_0^t\sum_{\substack{m_2\\
m_2\neq m } }\gamma_{mmm_2m}|u_m(t')|^2u_m(t')\ov{u_{m_2}}(t')\dd t'\notag \\
-&2\Im\int_0^t\sum_{\substack{m_2\\
		m_1\neq m } }\gamma_{m_1m_1m_1m}|u_{m_1}(t')|^2u_{m_1}(t')\ov{u_{m}}(t')\dd t'
\end{align}

In the last line, we benefit from a crucial cancellation that eliminates the problematic resonant interactions involving $m\in\{m_{1},m_{3}\}$ (except for over-paired cases). This strategy draws inspiration from similar computations performed in~\cite{OTW20} and~\cite{Tsutsumi}, at the difference that we perform a linear correction of the frequencies rather than a nonlinear gauge transform.

The above discussion leads us to introduce the following quadrilinear form:
\begin{align}
	\label{eq:N4}
&\mathcal{N}_n^{(4)}(f_1, f_{2}, f_{3}, f_{4})(t)
	\notag \\
	:=&2\Im \sum_{\{m_1,m_3 \}\cap \{m_2,m\}=\emptyset }
	\gamma_{nnmm} \, \gamma_{mm_1m_2m_3}
	\int_0^t f_{1,m_1}(t') \, \overline{f_{2,m_2}}(t') \, f_{3,m_3}(t') \, \overline{f_{4,m}}(t') \, \mathrm{d}t' \notag \\
	+&4\Im\sum_{\substack{m_1,m_2,m_3\\
m_3\neq m,  m_1=m_2	
 } } \gamma_{nnmm}\gamma_{mm_1m_1m_3}		\int_0^t f_{1,m_1}(t') \, \overline{f_{2,m_1}}(t') \, f_{3,m_3}(t') \, \overline{f_{4,m}}(t') \, \mathrm{d}t' \notag \\
-&2\Im\sum_{m_2:\; m_2\neq m}\gamma_{nnmm}\gamma_{mmm_2m}		\int_0^t f_{1,m}(t') \, \overline{f_{2,m_2}}(t') \, f_{3,m}(t') \, \overline{f_{4,m}}(t') \, \mathrm{d}t'\notag \\
-&2\Im \sum_{m_1:\; m_1\neq m}\gamma_{nnmm}\gamma_{m_1m_1m_1m}	\int_0^t f_{1,m_1}(t') \, \overline{f_{2,m_1}}(t') \, f_{3,m_1}(t') \, \overline{f_{4,m}}(t') \, \mathrm{d}t'.
\end{align}
To organise the various contributions in the equation for $u_n$, we also introduce the trilinear form:
\begin{equation}
	\label{eq:N3}
	\mathcal{N}_n^{(3)}(f_1, f_{2}, f_{3})(t)
	:=
	\sum_{\substack{n_1, n_2, n_3 \\ n \notin \{n_1, n_3\}
	}}
	\gamma_{nn_1n_2n_3} \,
	f_{1,n_1}(t) \, \overline{f_{2,n_2}}(t) \, f_{3,n_3}(t)\,.
\end{equation}
We now introduce the term involving constants and initial data:
\begin{equation}
	\label{eq:N2}
	\begin{aligned}
		\mathcal{N}_n^{(2)}(f_{2}, f_{3})(t)
		:=\;&
		2\sum_{\substack{n_2, n_3 \\ n_2 \neq n_3}}
		\gamma_{nnn_2n_3} \, \overline{f_{2,n_2}}(t) \, f_{3,n_3}(t) \\
		&+
		2\sum_{m}
		\gamma_{nnmm} \left( f_{3,m}(0) \, \overline{f_{2,m}}(0) - \mathbb{E}|u_{m}(0)|^{2} \right).
	\end{aligned}
\end{equation}
The equation for $u_n$ then takes the form:
\begin{equation}
	\label{eq:un}
	(i\partial_t - \lambda_n^2) u_n
	= \mathcal{N}_n^{(3)}(u)
	+ \big( \mathcal{N}_n^{(2)}(u) + \mathcal{N}_n^{(4)}(u) \big) u_n
	- \gamma_{nnnn} \, |u_n|^2 u_n.
\end{equation}
The term $\mathcal{N}^{(3)}$ exhibits a smoothing effect at the level of the second Picard iteration, even in the presence of high$\times$low frequency interactions.
By contrast, such interactions in $\mathcal{N}^{(2)}$ and $\mathcal{N}^{(4)}$ do not yield any regularization.

\subsubsection{The random averaging operator ansatz}
We now implement a \emph{random averaging operator} ansatz to capture high$\times$low frequency interactions. This approach is a frequency-localized variant of the paracontrolled ansatz of Gubinelli--Imkeller--Perkowski~\cite{GIP15}, implemented in the context of stochastic wave equations in~\cite{GKO24}. For a related work in the context of wave maps, see \cite{sterbenz2010energy}. Random averaging operators were introduced by Deng--Nahmod--Yue~\cite{DNY24}, building on key ideas from Bringmann~\cite{Bri20}. The adaptation to bounded domains different from~$\mathbb{T}^{d}$ relies on the framework developed in~\cite{BCST24}.

Our goal is to prove the convergence of the sequence
\[
u_N(t) := \sum_{n} u_{N,n}(t) \mathbf{e}_n
\]
in $H_\mathrm{rad}^{\alpha-\frac{1}{2}-}(\mathbf{B})$,
where $u_N$ solves~\eqref{eq:nlsN}. Let $\psi_N$ denote the colored Gaussian term --- or term of type (C) --- defined by
\[
\psi_N(t)
= \sum_{\frac{N}{2} < n \leq N} \psi_{N,n}(t) \mathbf{e}_n,
\quad \psi_N(0) = \mathbf{P}_N \phi_\alpha(\omega),
\]
where each mode satisfies the linear ODE:
\begin{equation}
	\label{eq:psi}
	(i\partial_t - \lambda_n^2) \psi_{N,n}
	= \left( \mathcal{N}_n^{(2)}(u_{\frac{N}{2}})(t) + \mathcal{N}_n^{(4)}(u_{\frac{N}{2}})(t) \right)\cdot \psi_{N,n},
\end{equation}
with initial condition
\[
\psi_{N,n}(0) = \mathbf{1}_{\frac{N}{2} < n \leq N} \cdot \frac{g_n(\omega)}{n^\alpha}.
\]
Equation~\eqref{eq:psi} is a linear ODE in $\mathbb{C}$ for each mode $\psi_{N,n}$, so that, in this model, the random averaging operator acts as multiplication by a complex phase.

Knowing $u_{\frac{N}{2}}$, we define, for $\frac{N}{2}<n\leq N$,
\[
\Theta_n^N(t) := \mathcal{N}_{n}^{(2)}(u_{\frac{N}{2}})(t) + \mathcal{N}_{n}^{(4)}(u_{\frac{N}{2}})(t) \in\R\,,
\]
and $\Theta_{n}^N(t)=0$ otherwise.
Hence $\psi_{N,n}(t)=\mathcal{H}_n^N(t)\cdot \psi_{N,n}(0)$, where the associated \emph{random averaging phase} is the time-dependent function
\[
\mathcal{H}_n^N(t) :=
\exp\left( -i \Big(t\lambda_{n}^{2}+\int_0^t \Theta_n^N(\tau) \, \mathrm{d}\tau\Big) \right).
\]

Consequently, the colored Gaussian variable can be written explicitly as
\begin{equation}
	\label{eq:C}
	\psi_N(t) = \sum_{\frac{N}{2} < n \leq N} \mathcal{H}_n^N(t) \cdot \frac{g_n(\omega)}{n^\alpha} \mathbf{e}_n.
\end{equation}
Since $\Theta_n^N$ is real-valued, we have the following property 
\[
	|\mathcal{H}_n^N(t)| = 1.
\]
It is not needed in the sequel. Moreover, the map $t \mapsto \Theta_n^N(t)$ is $\mathcal{B}_{\leq \frac{N}{2}}$-measurable, hence independent from the collection of Gaussian variables $(g_{n})_{\frac{N}{2}<n\leq N}$. 

Later we will rigorously define the modification $\mathcal{H}_n^{N,\dag}(t)$ through an inductive time-localization procedure, following~\cite{BCST24}.

\medskip

Set $v_N:=u_N-u_{\frac{N}{2}}$. We have
\begin{align*}
	(i\partial_t-\lambda_n^2)v_{N,n}=
	&\mathcal{N}_{n}^{(3)}(u_{N})-\mathcal{N}_{n}^{(3)}(u_{\frac{N}{2}})
	\\+&\Theta^N_{n}\cdot v_{N,n}+\sum_{j=2,4}\big(\mathcal{N}_{n}^{(j)}(u_{\frac{N}{2}}+v_{N})-\mathcal{N}_{n}^{(j)}(u_{\frac{N}{2}})\big)\cdot u_{N,n} \\
	+& \mathcal{R}_{n}(u_N)-\mathcal{R}_{n}(u_{\frac{N}{2}}),
\end{align*}
where the formal remainder is
\begin{equation}
	\label{def:R1}
	\mathcal{R}_{n}(f):=\gamma_{nnnn}|f_n|^2 f_n.
\end{equation}
The colored term $\psi_N$ solves
\[
(i\partial_t-\lambda_n^2)\psi_{N,n}=\Theta^N_n(t)\psi_{N,n},\quad \psi_{N,n}|_{t=0}=\mathbf{1}_{\frac{N}{2}<n\leq N}\phi_{\alpha,n}.
\]
We now impose the decomposition ansatz:
\[
u_N = u_{\frac{N}{2}} + v_N,\quad v_N=\psi_N+w_N.
\]
By construction, each mode of the smoother remainder $w_{N}$ solves
\begin{align}
	\label{eq:w}
	(i\partial_{t}-\lambda_{n}^{2})w_{N,n}
	=&\mathcal{N}^{(3)}_{n}(u_{\frac{N}{2}}+v_N)-
	\mathcal{N}_{n}^{(3)}(u_{\frac{N}{2}})\notag \\
	+&\Theta_n^N\cdot w_{N,n} +
	\sum_{j=2,4}\big(
	\mathcal{N}_{n}^{(j)}(u_{\frac{N}{2}}+v_N)-\mathcal{N}_{n}^{(j)}(u_{\frac{N}{2}})
	\big)
	\cdot u_{N,n} \notag \\
	+&\mathcal{R}_{n}(u_{\frac{N}{2}}+v_N)-\mathcal{R}_{n}(u_{\frac{N}{2}}),
\end{align}
with initial condition
\[
w_{N,n}|_{t=0}=0.
\]
The key cancellation occurs in the equation for the remainder~$w_N$. Since $v_N = \psi_N + w_N$ and $\psi_N$ solves~\eqref{eq:psi}, the most singular high$\times$low$\times$low interaction is absorbed: in the second line of~\eqref{eq:w}, the identity
\[
\Theta_{n}^{N}\cdot w_{N,n} = \Theta_{n}^{N}\cdot v_{N,n} - \Theta_{n}^{N}\cdot\psi_{N,n}
\]
shows that the contribution of~$\Theta_n^N$ to~$\psi_{N,n}$ is precisely subtracted, leaving only the action on the smoother term~$w_{N,n}$.

\subsubsection{The modulation analysis and quantitative statement}
The low-regularity regime necessitates new Strichartz-type estimates and delicate deterministic analysis adapted to the presence of non-resonant interactions, similar in spirit to those required in our forthcoming work~\cite{BCST24-2}. In particular, we develop in Sections~\ref{sec:detertool} and~\ref{sec:trilinear} new Strichartz-type estimates and an almost-orthogonality argument that play a key role in the modulation analysis. However, in contrast with our other works~\cite{BCST24,BCST24-2} on $\mathbb{S}^{2}$, there is no analysis in the physical space here.

\medskip

We now state the quantitative version of Theorem~\ref{thm:main1}.
\begin{theorem}[Quantitative statement]
	\label{thm:quan}
	Let $\alpha>17/18$. There exist absolute constants $C_1>c_1>0$ and $\delta_0>0$ such that the following holds. For any $R\geq 1$, let $\tau_R=R^{-C_1}$. There exists an event $\Omega_R$ with
	\[
	\mathbb{P}(\Omega_{R}^c)<C_1\e^{-c_1R^{\delta_0}},
	\]
	such that for all $\omega\in\Omega_R$, the sequence of smooth solutions $u_N$ of~\eqref{eq:nlsN} with initial data $\Pi_N\phi_{\alpha}(\omega,\cdot)$ converges to $u$ in $C([-\tau_R,\tau_R];H_{\mathrm{rad}}^{\alpha-\frac{1}{2}-}(\mathbf{B}))$. Moreover, the limit $u$ solves~\eqref{eq:NLS} in the distributional sense and can be written as
	\begin{align*}
		u(t)=\psi_{\infty}(t)+w_{\infty}(t),
	\end{align*}
	where $\psi_{\infty}$ is the limit of the colored Gaussian terms~\eqref{eq:C} in $C([-\tau_R,\tau_R];H_{\mathrm{rad}}^{\alpha-\frac{1}{2}-}(\mathbf{B}))$, and $w_{\infty}$ is the limit of the smoother remainders solving~\eqref{eq:w} in $C([-\tau_R,\tau_R];H_{\mathrm{rad}}^{s_0}(\mathbf{B}))$ for some $s_0>1/2$.
\end{theorem}
In particular, this proves Theorem \ref{thm:main1} through the map $\omega\mapsto \phi_{\alpha}(\omega,\cdot)$.
\begin{remark}
	As suggested in~\cite[Remark~6.3]{DNY24}, it is possible to extend the structural information globally in time when $\alpha=1$, using the invariant measure argument. We will address this issue in a forthcoming work.
\end{remark}

\subsection{Further remarks}
We conclude this introduction with some further remarks on our main results and possible perspectives.

\begin{remark}[Other dimensions]
	The case $d\geq 4$ appears to be out of reach for several reasons.
	In particular, for $d=4$ the cubic NLS on $\mathbb{S}^4$ is energy-critical,
	and the Strichartz estimates on the sphere suffer from a derivative loss that
	prevents one from closing a critical well-posedness argument.
	This is in sharp contrast with the case of $\mathbb{T}^{4}$, where the
	arithmetic structure of the torus yields stronger Strichartz estimates, and a
	global flow in the energy space was constructed by Yue~\cite{Yue21}.
	
	The Gibbs measure problem in the case $d=2$, however, is subcritical with respect to
	the probabilistic scaling and was addressed by
	Tzvetkov~\cite{Tzv08}. A refined analysis establishing the flow-map property for NLS with general nonlinearities
	was recently carried out in~\cite{FW25}.
\end{remark}

\begin{remark}[On the parameter $\alpha$]
	\label{rem:alpha}
	The main restriction on the parameter $\alpha$ arises from the conditions
\begin{align}\label{eq:constraintalpha}
		\frac12<s<4\alpha-3,\quad s<\frac{\alpha}{2}+\frac{1}{4},
	\quad
	s+4\alpha-\frac{9}{2}>0.
\end{align}	The first constraint is dictated naturally by the regularity of the second Picard iteration after removal of the resonant frequency interaction.
	Heuristically, performing the probabilistic scaling analysis as in~\cite{DNY24,DNY22}, after having removed the critical term, the typical regularity of the non-resonant high$\times$high interaction is $N^{-3\alpha+2}$, which is significantly smoother than the initial data $N^{-\alpha+1/2}$ provided $\alpha>3/4$.
	In our analysis, an additional factor $N^{1-\alpha}$ is lost due to non-optimal counting estimates associated with the modified dispersion relation $\lambda_n^2=n^2+\mu_n^2$.

	The second and the third constraints are more technical and originates from the high--high interaction in the quintic term, where we rely solely on deterministic estimates.
	
	We believe that $\alpha>3/4$ should represent the optimal threshold for probabilistic well-posedness within the present approach, but we do not pursue this issue here in order to keep the presentation concise.
\end{remark}

\begin{remark}[Zonal spherical harmonics on $\mathbb{S}^{3}$]
	The analysis developed in this work is expected to extend to the cubic NLS on the three-dimensional sphere $\mathbb{S}^{3}$ for zonal spherical harmonics.
	In fact, this setting should be technically simpler, as no boundary effects are present. The key ingredient needed to overcome the boundary effect is the decay estimate of the correlation function stated in~\eqref{eq:gamma3}.
\end{remark}

\subsection{Organization of the article}
The article is organized as follows. In Section~\ref{sec:functional}, we present a more precise statement of Theorem~\ref{thm:main1}, the functional framework, the rigorous iteration scheme, as well as key multilinear estimates. In Sections~\ref{sec:detertool} and~\ref{sec:probatool}, we establish Strichartz-type inequalities and probabilistic moment and operator bounds that are used substantially in the proof of the multilinear estimates. Sections~\ref{sec:trilinear} and~\ref{sec:quintilinear} are devoted to the proof of the key multilinear estimates, namely Proposition~\ref{prop:trilinear} and Proposition~\ref{prop:quintilinear}, which occupy the major part of this article.

\subsection*{Acknowledgments}
We are grateful to Yu Deng and Bjoern Bringmann for stimulating discussions, respectively in September 2024 at Orsay and in April 2025 at Princeton.
This research was supported by the European Research Council (ERC) under the European Union's Horizon 2020 research and innovation programme (Grant agreement 101097172 -- GEOEDP). C.S.\ and N.T.\ were partially supported by the ANR project Smooth ANR-22-CE40-0017.
N.C.\ benefited from the support of the Centre Henri Lebesgue ANR-11-LABX-0020-0.


	\section{Functional setup, induction scheme and proof of the main Theorem}\label{sec:functional}

In this section, we set up the functional framework for the proof of Theorem~\ref{thm:main1}. We begin with multilinear eigenfunction estimates and introduce the relevant function spaces. We then describe the rigorous induction scheme and state the key multilinear estimates that drive the argument.

\subsection{Eigenfunction estimates}
We state some important bounds on the correlation coefficient. The first is taken from~\cite{BB14}, the second is nontrivial due to boundary effects, and the third is a lower bound in the high$\times$high regime that plays a role in Appendix~\ref{sec:scaling} for probabilistic scaling considerations.

\begin{lemma}[Correlation bound]\label{lem:gamma}
	Given $n,n_{1},n_{2},n_{3}\in\N^{4}$, let~$\gamma_{nn_{1}n_{2}n_{3}}$ be the correlation coefficient defined in~\eqref{eq:gamma'}.
	\begin{enumerate}
		\item There exists~$C>0$ such that for $(n,n_1,n_2,n_3)\in\N^4$,
		\begin{equation}
			\label{eq:gamma'}
			\gamma_{nn_{1}n_{2}n_{3}}
			\leq C
			\min(n,n_{1},n_{2},n_{3})\,.
		\end{equation}
		\item Moreover, if~$n\gg n_1+n_2+n_3$, then
		\begin{equation}
			\label{eq:gamma3}
			\gamma_{nn_1n_2n_3}\leq C n^{-2}.
		\end{equation}
		\item There exists~$c>0$ such that for all~$(n,m)$ with~$n/2\leq m \leq 2n$,
		\begin{equation}
			\label{eq:gamma2}
			c\,n \leq \gamma_{nnmm}\,.
		\end{equation}
	\end{enumerate}
\end{lemma}
We refer to~\cite{BB14} (Section 2 in \cite{BB14}) for a short proof of the first statement, and to Appendix~\ref{sub:gamma} for the proof of the second and third statements.

\subsection{Notations and parameters}

We introduce a hierarchy of small parameters as in~\cite{BCST24}.
\begin{definition}[Hierarchy of small parameters]\label{hierarchy}
	Fix a universal small parameter~$0<\sigma\ll 1$, and set~$(s,q,\gamma,b,\delta)=(s_{\sigma},q_{\sigma},\gamma_{\sigma},b_{\sigma},\delta_{\sigma})$ satisfying
	\begin{align*}
		&\frac{1}{q_{\sigma}}=\sigma, \qquad
		b_{\sigma}-\frac{1}{2}=\gamma_{\sigma}-\frac{1}{q_{\sigma}'}=\sigma^{10},\qquad
		\delta_{\sigma}=\sigma^{20},\\
		&
		\frac{1}{2}+2^{200}\sigma<s_{\sigma}<\frac{\alpha}{2}+\frac{1}{4}-2^{100}\sigma
		,\qquad
		s_{\sigma}+4\alpha-\frac{9}{2}>2^{100}\sigma.
	\end{align*}
\end{definition}
\begin{remark}\label{rem:hierarchy}
	For sufficiently small~$\sigma>0$, the range of~$s_{\sigma}$ is non-empty if~$\alpha\in(\frac{17}{18},1]$. Under this range, we automatically have 
	$ \frac{\alpha}{2}+\frac{1}{4}<4\alpha-3,
	$
	which verifies the constraint  \eqref{eq:constraintalpha}.
\end{remark}

\subsection{Functional spaces}
In light of the decomposition of the nonlinearity presented in the introduction, it is natural to use the Fourier-restriction norm associated with the twisted Laplace operator. For a space-time function~$f(t,x)=\sum_{n}f_n(t)\mathbf{e}_n(x)$, define
\begin{equation}
	\label{eq:xsb}
	\|f\|_{X^{s,b}} =
	\|
	\langle n\rangle^{s}\langle \tau + \lambda_{n}^{2} \rangle^{b}
	\widehat{f_{n}}(\tau)
	\|_{\ell_{n}^{2}L_{\tau}^{2}}
	=
	\|\langle n\rangle^{s}
	\langle \kappa \rangle^{b}
	\widetilde{f_{n}}(\kappa)
	\|_{\ell_{n}^{2}L_{\kappa}^{2}},
\end{equation}
where
\[
\widetilde{f_{n}}(\kappa):=\widehat{f_{n}}(\kappa-\lambda_{n}^{2})
= \widehat{\e^{it\lambda_{n}^{2}}f_{n}}(\kappa)\,.
\]
We will show that the Fourier-restriction space associated with the norm~\eqref{eq:xsb}, thanks to the bound~\eqref{eq:gamma'} on~$\mu_{n}$, has the same properties as the standard one.

We also use more general Fourier-restriction norms:
\[
\|f\|_{X_{q,r}^{s,\gamma}}:=\|\la n\ra^s\la\kappa\ra^{\gamma}\widetilde{f}_n(\kappa) \|_{\ell_n^qL_{\kappa}^r}.
\]
For the random averaging operator~$\mathcal{H}_n^N(t)$ (a time-dependent function) with a fixed~$n$, we denote
\[
\|\mathcal{H}_n^N \|_{\widetilde{\mathcal{F}}L_q^{\gamma}}:=\|\la\kappa\ra^{\gamma}\widetilde{H}_n^N(\kappa)\|_{L_{\kappa}^q},
\]
where the wide tilde stands for the twisted Fourier transform, to compare with the usual Fourier--Lebesgue norm
\[
\|F\|_{\mathcal{F}L_q^{\gamma}}:=\|\la\tau\ra^{\gamma}\widehat{F}(\tau) \|_{L_{\tau}^q}.
\]
We introduce the Duhamel operator, acting on a space-time function~$F$ by
\begin{align}\label{def:Duhamel}
	\mathcal{I}F(t):=-i\int_0^t\e^{i(t-t')\widetilde{\Delta}}F(t') \mathrm{d}t'.
\end{align}
Decomposing~$F=\sum_{n}F_n\mathbf{e}_n$, we denote the mode-by-mode Duhamel operator
\begin{align}\label{def:Duhameln}
	\mathcal{I}_n(F)(t):=-i\int_0^t\e^{-i\lambda_n^2(t-t')}F_n(t')\mathrm{d}t'.
\end{align}

\subsection{Rigorous iteration scheme}

Following~\cite{BCST24}, we first define the extension objects. Fix a bump function~$\chi\in C_c^{\infty}((-1,1))$ such that~$\chi(t)\equiv 1$ for~$|t|\leq \frac{1}{2}$. For small~$0<T<1$, we set~$\chi_T(\cdot):=\chi(T^{-1}\cdot)$.
We define the truncated Duhamel operator
\begin{align}\label{def:DuhamelT}
	\mathcal{I}_{\chi_T}(F)(t):=-i\chi_T(t)\int_0^t \e^{-i(t-t')\widetilde{\Delta}}\chi_T(t')F(t')\mathrm{d}t',
\end{align}
and the mode-by-mode version
\begin{align}\label{def:DuhamelTn}
	\mathcal{I}_{\chi_T,n}(F)(t):=-i\chi_T(t)\int_0^t\e^{-i\lambda_n^2(t-t')}\chi_T(t')F_n(t') \mathrm{d}t'.
\end{align}

We denote by~$\mathcal{B}_{\leq N}$ the $\sigma$-algebra generated by~$(g_n)_{n\leq N}$, and set~$\mathcal{B}_{\leq\infty}:=\sigma(\bigcup_N\mathcal{B}_{\leq N})$. 
\medskip
\noi
$\bullet${\bf A. Initialization :} 
Fix a small time $T\in(0,1/2)$. For the initial frequency~$N_0=1$, we define the following objects at scale~$N_0$:
\[
\big(\mathcal{H}_1^{1,\dag}(t),\;
g^{\dag}_1(t),\; \psi_1^{\dag}(t),\;
w_1^{\dag}(t)\big):=(\chi(t)\mathcal{H}_1^{1}(t),\; \chi(t)g_1,\; \chi_T(t)g_1^{\dag}(t)\mathbf{e}_1,\; 0 ).
\]
\medskip
$\bullet${\bf B. Induction step :} We inductively define the objects for  scales~$N\geq N_0$ 
\[
((\mathcal{H}_n^{N,\dag}(t))_{N/2<n\leq N},\; (g_n^{\dag}(t))_{N/2<n\leq N},\; \psi_N^{\dag}(t),\; w_N^{\dag}(t) )
\]
such that they have time support~$(-1,1)$ and coincide on~$[-T/2,T/2]$ with
\[
((\mathcal{H}_n^{N}(t))_{N/2<n\leq N},\; (g_n(t))_{N/2<n\leq N} ,\; \psi_N(t),\; w_N(t))\,.
\]
Assuming that the above objects are well-defined on $[-T,T]$
for all $M\leq N$, we set
\[
u_N^{\dag}:=\sum_{M\leq N}(\psi_M^{\dag}+w_M^{\dag}).
\]
By construction, $u_N^{\dag}(t) = u_N(t)$ for~$|t|\leq T/2$.
Next, for~$N<n\leq 2N$, set
\begin{align}\label{Theta2Ndag}
	\Theta_n^{2N,\dag}(t):=\mathcal{N}_n^{(2)}(u^{\dag}_{N}(t))+\mathcal{N}_n^{(4)}(u_N^{\dag}(t)).
\end{align}
\noi
$\bullet${\bf Step 1 : Random averaging operator and colored Gaussians.}

We then define~$\mathcal{H}_n^{2N,\dag}(t)$ for~$N<n\leq 2N$ by solving the fixed-point problem
\begin{align}\label{H2Ndag}
	\mathcal{H}_n^{2N,\dag}(t):=\chi(t)\e^{-it\lambda_n^2}-i\chi_{2T}(t)\int_0^t
	\e^{-i\lambda_n^2(t-t')}\Theta_n^{2N,\dag}(t')\mathcal{H}_{n}^{2N,\dag}(t')\mathrm{d}t'.
\end{align}
If the fixed-point problem \eqref{H2Ndag} admits a unique solution (in a suitable function space to be specified), then for ~$|t|\leq T$, $\mathcal{H}_n^{2N,\dag}(t)$ solves the ODE
\[
i\frac{\mathrm{d}}{\mathrm{d}t}\mathcal{H}_n^{2N,\dag}(t)=\lambda_n^2\mathcal{H}_n^{2N,\dag}(t)+\Theta_n^{2N,\dag}(t)\mathcal{H}_n^{2N,\dag}(t),
\]
hence for~$|t|\leq T$,
\[
\mathcal{H}_n^{2N,\dag}(t)=\exp\Big(-i\Big(t\lambda_n^2+\int_0^t\Theta_n^{2N,\dag}(\tau)\mathrm{d}\tau\Big)\Big).
\]
Moreover, for~$|t|\leq T/2$, since~$\Theta_n^{2N,\dag}(t)=\Theta_n^{2N}(t)$ we have
\[
\mathcal{H}_n^{2N,\dag}(t)=\mathcal{H}_n^{2N}(t)\,.
\]
We define the time-dependent Gaussian random variable
\begin{align}\label{def:gndag}
	g^{2N,\dag}_n(t;\omega):=\mathcal{H}_n^{2N,\dag}(t)\cdot g_n(\omega),\quad N<n\leq 2N,
\end{align}
and
\begin{align}\label{def:psi2Ndag}
	\psi_{2N}^{\dag}(t):=\chi_T(t)\sum_{N<n\leq 2N}\frac{g^{2N,\dag}_n(t,\omega)}{n^{\alpha}}\mathbf{e}_n(x).
\end{align}
In the sequel, we abbreviate~$g_n^{2N,\dag}(t)$ for~$g_n^{2N,\dag}(t;\omega)$. Consequently,
\[
\widetilde{g}_n^{2N,\dag}(\kappa):=\widetilde{\mathcal{H}}_n^{2N,\dag}(\kappa)g_n(\omega).
\]
Since the random variables~${\Theta}_n^{2N,\dag}$ are~$\mathcal{B}_{\leq N}$-measurable and have modulus~$1$ on~$[-T,T]$, the variable~$g_n^{2N,\dag}(t)$ has the same law as~$g_n(\omega)$ for any~$t\in[-T,T]$. However, $\widetilde{g}_n^{2N,\dag}(\kappa)$ does not have the same law as~$g_n$, since the invariance in law holds only pointwise in time and~$\widetilde{\mathcal{H}}_n^{2N,\dag}(\kappa)$, which involves an average over all times, is not of modulus one. Nevertheless, for all~$M\leq 2N$, $m\approx M$, the random variables~$(\widetilde{\mathcal H}_m^{M,\dag})$ are independent of~$(g_n(\omega))_{N<n\leq 2N}$.
We note that the equation
\[
(i\partial_t-\lambda_n^2)\psi^{\dag}_{2N,n}=\Theta_{n}^{2N,\dag}(t)\psi_{2N,n}^{\dag}
\]
holds only for~$t\in[-T/2,T/2]$.

\medskip
\noi$\bullet${\bf Step 2 : Remainder. }

Finally, we define the remainder~$w_{2N}^{\dag}$ by solving the problem
\begin{align}\label{w2Ndag}
	w_{2N,n}^{\dag}(t)=&
	\mathcal{I}_{\chi_T,n}\big[
	\mathcal{N}^{(3)}_n(u^{\dag}_N+\psi_{2N}^{\dag}+w^{\dag}_{2N})-\mathcal{N}^{(3)}_{n}(u^{\dag}_N)
	\big]\notag \\+
	& \mathcal{I}_{\chi_T,n}(\Theta^{2N,\dag}_{n} w^{\dag}_{2N,n})
	+\mathcal{I}_{\chi_T,n}\Big(
	\sum_{j=2,4}\big(\mathcal{N}_{n}^{(j)}(u_{2N}^{\dag})-\mathcal{N}_{n}^{(j)}(u_N^{\dag})  \big)
	u^{\dag}_{2N,n} \Big)\notag \\ +&\mathcal{I}_{\chi_T,n}\big(\mathcal{R}_{n}(u_N^{\dag}+\psi_{2N}^{\dag}+w_{2N}^{\dag}) -\mathcal{R}_{n}(u_N^{\dag})  \big).
\end{align}
In summary, when 
both problems \eqref{H2Ndag} and \eqref{w2Ndag} admits unique solutions with the same $T>0$ as in the previous steps, 
we obtain the objects
\[
((\mathcal{H}^{2N,\dag}_n(t))_{N<n\leq 2N},\; (g^{\dag}_n(t))_{N<n\leq 2N} ,\; \psi^{\dag}_{2N}(t),\; w^{\dag}_{2N}(t))\,
\]
at the scale $2N$.

\subsection{Definition of the objects and the main estimates}

We now introduce the two types of inputs that appear in the multilinear estimates: colored Gaussian terms of type~$\mathrm{(C)}$ and deterministic remainders of type~$\mathrm{(D)}$.

\begin{definition}\label{def:CD}
	Assume that~$(s,q,\gamma,b,\delta)=(s_{\sigma},q_{\sigma},\gamma_{\sigma},b_{\sigma},\delta_{\sigma})$ satisfies the hierarchy in Definition~\ref{hierarchy}.
	We define type~$\mathrm{(C)}$ and type~$\mathrm{(D)}$ terms as follows:
	\begin{enumerate}
		\item Type~$\mathrm{(C)}$: an input function~$v$ is of type~$\mathrm{(C)}$ at scale~$N$ of size~$R$, if there is a~$\mathcal{B}_{\leq \frac{N}{2}}$-measurable function~$\mathcal{T}_n^N(t)\in\widetilde{\mathcal{F}}L_q^{\gamma}(\R)$, compactly supported in~$t\in(-2T,2T)$, such that
		\[
		v(t,x) = \sum_{\frac{N}{2}<n\leq N}
		\mathcal{T}^N_{n}(t)\frac{g_{n}(\omega)}{n^{\alpha}}\e^{-it\lambda_{n}^{2}}\cdot\mathbf{e}_{n}(x)\,,
		\]
		satisfying
		\begin{align}\label{bddef:TypeC}
		\sup_{\frac{N}{2}<n\leq N}\big(\|\langle\kappa\rangle^{\gamma}\widetilde{\mathcal{T}}^N_{n}(\kappa)\|_{L_{\kappa}^{q}} + \|\la\kappa\ra^b\widetilde{\mathcal{T}}^N_n(\kappa) \|_{L_{\kappa}^2}\big)
		\leq R  \,,
		\end{align}
		and
		\begin{align}\label{bddef:TypeC'}
		&\|v(t)\|_{X^{0,b}}\leq RT^{-(b-\frac{1}{2})}N^{-\alpha+\frac{1}{2}+\delta}, \notag \\ &\|v(t)\|_{X_{q,q}^{0,\gamma}}\leq RT^{-(\gamma-\frac{1}{q'})}N^{-\alpha+\frac{1}{q}+\delta}.
		\end{align}
		\item Type~$\mathrm{(D)}$: an input function~$v$ is of type~$\mathrm{(D)}$ at scale~$N$, if~$v$ is compactly supported in~$t\in(-T,T)$, with the estimates
		\[
		\|v\|_{X^{0,b}}
		\leq N^{-s},\quad \|(\mathrm{Id}-\Pi_L)v\|_{X^{0,b}}\leq  \big(\frac{N}{L}\big)N^{-s}
		\]
		for~$L\geq 4N$.
	\end{enumerate}
\end{definition}

\begin{remark}
For type (C) in Definition \ref{def:CD}, the associate operator $\mathcal{T}_n^N(t)$ is also a random variable. In the specific application,  \eqref{bddef:TypeC}, \eqref{bddef:TypeC'} only hold on some $\mathcal{B}_{\leq \frac{N}{2}}$-measurable set $\Xi_{\frac{N}{2}}$, we will explicitly mention this dependence. 
\end{remark}

\begin{definition}\label{R-certainly}
	Given~$R>1$ and~$0<\theta<1$, a statement~$\mathrm{(S)}$ is said to hold~$(R,\theta)$-certainly if there exists a measurable set~$\Sigma$ such that~$\mathrm{(S)}$ holds on~$\Sigma$ and~$\mathbb{P}(\Sigma)>1-\e^{-R^{\theta}}$. If~$\Sigma$ is~$\mathcal{G}$-measurable with respect to some sub-$\sigma$-field of~$\mathcal{B}_{\leq \infty}$, we say that~$\mathrm{(S)}$ holds~$(R,\theta;\mathcal{G})$-certainly.
\end{definition}

We now define the quantitative statement which needs to be propagated in the proof :
\begin{definition}[$\mathrm{Loc}(N)$]\label{Databound}
	Let~$R>1, T>0$, $N\in 2^{\N}$, and assume that parameters~$(s,q,\gamma,b,\delta)=(s_{\sigma},q_{\sigma},\gamma_{\sigma},b_{\sigma},\delta_{\sigma})$ satisfying Definition~\ref{hierarchy}. We say that $\mathrm{Loc}(N)$ on a data set $\Sigma$ holds, if for any initial data from $\Sigma$, the fixed-point problems  \eqref{H2Ndag} and \eqref{w2Ndag} admits solutions
	$$ \big( (\mathcal{H}_n^{N,\dag}(t))_{N/2<n\leq N},\; w_N^{\dag}(t) \big)
	$$
	with
	 the following estimates :
	\begin{enumerate}
		\item For all~$\frac{N}{2}<n\leq N$,
		\begin{align}\label{R-bound3}
		&	\sup_{\frac{N}{2}<n\leq N}\|\mathcal{H}^{N,\dag}_n(t)-\chi(t)\e^{-it\lambda_n^2}\|_{\widetilde{\mathcal{F}}L_q^{\gamma}
				\cap\widetilde{\mathcal{F}}L_2^b
			}\leq R^{-1},\\
&			\sup_{\frac{N}{2}<n\leq N}\|g_{n}^{N,\dag}\|_{\widetilde{\mathcal{F}}L_q^{\gamma}
				\cap\widetilde{\mathcal{F}}L_2^b
			}\leq RN^{\delta},
		\end{align}
		\item For any~$n_0\in\N$,
		\begin{equation}\label{R-bound2}
			\Big|\sum_{\frac{N}{2}<m\leq N}\gamma_{n_0n_0mm}\frac{|g_m^{N,\dag}(0)|^2-1}{m^{2\alpha}}\Big|\leq RN^{-2\alpha+\frac{1}{2}+\delta}\min(N,n_0),
		\end{equation}
		\item For any dyadic number~$L\geq 4N$,
		\begin{equation}\label{R-bound1}
			\|w_N^{\dag}\|_{X^{0,b}}\leq N^{-s}, \quad \|(\mathrm{Id}-\Pi_L)w_N^{\dag}\|_{X^{0,b}}\leq \big(\frac{N}{L}\big)N^{-s}.
		\end{equation}
	\end{enumerate}
\end{definition}

Consequently, with the property Loc$(N)$, we have from Proposition~A.1 of~\cite{BCST24} the following estimates :
\begin{align}\label{R-bound'}
	&\|\psi_N^{\dag}\|_{X_{q,q}^{0,\gamma}}\leq RT^{-(\gamma-\frac{1}{q'})}N^{-\alpha+\frac{1}{q}+\delta},\\  &\|\psi_N^{\dag}\|_{X^{0,b}}\leq RT^{-(b-\frac{1}{2})}N^{-\alpha+\frac{1}{2}+\delta}.
\end{align}

Following~\cite{BCST24} and \cite{DNY24}, the key induction step can be stated as follows.
\begin{proposition}[Key induction]\label{prop:key}
	Assume that~$(s,q,\gamma,b,\delta)=(s_{\sigma},q_{\sigma},\gamma_{\sigma},b_{\sigma},\delta_{\sigma})$ satisfies the hierarchy in Definition~\ref{hierarchy}. Let $R\geq 1$.
	Then there exist a numerical constant~$c_0>0$ and $T=T_R>0$, such that for sufficiently small~$\sigma>0$, the following statement holds : If $\mathrm{Loc}(M)$ holds for all $M\leq N$,  then~$(N^{\delta}R,c_0;\mathcal{B}_{\leq 2N})$-certainly, $\mathrm{Loc}(2N)$ holds true. In particular, a part from a set of probability smaller than $C_{\delta}\e^{-c_0R^{\delta}}$, $\mathrm{Loc}(N)$ holds true for all $N\in 2^{\N}$.
\end{proposition}
	
In order to prove Proposition~\ref{prop:key}, we decompose the nonlinearity appearing on the right-hand side of~\eqref{eq:w} dyadically.
	\medskip

\paragraph{\bf Cubic terms.}
We decompose the trilinear non-resonant interactions as
\[
\mathcal{N}^{(3)}(\cdot)
=\mathcal{N}_{[123]}^{(3)}(\cdot)
+\mathcal{N}_{(12)}^{(3)}(\cdot)
+\mathcal{N}_{(23)}^{(3)}(\cdot),
\]
with:

\emph{(i) Fully non-resonant interactions:}
\begin{equation}
	\label{def:N^3(123)}
	(\mathcal{N}_{[123]}^{(3)})_n(f_1,f_2,f_3)
	:=\sum_{\substack{n_1,n_2,n_3\\
			\{n,n_2\}\cap\{n_1,n_3\}=\emptyset}}
	\gamma_{nn_1n_2n_3}\,
	f_{1,n_1}\,\overline{f_{2,n_2}}\,f_{3,n_3}.
\end{equation}

\emph{(ii) Partially resonant interactions:}
\begin{align}
	\label{def:N^3(12)}
	(\mathcal{N}_{(12)}^{(3)})_n(f_1,f_2,f_3)
	&:=\sum_{\substack{n_1,n_2,n_3\\
			n_1=n_2,\, n_3\neq n}}
	\gamma_{nn_1n_2n_3}\,
	f_{1,n_1}\,\overline{f_{2,n_2}}\,f_{3,n_3}\,,
	\\
	\label{def:N^3(23)}
	(\mathcal{N}_{(23)}^{(3)})_n(f_1,f_2,f_3)
	&:=\sum_{\substack{n_1,n_2,n_3\\
			n_1\neq n,\, n_2=n_3}}
	\gamma_{nn_1n_2n_3}\,
	f_{1,n_1}\,\overline{f_{2,n_2}}\,f_{3,n_3}.
\end{align}
The trilinear estimates for~$\mathcal{N}_{(12)}^{(3)}$ and~$\mathcal{N}_{(23)}^{(3)}$ require a treatment that differs substantially from the one developed in~\cite{BCST24}, due to the absence of a cubic Wick renormalization in the present setting, an effect of the non-translation-invariant geometry.
By symmetry of the indices, it suffices to focus on the analysis of~$\mathcal{N}_{(23)}^{(3)}$.

\paragraph{\bf Quadratic and quartic terms.}
For functions~$u=\sum_{n}u_n\mathbf{e}_n$ and~$v=\sum_{n}v_n\mathbf{e}_n$, we introduce the diagonal product
\begin{equation}
	\label{def:dotcirc}
	(u\dotcirc v)(x):=\sum_{n}u_nv_n\,\mathbf{e}_n(x).
\end{equation}
This operation naturally arises in the terms~$\mathcal{N}^{(2)}$ and~$\mathcal{N}^{(4)}$.

We now summarize the estimates leading to the proof of Proposition~\ref{prop:key}.
Throughout, we assume that the Loc$(M)$ holds for all $M\leq N$, and that~$N_1,N_2,N_3$ are dyadic integers with non-increasing rearrangement
\[
N_{(1)}\geq N_{(2)}\geq N_{(3)}.
\]
The proof of Proposition \ref{prop:key} relies crucially on the following multi-linear estimates, which we will state for general type (C) and type (D) terms.
More precisely, write~$v_M^{(\mathrm{C})}$ or~$v_M^{(\mathrm{D})}$ when~$v$ is of type~(C) or~(D) at frequency scale~$M$ and of size~$R$.

\begin{proposition}[Trilinear estimates]
	\label{prop:trilinear}
 Assume that $R\geq 1$ and $T>0$. Let $(v_M^{(\mathrm{C})}, v_M^{(\mathrm{D})})_{M\in 2^{\N}}$ are sequences of type $(\mathrm{C})$ and type $(\mathrm{D})$ terms according to Definition \ref{def:CD}.
 Then with an absolute constant $c_0>0$ and for all sufficiently small parameter $\sigma>0$ defining the parameters $(s,q,\gamma,b,\delta)$, the following estimates hold
	$((2N)^{\delta}R,c_0;\mathcal{B}_{\leq2N})$-certainly for any choice of types~$(\ast_1,\ast_2,\ast_3)\in\{\mathrm{C},\mathrm{D}\}^3$:

	\emph{(A) Fully non-resonant cubic interactions:} If $\max(N_1,N_2,N_3)=2N$, 
	\begin{equation}
		\label{eq:tri-cubic}
		\big\|\mathcal{I}_{\chi_T}\mathcal{N}_{[123]}^{(3)}
		(v_{N_1}^{(\ast_1)},v_{N_2}^{(\ast_2)},v_{N_3}^{(\ast_3)})\big\|_{X^{0,b}}
		\lesssim R^3T^{\sigma}N_{(1)}^{-s}N_{(2)}^{-\delta},
	\end{equation}

	\emph{(B) Resonant cubic interaction:} If $\max(N_1,N_2,N_3)=2N$,
	\begin{equation}
		\label{eq:tri-res}
		\big\|\mathcal{I}_{\chi_T}\mathcal{N}_{(23)}^{(3)}
		(v_{N_1}^{(\ast_{1})},v_{N_2}^{(\ast_{2})},v_{N_3}^{(\ast_{3})})\big\|_{X^{0,b}}
		\lesssim R^3T^{\sigma}N_{(1)}^{-s}N_{(2)}^{-\delta}.
	\end{equation}

	\emph{(C) Quadratic interaction:} For $\max(N_1,N_2,N_3)=2N$, the following estimate holds except when~$N_2,N_3\leq N_1/2$ and~$\ast_1=\mathrm{(C)}$:
	\begin{equation}
		\label{eq:circ}
		\big\|\mathcal{I}_{\chi_T}\big(
		v_{N_1}^{(\ast_{1})}\dotcirc\mathcal{N}^{(2)}(v_{N_2}^{(\ast_{2})},v_{N_3}^{(\ast_{3})})
		\big)\big\|_{X^{0,b}}
		\lesssim R^3T^{\sigma}N_{(1)}^{-s}N_{(2)}^{-\delta}.
	\end{equation}

	\emph{(D) Operator bounds.}
	If~$\max(N_2,N_3)=N$, the operator
	\[
	f \longmapsto \mathcal{I}_{\chi_T}\big(
	f\dotcirc\mathcal{N}^{(2)}(v_{N_2}^{(\ast_{2})},v_{N_3}^{(\ast_{3})})
	\big)
	\]
	satisfies
	\begin{align}
		\label{eq:op-X}
		\| \cdot \|_{X^{0,b}\to X^{0,b}}
		&\lesssim R^2T^{\sigma}\max(N_2,N_3)^{-\delta},\\
		\label{eq:op-FL}
		\| \cdot \|_{X_{q,q}^{0,\gamma}\to X_{q,q}^{0,\gamma}}
		&\lesssim R^2T^{\sigma}\max(N_2,N_3)^{-\delta}.
	\end{align}
\end{proposition}

We briefly indicate how each group of estimates is proved.

\emph{$(\mathrm{A})$ Fully non-resonant cubic interactions.}
We distinguish several cases according to the input types and the frequency configuration:
\begin{itemize}
	\item If at least two inputs are of type~(D), the bound~\eqref{eq:tri-cubic} follows from Proposition~\ref{linear:type1} and a careful dyadic summation procedure presented in Section~\ref{sec:mainproof}.

	\item The same deterministic argument applies, for arbitrary input types, in the high$\times$low regime~$N_{(2)}\lesssim N_{(1)}^{1/2}$.

	\item By contrast, in the regime~$N_{(2)}\gg N_{(1)}^{1/2}$ and when at least two inputs are of type~(C), the proof of~\eqref{eq:tri-cubic} requires probabilistic estimates, which are developed in Section~\ref{sec:probatool}.

	\item If at least two inputs are of type~(C) and~$N_{(2)}\ll N_{(1)}^{\epsilon}$, the random matrix estimate developed in Section~\ref{sec:probatool} is insufficient due to the loss of a~$\log(N_{(1)})$-factor. We remedy this by an almost-orthogonality argument (Proposition~\ref{prop:highhighverylow}).
\end{itemize}

\emph{$(\mathrm{B})$ Partially resonant cubic interactions.}
For the component~$\mathcal{N}_{(23)}^{(3)}$, the deterministic analysis applies for all choices of types.
The bound~\eqref{eq:tri-res} is obtained within the proof of Proposition~\ref{linear:type1}.

\emph{$(\mathrm{C})$ Quadratic interaction.}
The estimate~\eqref{eq:circ} for the quadratic contribution
is also proved in the course of the proof of Proposition~\ref{linear:type1}.

\emph{$(\mathrm{D})$ Operator bounds.}
The operator estimates~\eqref{eq:op-X} and~\eqref{eq:op-FL} are direct consequences of
Corollary~\ref{co:2-linear:1}.
\medskip

\paragraph{ \bf Quintic terms.}
Rewritten in terms of the diagonal product~\eqref{def:dotcirc}, the quintic contribution on the right-hand side of~\eqref{w2Ndag} takes the form
\[
u_{2N}^{\dag}\dotcirc \mathcal{N}^{(4)}(u_{2N}^{\dag}),
\]
where the quadrilinear form~$\mathcal{N}^{(4)}$ is defined in~\eqref{eq:N4}.
Throughout, we assume that Loc$(M)$ holds for all $M\leq N$, and that~$N_1,\dots,N_5$ are dyadic integers with non-increasing rearrangement
\[
 N_{(1)} \geq N_{(2)} \geq \cdots \geq N_{(5)}.
\]

\begin{proposition}[Quintilinear estimates]
	\label{prop:quintilinear}
 Assume that $R\geq 1$ and $T>0$. Let $(v_M^{(\mathrm{C})}, v_M^{(\mathrm{D})})_{M\in 2^{\N}}$ are sequences of type $(\mathrm{C})$ and type $(\mathrm{D})$ terms according to Definition \ref{def:CD}.
Then with an absolute constant $c_0>0$ and for all sufficiently small parameter $\sigma>0$ defining the parameters $(s,q,\gamma,b,\delta)$, the following estimates hold
$((2N)^{\delta}R,c_0;\mathcal{B}_{\leq2N})$-certainly for any choice of types~$(\ast_1,\ast_2,\ast_3)\in\{\mathrm{C},\mathrm{D}\}^3$:  
	\begin{align*}
		\big\|\mathcal{I}_{\chi_T}\big(
		\mathcal{N}^{(4)}(v_{N_1}^{(\ast_1)},v_{N_2}^{(\ast_2)},v_{N_3}^{(\ast_3)},v_{N_4}^{(\ast_4)})
		\dotcirc v_{N_5}^{(\ast_5)}
		\big)\big\|_{X^{0,b}}
		&\lesssim R^5T^{\sigma}N_{(1)}^{-s}N_{(2)}^{-\delta},
	\end{align*}
whenever 
\begin{align*}
&\max(N_1,N_2,N_3,N_4,N_5)=2N,\\
	&N_5\leq \max(N_1,N_2,N_3,N_4),\text{ or } \max(N_1,N_2,N_3,N_4)\leq N_5/2 \text{ and }v_{N_5}^{(\ast_5)}\text{ is of type~(D)}.
	\end{align*}
	Moreover, if~$\max(N_1,N_2,N_3,N_4)=N$, the following operator bounds hold:
	\begin{align*}
		\big\|\cdot\mapsto
		\mathcal{I}_{\chi_T}\big(
		\cdot\dotcirc\mathcal{N}^{(4)}(v_{N_1}^{(\ast_1)},v_{N_2}^{(\ast_2)},v_{N_3}^{(\ast_3)},v_{N_4}^{(\ast_4)})
		\big)\big\|_{X^{0,b}\to X^{0,b}}
		&\lesssim R^4T^{\sigma}\max(N_1,\dots,N_4)^{-\delta},\\
		\big\|\cdot\mapsto
		\mathcal{I}_{\chi_T}\big(
		\cdot\dotcirc\mathcal{N}^{(4)}(v_{N_1}^{(\ast_1)},v_{N_2}^{(\ast_2)},v_{N_3}^{(\ast_3)},v_{N_4}^{(\ast_4)})
		\big)\big\|_{X_{q,q}^{0,\gamma}\to X_{q,q}^{0,\gamma}}
		&\lesssim R^4T^{\sigma}\max(N_1,\dots,N_4)^{-\delta}.
	\end{align*}
\end{proposition}

The proof of Proposition~\ref{prop:quintilinear} is given in Section~\ref{sec:quintilinear}.
Using Lemma~\ref{lem:N4mapping}, it essentially reduces to estimating the~$L_t^{\infty}$-norm of the quadrilinear term~$\mathcal{N}_n^{(4)}(\cdots)$.

\emph{Deterministic regime.}
For interactions of type~(DDDD), (DDDC), or~(DDCC), the bounds in
Proposition~\ref{prop:quintilinear} follow from Corollary~\ref{cor:4-linear-1}
and rely solely on deterministic estimates.

\emph{Probabilistic regime.}
When at least three inputs are of type~(C), deterministic arguments are no longer sufficient.
In this case, a probabilistic gain is required and is provided by
Corollary~\ref{cor:4-linearprob}, which exploits the moment bound
from Lemma~\ref{momentbd'}.

\subsection{Proof of Proposition~\ref{prop:key} and Theorem~\ref{thm:quan} assuming the multilinear estimates}

Following~\cite{BCST24}, we sketch how the key induction step, Proposition~\ref{prop:key}, follows from
the dyadic decomposition of the right-hand side of~\eqref{w2Ndag} together with the
multilinear bounds in Proposition~\ref{prop:trilinear} (cubic estimates) and
Proposition~\ref{prop:quintilinear} (quintic estimates). We then explain briefly how
Theorem~\ref{thm:quan} follows by iterating Proposition~\ref{prop:key} over dyadic scales.

\begin{proof}[Proof of Proposition~\ref{prop:key}]
	Fix a dyadic number~$N\ge 2$ and assume that~$\mathrm{Loc}(M)$ hold for scales $M\leq N$. More precisely, our goal is to construct related objects at scale~$2N$ by solving~\eqref{H2Ndag} and~\eqref{w2Ndag} and then defining~$\psi_{2N}^{\dag}$ through~\eqref{def:psi2Ndag}. Note that since Loc$(M)$ holds for all $M\leq N$, there is a $\mathcal{B}_{\leq N}$-measurable set $\Xi_N$, such that for all $M\leq N$, $\psi_{M}^{\dag}$ are of type (C), according to Definition \ref{def:CD} on $\Xi_N$.

	\noi
	$\bullet${\bf Step~1: construction of~$\mathcal{H}_n^{2N,\dag}$}
	
	Rewrite~\eqref{H2Ndag} as the fixed-point problem
	\begin{align}\label{H2N-fp}
	\mathcal{H}_n^{2N,\dag}=\chi(t)\e^{-it\lambda_n^2}\;+\;\mathcal{I}_{\chi_{2T},n}\big(\Theta_n^{2N,\dag}\,\mathcal{H}_n^{2N,\dag}\big),
	\qquad N<n\le 2N,
	\end{align}
	where~$\Theta_n^{2N,\dag}$ is given by~\eqref{Theta2Ndag}. Using the decomposition
	\[
	u_N^{\dag}=\sum_{M\le N}(\psi_M^{\dag}+w_M^{\dag})
	\]
	and the dyadic splitting of the nonlinearities (as encoded in
	Proposition~\ref{prop:trilinear} and Proposition~\ref{prop:quintilinear}),
	the coefficient~$\Theta^{2N,\dag}$ can be written as a finite sum of dyadic pieces of the form~$\mathcal{N}^{(2)}(\cdot,\cdot)$ and~$\mathcal{N}^{(4)}(\cdot,\cdot,\cdot,\cdot)$ with inputs
	of type~(C) or~(D) at scales~$\le N$.
	The operator bounds of Proposition~\ref{prop:trilinear} and
	Proposition~\ref{prop:quintilinear} then yield, $((2N)^{\delta}R,c_0;\mathcal{B}_{\le 2N})$-certainly,
	the uniform estimate
	\[
	\sup_{N<n\le 2N}
	\big\|\ \cdot\ \mapsto \mathcal{I}_{\chi_{2T},n}\big(\Theta_n^{2N,\dag}\,\cdot\big)\big\|_
	{\widetilde{\mathcal{F}}L_q^{\gamma}\to\widetilde{\mathcal{F}}L_q^{\gamma}}
	\ \lesssim\ R^5T^{\sigma}.
	\]

	With the choice~$T=R^{-C_1}$ (as in Theorem~\ref{thm:quan}) and~$C_1$ sufficiently large,
	the right-hand side of \eqref{H2N-fp} can be made less than~$1/2$. Hence~\eqref{H2Ndag} has a unique solution~$\mathcal{H}_n^{2N,\dag}\in \widetilde{\mathcal{F}}L_q^{\gamma}$, and moreover
	\[
	\sup_{N<n\le 2N}\big\|\mathcal{H}_n^{2N,\dag}-\chi(t)\e^{-it\lambda_n^2}\big\|_{\widetilde{\mathcal{F}}L_q^{\gamma}\cap \widetilde{\mathcal{F}}L_{2}^{b}}
	\ \lesssim\ R^5T^{\sigma}
	\ \le\ R^{-1},
	\]
	which is exactly the first bound in~\eqref{R-bound3} at scale~$2N$.

	Next, by~\eqref{def:gndag} and the bound~$\|\mathcal{H}_n^{2N,\dag}\|_{\widetilde{\mathcal{F}}L_q^{\gamma}}\lesssim 1$,
	we have~$\|g_n^{2N,\dag}\|_{\widetilde{\mathcal{F}}L_q^{\gamma}}\lesssim |g_n(\omega)|$. By a standard Gaussian
	large-deviation estimate and a union bound over~$N<n\le 2N$, we obtain $((2N)^{\delta}R,c_0)$-certainly
	\[
	\sup_{N<n\le 2N}|g_n(\omega)|\ \le\ R(2N)^{\delta}\,.
	\]
	This gives the second estimate in~\eqref{R-bound3}
	at scale~$2N$. Consequently, $\psi_{2N}^{\dag}$ is of type~(C) at scale~$2N$ of size~$R$.
\medskip

	\noi
	$\bullet${\bf Step~2: bound~\eqref{R-bound2} at scale~$2N$}

	Since~$\mathcal{H}_m^{2N,\dag}(0)=\chi(0)=1$, we have~$g_m^{2N,\dag}(0)=g_m(\omega)$.
	Hence the quantity in~\eqref{R-bound2} at scale~$2N$ is a sum of independent centered random variables
	depending only on~$(g_m)_{N<m\le 2N}$, and is independent of~$\mathcal{B}_{\le N}$.
	Hypercontractivity together with the bound~$\gamma_{n_0n_0mm}\lesssim \min(n_0,m)$ (from~\eqref{eq:gamma'}) yields~\eqref{R-bound2} at scale~$2N$, $((2N)^{\delta}R,c_0)$-certainly.

	\emph{Step~3: construction of~$w_{2N}^{\dag}$ and bounds.}
	We solve~\eqref{w2Ndag} by a contraction mapping argument in the ball
	\[
	\mathfrak{B}_{2N}:=\{w\in X^{0,b}:\ \|w\|_{X^{0,b}}\le (2N)^{-s}\}.
	\]
	Using the dyadic decomposition of the cubic term~$\mathcal{N}^{(3)}(u_N^{\dag}+\psi_{2N}^{\dag}+w)-\mathcal{N}^{(3)}(u_N^{\dag})$
	and applying Proposition~\ref{prop:trilinear} to each dyadic piece,
	we obtain a bound of the form
	\[
	\big\|\mathcal{I}_{\chi_T}\big(\mathcal{N}^{(3)}(\cdots)-\mathcal{N}^{(3)}(u_N^{\dag})\big)\big\|_{X^{0,b}}
	\ \lesssim\ R^3T^{\sigma}(2N)^{-s}
	\]
	whenever~$w\in\mathfrak{B}_{2N}$. Similarly, the quintic contribution in~\eqref{w2Ndag}
	is controlled by Proposition~\ref{prop:quintilinear}, and the linear term~$\mathcal{I}_{\chi_T,n}(\Theta_n^{2N,\dag}w_{2N,n}^{\dag})$ is controlled by the operator bounds
	in Proposition~\ref{prop:trilinear} and Proposition~\ref{prop:quintilinear}.
	The remainder term involving~$\mathcal{R}_{2N,n}$ is easier since it is diagonal in~$n$
	(there is no frequency summation), and it is handled by the same~$X^{0,b}$-based contraction argument.
	Choosing again~$T=R^{-C_1}$ with~$C_1$ large, these bounds make the right-hand side of~\eqref{w2Ndag}
	a strict contraction on~$\mathfrak{B}_{2N}$, so a unique solution~$w_{2N}^{\dag}\in\mathfrak{B}_{2N}$
	exists and satisfies
	\[
	\|w_{2N}^{\dag}\|_{X^{0,b}}\le (2N)^{-s}\,.
	\]
	This is the first inequality in~\eqref{R-bound1} at scale~$2N$.

	The second inequality in~\eqref{R-bound1} follows from the same argument. We briefly explain the factor~$(N/L)$ when estimating~$\|(\mathrm{Id}-\Pi_{L})w_{2N}^{\dag}\|_{X^{0,b}}$ for~$L\gg 2N$, which differs from the case of~$\mathbb{S}^2$ in~\cite{BCST24} (where the correlation coefficient vanishes in this regime, yielding a larger decay). We observe that~$\psi_{N_j}^{\dag}$ is exactly supported in frequencies~$n_j\in(N_j/2,N_j]$. If~$L\gg 2N$ and all inputs are of the form~$\psi_{N_j}^{\dag}$, we use the decay rate~$L^{-2}$ of the correlation coefficients to obtain the corrected bound~$L^{-2}N_{(1)}^{-s}N_{(2)}^{-\delta}$ in the proof of Proposition~\ref{prop:trilinear} and Proposition~\ref{prop:quintilinear}. Otherwise, at least one input is a type~(D) term~$w_{N_j}^{\dag}$ with~$N_j\leq 2N$ and frequency support on~$n_j\gtrsim L$. This generates a correction~$(N_j/L)\leq (2N/L)$ in the output of Proposition~\ref{prop:trilinear} and Proposition~\ref{prop:quintilinear}, which allows us to conclude.
\medskip

	Collecting the previous two steps, we obtain $\mathrm{Loc}(2N)$, $((2N)^{\delta}R,c_0)$-certainly, and the construction is~$\mathcal{B}_{\le 2N}$-measurable. This completes the proof of Proposition~\ref{prop:key}.
\end{proof}

\paragraph{\bf Proof of Theorem~\ref{thm:quan}.}
Fix~$R\ge 1$ and set~$\tau_R=R^{-C_1}$ as in Theorem~\ref{thm:quan}. We only prove the convergence for dyadic sequences, as
it is now standard (see~\cite{DNY24,BCST24}) to pass from the dyadic convergent sequence to the convergence of the full sequence.
Starting from the initial scale~$N_0=2$,
standard Gaussian large-deviation bounds lead to the property $\mathrm{Loc}(2)$ outside a set of probability~$\lesssim \e^{-cR^{c_0}}$. Iterating Proposition~\ref{prop:key} over dyadic scales~$N=2^k$ and taking the intersection
of the resulting good events, we obtain a set~$\Omega_R$ with
\[
\mathbb{P}(\Omega_R^c)\ \lesssim\ \sum_{k\ge 1}\e^{-( (2^{k})^{\delta}R)^{c_0}}\ \lesssim\ \e^{-cR^{c_0}}.
\]
Transferring this estimate to the induced measure~$\mu_{\alpha}$ gives the set~$\Sigma_R$ in Theorem~\ref{thm:quan}.
On~$\Omega_R$, the bounds~\eqref{R-bound1} imply in particular
\[
\sum_{k\ge 1}\|w_{2^k}^{\dag}\|_{X^{0,b}}\ \lesssim\ \sum_{k\ge 1}2^{-ks}<\infty.
\]
By frequency localization, choosing any~$s_0\in(\frac12,s)$, we also have~$\sum_{k\ge 1}\|w_{2^k}^{\dag}\|_{C_tH_x^{s_0}}<\infty$, hence~$w_{\infty}:=\sum_{k\ge 1}w_{2^k}^{\dag}\in C([-\tau_R,\tau_R];H_{\mathrm{rad}}^{s_0}(\mathbf{B}))$.
Similarly, the bounds for~$\psi_{2^k}^{\dag}$ yield that~$\psi_{\infty}:=\sum_{k\ge 1}\psi_{2^k}^{\dag}$ converges in~$C([-\tau_R,\tau_R];H_{\mathrm{rad}}^{\alpha-\frac12-}(\mathbf{B}))$.
Setting~$u:=\psi_{\infty}+w_{\infty}$ gives the claimed decomposition in Theorem~\ref{thm:quan}.
Since each~$u_N$ solves~\eqref{eq:nlsN} and all nonlinear contributions are controlled uniformly along the
iteration by Proposition~\ref{prop:trilinear} and Proposition~\ref{prop:quintilinear},
one passes to the limit in the Duhamel formula to obtain that~$u$ solves the cubic NLS
in the distributional sense on~$[-\tau_R,\tau_R]$.
This completes the proof of Theorem \ref{thm:quan}.

	
	\section{Counting and Strichartz-type estimates}\label{sec:detertool}

For expository reasons, we only present statements and proofs for~$\alpha<1$. The modification for the case~$\alpha=1$ is straightforward: every factor~$N^{1-\alpha}$ is replaced by a factor~$\log N$.

For~$\kappa\in\R$ and fixed constants~$C_0,C_1>1$ and dyadic integers~$N_1,N_2$, define:
\begin{align}\label{def:E+-}
	E_{+}(\kappa;N_1,N_2)&:=\Big\{(n_1,n_2):\; \tfrac{1}{C_1}N_i<n_i<C_1 N_i,\;\notag\\
	&\qquad\qquad\qquad |\lambda_{n_1}^2+ \lambda_{n_2}^2-\kappa|\leq C_0\Big\},\notag \\
	E_-(\kappa;N_1,N_2)&:=\Big\{(n_1,n_2):\; n_1\neq n_2,\; \tfrac{1}{C_1}N_i<n_i< C_1N_i,\;\notag\\
	&\qquad\qquad\qquad |\lambda_{n_1}^2-\lambda_{n_2}^2-\kappa|\leq C_0 \Big\}.
\end{align}

\subsection{Counting estimates}
In this section we prove several counting estimates for multi-indices restricted to constant-range neighborhood of level sets of the resonance function, from which we deduce new Strichartz-type estimates.

\begin{lemma}[Divisor bound with perturbed frequencies]
	\label{lem:db}
	Let~$N_{1}, N_{2}$ be dyadic integers with~$N_{2} \le N_{1}$, and let~$C_1, C_0$ be numerical constants defining the sets~$E_{\pm}(\kappa;N_1,N_2)$. Then for every~$\varepsilon>0$,
	\[
	\sup_{\kappa\in\R}\# E_{\pm}(\kappa;N_1,N_2)
	\;\leq\; C_{\varepsilon}\min(N_2, N_1^{2(1-\alpha)}N_2^{\varepsilon}),
	\]
	where~$C_{\varepsilon}$ depends only on~$C_0$, $C_1$, and~$\varepsilon$.
	Moreover, there exists~$C_2:=C_2(C_1,C_0)\geq 1$ such that if~$C_2\cdot N_2 < N_1^{1/2}$, then
	\[
	\sup_{\kappa\in\R} \# E_{\pm}(\kappa;N_1,N_2)\leq C,
	\]
	where the constant~$C$ is independent of~$N_1$ and~$N_2$.
\end{lemma}

\begin{proof}
	Given~$\kappa\in \R$, we note that~$(n_1,n_2)\in E_{\pm}(\kappa;N_1,N_2)$ forces
	\[
	|n_{1}^{2}\pm n_{2}^{2}-\kappa|\;\leq CN_{1}^{2(1-\alpha)}\,.
	\]
	For each fixed~$n_{2}\sim N_{2}$, and since~$2(1-\alpha)<1$, at most one~$n_{1}$ can satisfy the above relation when~$N_{1}$ is sufficiently large. Thus the number of pairs~$(n_{1},n_{2})$ is controlled by~$O(N_{2})$. Alternatively, for a given~$\kappa$, the quantity to estimate can be bounded from above by
	\[
	\#\big\{
	(n_{1},n_{2})\in \N^{2}:\ C_1^{-1}N_i\leq n_{i}\leq C_1N_{i}\,,\quad |n_{1}^{2}\pm n_{2}^{2} - \kappa|\leq CN_{1}^{2(1-\alpha)}
	\big\}
	\]
	so that the bound~$N_{1}^{2(1-\alpha)}N_{2}^{\varepsilon}$ follows from the standard divisor bound in~\cite[Lemma~3.2]{BGT05}.

	To prove the improved bound in the regime~$C_2N_2 < N_1^{1/2}$, it suffices to consider sufficiently large~$N_1$. We claim that if~$(n_1,n_2),(n_1^*, n_2^*)\in E_{\pm}(\kappa;N_1,N_2)$, then~$n_1=n_1^*$. Indeed, for any such tuples,
	\[
	|(\lambda_{n_1}^2-\lambda_{n_1^*}^2)\pm (\lambda_{n_2}^2-\lambda_{n_2^*}^2)|\leq 2C_0,
	\]
	which implies
	\[
	|\lambda_{n_1}^2-\lambda_{n_1^*}^2|\leq 2C_0+2C_1^2N_2^2+CN_2^{2(1-\alpha)}\leq 2C_0+(2C_1^2+C)N_2^2.
	\]
	If~$n_1\neq n_1^*$, then for sufficiently large~$N_{1}$,
	\[
	|\lambda_{n_1}^2-\lambda_{n_1^*}^2|\geq|(n_1-n_1^*)(n_1+n_1^*)|-CN_1^{2(1-\alpha)}\geq 2C_1^{-1}N_1-CN_1^{2(1-\alpha)}>\frac{N_1}{2C_1}.
	\]
	This implies
	\[
	N_1<4C_0C_1+(4C_1^2+2C)C_1N_2^2\leq (4C_0+4C_1^2+2C)C_1N_2^2,
	\]
	which is impossible if~$C_2N_2<N_1^{1/2}$ with~$C_2^2>(4C_0+4C_1^2+2C)C_1$. Therefore, there exists a unique~$n_1^0$ such that~$(n_1,n_2)\in E_{\pm}(\kappa;N_1,N_2)$ implies~$n_1=n_1^0$. It then suffices to count~$n_2\sim N_2$ for the fixed~$n_1=n_1^0$. With~$n_1=n_1^0$ fixed, the constraint~$|\lambda_{n_1^0}^2\pm\lambda_{n_2}^2-\kappa|\le C_0$ restricts~$n_2$ to an interval of length~$O(1)$, so there are~$O(1)$ admissible choices of~$n_2$.
\end{proof}

We also need an improved counting estimate. For~$\kappa\in \mathbb{R}$ and~$\vec{N}=(N_1,N_2,N_3,N_4)$, let~$N_{(1)}\geq N_{(2)}\geq N_{(3)}\geq N_{(4)}$ be the non-increasing rearrangement of~$N_1,N_2,N_3,N_4$. For~$\vec{n}=(n_1,n_2,n_3,n_4)$, define the resonance function
\begin{equation}\label{def:Omega}
	\Omega(\vec{n}):=\lambda_{n_1}^2-\lambda_{n_2}^2+\lambda_{n_3}^2-\lambda_{n_4}^2.
\end{equation}
We set
\[
\Gamma_{\vec{N},\kappa}:=\{(n_1,n_2,n_3,n_4):\; n_j\sim N_j,\quad \{n_{1},n_{3}\}\cap \{n_{2},n_{4}\}=\varnothing\,,\quad \Omega(\vec{n})=\kappa+O(1)\}.
\]
For~$j\in\{1,2,3,4\}$ we denote by~$\Gamma_{\vec{N},\kappa}(n_j)$ the section of~$\Gamma_{\vec{N},\kappa}$ with~$n_j$ fixed, and similarly~$\Gamma_{\vec{N},\kappa}(n_i,n_j)$ for the section with~$n_i$ and~$n_j$ fixed.

\begin{lemma}[Improved counting estimates]\label{improvedcounting}
	Let~$\frac{1}{2}< \alpha< 1$. For any~$\varepsilon>0$, there exists~$C_{\varepsilon}>0$ such that
	for any~$\vec{N}=(N_1,N_2,N_3,N_4)$, $j\in\{1,2,3,4\}$, $n_j\sim N_j$, and~$\kappa\in\R$,
	\begin{align}\label{improvedcounting1}
		\#\Gamma_{\vec{N},\kappa}(n_j)\leq C_{\varepsilon}
		\min\Big(M_2^{\varepsilon}\Big(1+\frac{M_{2}^{2}}{M_{1}}\Big),\;
		M_2^{\varepsilon}M_{1}^{2(1-\alpha)}M_{3} \Big),
	\end{align}
	where~$M_1\geq M_2\geq M_3$ is the non-increasing rearrangement of~$\{N_i:\ i\in\{1,2,3,4\}\setminus\{j\}\}$.
\end{lemma}

\begin{proof}
	Without loss of generality, we estimate~$\#\Gamma_{\vec{N},\kappa}(n_4)$ and assume that~$N_1\geq N_2\geq N_3$. For fixed~$n_4$, the constraint~$\Omega(\vec{n})=\kappa+O(1)$ gives
	\begin{equation}\label{eq:Omega-fixed-n4}
		\lambda_{n_1}^2-\lambda_{n_2}^2+\lambda_{n_3}^2=\kappa+\lambda_{n_4}^2+O(1).
	\end{equation}
	To obtain the bound~$M_2^{\varepsilon}M_1^{2(1-\alpha)}M_3$, fix~$n_3\sim N_3$. Then~\eqref{eq:Omega-fixed-n4} becomes
	\[
	\lambda_{n_1}^2-\lambda_{n_2}^2=\kappa+\lambda_{n_4}^2-\lambda_{n_3}^2+O(1).
	\]
	By Lemma~\ref{lem:db}, for each fixed~$n_3$ and~$n_4$ we have~$\#\Gamma_{\vec{N},\kappa}(n_3,n_4)\le C_\varepsilon N_2^\varepsilon N_1^{2(1-\alpha)}$.
	Summing over~$n_3\sim N_3$ yields
	\[
	\#\Gamma_{\vec{N},\kappa}(n_4)\le C_\varepsilon\,N_2^\varepsilon\,N_1^{2(1-\alpha)}\,N_3.
	\]

	It remains to prove the other bound. Let
	\[
	\mathrm{Proj}_1(\Gamma_{\vec{N},\kappa}(n_4))
	:=\left\{n_1\sim N_1:\ \exists\,n_2\sim N_2,\ n_3\sim N_3\ \text{with } (n_1,n_2,n_3,n_4)\in\Gamma_{\vec{N},\kappa}
	\right\}\,.
	\]
	We claim that there exists an interval~$J\subset\R$ with~$|J|\lesssim 1+\frac{N_2^2}{N_1}$ such that
	\[
	\mathrm{Proj}_1(\Gamma_{\vec{N},\kappa}(n_4))\subset J\,.
	\]
	Indeed, for any~$(n_1,n_2,n_3,n_4)\in\Gamma_{\vec{N},\kappa}(n_4)$, rearranging~\eqref{eq:Omega-fixed-n4} gives
	\[
	\lambda_{n_1}^2=\kappa+\lambda_{n_4}^2+\lambda_{n_2}^2-\lambda_{n_3}^2+O(1)
	=\kappa+\lambda_{n_4}^2+O(N_2^2),
	\]
	since~$n_2\sim N_2$ and~$n_3\sim N_3\le N_2$. Hence for two admissible~$n_1,n_1^\ast\sim N_1$,
	\[
	|\lambda_{n_1}^2-\lambda_{n_1^\ast}^2|=O(N_2^2).
	\]
	Using~$\lambda_n^2=n^2+\mu_n^2$ and~$|\mu_n^2|\lesssim n^{2(1-\alpha)}$, we obtain
	\[
	|(n_1-n_1^\ast)(n_1+n_1^\ast)|
	=|n_1^2-n_1^{\ast\,2}|
	\le |\lambda_{n_1}^2-\lambda_{n_1^\ast}^2|+|\mu_{n_1}^2-\mu_{n_1^\ast}^2|
	=O\big(N_2^2+N_1^{2(1-\alpha)}\big).
	\]
	Since~$n_1+n_1^\ast\sim N_1$ and~$\alpha>\frac12$, this yields
	\[
	|n_1-n_1^\ast|\ \lesssim\ \frac{N_2^2}{N_1}+N_1^{1-2\alpha}
	\ \lesssim\ 1+\frac{N_2^2}{N_1},
	\]
	which proves the claim.

	For fixed~$n_1\in J\cap\N$ and~$n_4$, applying Lemma~\ref{lem:db} to~\eqref{eq:Omega-fixed-n4} (with~$n_1$ and~$n_4$ fixed)
	gives~$\#\Gamma_{\vec{N},\kappa}(n_1,n_4)\le C_\varepsilon N_2^\varepsilon$.
	Therefore,
	\[
	\#\Gamma_{\vec{N},\kappa}(n_4)
	\le \sum_{n_1\in J\cap\N}\#\Gamma_{\vec{N},\kappa}(n_1,n_4)
	\le C_\varepsilon\,N_2^\varepsilon\,|J|
	\ \lesssim\ C_\varepsilon\,N_2^\varepsilon\Big(1+\frac{N_2^2}{N_1}\Big).
	\]
	This proves~\eqref{improvedcounting1}.
\end{proof}

\subsection{Strichartz-type estimates}

Equipped with the improved counting estimates, we obtain the following Strichartz-type bounds.

\begin{lemma}\label{timeStrichart}
	Let~$\eta\in\mathcal{S}(\R;\R_{+})$. For~$j=1,2,3,4$, let~$\mathbf{a}^{(j)}\in\ell^{2}(\N)$ be
	supported in~$\{k_j\sim K_j\}$. Denote by~$K_{(1)}\ge K_{(2)}\ge K_{(3)}\ge K_{(4)}$ the non-increasing rearrangement of~$K_1,K_2,K_3,K_4$. Let
	\[
	\mathcal{M}_{\vec{K}}((a^{(j)})_{1\leq j\leq 4})
	:=
	\sup_{\kappa\in\R}
	\sum_{\substack{k_{1},k_{2},k_{3},k_{4}\\
			\{k_{1},k_{3}\}\cap\{k_{2},k_{4}\}=\emptyset}}
	\eta(\Omega(\vec k)-\kappa)\prod_{j=1}^{4}|\mathbf{a}^{(j)}_{k_{j}}|\,.
	\]
	Then, for every~$\varepsilon>0$, there exists~$C_{\varepsilon,\eta}>0$ such that
	\begin{equation}
		\label{Strichartztype1}
		\mathcal{M}_{\vec{K}}((a^{(j)})_{1\leq j\leq 4})
		\le
		C_{\varepsilon,\eta}\,
		K_{(1)}^{2(1-\alpha)}\,
		K_{(2)}^{\varepsilon}\,
		\min
		\Big(
		\prod_{j=1}^{4}\|\mathbf{a}^{(j)}\|_{\ell^{2}}\,,\;
		K_{(3)}K_{(4)}
		\prod_{j=1}^{4}
		\|a^{(j)}\|_{\ell^{\infty}}
		\Big)\,.
	\end{equation}
	If in addition\footnote{More precisely, $C_2K_{(3)}< K_{(2)}^{1/2}$ with the numerical constant~$C_2$ from Lemma~\ref{improvedcounting}.} $K_{(1)}\sim K_{(2)}\gg K_{(3)}^2$, then
	\begin{equation}\label{Strichartztype1'}
		\mathcal{M}_{\vec{K}}((a^{(j)})_{1\leq j\leq 4})
		\le
		C_{\eta}\,
		\prod_{j=1}^{4}\|\mathbf{a}^{(j)}\|_{\ell^{2}}.
	\end{equation}
	Moreover, for any~$j_0\in\{1,2,3,4\}$ and any~$j_1\in\{1,2,3,4\}\setminus\{j_0\}$,
	let~$M_1\ge M_2\ge M_3$ be the non-increasing rearrangement of the scales~$\{K_j:\ j\in\{1,2,3,4\}\setminus\{j_1\}\}$. Then
	\begin{multline}
		\label{Strichartztype2}
		\mathcal{M}_{\vec{K}}((a^{(j)})_{1\leq j\leq 4})
		\le
		C_{\varepsilon,\eta}\,
		\min(K_{j_0},K_{j_1})^{\varepsilon}M_1^{1-\alpha} \,
		\min\Big(\Big(1+\frac{M_{2}^{2}}{M_{1}}\Big)^\frac12,\; M_{1}^{1-\alpha}M_{3}^\frac12\Big)
		\\
		\times\|\mathbf{a}^{(j_0)}\|_{\ell^{\infty}}\prod_{j\neq j_0}\|\mathbf{a}^{(j)}\|_{\ell^{2}}.
	\end{multline}
\end{lemma}

\begin{remark}
	In the deterministic analysis, the guiding principle for the choice between~\eqref{Strichartztype1} and~\eqref{Strichartztype2} is as follows: when dealing with a multilinear expression, we use~\eqref{Strichartztype2} if there is a type~(C) input on the relative top frequencies~$\geq N_{(1)}^{1/2}$; otherwise we apply~\eqref{Strichartztype1} or~\eqref{Strichartztype1'}.
\end{remark}

\begin{remark}
	If the type~(C) term appears on the top frequency, i.e.\ $j_0\in\{1,2,3,4\}$ satisfies~$K_{j_0}\in\{K_{(1)},K_{(2)}\}$, then~\eqref{Strichartztype2} gives
	\begin{multline}
		\label{Strichartztype2'}
		\mathcal{M}_{\vec{K}}((a^{(j)})_{1\leq j\leq 4})
		\leq C_{\varepsilon,\eta}\,
		K_{(2)}^{\varepsilon}K_{j_0}^{1-\alpha} \,
		\min\Big(\Big(1+\frac{K_{(3)}^{2}}{K_{j_0}}\Big)^\frac12,\; K_{j_0}^{1-\alpha}K_{(4)}^\frac12\Big)
		\\
		\times
		\|\mathbf{a}^{(j_0)}\|_{\ell^{\infty}}\prod_{j\neq j_0}\|\mathbf{a}^{(j)}\|_{\ell^{2}}.
	\end{multline}
	To this in the case when $K_{j_0}=K_{(2)}$, take $j_1$ such that $K_{j_1}=K_{(1)}$.
\end{remark}
\begin{proof}
	Since we assume~$\{n_{1},n_{3}\}\cap\{n_{2},n_{4}\}=\varnothing$, and Lemma~\ref{lem:db}
	provides bounds for both~$E_{+}$ and~$E_{-}$, we may assume without loss of generality that
	\begin{equation}
		\label{eq:orderK}
		K_1\ge K_2\ge K_3\ge K_4.
	\end{equation}
	We write~$\vec K=(K_1,K_2,K_3,K_4)$ and first prove~\eqref{Strichartztype1} and~\eqref{Strichartztype1'}.

	We decompose~$\R$ into unit intervals. For~$m,\ell\in\Z$, define the level sets
	\[
	E_m:=\{(k_1,k_4): k_1\sim K_1,\ k_4\sim K_4,\ k_1\neq k_4,\
	\lambda_{k_1}^2-\lambda_{k_4}^2\in[m,m+1]\},
	\]
	\[
	F_\ell:=\{(k_2,k_3): k_2\sim K_2,\ k_3\sim K_3,\ k_2\neq k_3,\
	\lambda_{k_2}^2-\lambda_{k_3}^2\in[\ell,\ell+1]\}.
	\]
	With this notation,
	\begin{multline}
		\label{eq:levelsetdec-new}
		\mathcal{M}_{\vec{K}}\big((a^{(j)})_{1\le j\le 4}\big)
		\le
		\sum_{m,\ell\in\Z}
		\Big(\sup_{|\nu|\le 1}\eta(m-\ell+\nu-\kappa)\Big)
		\\
		\times
		\Big(\sum_{(k_1,k_4)\in E_m}
		\mathbf{a}^{(1)}_{k_1}\mathbf{a}^{(4)}_{k_4}\Big)
		\Big(\sum_{(k_2,k_3)\in F_\ell}
		\mathbf{a}^{(2)}_{k_2}\mathbf{a}^{(3)}_{k_3}\Big).
	\end{multline}

	By Cauchy--Schwarz,
	\[
	\sum_{(k_1,k_4)\in E_m}
	\mathbf{a}^{(1)}_{k_1}\mathbf{a}^{(4)}_{k_4}
	\le
	\Big(
	\sum_{(k_1,k_4)\in E_m}
	|\mathbf{a}^{(1)}_{k_1}\mathbf{a}^{(4)}_{k_4}|^2
	\Big)^{\frac12}
	(\#E_m)^{\frac12}.
	\]
	Lemma~\ref{lem:db} gives uniformly in~$m$
	\[
	\#E_m \le C_\varepsilon K_1^{2(1-\alpha)}K_4^\varepsilon,
	\]
	and therefore
	\begin{equation}
		\label{eq:Em-bound}
		\sum_{(k_1,k_4)\in E_m}
		\mathbf{a}^{(1)}_{k_1}\mathbf{a}^{(4)}_{k_4}
		\le
		C_\varepsilon K_1^{1-\alpha}K_4^\varepsilon
		\Big(
		\sum_{(k_1,k_4)\in E_m}
		|\mathbf{a}^{(1)}_{k_1}\mathbf{a}^{(4)}_{k_4}|^2
		\Big)^{\frac12}.
	\end{equation}
	Similarly,
	\[
	\#F_\ell \le C_\varepsilon K_2^{2(1-\alpha)}K_3^\varepsilon,
	\]
	and hence
	\begin{equation}
		\label{eq:Fl-bound}
		\sum_{(k_2,k_3)\in F_\ell}
		\mathbf{a}^{(2)}_{k_2}\mathbf{a}^{(3)}_{k_3}
		\le
		C_\varepsilon K_2^{1-\alpha}K_3^\varepsilon
		\Big(
		\sum_{(k_2,k_3)\in F_\ell}
		|\mathbf{a}^{(2)}_{k_2}\mathbf{a}^{(3)}_{k_3}|^2
		\Big)^{\frac12}.
	\end{equation}
	Combining~\eqref{eq:Em-bound}--\eqref{eq:Fl-bound} with
	\[
	\sup_{\ell}\sum_{m\in\Z}
	\sup_{|\nu|\le 1}
	\eta(m-\ell+\nu-\kappa)
	\lesssim_\eta 1,
	\]
	and applying Schur's test in~\eqref{eq:levelsetdec-new}, we obtain
	\[
	\mathcal{M}_{\vec{K}}\big((a^{(j)})_{1\le j\le 4}\big)
	\lesssim_{\varepsilon,\eta}
	K_1^{2(1-\alpha)}K_3^\varepsilon
	\prod_{j=1}^4\|\mathbf{a}^{(j)}\|_{\ell^2},
	\]
	which proves the first bound in~\eqref{Strichartztype1}. If in addition~$K_{(1)}\sim K_{(2)}\gg K_{(3)}^2$, then~$K_1\gg K_4^2$ and~$K_2\gg K_3^2$, so~$\#E_m\lesssim 1$ and~$\# F_{\ell}\lesssim 1$ uniformly in~$m$ and~$\ell$, which gives~\eqref{Strichartztype1'}.

	\emph{Alternative bound in~\eqref{Strichartztype1}.}
	To obtain the second estimate in~\eqref{Strichartztype1}, we use the rapid decay of~$\eta$:
	\begin{align*}
		\mathcal{M}_{\vec{K}}\big((a^{(j)})_{1\le j\le 4}\big)
		&\lesssim
		\sup_{\kappa}
		\sum_{\substack{k_{1},k_{2},k_{3},k_{4}\\
				\{k_{1},k_{3}\}\cap\{k_{2},k_{4}\}=\varnothing}}
		\langle\Omega(\vec k)-\kappa\rangle^{-100}
		\prod_{j=1}^4 |a_{k_j}^{(j)}|
		\\
		&\lesssim
		\sup_{\ell,\kappa}
		\sum_{\substack{k_{1},k_{2},k_{3},k_{4}\\
				\{k_{1},k_{3}\}\cap\{k_{2},k_{4}\}=\varnothing}}
		\mathbf{1}_{|\Omega(\vec k)-\ell-\kappa|\le 1}
		\prod_{j=1}^4 |a_{k_j}^{(j)}|
		\\
		&\lesssim
		K_3K_4
		\Big(
		\sup_{\kappa'}\#E_{-}(\kappa',K_1,K_2)
		\Big)
		\prod_{j=1}^4\|a^{(j)}\|_{\ell^\infty}\,.
	\end{align*}
	Invoking Lemma~\ref{lem:db}, we obtain the second part of~\eqref{Strichartztype1}.

	We now prove~\eqref{Strichartztype2}.
	Fix~$j_0\in\{1,2,3,4\}$ and~$j_1\in\{1,2,3,4\}\setminus\{j_0\}$.
	Let~$\{\ell_{0},\ell_{1}\}:=\{1,2,3,4\}\setminus\{j_0,j_1\}$.
	Using~$\|\mathbf{a}^{(j_0)}\|_{\ell^\infty}$ to bound the~$k_{j_0}$-sum, we obtain
	\begin{equation}
		\label{eq:pf-Stri2-1-gen}
		\mathcal{M}_{\vec{K}}((a^{(j)})_{1\leq j\leq 4})
		\le
		\|\mathbf{a}^{(j_0)}\|_{\ell^\infty}\,
		\sup_{\kappa\in\R}
		\sum_{k_{j_1},k_{\ell_0},k_{\ell_1}}
		\sum_{k_{j_0}}
		\eta(\Omega(\vec k)-\kappa)\,
		|\mathbf{a}^{(j_1)}_{k_{j_1}}\mathbf{a}^{(\ell_{0})}_{k_{\ell_{0}}}\mathbf{a}^{(\ell_{1})}_{k_{\ell_{1}}}|,
	\end{equation}
	where we do not explicitly display the non-resonant condition~$\{k_1,k_3\}\cap\{k_2,k_4\}=\emptyset$ on the right-hand side.
	Since~$\eta\ge0$ is Schwartz, we discretize the modulation into unit intervals as before and apply Cauchy--Schwarz in~$(k_{j_0},k_{j_1},k_{\ell_{0}},k_{\ell_{1}})$.
	Taking the supremum in~$\kappa$, we obtain
	\begin{multline}\label{eq:pf-Stri2-2-gen}
		\mathcal{M}_{\vec{K}}((a^{(j)})_{1\leq j\leq 4})
		\lesssim_\eta
		\|\mathbf{a}^{(j_0)}\|_{\ell^\infty}\,
		\\
		\times\sup_{\kappa\in\R}
		\left(
		\sum_{k_{j_1}}|\mathbf{a}^{(j_1)}_{k_{j_1}}|^2\ \#\Gamma_{\vec K,\kappa}(k_{j_1})
		\right)^{\frac12}
		\left(
		\sum_{k_{\ell_{0}},k_{\ell_{1}}}|\mathbf{a}^{(\ell_{0})}_{k_{\ell_{0}}}\mathbf{a}^{(\ell_{1})}_{k_{\ell_{1}}}|^2\ \#\Gamma_{\vec K,\kappa}(k_{\ell_{0}},k_{\ell_{1}})\right)^{\frac12}.
	\end{multline}
	We now estimate the counting factors.
	Applying Lemma~\ref{improvedcounting} with the section index~$j=j_1$,
	writing~$M_1\ge M_2\ge M_3$ for the rearrangement of~$\{K_j:\ j\neq j_1\}$, we have
	\begin{equation}\label{eq:count-kjstar-sqrt}
		\sup_{\kappa\in\R}\sup_{k_{j_1}\sim K_{j_1}}
		\big(\#\Gamma_{\vec K,\kappa}(k_{j_1})\big)^{\frac12}
		\ \le\
		C_\varepsilon\,M_2^{\varepsilon}\,
		\min\Big(\Big(1+\frac{M_2^2}{M_1}\Big)^{\frac12},\ M_1^{1-\alpha}M_3^{\frac12}\Big).
	\end{equation}
	Moreover, by the divisor bound (Lemma~\ref{lem:db}),
	\begin{equation}\label{eq:count-kpq-sqrt-gen}
		\sup_{\kappa\in\R}\sup_{k_{\ell_0}\sim K_{\ell_0},\ k_{\ell_1}\sim K_{\ell_1}}
		\big(\#\Gamma_{\vec K,\kappa}(k_{\ell_0},k_{\ell_1})\big)^{\frac12}
		\ \le\
		C_\varepsilon\,M_1^{1-\alpha}\min(K_{j_0},K_{j_1})^{\varepsilon}.
	\end{equation}
	Inserting~\eqref{eq:count-kjstar-sqrt} and~\eqref{eq:count-kpq-sqrt-gen} into~\eqref{eq:pf-Stri2-2-gen}
	gives~\eqref{Strichartztype2}.
\end{proof}


	\section{Probabilistic bounds for fully non-resonant trilinear interactions}
\label{sec:probatool}

In this section we collect moment bounds for multilinear expansions involving terms of type~(C) defined in Definition~\ref{def:CD}, as well as random operator estimates in modulated Fourier--Lebesgue spaces.
The goal is to exploit randomness through Wiener chaos expansions and non-commutative Khintchine inequalities, reducing the analysis to deterministic counting estimates with a gain of a square root.

Thought this section, for a type (C) term at scale $N$ of size $R$, defined in Definition \ref{def:CD}, we denote
\begin{align}\label{def:gnN}
g_n^N(t):=\mathcal{T}_n^N(t)\cdot g_n(\omega).
\end{align}

\subsection{Moment bounds}

The following lemma bounds, conditionally on the low-frequency sigma-algebra, multilinear quantities involving type~(C) terms at a fixed frequency scale~$N$.
The key input is the probabilistic structure of these terms, and in particular the independence between the low- and high-frequency Gaussian variables.

\begin{lemma}\label{k-linearproba}
	Fix~$N\geq1$, $k\in\N$, and let~$f_{n_1,\dots,n_k}(\kappa_1,\dots,\kappa_k)$ be a~$\mathcal{B}_{\leq \frac{N}{2}}$-measurable function in
	\[
	L_\omega^r\,\ell_{n_1,\dots,n_k}^2 L_{\kappa_1,\dots,\kappa_k}^{q'}.
	\]
Assume that~$g_n^{N}(t):=\mathcal{T}_n^{N}(t)\cdot g_n(\omega)$, defined in \eqref{def:gnN}. 
	Then there exists a constant~$C=C(k)>0$ such that for all~$r\geq2$, almost surely,
	\begin{multline*}
		\Big\|
		\int_{\R^k}
		\sum_{\substack{n_1,\dots,n_k\\ \mathrm{non\text{-}paired}}}
		\Big(\prod_{j=1}^k (\widetilde g_{n_j}^{N})^{\pm}(\kappa_j)\Big)
		f_{n_1,\dots,n_k}(\kappa_1,\dots,\kappa_k)
		\,\mathrm{d}\kappa_1\cdots\mathrm{d}\kappa_k
		\Big\|_{L_\omega^r\mid\mathcal{B}_{\leq N}} \\
		\leq C\,r^{\frac k2}\|\widetilde{\mathcal{T}}_{n}^{N}(\kappa)\|_{\ell_n^{\infty}\widetilde{\mathcal{F}}L_q^{\gamma}}^k
		\Big\|
		\prod_{j=1}^k \langle\kappa_j\rangle^{-\gamma}
		f_{n_1,\dots,n_k}(\kappa_1,\dots,\kappa_k)
		\Big\|_{\ell_{n_1,\dots,n_k}^2 L_{\kappa_1,\dots,\kappa_k}^{q'}}.
	\end{multline*}
\end{lemma}

The proof relies on conditional Wiener chaos estimates together with the specific structure and bounds satisfied by type~(C) terms. We refer to~\cite[Lemma~8.2]{BCST24} for a proof in the slightly more general case where the spatial variable plays a role.

Next, we establish a general moment bound for multilinear expressions built from type~(C) terms.

\begin{lemma}[Moment bound for unpaired trilinear interactions]
	\label{momentbd'}
Assume that $(v^{(\mathrm{C})}_{M})_{M\leq N}$ are type $(\mathrm{C})$ terms in Definition \ref{def:CD} with associated Gaussians $g_{m}^M(t)$ defined via \eqref{def:gnN}.
Let~$\mathfrak{h}$ be a deterministic tensor.
	Then, for every~$r\geq2$, $\frac12<b_1<1$, and~$N_1,N_2,N_3\leq N$,
	\begin{multline}
		\label{eq:momentbd}
		\hspace{-15pt}\Bigg\|
		\langle\kappa\rangle^{-b_{1}}
		\int_{\R^{3}}
		\sum_{\substack{n_{1},n_{2},n_{3}\\ \mathrm{distinct}}}
		\mathfrak{h}(\vec n)\,
		\widehat{\chi}(\widetilde{\kappa}-\Omega(\vec{n}))\,
		\frac{
			\widetilde{g}_{n_{1}}^{N_{1}}(\kappa_{1})\,
			\overline{\widetilde{g}_{n_{2}}^{N_{2}}(\kappa_{2})}\,
			\widetilde{g}_{n_{3}}^{N_{3}}(\kappa_{3})
		}{
			n_{1}^{\alpha}
			n_{2}^{\alpha}
			n_{3}^{\alpha}}
		\mathrm{d}\kappa_1\,\mathrm{d}\kappa_2\,\mathrm{d}\kappa_3
		\Bigg\|_{L_{\omega}^{r}(\Xi_{N};L_{\kappa}^2\ell_{n}^{2})}
		\\
		\leq
		C\,(N_1N_2N_3)^{-\alpha}\,r^{\frac{3}{2}}R^3
		\sup_{\kappa\in\R}
		\big\|
		\mathfrak{h}(\vec n)\,
		\widehat{\chi}(\kappa-\Omega(\vec{n}))
		\big\|_{\ell_{n,n_{1},n_{2},n_{3}}^{2}}.
	\end{multline}
\end{lemma}

\begin{proof}
	Without loss of generality, assume~$N_1\geq N_2\geq N_3$.
	We distinguish the frequency generations of the random variables~$(\widetilde g_{n_j}^{N_j})$ in order to apply Lemma~\ref{k-linearproba} with~$k=1$ iteratively.
	For clarity, we treat the case~$N_1>N_2>N_3$; the remaining cases follow by a similar argument. For instance, when~$N_{1}=N_{2}=N_{3}$, the type~(C) terms come from the same generation and Lemma~\ref{k-linearproba} can be applied directly with~$k=3$.

	\emph{Step~1: conditioning at scale~$N_1/2$.}
	Define the~$\mathcal{B}_{\leq N_1/2}$-measurable coefficient
	\[
	f_{n,n_1}(\kappa,\kappa_1)
	:=
	\int_{\R^2}
	\sum_{\substack{n_2,n_3\\ n_2\neq n_3,n_1\\ 
	n_3\neq n}}
	\mathfrak{h}(\vec n)\,
	\widehat{\chi}(\widetilde{\kappa}-\Omega(\vec{n}))
	\frac{
		\overline{\widetilde g_{n_2}^{N_2}(\kappa_2)}\,
		\widetilde g_{n_3}^{N_3}(\kappa_3)
	}{
		n_1^{\alpha}
		n_2^{\alpha}
		n_3^{\alpha}}
	\,\mathrm{d}\kappa_2\,\mathrm{d}\kappa_3,
	\]
	The quantity we want to bound in~\eqref{eq:momentbd} can be rewritten as
	\[
	\langle\kappa\rangle^{-b_1}
	\int_{\R}\sum_{n_1}
	\widetilde g_{n_1}^{N_1}(\omega,\kappa_1)\,
	f_{n,n_1}(\kappa,\kappa_1)\,\mathrm{d}\kappa_1.
	\]

	Using conditional expectation, Minkowski's inequality (since~$r\geq2$),
	Lemma~\ref{k-linearproba} with~$k=1$, and Hölder's inequality, we obtain
	\begin{multline}
		\label{eq:nrenn}
		\Big\|
		\langle\kappa\rangle^{-b_1}
		\int_{\R}\sum_{n_1}
		\widetilde g_{n_1}^{N_1}(\kappa_1)\,
		f_{n,n_1}(\kappa,\kappa_1)\,\mathrm{d}\kappa_1
		\Big\|_{L_\omega^r(\Xi_N;L_\kappa^2\ell_n^2)}
		\\
		\lesssim
		r^{\frac12}R
		\big\|
		\langle\kappa\rangle^{-b_1}\langle\kappa_1\rangle^{-\gamma}
		f_{n,n_1}(\kappa,\kappa_1)
		\big\|_{L_\omega^r(\Xi_N;L_{\kappa,\kappa_1}^{q'}\ell_{n,n_1}^2)}.
	\end{multline}
	Here we used that on~$\Xi_N$, the variable~$\widetilde g^{N_1}$ is of type~(C)
	at scale~$N_1$ and of size~$R$.

	\emph{Step~2: conditioning at scale~$N_2/2$.}
	For fixed~$(n,n_1,n_{2},\kappa,\kappa_1,\kappa_{2})$, define
	\[
	f_{n,n_1,n_2}(\kappa,\kappa_1,\kappa_2)
	:=
	\int_{\R}
	\sum_{\substack{n_3\\ n_3\neq n_1,n_2}}
	\mathfrak{h}(\vec n)\,
	\widehat{\chi}(\widetilde{\kappa}-\Omega(\vec{n}))
	\frac{\widetilde g_{n_3}^{N_3}(\kappa_3)}
	{n_1^{\alpha}n_2^{\alpha}n_3^{\alpha}}
	\,\mathrm{d}\kappa_3,
	\]
	so that
	\[
	f_{n,n_1}(\kappa,\kappa_1)
	=
	\int_{\R}\sum_{\substack{n_{2}\\n_2\neq n_1}}
	\widetilde g_{n_2}^{N_2}(\kappa_2)\,
	f_{n,n_1,n_2}(\kappa,\kappa_1,\kappa_2)\,\mathrm{d}\kappa_2.
	\]
	Applying Lemma~\ref{k-linearproba} with~$k=1$ again bounds~\eqref{eq:nrenn} by
	\[
	rR^2
	\big\|
	\langle\kappa\rangle^{-b_1}
	\langle\kappa_1\rangle^{-\gamma}
	\langle\kappa_2\rangle^{-\gamma}
	f_{n,n_1,n_2}(\kappa,\kappa_{1},\kappa_{2})
	\big\|_{L_\omega^r(\Xi_N;L_{\kappa,\kappa_1,\kappa_2}^{q'}\ell_{n,n_1,n_2}^2)}.
	\]

	\emph{Step~3: final iteration.}
	Setting
	\[
	f_{n,n_1,n_2,n_3}(\kappa,\kappa_{1},\kappa_{2},\kappa_{3})
	:=
	\mathfrak{h}(\vec n)\,
	\widehat{\chi}(\widetilde{\kappa}-\Omega(\vec{n}))
	\prod_{j=1}^3n_j^{-\alpha},
	\]
	and repeating the argument one last time, we obtain the bound
	\[
	r^{\frac32}R^3
	\big\|
	\langle\kappa\rangle^{-b_1}
	\prod_{j=1}^3\langle\kappa_j\rangle^{-\gamma}
	f_{n,n_1,n_2,n_3}(\kappa,\kappa_{1},\kappa_{2},\kappa_{3})
	\big\|_{L_\kappa^2L_{\kappa_1,\kappa_2,\kappa_3}^{q'}\ell_{n,n_1,n_2,n_3}^2}.
	\]

	Since~$\widetilde{\kappa}=\kappa_1-\kappa_2+\kappa_3-\kappa$ and~$2b_{1}>1$, $q'\gamma>1$, this is bounded by
	\[
	(N_1N_2N_3)^{-\alpha}\,r^{\frac32}R^3
	\sup_{\kappa\in\R}
	\big\|
	\mathfrak{h}(\vec n)\,
	\widehat{\chi}(\kappa-\Omega(\vec{n}))
	\big\|_{\ell_{n,n_1,n_2,n_3}^2},
	\]
	which proves~\eqref{eq:momentbd}.
\end{proof}

\subsection{The cubic fully non-resonant interaction with type~(C) terms}
\label{CCC}

In this section we study the (C)(C)(C) configurations that cannot be handled by purely deterministic arguments. The analysis combines conditional Wiener chaos estimates with the deterministic Strichartz-type bounds established earlier.
\begin{proposition}\label{prop:CCC}
	Let $(v_M^{(\mathrm{C})})_{M\leq N}$ be type (C) terms on a $\mathcal{B}_{\leq N}-$ measurable set $\Xi_N$. Then, $(NR,\delta_{0},\mathcal{B}_{\leq 2N})$-certainly on $\Xi_N$, for all dyadic integers~$N_{1},N_{2},N_{3}$ with $$ \max(N_{1},N_{2},N_{3})=2N,\text{ and } N_{0}\lesssim N,$$
	\[
	\big\|
	\chi\,\mathbf{P}_{N_{0}}
	\mathcal{N}_{[1,2,3]}^{(3)}(v^{(\mathrm{C})}_{N_{1}},
v^{(\mathrm{C})}_{N_{2}},v^{(\mathrm{C})}_{N_{3}})
	\big\|_{X^{0,-b_{1}}}
	\lesssim R^{4}
	\,
	N^{3(1-\alpha)-\alpha+\delta}.
	\]
\end{proposition}

\begin{remark}\label{rmk:numerology}
	According to the constraints on~$s$ in Definition~\ref{hierarchy}, this is possible provided that
	\[
	4\alpha - 3> \frac{9}{2}-4\alpha +2^{100}\sigma + 3\delta\,,
	\]
	namely when~$\alpha>15/16$ and~$\sigma$ is sufficiently small. Combined with the interpolation argument in Section~\ref{sec:mainproof}, this completes the proof of~\eqref{eq:tri-cubic} for the (C)(C)(C) case in Proposition~\ref{prop:trilinear}.
\end{remark}

\begin{proof}
	Up to an interpolation argument presented in Section~\ref{sec:mainproof} with a crude bound on the~$X^{0,0}$ norm, we may assume that~$b_{1}>\frac{1}{2}$. From the moment bounds in the previous lemma, we deduce that $(RN,\delta_{0},\mathcal{B}_{\leq 2N})$-certainly,
	\begin{multline*}
		\big\|
		\chi\,\mathbf{P}_{N_{0}}
		\mathcal{N}_{[1,2,3]}^{(3)}(v^{(\mathrm{C})}_{N_{1}},
	v^{(\mathrm{C})}_{N_{2}},v^{(\mathrm{C})}_{N_{3}})\big\|_{X^{0,-b_{1}}}
		\lesssim R^{4}N^{\delta/2}
		\cdot
		(N_{1}N_{2}N_{3})^{-\alpha}\\
		\sup_{\kappa\in\R}
		\|\gamma_{n_{0}n_{1}n_{2}n_{3}}\mathfrak{h}_{\vec{N}}(\vec{n})
		\widehat{\chi}(\kappa-\Omega(\vec{n}))\|_{\ell_{n_{0},n_{1},n_{2},n_{3}}^{2}}\,,
	\end{multline*}
	where
	\[
	\mathfrak{h}_{\vec{N}}(\vec{n})
	:=\mathbf{1}_{\{n_{1},n_{3}\}\cap\{n_{0},n_{2}\}=\varnothing}\prod_{i=0}^{3}
	\mathbf{1}_{n_{i}\sim N_{i}}\,.
	\]
	Without loss of generality, assume~$N_{3}\leq N_{2}\leq N_{1}=2N$ and~$N_{0}\lesssim N$. The Strichartz-type bound~\eqref{Strichartztype1} gives
	\begin{multline*}
		\sup_{\kappa\in\R}
		\|
		\mathfrak{h}_{\vec{N}}(\vec{n})
		\widehat{\chi}(\kappa-\Omega(\vec{n}))
		\|_{\ell_{n_{0},n_{1},n_{2},n_{3}}^{2}}
		\lesssim
		\sup_{\kappa}
		\Big(\sum_{\vec{n}}|\mathfrak{h}_{\vec{N}}(\vec{n})
		\widehat{\chi}(\kappa-\Omega(\vec{n}))|
		\Big)^{\frac{1}{2}}
		\\
		\lesssim_{\varepsilon} N_{1}^{1-\alpha+\varepsilon}(N_{2}N_{3})^{\frac{1}{2}}\,.
	\end{multline*}
	Together with the bound~\eqref{eq:gamma'} on the correlation function, we control the above by
	\[
	R^{4}N^{\delta/2}
	\underbrace{(N_{1}N_{2}N_{3})^{-\alpha}}_{\text{weight}}
	\underbrace{N_{3}}_{\text{correlation factor}}
	\,\underbrace{
		N_{1}^{1-\alpha+\varepsilon}(N_{2}N_{3})^{\frac{1}{2}}
	}_{\text{Strichartz-type bound}}\,,
	\]
	which is bounded by (controlling $N_1^{\epsilon}$ by $N^{\delta/2}$ )
	\[
	R^{4}
	N_{1}^{1-2\alpha+\delta}N_{2}^{\frac{1}{2}-\alpha}N_{3}^{\frac{3}{2}-\alpha}
	\lesssim
	R^{4}
	N_{1}^{1-2\alpha+\delta}N_{2}^{2(1-\alpha)}
	\lesssim
	R^{4}
	N_{1}^{3-4\alpha+\delta}\,.
	\]
	This completes the proof.
\end{proof}

\subsection{Random operator estimates}
In this section we prove~\eqref{eq:tri-cubic} when there is exactly one type~(D) input, a configuration not amenable to the deterministic estimates developed earlier.
We restrict to the case~$\ast_{1}= \mathrm{D}$ and~$\ast_{2}=\ast_{3}=\mathrm{C}$.
For the interaction~$\mathcal{N}_{[1,2,3]}^{(3)}$, the position of the inputs plays no role in our analysis, and the remaining configurations are handled in exactly the same way.

We first establish tensor estimates that follow from the deterministic counting estimates in Section~\ref{sec:detertool}.

\begin{definition}[Tensor and tensor norms] \label{def:tens}
	Let~$\vec{\kappa}=(\kappa_{0},\kappa_{1},\kappa_{2},\kappa_{3})$,
	$\vec{n}=(n_{0},n_{1},n_{2},n_{3})$, and
	$\vec{N}=(N_{0},N_{1},N_{2},N_{3})$.
	We define the tensor
	\[
    \hspace{-15pt}\mathcal{G}_{\vec N}(\vec{n},\vec{\kappa})
	:=\mathfrak{h}_{\vec{N}}(\vec n)\,
	\widehat{\chi}\big(\widetilde{\kappa}-\Omega(\vec n)\big),
	\qquad
	\mathfrak{h}_{\vec{N}}(\vec n)
	:=\mathbf{1}_{\{n_{1},n_{3}\}\cap\{n_{0},n_{2}\}=\varnothing}
	\prod_{i=0}^{3}\mathbf{1}_{n_{i}\sim N_{i}},
	\]
	and its merged tensor norm
	\begin{multline*}
	\|\mathcal{G}_{\vec{N}}\|_{\mathsf{Ten}(n_{1}\to n_{0})}
\\
    :=\max\Big\{
	\|\mathcal{G}_{\vec{N}}\|_{n_{1},n_{2},n_{3}\to n_{0}},\
	\|\mathcal{G}_{\vec{N}}\|_{n_{1},n_{2}\to n_{0},n_{3}},\
	\|\mathcal{G}_{\vec{N}}\|_{n_{1},n_{3}\to n_{0},n_{2}},\
	\|\mathcal{G}_{\vec{N}}\|_{n_{1}\to n_{0},n_{2},n_{3}}
	\Big\},
	\end{multline*}
where the norms inside the max are operator norms on $\ell^2$-type spaces.
The other merged tensor norms $\|\mathcal{G}_{\vec{N}}\|_{\mathsf{Ten}(n_j\to n_0) }$ is defined in an analogue way. 
\end{definition}


	\begin{lemma}[Tensor estimates]\label{lem:tens}
		Let~$\vec N=(N_0,N_1,N_2,N_3)$ and let
		\[
		N_{\max}\ge N_{\mathrm{med}}\ge N_{\min}
		\]
		denote the non-increasing rearrangement of~$N_1,N_2,N_3$.
		Assume that either~$N_0\sim N_{\max}$, or that~$N_0\lesssim N_{\max}$ and~$N_{\mathrm{med}}\sim N_{\max}$.
		
		Then, for any~$j\in\{1,2,3\}$ and every~$\varepsilon>0$,
		\[
		\sup_{\vec\kappa}
		\|\mathcal G_{\vec N}(\,\cdot\,,\vec\kappa)\|_{\mathsf{Ten}(n_{j}\to n_{0})}
		\;\lesssim_{\varepsilon}\;
		\begin{cases}
			N_{\max}^{\,1-\alpha+\varepsilon}\,N_{\mathrm{med}}^{1/2},
			& \text{if } N_{\mathrm{med}}\gtrsim N_{\max}^{1/2},\\[0.2cm]
			N_{\mathrm{max}}^{\,1-\alpha+\varepsilon},
			& \text{if } N_{\mathrm{med}}\ll N_{\max}^{1/2},\\[0.2cm]
			N_{\mathrm{med}}^{2(1-\alpha)+\epsilon}, &\text{if } N_{\mathrm{med}}\ll N_{\max}^{1/2},\; N_j=N_{\max}.
		\end{cases}
		\]
	\end{lemma}
	
	\begin{proof}
		Without loss of generality, we prove the case~$j=1$ and assume~$N_2\ge N_3$.
		We estimate the tensor norm by controlling all operator norms
		appearing in~$\|\cdot\|_{\mathsf{Ten}(n_{1}\to n_{0})}$.
		
		\emph{Step~1: the norm~$\|\cdot\|_{n_1,n_2,n_3\to n_0}$.}
		Fix~$\vec\kappa$. By Schur's test,
		\[
		\|\mathcal G_{\vec N}(\cdot,\vec\kappa)\|_{n_1,n_2,n_3\to n_0}
		\lesssim
		\Big(\sup_{n_1,n_2,n_3}\sum_{n_0}|\mathcal G_{\vec N}|\Big)^{1/2}
		\Big(\sup_{n_0}\sum_{n_1,n_2,n_3}|\mathcal G_{\vec N}|\Big)^{1/2}.
		\]
		The first factor is uniformly bounded.
		For the second, we write
		\[
		\sup_{n_0}\sum_{n_1,n_2,n_3}|\mathcal G_{\vec N}|
		\lesssim
		\sup_{n_0}\sum_{m\ge1}\sup_{|\nu|\le1}
		\widehat\chi(m+\nu)\,
		\#\Gamma_{\vec N,\widetilde\kappa+m}(n_0).
		\]
		Applying Lemma~\ref{improvedcounting}, we obtain
		\[
		\sup_{n_0}\sum_{n_1,n_2,n_3}|\mathcal G_{\vec N}|
		\lesssim_{\varepsilon}
		\begin{cases}
			N_{\max}^{\,2(1-\alpha)+\varepsilon}\,N_{\min},
			&\quad \text{if}\ N_{\mathrm{med}}\gtrsim N_{\max}^{1/2},\\
			N_{\mathrm{med}}^{\varepsilon},
			&\quad \text{if}\ N_{\mathrm{med}}\ll N_{\max}^{1/2}.
		\end{cases}
		\]
		Combining the two factors yields
		\begin{equation}
			\label{eq:tensor-basic}
			\sup_{\vec\kappa}
			\|\mathcal G_{\vec N}(\cdot,\vec\kappa)\|_{n_1,n_2,n_3\to n_0}
			\lesssim_{\varepsilon}
			\begin{cases}
				N_{\max}^{\,1-\alpha+\varepsilon}\,N_{\mathrm{med}}^{1/2},
				&\quad \text{if}\  N_{\mathrm{med}}\gtrsim N_{\max}^{1/2},\\
				N_{\mathrm{med}}^{\varepsilon},
				&\quad \text{if}\  N_{\mathrm{med}}\ll N_{\max}^{1/2}.
			\end{cases}
		\end{equation}
		
		\emph{Step~2: the norm~$\|\cdot\|_{n_1\to n_0,n_2,n_3}$.}
		We distinguish according to the size of~$N_2$ relative to~$N_0^{1/2}$.
		The same counting argument yields
		\[
		\sup_{\vec\kappa}
		\|\mathcal G_{\vec N}(\cdot,\vec\kappa)\|_{n_1\to n_0,n_2,n_3}
		\lesssim_{\varepsilon}
		\begin{cases}
			\max(N_0,N_2)^{\,1-\alpha+\varepsilon}\min(N_0,N_3)^{1/2},
			& N_2\gtrsim N_0^{1/2},\\
			N_2^{\varepsilon}\lesssim N_{\mathrm{med}}^{\varepsilon},
			& N_2\ll N_0^{1/2}.
		\end{cases}
		\]
		If~$N_0\sim N_{\max}$, this coincides with~\eqref{eq:tensor-basic}.
		Otherwise, $N_0\ll N_{\max}$ and~$N_2\sim N_{\max}$ under our assumptions, so similar bound follows, replacing $N_{\min}^{1/2}$ by $N_{\mathrm{med}}^{1/2}$ (in the case where $N_1=N_{\min}$).
		
		\emph{Step~3: remaining tensor norms.}
		We now estimate
		\[
		\|\cdot\|_{n_1,n_2\to n_0,n_3}
		\quad\text{and}\quad
		\|\cdot\|_{n_1,n_3\to n_0,n_2}.
		\]
		Using the same argument together with Lemma~\ref{lem:db}, we obtain:
		\begin{itemize}
			\item If~$N_{\mathrm{med}}\gtrsim N_{\max}^{1/2}$,
			\[
			\sup_{\vec\kappa}
			\big(
			\|\mathcal G_{\vec N}\|_{n_1,n_2\to n_0,n_3}
			+
			\|\mathcal G_{\vec N}\|_{n_1,n_3\to n_0,n_2}
			\big)
			\lesssim_{\varepsilon}
			N_{\max}^{2(1-\alpha)+\varepsilon}\lesssim N_{\max}^{1-\alpha+\epsilon}N_{\mathrm{med}}^{1/2}.
			\]
			\item If~$N_{\mathrm{med}} \ll N_{\max}^{1/2}$ and~$N_{\mathrm{med}}\ll N_0^{1/2}$, both norms are~$\mathcal O(1)$.
			\item If~$N_{\mathrm{med}}\ll N_{\max}^{1/2}$ and~$N_{\mathrm{med}}\gtrsim N_0^{1/2}$,
			\[
		\sup_{\vec\kappa}
			\big(
			\|\mathcal G_{\vec N}\|_{n_1,n_2\to n_0,n_3}
			+
			\|\mathcal G_{\vec N}\|_{n_1,n_3\to n_0,n_2}
			\big)
			\lesssim_{\varepsilon}
			\max(N_0,N_{2})^{\,1-\alpha+\varepsilon/2}
			\lesssim
			N_{\mathrm{max}}^{\,1-\alpha+\varepsilon/2}.
			\]
			Moreover, when $N_1=N_{\max}$, we have $N_2\leq N_{\mathrm{med}}$, and $N_0\leq N_{\mathrm{med}}^2$. Hence \[\max(N_0,N_2)^{1-\alpha+\varepsilon/2}\lesssim N_{\mathrm{med}}^{2(1-\alpha)+\epsilon}.
	      \]
		\end{itemize}
		
		Collecting all cases proves the lemma.
	\end{proof}
	

Equipped with the tensor estimates, we establish the main random operator bound in Fourier-restriction spaces.
	
	\begin{proposition}\label{prop:random-opdyadic}
	Assume that~$b,b_1>1/2$. Let~$N_0,N_1,N_2,N_3$ be dyadic integers, let~$N_{\mathrm{max}}\geq N_{\mathrm{med}}\geq N_{\mathrm{min}}$ be the non-increasing rearrangement of~$N_1,N_2,N_3$, and set~$M:=\max(N_2,N_3)$. Assume that either~$N_0\lesssim N_{\mathrm{max}}$ or~$N_{\mathrm{med}}\sim N_{\mathrm{max}}$. Let $v_{N_2}^{(\mathrm{C})},v_{N_3}^{(\mathrm{C})}$ are type (C) terms according to Definition \ref{def:CD}. Then $(MR,\delta_0;\mathcal{B}_{\leq M})$-certainly,
	\[
	\|\chi(t)\mathbf{P}_{N_0}\mathcal{N}_{[1,2,3]}^{(3)}(\mathbf{P}_{N_1}\cdot,v_{N_2}^{(\mathrm{C})},v_{N_3}^{\mathrm{(C)}}) \|_{X^{0,b}\to X^{0,-b_1}}\lesssim R^2\Lambda_{\vec{N}},
	\]
	where
	\[
	\Lambda_{\vec{N}}=(N_2N_3)^{-\alpha}\max(N_2,N_3)^{\delta_0}\min(N_{\mathrm{min}},N_0)\log(N_{\mathrm{max}})\sup_{\vec{\kappa}}\|\mathcal{G}_{\vec{N}}(\vec{\kappa})\|_{\mathsf{Ten}(n_{1}\to n_{0})},
	\]
	and the tensor bound~$\|\mathcal{G}_{\vec{N}}(\vec{\kappa})\|_{\mathsf{Ten}(n_{1}\to n_{0})}$ is given in Lemma~\ref{lem:tens}.
\end{proposition}

\begin{remark}\label{prop:random-opdyadicremark}
	Similarly, as an operator ($\mathbb{C}$-antilinear),
	\[
	\|\chi(t)\mathbf{P}_{N_0}\mathcal{N}_{[1,2,3]}^{(3)}(v_{N_1}^{(\mathrm{C})},\mathbf{P}_{N_2}\cdot,v_{N_3}^{(\mathrm{C})}) \|_{X^{0,b}\to X^{0,-b_1}}\lesssim R^2\Lambda_{\vec{N}},
	\]
	where
	\[
	\Lambda_{\vec{N}}=(N_1N_3)^{-\alpha}\max(N_1,N_3)^{\delta_0}\min(N_{\mathrm{min}},N_0)\log(N_{\mathrm{max}})\sup_{\vec{\kappa}}\|\mathcal{G}_{\vec{N}}(\vec{\kappa})\|_{\mathsf{Ten}(n_{2}\to n_{0})}.
	\]
\end{remark}

The proof combines decoupling with the non-commutative Khintchine inequality, following the general strategy of~\cite{Kan25}. This inequality is originally due to~\cite{HP93} (see also~\cite{vHan17} for a modern treatment).
A logarithmic loss in this inequality requires an additional input when~$\max(N_{2},N_{3})$ is much smaller than the remaining scales.
This issue already appears in~\cite[Claim~5.2]{DNY24} and~\cite[Proposition~5.1]{BR25}, and is resolved here in Proposition~\ref{prop:highhighverylow} by exploiting a gain from the modulation analysis as  in the treatment of the partially resonant interaction~$\mathcal{N}_{(23)}^{(3)}$. 

\begin{proof}
	Without loss of generality, assume~$\max(N_2,N_3)=N_2$.
	Since all terms are frequency-localized, we remove the spatial weights in the denominators, gaining a factor~$(N_{2}N_{3})^{-\alpha}$, and replace the correlation coefficient~$\gamma_{nn_{1}n_{2}n_{3}}$ by~$\min(N_{0},N_{\min})$.

	\emph{Step~1: duality.}
	Given~$u$ and~$v$ with~$\|u\|_{X^{0,0}}, \|v\|_{X^{0,0}}\le 1$, we need to bound
	\begin{equation}
		\label{eq:du}
		\left|
		\sum_{\vec{n}}
		\mathfrak{h}_{\vec{N}}(\vec n)
		\int
		\widehat{\chi}\big(\widetilde{\kappa}-\Omega(\vec n)\big)\,
		\overline{\widetilde{v}_{n_{0}}(\kappa_{0})}\widetilde{u}_{n_{1}}(\kappa_{1})\,
		\overline{\widetilde{g}_{n_{2}}^{N_{2}}(\kappa_{2})}\,
		\widetilde{g}_{n_{3}}^{N_{3}}(\kappa_{3})\,
		\frac{\mathrm{d}\vec \kappa}
		{\langle\kappa_{0}\rangle^{b_{1}}\langle\kappa_{1}\rangle^{b}}
		\right|\,,
	\end{equation}
	where~$\mathfrak{h}_{\vec{N}}$ is as in Lemma~\ref{lem:tens} and~$b_{1}>\frac{1}{2}$. Isolating the~$X^{0,0}$-norms of~$u$ and~$v$ from the random inputs, we bound the above by
	\begin{multline*}
		\left\|\langle\kappa_{0}\rangle^{-b_{1}}
		\langle\kappa_{1}\rangle^{-b}
		\left\|
		\sum_{n_{2},n_{3}}
		\mathfrak{h}_{\vec{N}}(\vec n)
		\int
		\overline{\widetilde{g}_{n_{2}}^{N_{2}}(\kappa_{2})}\,
		\widetilde{g}_{n_{3}}^{N_{3}}(\kappa_{3})\,
		\widehat{\chi}\big(\widetilde{\kappa}-\Omega(\vec n)\big)\,
		\mathrm{d}\kappa_{2}\mathrm{d}\kappa_{3}
		\right\|_{n_{1}\to n_{0}}
		\right\|_{L_{\kappa_{0},\kappa_{1}}^{2}}
		\\
		\times\|\mathbf{P}_{N_{0}}v\|_{X^{0,0}}
		\|\mathbf{P}_{N_{1}}u\|_{X^{0,0}}\,.
	\end{multline*}
	For fixed~$(\kappa_{0},\kappa_{1})$ and~$(n_{0},n_{1})$, set
	\[
	\mathcal{T}_{n_{0},n_{1}}(\omega;\kappa_{0},\kappa_{1})
	:=
	\sum_{n_{2},n_{3}}
	\mathfrak{h}_{\vec{N}}(\vec n)
	\int
	\overline{\widetilde{g}_{n_{2}}^{N_{2}}(\kappa_{2})}\,
	\widetilde{g}_{n_{3}}^{N_{3}}(\kappa_{3})\,
	\widehat{\chi}\big(\widetilde{\kappa}-\Omega(\vec n)\big)\,
	\mathrm{d}\kappa_{2}\mathrm{d}\kappa_{3}.
	\]
	The bound for~\eqref{eq:du} then reads
	\begin{equation}
		\label{ch}
		\left\|
		\langle\kappa_{0}\rangle^{-b_{1}}
		\langle\kappa_{1}\rangle^{-b}
		\left\|
		\mathcal{T}(\omega;\kappa_{0},\kappa_{1})
		\right\|_{n_{1}\to n_{0}}
		\right\|_{L_{\kappa_{0},\kappa_{1}}^{2}}
		\|\mathbf{P}_{N_{0}}u\|_{X^{0,0}}
		\|\mathbf{P}_{N_{1}}v\|_{X^{0,0}}\,.
	\end{equation}

	\emph{Step~2: moments and non-commutative Khintchine with successive removal of RAOs.}
	For~$p\ge 2$, using~$2b>1$ and~$2b_{1}>1$,
	\begin{multline}
		\label{eq:01}
		\left\|
		\left\|
		\langle\kappa_{0}\rangle^{-b_{1}}\langle\kappa_{1}\rangle^{-b}\,
		\left\|\mathcal{T}(\omega;\kappa_{0},\kappa_{1})\right\|_{n_{1}\to n_{0}}
		\right\|_{L^{2}_{\kappa_{0},\kappa_{1}}}
		\right\|_{L_{\omega}^{p}}
		\\
		\le
		\left\|
		\langle\kappa_{0}\rangle^{-b_{1}}\langle\kappa_{1}\rangle^{-b}\,
		\left\|
		\|\mathcal{T}(\omega;\kappa_{0},\kappa_{1})\|_{n_{1}\to n_{0}}
		\right\|_{L_{\omega}^{p}}
		\right\|_{L^{2}_{\kappa_{0},\kappa_{1}}}
		\\
		\lesssim
		\sup_{\kappa_{0},\kappa_{1}}
		\left\|
		\|\mathcal{T}(\omega;\kappa_{0},\kappa_{1})\|_{n_{1}\to n_{0}}
		\right\|_{L_{\omega}^{p}}\,.
	\end{multline}
	Fix~$(\kappa_{0},\kappa_{1},n_{0},n_{1})$ and write
	\[
	\mathcal{T}_{n_{0},n_{1}}(\omega;\kappa_{0},\kappa_{1})
	=
	\sum_{n_{2}\neq n_{3}}
	\overline{g_{n_{2}}}\,g_{n_{3}}\,
	\mathfrak{h}_{\vec{N}}(\vec n)
	\int
	\widehat{\chi}\big(\widetilde{\kappa}-\Omega(\vec n)\big)\,
	\overline{\widetilde{\mathcal{T}}_{n_{2}}^{N_{2}}(\kappa_{2})}\,
	\widetilde{\mathcal{T}}_{n_{3}}^{N_{3}}(\kappa_{3})\,
	\mathrm{d}\kappa_{2}\mathrm{d}\kappa_{3}\,.
	\]
	Introduce the random tensor
	\[
	\mathcal{L}_{n_{0},n_{1},n_{2}}(\omega;\kappa_{0},\kappa_{1})
	:=
	\sum_{n_{3}}
	g_{n_{3}}\,
	\mathfrak{h}_{\vec{N}}(\vec n)
	\int
	\widehat{\chi}(\widetilde{\kappa}-\Omega(\vec n))\,
	\overline{\widetilde{\mathcal{T}}_{n_{2}}^{N_{2}}(\kappa_{2})}\,
	\widetilde{\mathcal{T}}_{n_{3}}^{N_{3}}(\kappa_{3})\,
	\mathrm{d}\kappa_{2}\mathrm{d}\kappa_{3},
	\]
	so that
	\[
	\mathcal{T}_{n_{0},n_{1}}(\omega;\kappa_{0},\kappa_{1})
	=\sum_{n_{2}}
	\overline{g_{n_{2}}}\,
	\mathcal{L}_{n_{0},n_{1},n_{2}}(\omega;\kappa_{0},\kappa_{1}).
	\]
	We only use deterministic bounds for the random averaging operators~$\mathcal{T}^{N_{2}}=(\mathcal{T}_{n_2}^{N_2})_{n_2\sim N_2}$ and~$\mathcal{T}^{N_{3}}=(\mathcal{T}_{n_3}^{N_3})_{n_3\sim N_3}$.
	However, $\mathcal{T}^{N_{2}}$ (resp.\ $\mathcal{T}^{N_{3}}$) is~$\mathcal{B}_{\leq N_{2}}$-measurable (resp.\ $\mathcal{B}_{\leq N_{3}}$-measurable).
	Before invoking probabilistic tools, we eliminate these operators by applying Hölder's inequality in the modulation variables, replacing them by their~$X^{0,\gamma}_{\infty,q}$ norms.
	We distinguish two cases.

	\noindent\emph{Case~$N_{2}>N_{3}$.}
	For $N\in 2^{\N}$, denote $\mathcal{B}_{N}$ the $\sigma$-algebra generated by $\{g_n\}_{n\sim N}$.
	In this case,  $\mathcal{T}^{N_{2}}(t)$ is not independent to random variables ~$\{g_{n_{3}}\}_{n_{3}\sim N_{3}}$, but is independent from~$\mathcal{B}_{N_2}$, generated by $\{g_{n_2} \}_{n_2\sim N_2}$.  We proceed in two steps.

	\emph{(i)} Apply the non-commutative Khintchine inequality, as stated in~\cite[(1.1)]{Kan25}, to the sum over~$n_{2}$ conditionally on~$\mathcal{B}_{\leq N_{2}/2}$: we obtain that almost surely,
	\begin{multline*}
		\Big\|
		\|\mathcal{T}(\omega;\kappa_{0},\kappa_{1})\|_{n_{1}\to n_0}
		\Big\|_{L^{p}_{\omega}\mid \mathcal{B}_{\leq N_{2}/2}}
		\lesssim
		\sqrt{p\log(N_{\mathrm{max}})}\,
		\\
		\times\max
		\left(
		\|\mathcal{L}(\omega;\kappa_{0},\kappa_{1})\|_{n_{1},n_{2}\to n_{0}},\
		\|\mathcal{L}(\omega;\kappa_{0},\kappa_{1})\|_{n_{1}\to n_{0},n_{2}}
		\right).
	\end{multline*}
	We then bound~$\mathcal{T}^{N_{2}}$ deterministically: using~$\|\mathcal{T}^{N_{2}}\|_{X^{0,\gamma}_{\infty,q}}\le R$ and Hölder in~$\kappa_{2}$ (with~$q'\gamma>1$), we obtain almost surely
	\begin{multline*}
		\|\mathcal{L}(\omega;\kappa_{0},\kappa_{1})\|_{n_{1},n_{2}\to n_{0}}
		\\
		\lesssim
		R\,
		\left\|
		\langle\kappa_{2}\rangle^{-\gamma}
		\left\|
		\sum_{n_{3}} g_{n_{3}}\,
		\mathfrak{h}_{\vec{N}}(\vec n)
		\int
		\widehat{\chi}(\widetilde{\kappa}-\Omega(\vec n))\,
		\widetilde{\mathcal{T}}_{n_{3}}^{N_{3}}(\kappa_{3})\,
		\mathrm{d}\kappa_{3}
		\right\|_{n_{1},n_{2}\to n_{0}}
		\right\|_{L^{q'}_{\kappa_{2}}}\,.
	\end{multline*}
	We proceed similarly for the tensor norm~$\|\cdot\|_{n_{1}\to n_{0},n_{2}}$.

	\emph{(ii)} Apply the non-commutative Khintchine inequality to the sum over~$n_{3}$ conditionally on~$\mathcal{B}_{\leq N_{3}/2}$.
	This yields another loss of~$\sqrt{p\log(N_{\max})}$, and using Hölder in~$\kappa_{3}$ together with the deterministic bound for~$\mathcal{T}^{N_{3}}$ gives
	\begin{multline*}
		\|\|\mathcal{T}(\omega;\kappa_{0},\kappa_{1})\|_{n_{1}\to n_{0}}
		\|_{L_{\omega}^{p}|\mathcal{B}_{\leq N_3/2}}
		\\
		\lesssim\ p\log(N_{\mathrm{max}})\,R^{2}\,
		\|
		\langle\kappa_{2}\rangle^{-\gamma}
		\langle\kappa_{3}\rangle^{-\gamma}\,
		\|\mathcal{G}_{\vec{N}}(\vec{\kappa})\|_{\mathsf{Ten}(n_{1}\to n_{0})}
		\|_{L_{\kappa_{2},\kappa_{3}}^{q'}}\,,
	\end{multline*}
	where the merged tensor norm on the right-hand side was defined in Definition~\ref{def:tens}. Since~$\gamma q'>1$,
	\begin{equation}
		\label{eq:pmoment-pre}
		\sup_{\kappa_{0},\kappa_{1}}
		\|\|\mathcal{T}(\omega;\kappa_{0},\kappa_{1})\|_{n_{1}\to n_{0}}
		\|_{L_{\omega}^{p}}
		\lesssim
		p\log(N_{\mathrm{max}})R^{2}
		\sup_{\vec{\kappa}}
		\|\mathcal{G}_{\vec{N}}(\vec{\kappa})\|_{\mathsf{Ten}(n_{1}\to n_{0})}\,.
	\end{equation}

	\noindent\emph{Case~$N_{2}=N_{3}$.}
	Using the constraint~$n_{2}\neq n_{3}$, we decouple the Gaussians by introducing
	an independent copy~$\{g_{n_{3}}(\omega')\}$, which is independent of the $\sigma$-algebra $\mathcal{B}_{\leq N_3}$. 
	We then apply the non-commutative Khintchine inequality twice
	(first in~$\omega$, then in~$\omega'$),
	together with the deterministic RAO bounds.
	This yields again~\eqref{eq:pmoment-pre}.

	By standard large-deviation arguments derived from~$p$-moment bounds,
	we conclude that $(MR,\delta_{0};\mathcal{B}_{\leq M})$-certainly,
	\[
	\Big\|
	\langle\kappa_{0}\rangle^{-b_{1}}
	\langle\kappa_{1}\rangle^{-b}
	\|\mathcal{T}(\omega;\kappa_{0},\kappa_{1})\|_{n_{1}\to n_{0}}
	\Big\|_{L_{\kappa_{0},\kappa_{1}}^{2}}
	\lesssim
	M^{\delta_0}R^{2+\delta_0}\log(N_{\max})
	\sup_{\vec{\kappa}}
	\|\mathcal{G}_{\vec{N}}(\vec{\kappa})\|_{\mathsf{Ten}(n_{1}\to n_{0})}.
	\]

	\emph{Step~3: conclusion.}
	Inserting this bound into~\eqref{ch}, we obtain for all~$N_{0}$, $N_{1}$ and all~$u$, $v$,
	\begin{multline}
		\label{din}
		\left|
		\sum_{n_{0}\sim N_{0}}
		\int
		v_{n_{0}}(t)\,
		\chi(t)\,
		\mathcal{N}_{[1,2,3]}^{(3)}
		\big(
		\mathbf{P}_{N_{1}}u,
		v_{N_{2}}^{(\mathrm{C})},
		v_{N_{3}}^{(\mathrm{C})}
		\big)_{n_{0}}
		\,\mathrm{d}t
		\right|
		\\
		\lesssim
		R^{3}\Lambda_{\vec{N}}
		\|\mathbf{P}_{N_{1}}u\|_{X^{0,b}}
		\|\mathbf{P}_{N_{0}}v\|_{X^{0,b_{1}}},
	\end{multline}
	where
	\[
	\Lambda_{\vec{N}}
	:=
	(N_{2}N_{3})^{-\alpha}
	\min(N_{0},N_{\mathrm{min}})
	\log(N_{\max})
	N_{2}^{\delta}
	\sup_{\vec{\kappa}}
	\|\mathcal{G}_{\vec{N}}(\vec{\kappa})\|_{\mathsf{Ten}(n_{1}\to n_{0})}.
	\]
	This completes the proof.
\end{proof}
	\begin{corollary}[Random operator bound]\label{cor:randomop}
		\label{prop:random-op}
		There exists~$b_{1}$ satisfying~$0<b_{1}<\tfrac12<b$ and~$b+b_{1}<1$,
		such that for all dyadic integers~$N_{2},N_{3}$, setting~$M:=\max(N_{2},N_{3})$, the following holds.
Assume that $(v_{M'}^{(\mathrm{C})})_{M'\leq M}$ are type (C) terms on a $\mathcal{B}_{\leq M}$-measurable set $\Xi_M$. Then
		$(MR,\delta_{0};\mathcal{B}_{\leq M})$-certainly, for any~$0\leq s_0< 2\alpha-1-2\delta_0$,
		\begin{equation}
			\label{eq:random-op}
			\Big\|
			\chi\,
			\mathcal{N}_{[1,2,3]}^{(3)}
			\big(\,\cdot\,,v_{N_{2}}^{(\mathrm{C})},v_{N_{3}}^{(\mathrm{C})}\big)
			\Big\|_{X^{s_0,b}\to X^{0,-b_{1}}}
			\;\lesssim\;
			R^{2}M^{-s_0-\delta_0}.
		\end{equation}
	\end{corollary}
	
	As a consequence, the trilinear bound~\eqref{eq:tri-cubic} in the case of exactly two type~(C) inputs follows. More precisely, if the input~$w_{N_1}^{\dag}$ satisfies~$N_1=N_{(1)}$, we apply it with~$s_0=0$, while if~$N_2=N_{(1)}$ we apply it with~$s_0=s$.
	
	\begin{proof}
		We begin with the following standard preliminary reductions.
		\begin{itemize}
			\item Without loss of generality, assume~$\max(N_{2},N_{3})=N_{2}$.
			\item We decompose both the input and the output dyadically, with frequency scales~$N_{1}$ and~$N_{0}$ respectively, and denote by~$N_{\max}\geq N_{\mathrm{med}}\geq N_{\min}$ the non-increasing rearrangement of~$N_{1},N_{2},N_{3}$.
			\item We may further assume that either~$N_{0}\sim N_{\max}$, or~$N_{0}\ll N_{\max}$ and~$N_{\mathrm{med}}\sim N_{\max}$.
			Indeed, in all other cases the correlation coefficient decays rapidly in~$\max(N_{0},N_{\max})$, by Lemma~\ref{lem:highlowlowlow}.
			\item It suffices to establish the bound for all~$b_{2}>\frac12$.
			Indeed, the statement follows by interpolation with a crude~$X^{0,0}$ bound with loss~$\mathcal{O}(N_{2}^{C})$, which we prove now. If~$N_{2}\sim N_{\max}$, any crude estimate will do. Otherwise, $N_{0}\sim N_{1}$ and we get, for all~$t\in\R$,
			\begin{multline*}
				\left(
				\sum_{n_{0}\sim N_{0}}
				\left|\sum_{n_{1},n_{2},n_{3}}
				\left(
				\prod_{i=0}^{3}\mathbf{1}_{n_{i}\sim N_{i}}
				\right)
				\gamma_{n_{0}n_{1}n_{2}n_{3}}
				u_{n_{1}}(t)v_{N_{2},n_{2}}^{(\mathrm{C})}(t)v_{N_{3},n_{3}}^{(\mathrm{C})}(t)
				\right|^{2}
				\right)^{\frac{1}{2}}
				\\
				\lesssim N_{2}\|v_{N_{2},n_{2}}^{(\mathrm{C})}(t)\|_{\ell_{n_{2}}^{1}}
				\|v_{N_{3},n_{3}}^{(\mathrm{C})}(t)\|_{\ell_{n_{3}}^{1}}\|u_{n_{1}}(t)\|_{\ell_{n_{1}}^{2}}
				\lesssim N_{2}^{10}\|u_{n_{1}}(t)\|_{\ell_{n_{1}}^{2}}\,.
			\end{multline*}
		\end{itemize}
		For any~$v\in X^{0,b}$, from the preliminary reductions, it suffices to estimate, for any~$b_2>1/2$,
		\[
		\Big(
		\sum_{N_0}\Big(\sum_{\substack{N_1\\
				N_{(1)}\sim N_{(2)}
		}}
		\big\|\chi(t)\mathbf{P}_{N_0}\mathcal{N}_{[1,2,3]}^{(3)}(\mathbf{P}_{N_1}v,v_{N_2}^{(\mathrm{C})},v_{N_3}^{(\mathrm{C})}) \big\|_{X^{0,-b_2}}
		\Big)^2
		\Big)^{1/2}.
		\]
		Denote
		\[
		a_{N_1}:=N_1^{s_0}\|\mathbf{P}_{N_1}v\|_{X^{0,b}}.
		\]
		By Proposition~\ref{prop:random-opdyadic}, the expression above is bounded by
		\begin{align}\label{din'}
			R^2\Big(\sum_{N_0}
			\Big(
			\sum_{\substack{N_1\\
					N_{(1)}\sim N_{(2)}
			} } \Lambda_{\vec{N}}
			\cdot N_1^{-s_0} a_{N_1}
			\Big)^2
			\Big)^{1/2}.
		\end{align}
		
		It remains to estimate~$\Lambda_{\vec{N}}$ using Lemma~\ref{lem:tens} and perform the dyadic summation over~$N_0$ and~$N_1$.
		We distinguish cases.
		
		\emph{Case~1: $N_{\max}=N_{1}$.}
		We split according to the size of~$N_{2}$.
		
		\emph{Case~1a: $N_{2}\gtrsim N_{1}^{1/2}$.}
		Lemma~\ref{lem:tens} yields
		\[
		\Lambda_{\vec{N}}
		\lesssim_{\varepsilon}
		(N_{2}N_{3})^{-\alpha}
		N_{3}
		N_{2}^{\delta_0}
		N_{1}^{1-\alpha+\varepsilon}
		N_{2}^{1/2}\lesssim N_2^{\frac{3}{2}-\alpha}N_2^{-\alpha+\delta_0}N_1^{1-\alpha+\varepsilon}
		\lesssim
		N_{1}^{\frac{7}{4}-2\alpha+2\delta_0}.
		\]
		Hence if~$\alpha>\frac{7}{8}+2\delta_0$,
		\[
		\Lambda_{\vec{N}}\cdot N_1^{-s_0}\leq N_2^{-s_0-\delta_0}N_1^{-\delta_0},
		\]
		and the remaining dyadic sum over~$N_0$ and~$N_1$ converges.
		
		\emph{Case~1b: $N_{1}^{\delta_0}\lesssim N_{2}\ll N_{1}^{1/2}$.}
		Here~$N_{0}\sim N_{1}$, so~$N_{0}$ and~$N_{1}$ can be summed in~\eqref{din'}.
		We obtain
		\[
		\Lambda_{\vec{N}}
		\lesssim_{\varepsilon}
		(N_{2}N_{3})^{-\alpha}
		N_{3}\log(N_{1})
		N_{2}^{2(1-\alpha)+\varepsilon+\delta_0}
		\lesssim
		N_{2}^{3(1-\alpha)-\alpha+2\delta_0}.
		\]
		A similar argument is conclusive, provided~$\alpha>\frac{3}{4}+2\delta_0$.
		
		\emph{Case~1c: $N_{2}\ll N_{1}^{\delta_0}$.}
		The argument is deterministic and exploits the large size of the resonance function to gain a power of the highest frequency.
		Using Proposition~\ref{prop:highhighverylow} and~\eqref{R-bound'} directly with~$b_{1}<\frac{1}{2}$ as in the statement, we obtain
		\begin{multline*}
			\|
			\mathbf{P}_{N_{0}}
			\mathcal{N}_{[1,2,3]}^{(3)}
			(\mathbf{P}_{N_{1}}v,v_{N_2}^{(\mathrm{C})},v_{N_{3}}^{(\mathrm{C})})
			\|_{X^{0,-b_{1}}}
			\lesssim
			R^{2}T^{-2(b-\frac12)}
			N_{1}^{\frac12-2b_{1}}
			(N_{2}N_{3})^{1-\alpha}
			\|\mathbf{P}_{N_{1}}u\|_{X^{0,b}}
			\\
			\lesssim
			R^{2}T^{-2(b-\frac12)}
			N_{1}^{\frac12-2b_{1}-s_0+2\delta_0(1-\alpha)}\|\mathbf{P}_{N_1}v\|_{X^{s_0,b}}.
		\end{multline*}
		This is acceptable since~$b_{1}$ is close to~$\tfrac12$ and~$\delta_0$ is small.
		The dyadic sum is taken over~$N_1\sim N_0\gg N_2$, leading to the final bound
		\[
		R^2T^{-2(b-\frac{1}{2})}N_2^{-s_0-\delta_0}\|v\|_{X^{s_0,b}}.
		\]
		
		\emph{Case~2: $N_{\max}=N_{2}$.}
		This regime is favorable, as we gain a negative power of~$N_{\max}$.
		Using Lemma~\ref{lem:tens},
		\[
		N_1^{-s_0}\Lambda_{\vec{N}}
		\lesssim_{\varepsilon}
		\begin{cases}
			N_1^{1-s_0}N_3^{\frac{1}{2}-\alpha}
			N_2^{1-2\alpha+\delta_0+\varepsilon}, & N_1\leq N_3,\; N_3\gtrsim N_2^{1/2},\\
			N_3^{1-\alpha}N_1^{\frac{1}{2}-s_0}
			N_2^{1-2\alpha+\delta_0+\varepsilon}, & N_3\leq N_1,\; N_3\gtrsim N_2^{1/2},\\
			N_1^{1-s_0}N_3^{-\alpha}N_2^{1-2\alpha+\delta_0+\varepsilon}, & N_1\leq N_3\ll N_2^{1/2},\\
			N_1^{-s_0}N_3^{1-\alpha}N_2^{1-2\alpha+\delta_0+\varepsilon}, & N_3\leq N_1\ll N_2^{1/2},\\
	        N_3^{1-\alpha}N_1^{\frac{1}{2}-s_0}
			N_2^{1-2\alpha+\delta_0+\varepsilon}, & N_3\ll N_2^{1/2}\lesssim N_1\leq N_2.
		\end{cases}
		\]
	Note that we do not distinguish whether $0\leq s_0\leq \frac{1}{2}$ or $\frac{1}{2}<s_0<1$. The dyadic summation in \eqref{din'} is bounded by
		\[
		\|a_{N_1}\|_{\ell_{N_1}^2}\cdot N_2^{1-2\alpha+\delta_0+\varepsilon}\big[N_3^{\frac{3}{2}-\alpha-s_0}+N_3^{1-\alpha}N_2^{\frac{1}{2}-s_0}\big],\quad \text{if } N_3\gtrsim N_2^{1/2}.
		\]
		If $N_3\ll N_2^{1/2}$, the same dyadic summation is bounded by
		\begin{align*}
		&\|a_{N_1}\|_{\ell_{N_1}^2}\big[N_3^{1-s_0-\alpha}N_2^{1-2\alpha+\delta_0+\varepsilon}
		+N_3^{1-\alpha}\big(N_2^{\frac{1}{2}(\frac{1}{2}-s_0)}N_2^{1-2\alpha+\delta_0+\varepsilon}+N_2^{\frac{1}{2}-s_0}N_2^{1-2\alpha+\delta_0+\varepsilon} \big)
		\big]\\
		\lesssim & 
		\|a_{N_1}\|_{\ell_{N_1}^2}\cdot \big[N_2^{1-2\alpha+\delta_0+\varepsilon}+N_2^{\frac{1}{2}(1-s_0-\alpha)+1-2\alpha+\delta_0+\varepsilon}+N_2^{\frac{7}{4}-\frac{5\alpha}{2}-\frac{s_0}{2}+\delta_0+\varepsilon }
		+N_2^{2-\frac{5\alpha}{2}-s_0+\delta_0+\varepsilon}  \big],
		\end{align*}
	where we distinguish two cases according to the sign of $1-s_0-\alpha$.

		For~$\alpha>\frac{5}{6}+\delta_0$ and~$0\leq s_0<2\alpha-1-2\delta_0$, a straightforward check shows that~\eqref{din'} is bounded by
		\[
		R^2\|a_{N_1}\|_{\ell_{N_1}^2}\cdot N_2^{-s_0-\delta_0}.
		\]
		This completes the proof.
	\end{proof}
	
	
	
	\section{Multilinear estimates I: Trilinear interactions}\label{sec:trilinear}

In this section, we establish the key trilinear estimates that will be used in the sequel. We treat separately the non-resonant interactions, where the modulation is large, and the partially resonant interactions, where two of the frequencies coincide.

\subsection{Non-resonant interactions}

We begin with the non-resonant case. The following proposition provides bounds on the trilinear nonlinearity~$\mathcal{N}_{[123]}^{(3)}$ when all four frequency parameters are localized.

\begin{proposition}[Non-resonant trilinear estimates]\label{prop:non-resonant}
	Given $\vec{K}=(K_{1},K_{2},K_{3},K_4)$, denote by
	$K_{(1)}\geq K_{(2)}\geq K_{(3)}\geq K_{(4)}$ the non-increasing rearrangement of~$K_1,K_2,K_3,K_4$.
	Denote
	\[
	\Gamma_{K_1K_2K_3K_4}=
	\begin{cases}
		K_{(4)},& \text{if } K_{(1)}\sim K_{(2)},\\
		K_{(1)}^{-2},& \text{if } K_{(1)}\gg K_{(2)}.
	\end{cases}
	\]
	Then for all $b_{1},b>\tfrac12$, $q\ge2$, and $\gamma>1-\tfrac1q$,
	and for all
	$v^{(1)},v^{(2)},v^{(3)}\in X^{0,b}$, we have
	\begin{align*}
		\Big\|
		\chi\,\mathbf{P}_{K_{4}}
		\mathcal{N}_{[123]}^{(3)}
		(\mathbf{P}_{K_{1}}v^{(1)},\mathbf{P}_{K_{2}}v^{(2)},\mathbf{P}_{K_{3}}v^{(3)})
		\Big\|_{X^{0,-b_{1}}} \leq \Upsilon_{K_1K_2K_3K_4}\cdot \prod_{j=1}^3\|\mathbf{P}_{K_j}v^{(j)}\|_{X^{0,b}},
	\end{align*}
	where for any $\epsilon>0$,
	\begin{align}
		\Upsilon_{K_1K_2K_3K_4}&\leq C\;\Gamma_{K_1K_2K_3K_4},\quad \text{if}\quad K_{(1)}\sim K_{(2)}\gg K_{(3)}^2, \label{eq:tri1'}
		\\
		\label{eq:tri1}
		&\leq C_{\epsilon}\;\Gamma_{K_1K_2K_3K_4}\; K_{(1)}^{2(1-\alpha)}K_{(3)}^{\epsilon}\quad \text{otherwise}.
	\end{align}

	Moreover, for any $j_0\in\{1,2,3\}$ and $j_1\in\{1,2,3,4\}\setminus\{j_0\}$, we have
	\begin{multline*}
		\Big\|
		\chi\,\mathbf{P}_{K_{4}}
		\mathcal{N}_{[123]}^{(3)}
		(\mathbf{P}_{K_{1}}v^{(1)},\mathbf{P}_{K_{2}}v^{(2)},\mathbf{P}_{K_{3}}v^{(3)})
		\Big\|_{X^{0,-b_{1}}} \\
		\leq \widetilde{\Upsilon}_{K_1K_2K_3K_4}\cdot \|\mathbf{P}_{K_{j_0}}v^{(j_0)}\|_{X_{q,q}^{0,\gamma}}\cdot \prod_{j\neq j_0}\|\mathbf{P}_{K_j}v^{(j)}\|_{X^{0,b}},
	\end{multline*}
	where for any $\epsilon>0$,
	\begin{align}\label{eq:tri2}
		\widetilde{\Upsilon}_{K_1K_2K_3K_4}\leq C_{\epsilon}\;\Gamma_{K_1K_2K_3K_4}\cdot K_{(2)}^{\epsilon}M_1^{1-\alpha}\min\Big(1+\frac{M_2}{M_1^{1/2}},\; M_1^{1-\alpha}M_3^{1/2}\Big),
	\end{align}
	where $M_1\geq M_2\geq M_3$ is the non-increasing rearrangement of~$\{K_j:\; j\neq j_1\}$.
\end{proposition}

\begin{proof}
	By duality, it suffices to bound
\begin{multline*}
	\sup_{\|v^{(4)}\|_{X^{0,b_{1}}}\le1}
	\left|
	\sum_{\substack{n_{i}\sim K_{i}\\ \{n_{1},n_{3}\}\cap\{n_{2},n_4\}=\emptyset}}
	\gamma_{n_{1}n_{2}n_{3}n_4}
	\int
	\widehat{\chi}(\widetilde{\kappa}-\Omega(\vec n))\right.\\
	\left.\times\,
	\widetilde{v_{n_{1}}^{(1)}}(\kappa_{1})
	\ov{\widetilde{v_{n_{2}}^{(2)}}}(\kappa_{2})
	\widetilde{v_{n_{3}}^{(3)}}(\kappa_{3})
	\ov{\widetilde{v_{n_4}^{(4)}}}(\kappa_{4})
	\mathrm{d}\vec\kappa
	\right|,
\end{multline*}
	where $\widetilde{\kappa}=\kappa_1-\kappa_2+\kappa_3-\kappa_4$.
	To prove~\eqref{eq:tri1}, we rewrite this as
	\begin{multline}
		\label{eq:1}
		\sup_{\|v^{(4)}\|_{X^{0,b_{1}}}\le1}
		\Bigg|
		\sum_{\substack{n_{i}\sim K_{i}\\ \{n_{1},n_{3}\}\cap\{n_{2},n_4\}=\emptyset}}
		\gamma_{n_{1}n_{2}n_{3}n_4}
		\int
		\widehat{\chi}(\widetilde{\kappa}-\Omega(\vec n))
		\\
		\times
		f_{n_{1}}^{(1)}(\kappa_{1})
		\ov{f_{n_{2}}^{(2)}}(\kappa_{2})
		f_{n_{3}}^{(3)}(\kappa_{3})
		\ov{f_{n_4}^{(4)}}(\kappa_{4})
		\Bigl(\prod_{i=1}^{3}\langle\kappa_{i}\rangle^{-b}\Bigr)
		\frac{\mathrm{d}\vec\kappa}{\langle\kappa_{4}\rangle^{b_{1}}}
		\Bigg|,
	\end{multline}
	where
	\[
	f_{n_{4}}^{(4)}(\kappa_{4})=\langle\kappa_{4}\rangle^{b_{1}}\widetilde{v_{n_{4}}^{(4)}}(\kappa_{4}),
	\qquad
	f_{n_{i}}^{(i)}(\kappa_{i})=\langle\kappa_{i}\rangle^{b}\widetilde{v_{n_{i}}^{(i)}}(\kappa_{i}),\quad i\in\{1,2,3\}.
	\]
	For fixed~$\vec\kappa$, we apply~\eqref{Strichartztype1} with
	$\mathbf{a}_{n_{j}}^{(j)}:=f_{n_{j}}^{(j)}(\kappa_{j})$
	and control the correlation factor~$\gamma$ by~$\Gamma_{K_1K_2K_3K_4}$ using Lemma~\ref{lem:gamma}.
	The bound~\eqref{eq:tri1} then follows from H\"older's inequality and the assumption~$b,b_{1}>\tfrac12$.

	The proof of~\eqref{eq:tri2} is similar: we replace the~$X^{0,b}$ norm of~$v^{(j_0)}$ by its Fourier--Lebesgue norm and apply~\eqref{Strichartztype2} instead.
\end{proof}

\subsection{Partially resonant interactions}

We now turn to the partially resonant interactions, in which two of the spatial frequencies in the nonlinearity coincide. In this regime, the resonance function degenerates to a difference of two eigenvalues, and the modulation analysis requires a different approach.

The following elementary lemma controls bilinear sums restricted to level sets of the resonance function.

\begin{lemma}\label{Schurtype}
	For positive sequences $(\mathbf{h}_n),(\mathbf{a}_n)$ in~$\ell^2$, we have
	\[
	\sup_{m\in \Z}\sum_{\substack{n_0\sim N_0,\, n_1\sim N_1\\
			n_0\neq n_1}} \mathbf{1}_{I_m}(|\lambda_{n_0}^2-\lambda_{n_1}^2|)\mathbf{h}_{n_0}\mathbf{a}_{n_1}\lesssim \Big(\frac{M}{K}\Big)^{\frac{1}{2}}\|\mathbf{h}_n\|_{\ell^2}\|\mathbf{a}_n\|_{\ell^2},
	\]
	where $K=\min(N_0,N_1)$ and $M=\max(N_0,N_1)$, and $I_m=\Big[\frac{m}{2}M,\frac{m+1}{2}M \Big]$.
\end{lemma}

\begin{proof}
	Without loss of generality, assume $K=N_0<N_1=M$. For fixed~$n_0$ and distinct~$n_1,n_1'\sim N_1$,
	\[
	|\lambda_{n_1}^2-\lambda_{n_1'}^2|\geq |(n_1-n_1')(n_1+n_1')|-\mathcal{O}(N_1^{2(1-\alpha)})>\frac{N_1}{2}
	\]
	for sufficiently large~$N_1$, since $\alpha>\frac{1}{2}$. Therefore, for any~$m$,
	\[
	\sup_{n_0}\sum_{\substack{n_1\sim N_1\\
			n_1\neq n_0}}\mathbf{1}_{I_m}(|\lambda_{n_0}^2-\lambda_{n_1}^2|)\lesssim 1.
	\]
	Similarly, for fixed~$n_1$ and distinct~$n_0,n_0'\sim N_0$, we have
	\[
	|\lambda_{n_0}^2-\lambda_{n_0'}^2|>\frac{N_0}{2}
	\]
	for sufficiently large~$N_0$, hence
	\[
	\sup_{n_1}\sum_{\substack{n_0\sim N_0\\
			n_0\neq n_1}}\mathbf{1}_{I_m}(|\lambda_{n_0}^2-\lambda_{n_1}^2|)\lesssim \frac{N_1}{N_0}.
	\]
	The desired estimate then follows from Schur's test.
\end{proof}

Recall that the partially resonant nonlinearities are given by
\begin{align*}
	\big(\mathcal{N}_{(23)}^{(3)}\big)_n(Z,F,G)
	&:=\sum_{\substack{n_1,n_2\\
			n_1\neq n}} \gamma_{n_1 n_2 n_2 n}\, Z_{n_1}\ov{F_{n_2}}G_{n_2}\,,\\
	\big(Z\dotcirc\mathcal{N}_{[23]}^{(2)}(F,G)\big)_{n}
	&:=2Z_{n}\sum_{\substack{n_{2},n_{3}\\n_{2}\neq n_{3}}}
	\gamma_{n n n_{2}n_{3}}\, \overline{F_{n_{2}}}G_{n_{3}}\,.
\end{align*}

The next proposition provides the key estimates for these partially resonant terms. It relies on a modulation analysis that allows us to gain a power of the unpaired frequencies from the large size of the resonance function.

\begin{proposition}\label{linear:type1}
	Assume $b>\frac{1}{2}>b_1>\frac{1}{4}$ with $b+b_1<1$. Let $\alpha\in\big(\frac{1}{2},1\big)$ satisfy
	\[
	2(1-\alpha)<2b_1-\frac{1}{2}.
	\]
	Define
	\[
	A(N_0,N_1,N_2):= \min(N_0,N_1,N_2)\, \frac{(N_0\vee N_1)^{1-2b_1}}{(N_0 \wedge N_1)^{\frac{1}{2}}}\,.
	\]
	Then for all $Z,F,G\in X^{0,b}$,
	\begin{multline}
		\label{resonant:1}
		\left\|
		\chi(t)\,\mathbf{P}_{N_0}\mathcal{N}_{(23)}^{(3)}(\mathbf{P}_{N_1}Z,\mathbf{P}_{N_2}F,\mathbf{P}_{N_2}G)
		\right\|_{X^{0,-b_1}}
		\lesssim A(N_0,N_1,N_2)\\
		\times\|\mathbf{P}_{N_1}Z\|_{X^{0,b}}\|\mathbf{P}_{N_2}F\|_{X^{0,b}}\|\mathbf{P}_{N_2}G\|_{X^{0,b}}\,,
	\end{multline}
	and
	\begin{multline}\label{resonant:2}
		\left\|
		\chi(t)\,\mathbf{P}_{N_1}
        \big(
		\mathbf{P}_{N_1}Z\dotcirc\mathcal{N}_{[23]}^{(2)}(\mathbf{P}_{N_2}F,\mathbf{P}_{N_3}G)
		\big)
		\right\|_{X^{0,-b_1}}
		\lesssim A(N_2,N_3,N_1)
		\\
		\times\|\mathbf{P}_{N_1}Z\|_{X^{0,b}}\|\mathbf{P}_{N_2}F\|_{X^{0,b}}\|\mathbf{P}_{N_3}G\|_{X^{0,b}}\,.
	\end{multline}
\end{proposition}

\begin{proof}
	To simplify notation, we write $Z=\mathbf{P}_{N_1}Z$, $F=\mathbf{P}_{N_2}F$, and $G=\mathbf{P}_{N_2}G$. We introduce
	\[
	\vec{\kappa}=(\kappa_0,\kappa_1,\kappa_2,\kappa_3),\qquad
	\widetilde{\kappa}:=\kappa_1-\kappa_2+\kappa_3-\kappa_0.
	\]
	We only prove~\eqref{resonant:1}, since the proof of~\eqref{resonant:2} is almost identical by switching the pairs of distinct indices from $n_0\neq n_1,n_2=n_3$ to $n_2\neq n_3, n_0=n_1$. By duality, we must show that for any $H=\mathbf{P}_{N_0}H\in L_{t,x}^2$ with $\|\widetilde{H}_n(\kappa)\|_{\ell_n^2 L_{\kappa}^2}\leq 1$,
	\begin{multline}
		\label{linear:type1-1}
		\int_{\mathbb{R}^4} \mathrm{d}\vec{\kappa} \sum_{\substack{n_0,n_1,n_2\\
				n_0\neq n_1}} \gamma_{n_0 n_1 n_2 n_2}\cdot \frac{\widehat{\chi}(\widetilde{\kappa}-\Omega(\vec{n}))}{\la\kappa_0\ra^{b_1}}\,\ov{\widetilde{H}_{n_0}}(\kappa_0)\,\widetilde{Z}_{n_1}(\kappa_1)\,\ov{\widetilde{F}_{n_2}}(\kappa_2)\,\widetilde{G}_{n_2}(\kappa_3) \\
		\lesssim A(N_0,N_1,N_2)\,\|Z\|_{X^{0,b}}\|F\|_{X^{0,b}}\|G\|_{X^{0,b}},
	\end{multline}
	where the sum runs over $n_0\sim N_0$, $n_1\sim N_1$, $n_2\sim N_2$.
	The absolute value of the left-hand side of~\eqref{linear:type1-1} is bounded by
	\begin{align*}
		\int_{\mathbb{R}^4}\mathrm{d}\vec{\kappa}\sum_{\substack{n_0,n_1\\
				n_0\neq n_1}}\min(N_0,N_1,N_2) \frac{|\widehat{\chi}(\widetilde{\kappa}-\Omega(\vec{n}))|}{\la\kappa_0\ra^{b_1}} |\widetilde{H}_{n_0}(\kappa_0)\,\widetilde{Z}_{n_1}(\kappa_1)|\, \|\widetilde{F}_{n_2}(\kappa_2)\|_{\ell_{n_2}^{2}}
		\|\widetilde{G}_{n_2}(\kappa_3)\|_{\ell_{n_2}^{2}}.
	\end{align*}
	We introduce
	\[
	\mathbf{a}_{n_1}(\kappa_1):=\la\kappa_1\ra^b|\widetilde{Z}_{n_1}(\kappa_1)|,\quad \mathbf{f}(\kappa_2):=\la\kappa_2\ra^b\|\widetilde{F}_{n_2}(\kappa_2)\|_{\ell_{n_2}^{2}},\quad \mathbf{g}(\kappa_3):= \la\kappa_3\ra^b\|\widetilde{G}_{n_2}(\kappa_3)\|_{\ell_{n_2}^{2}}.
	\]
	We need to estimate
	\begin{equation}
		\label{desired}
		\min(N_0,N_1,N_2)\cdot\int_{\mathbb{R}^4}\mathrm{d}\vec{\kappa}\;
		\frac{\mathbf{f}(\kappa_2)\,\mathbf{g}(\kappa_3)}{\la\kappa_2\ra^b\la\kappa_3\ra^b}
		\sum_{\substack{n_0,n_1\\
				n_0\neq n_1}}|\widehat{\chi}(\widetilde{\kappa}-\Omega(\vec{n}))|\,\frac{
			|\widetilde{H}_{n_0}(\kappa_0)|\,\mathbf{a}_{n_1}(\kappa_1)
		}{\la\kappa_0\ra^{b_1}\la\kappa_1\ra^b}\,.
	\end{equation}
	Set $K=\min(N_0,N_1)$ and $M=\max(N_0,N_1)$. For $m\in [1,C_0 M]$, recall that in Lemma \ref{Schurtype}, we define
	\[
	I_m:=\Big[\frac{m}{2}M,\frac{m+1}{2}M\Big].
	\]
	Since $\alpha>\frac{1}{2}$ and $n_0\neq n_1$, for sufficiently large~$M$,
	\[
	|\Omega(\vec{n})|=|n_0^2-n_1^2+\mu_{n_0}^2-\mu_{n_1}^2|\geq |(n_0-n_1)(n_0+n_1)|-\max(\mu_{n_0}^2,\mu_{n_1}^2)\geq \frac{m}{2}M.
	\]

	A direct application of Lemma~\ref{Schurtype} to the sum over~$n_0,n_1$ does not give a satisfactory estimate for~\eqref{desired}, and we need a finer decomposition. We carefully choose the order of integration and summation.
	Define
	\[
	J_m=\Big[\frac{m-1/2}{2}M, \frac{m+3/2}{2}M\Big]\supset I_m.
	\]
	We control~\eqref{desired} by
	\begin{align}\label{S_m+R_m}
		\min(K,N_2)\sum_{m=1}^{C_0 M}\int_{\mathbb{R}^4}\big(S_m(\vec{\kappa}) +R_m(\vec{\kappa})\big)\, \mathrm{d}\vec{\kappa},
	\end{align}
	where
	\[
	S_m(\vec{\kappa}):=\sum_{\substack{n_0,n_1\\
			n_0\neq n_1}}|\widehat{\chi}(\widetilde{\kappa}-\Omega(\vec{n}))|\,\mathbf{1}_{I_m}(\Omega(\vec{n}))\,\mathbf{1}_{J_m}(\widetilde{\kappa})\cdot\frac{
		|\widetilde{H}_{n_0}(\kappa_0)|\,\mathbf{a}_{n_1}(\kappa_1)\,
		\mathbf{f}(\kappa_2)\,\mathbf{g}(\kappa_3)}{\la\kappa_0\ra^{b_1}\la\kappa_1\ra^b\la\kappa_2\ra^b\la\kappa_3\ra^b},
	\]
	and
	\[
	R_m(\vec{\kappa}):=\sum_{\substack{n_0,n_1\\
			n_0\neq n_1}}|\widehat{\chi}(\widetilde{\kappa}-\Omega(\vec{n}))|\,\mathbf{1}_{I_m}(\Omega(\vec{n}))\,\mathbf{1}_{\mathbb{R}\setminus J_m}(\widetilde{\kappa})\cdot\frac{
		|\widetilde{H}_{n_0}(\kappa_0)|\,\mathbf{a}_{n_1}(\kappa_1)\,
		\mathbf{f}(\kappa_2)\,\mathbf{g}(\kappa_3)}{\la\kappa_0\ra^{b_1}\la\kappa_1\ra^b\la\kappa_2\ra^b\la\kappa_3\ra^b}.
	\]
	By the rapid decay of~$\widehat{\chi}$ away from the origin, for any $A>1$,
	\begin{align}\label{Rm}
		\int_{\mathbb{R}^4} \sum_{m=1}^{C_0 M}|R_m(\vec{\kappa})|\, \mathrm{d}\vec{\kappa} &\leq C_A M^{-A}\|\mathbf{a}_{n_1}(\kappa_1)\|_{L_{\kappa_1}^2\ell_{n_1}^2}
		\|\mathbf{f}(\kappa_2)\|_{L_{\kappa_2}^2}\|\mathbf{g}(\kappa_3)\|_{L_{\kappa_3}^2}\notag \\
		&\leq C_A M^{-A}\|Z\|_{X^{0,b}}\|F\|_{X^{0,b}}\|G\|_{X^{0,b}}.
	\end{align}
	It remains to estimate the main contribution $\sum_{m=1}^{C_0 M}\int_{\mathbb{R}^4}S_m(\vec{\kappa})\,\mathrm{d}\vec{\kappa}$.

	On~$J_m$, we have
	\[
	|\widetilde{\kappa}|=|\kappa_0-\kappa_1+\kappa_2-\kappa_3|\geq \frac{mM}{4}.
	\]
	Since $0<b_1<b$, for $\widetilde{\kappa}\in J_m$ we obtain
	\[
	\frac{1}{\la\kappa_0\ra^{b_1}\prod_{j=1}^3\la\kappa_j\ra^b}\lesssim \sum_{j=0}^3\frac{1}{(mM)^{b_1}}\cdot\frac{\la\kappa_j\ra^{b}}{\prod_{i=0}^3\la\kappa_i\ra^b}.
	\]
	This gives $S_m(\vec{\kappa})\lesssim \sum_{j=0}^3 S_m^{(j)}(\vec{\kappa})$, where
	\begin{multline*}
	S_m^{(j)}(\vec{\kappa}):=
	\frac{\mathbf{1}_{J_m}(\widetilde{\kappa})}{(mM)^{b_1}}\,\frac{\la\kappa_j\ra^b}{\prod_{i=0}^3\la\kappa_i\ra^b}\sum_{\substack{n_0\neq n_1}}|\widehat{\chi}(\widetilde{\kappa}-\Omega(\vec{n}))|\,\mathbf{1}_{I_m}(\Omega(\vec{n}))\\
	\times|\widetilde{H}_{n_0}(\kappa_0)|\,
	\mathbf{a}_{n_1}(\kappa_1)\cdot \mathbf{f}(\kappa_2)\,\mathbf{g}(\kappa_3).
\end{multline*}
	To estimate the integration, we first integrate in the~$\kappa_j$ variable using Cauchy--Schwarz, then carry out the sum over~$n_0\neq n_1$ via Lemma~\ref{Schurtype}, and finally integrate over the remaining variables~$\kappa_i$, $i\neq j$.
	Using $\|\widehat{\chi}(\widetilde{\kappa}-\Omega(\vec{n}))\|_{L_{\kappa_j}^2}\lesssim 1$ together with Lemma~\ref{Schurtype}, we obtain
	\begin{align*}
		\int_{\mathbb{R}}S_m^{(0)}(\vec{\kappa})\, \mathrm{d}\kappa_0 &\lesssim \frac{1}{(mM)^{b_1}}\Big(\frac{M}{K}\Big)^{\frac{1}{2}}
		\|\mathbf{1}_{J_m}(\widetilde{\kappa})\, \widetilde{H}_{n_0}(\kappa_0)\|_{L_{\kappa_0}^2\ell_{n_0}^2}\,\frac{\|\mathbf{a}_{n_1}(\kappa_1)\|_{\ell_{n_1}^2}\,
			\mathbf{f}(\kappa_2)\,\mathbf{g}(\kappa_3)
		}{\prod_{i\neq 0}\la\kappa_i\ra^b},\\
		\int_{\mathbb{R}}S_m^{(1)}(\vec{\kappa})\,\mathrm{d}\kappa_1 &\lesssim \frac{1}{(mM)^{b_1}}\Big(\frac{M}{K}\Big)^{\frac{1}{2}}
		\|\mathbf{1}_{J_m}(\widetilde{\kappa})\, \mathbf{a}_{n_1}(\kappa_1)\|_{L_{\kappa_1}^2\ell_{n_1}^2}\,\frac{\|\widetilde{H}_{n_0}(\kappa_0)\|_{\ell_{n_0}^2}\,
			\mathbf{f}(\kappa_2)\,\mathbf{g}(\kappa_3)
		}{\prod_{i\neq 1}\la\kappa_i\ra^b},\\
		\int_{\mathbb{R}}S_m^{(2)}(\vec{\kappa})\,\mathrm{d}\kappa_2 &\lesssim
		\frac{1}{(mM)^{b_1}}\Big(\frac{M}{K}\Big)^{\frac{1}{2}}\|\mathbf{1}_{J_m}(\widetilde{\kappa})\,\mathbf{f}(\kappa_2)\|_{L_{\kappa_2}^2}\,\frac{\|\widetilde{H}_{n_0}(\kappa_0)\|_{\ell_{n_0}^2}\,\|\mathbf{a}_{n_1}(\kappa_1)\|_{\ell_{n_1}^2}\,\mathbf{g}(\kappa_3)}{\prod_{i\neq 2}\la\kappa_i\ra^b},\\
		\int_{\mathbb{R}}S_m^{(3)}(\vec{\kappa})\,\mathrm{d}\kappa_3 &\lesssim
		\frac{1}{(mM)^{b_1}}\Big(\frac{M}{K}\Big)^{\frac{1}{2}}\|\mathbf{1}_{J_m}(\widetilde{\kappa})\,\mathbf{g}(\kappa_3)\|_{L_{\kappa_3}^2}\,\frac{\|\widetilde{H}_{n_0}(\kappa_0)\|_{\ell_{n_0}^2}\,\|\mathbf{a}_{n_1}(\kappa_1)\|_{\ell_{n_1}^2}\,\mathbf{f}(\kappa_2)}{\prod_{i\neq 3}\la\kappa_i\ra^b}.
	\end{align*}
	Summing over $m\leq C_0 M$, applying Cauchy--Schwarz, and integrating over the remaining variables, we obtain
	\begin{align*}
		&\sum_{j=0}^{3}\int_{\mathbb{R}^3}\prod_{i\neq j}\mathrm{d}\kappa_i \sum_{m=1}^{C_0 M}\int_{\mathbb{R}}S_m^{(j)}(\vec{\kappa})\, \mathrm{d}\kappa_j\\ &\lesssim \frac{M^{\frac{1}{2}-b_1}}{K^{\frac{1}{2}}}\Big(\sum_{m=1}^{C_0 M}\frac{1}{m^{2b_1}}\Big)^{\frac{1}{2}}\|\widetilde{H}_{n_0}(\kappa_0)\|_{\ell_{n_0}^2 L_{\kappa_0}^2}\,\|\mathbf{a}_{n_1}(\kappa_1)\|_{\ell_{n_1}^2 L_{\kappa_1}^2}\,\|\mathbf{f}(\kappa_2)\|_{L_{\kappa_2}^2}\,\|\mathbf{g}(\kappa_3)\|_{L_{\kappa_3}^2}\\
		&\lesssim
		\frac{M^{1-2b_1}}{K^{\frac{1}{2}}}\,\|Z\|_{X^{0,b}}\|F\|_{X^{0,b}}\|G\|_{X^{0,b}}.
	\end{align*}
	Combining this with~\eqref{S_m+R_m} and~\eqref{Rm} completes the proof.
\end{proof}

	
	A similar argument yields the following estimates for the non-resonant interactions in the regime where two frequencies are much larger than the others.

\begin{proposition}[High-high-very low estimates]\label{prop:highhighverylow}
	Assume $b>\frac{1}{2}>b_1>\frac{1}{4}$ with $b+b_1<1$. Let $\alpha\in \big(\frac{1}{2},1\big)$ satisfy
	\[
	2(1-\alpha)<2b_1-\frac{1}{2}.
	\]
	Then for all $F_1,F_2,F_3\in X^{0,b}$, if $N_{(3)}\leq N_{(1)}^{\epsilon_0}\sim N_{(2)}^{\epsilon_0}$ with $\epsilon_0<2^{-10}$, we have
	\[
	\left\|
	\chi(t)\,\mathbf{P}_{N_0}\mathcal{N}_{[123]}^{(3)}(\mathbf{P}_{N_1}F_1,\mathbf{P}_{N_2}F_2,\mathbf{P}_{N_3}F_3)
	\right\|_{X^{0,-b_1}}\leq A'(N_0,N_1,N_2,N_3)\prod_{j=1}^3\|\mathbf{P}_{N_j}F_j\|_{X^{0,b}},
	\]
	where
	\[
	A'(N_0,N_1,N_2,N_3)\lesssim N_{(1)}^{\frac{1}{2}-2b_1}N_{(3)}^{\frac{1}{2}}N_{(4)}^{\frac{1}{2}}.
	\]
\end{proposition}

\begin{proof}
	The proof is very similar to that of Proposition~\ref{linear:type1}, so we only sketch the differences. To fix ideas, assume $N_0\sim N_1$ and $N_3\leq N_2\leq N_0^{\epsilon_0}$.

	By duality, it suffices to show that for any $H=\mathbf{P}_{N_0}H\in L_{t,x}^2$ with $\|\widetilde{H}_n(\kappa)\|_{\ell_n^2 L_{\kappa}^2}\leq 1$,
	\begin{align*}
		&\int_{\mathbb{R}^4}\mathrm{d}\vec{\kappa}\sum_{\substack{n_0,n_1,n_2,n_3\\
				\{n_0,n_2\}\cap \{n_1,n_3\}=\emptyset
		}}\gamma_{n_0 n_1 n_2 n_3}\cdot\frac{\widehat{\chi}(\widetilde{\kappa}-\Omega(\vec{n}))}{\langle\kappa_0\rangle^{b_1}}\,\ov{\widetilde{H}_{n_0}}(\kappa_0)\prod_{j=1}^3\widetilde{F}^{\pm}_{n_j}(\kappa_j) \\
		&\qquad\lesssim A'(N_0,N_1,N_2,N_3)\prod_{j=1}^3\|F_j\|_{X^{0,b}},
	\end{align*}
	where $F_j^+=F_j$ and $F_j^-=\ov{F}_j$.
	We introduce
	\[
	\mathbf{a}_{n_j}(\kappa_j)=\langle\kappa_j\rangle^b|\widetilde{F}_{n_j}(\kappa_j)|, \quad j=1,2,3.
	\]
	We need to estimate
	\begin{align*}
		\min(N_0,N_1,N_2,N_3)\cdot\int_{\mathbb{R}^4}\mathrm{d}\vec{\kappa}\sum_{\substack{n_0,n_1,n_2,n_3\\
				\{n_0,n_2\}\cap\{n_1,n_3\}=\emptyset
		}}|\widehat{\chi}(\widetilde{\kappa}-\Omega(\vec{n}))|\cdot\frac{|\widetilde{H}_{n_0}(\kappa_0)|}{\langle\kappa_0\rangle^{b_1}}\prod_{j=1}^3\frac{\mathbf{a}_{n_j}(\kappa_j)}{\langle\kappa_j\rangle^b}.
	\end{align*}
	With $K=\min(N_0,N_1)\sim M=\max(N_0,N_1)$, we introduce the same intervals $I_m=[mM/2, (m+1)M/2]$ for $m\in[1,C_0 M]$ and the slight enlargement~$J_m$. We decompose as before into $S_m(\vec{\kappa})$ and $R_m(\vec{\kappa})$, and further split
	\[
	S_m(\vec{\kappa})\lesssim \sum_{j=0}^3 S_m^{(j)}(\vec{\kappa}),
	\]
	where
\begin{multline*}
		S_m^{(j)}(\vec{\kappa}):=\frac{\mathbf{1}_{J_m}(\widetilde{\kappa})}{(mM)^{b_1}}\,\frac{\langle\kappa_j\rangle^b}{\prod_{i=0}^3 \langle\kappa_i\rangle^{b}}\!\!\!\!\sum_{\substack{n_0,n_1,n_2,n_3\\
				\{n_0,n_2\}\cap\{n_1,n_3\}=\emptyset
		}}\!\!\!\!
		|\widehat{\chi}(\widetilde{\kappa}-\Omega(\vec{n}))|\,\mathbf{1}_{I_m}(\Omega(\vec{n}))\\
		\times|\widetilde{H}_{n_0}(\kappa_0)|\prod_{j=1}^3\mathbf{a}_{n_j}(\kappa_j).
	\end{multline*}
	To estimate $\int_{\mathbb{R}^4}S_m^{(j)}(\vec{\kappa})\,\mathrm{d}\vec{\kappa}$, we first integrate the~$\kappa_j$ variable using Cauchy--Schwarz, then sum over~$n_0\neq n_1$, then over~$n_2,n_3$, and finally integrate over the remaining~$\kappa_i$, $i\neq j$. For the sum over~$n_0\neq n_1$, we use the following variant of Lemma~\ref{Schurtype}:
	\begin{align}\label{Schurtypevariant}
		\sup_{\substack{m\in[1,C_0 M]\\
				|z|\leq M^{\epsilon_0}
		}}\sum_{\substack{n_0\sim M,\, n_1\sim M\\
				n_0\neq n_1
		}} \mathbf{1}_{I_m}(|\lambda_{n_0}^2-\lambda_{n_1}^2-z|)\,\mathbf{h}_{n_0}\mathbf{a}_{n_1}\lesssim \|\mathbf{h}_n\|_{\ell^2}\|\mathbf{a}_n\|_{\ell^2}.
	\end{align}
	Indeed, since $M^{\epsilon_0}\ll M$, the same argument as in Lemma~\ref{Schurtype} gives
	\[
	|\lambda_{n_0}^2-\lambda_{n_1}^2-z|>\frac{M}{2}
	\]
	for sufficiently large~$M$, and Schur's test applies with a bounded kernel.

	Finally, we sum over~$n_2,n_3$ using Cauchy--Schwarz in~$\ell^2$, which costs an extra factor of~$N_2^{1/2}N_3^{1/2}$ through the bound
	\[
	\|\mathbf{1}_{n_j\sim N_j}\mathbf{a}_{n_j}(\kappa_j)\|_{\ell_{n_j}^1}\lesssim N_j^{1/2}\|\mathbf{a}_{n_j}(\kappa_j)\|_{\ell_{n_j}^2}, \quad j=2,3.
	\]
	Combining these estimates yields the claimed bound.
\end{proof}

As a consequence of Propositions~\ref{linear:type1} and~\ref{prop:highhighverylow}, we obtain the following global trilinear estimates for the partially resonant nonlinearities.

\begin{corollary}\label{cor:pres1}
	Assume $b>\frac{1}{2}>b_1>\frac{1}{4}$ with $b+b_1<1$. Let $\alpha\in\big(\frac{1}{2},1\big)$ satisfy
	\[
	2(1-\alpha)<2b_1-\frac{1}{2}.
	\]
	Then for any $\frac{3}{4}-b_1<s_1<\alpha-\frac{1}{2}$ and all $z\in X^{0,b}$, $F,G\in X^{s_1,b}$,
	\begin{align}\label{pres1-1}
		\Big\|
		\chi(t)\,\mathcal{N}_{(23)}^{(3)}(z,F,G)
		\Big\|_{X^{0,-b_1}}\lesssim_{\chi} \|z\|_{X^{0,b}}\|F\|_{X^{s_1,b}}\|G\|_{X^{s_1,b}}.
	\end{align}
	Similarly,
	\begin{align}\label{pres1-2}
		\Big\|
		\chi(t)\,\mathcal{N}_{(23)}^{(3)}(F,G,z)
		\Big\|_{X^{0,-b_1}}
		+
		\Big\|
		\chi(t)\,\mathcal{N}_{(23)}^{(3)}(F,z,G)
		\Big\|_{X^{0,-b_1}}
		\lesssim_{\chi} \|z\|_{X^{0,b}}\|F\|_{X^{s_1,b}}\|G\|_{X^{s_1,b}}.
	\end{align}
\end{corollary}

\begin{proof}
	We first prove~\eqref{pres1-1}. By dyadic decomposition and~\eqref{resonant:1}, the left-hand side is bounded by
	\begin{align*}
		\sum_{N_0,N_1,N_2}\frac{\max(N_0,N_1)^{1-2b_1}\min(N_0,N_1,N_2)}{\min(N_0,N_1)^{1/2}N_2^{2s_1}}\,\|\mathbf{P}_{N_1}z\|_{X^{0,b}}\|\mathbf{P}_{N_2}F\|_{X^{s_1,b}}\|\mathbf{P}_{N_2}G\|_{X^{s_1,b}}.
	\end{align*}
	Without loss of generality, we assume that the dyadic sum is restricted to the additional constraint :
	\begin{align}\label{res-constraint:dyadic}
		\max(N_0,N_1)\lesssim N_2+\min(N_0,N_1).
	\end{align}
	Indeed, if the constraint~\eqref{res-constraint:dyadic} is violated, the correlation bound~\eqref{eq:gamma3} allows us to replace the factor $\min(N_0,N_1,N_2)$ by $\max(N_0,N_1,N_2)^{-2}$, and the resulting dyadic sum converges.

	Now, we claim that, under the constraint \eqref{res-constraint:dyadic} and the conditions of $b_1,s_1$,
	\[
	\frac{\max(N_0,N_1)^{1-2b_1}\min(N_0,N_1,N_2)}{\min(N_0,N_1)^{1/2}N_2^{2s_1}} \lesssim \frac{\min(N_0,N_1,N_2)^{1-2s_1}}{\max(N_0,N_1,N_2)^{2b_1-1/2}},
	\]
	consequently, the dyadic summation converges provided $s_1>\frac{3}{4}-b_1$.
	\medskip
	
	Indeed, in the case $N_2\gtrsim \max(N_0,N_1)$, we have
	\begin{align*}
	\frac{\max(N_0,N_1)^{1-2b_1}\min(N_0,N_1,N_2) }{\min(N_0,N_1)^{1/2}N_2^{2s_1} }\sim &\frac{\max(N_0,N_1)^{1-2b_1}\min(N_0,N_1,N_2) }{\min(N_0,N_1,N_2)^{1/2}\max(N_0,N_1,N_2)^{2s_1} } \\
	\lesssim & \frac{\min(N_0,N_1,N_2)^{1/2} }{\max(N_0,N_1,N_2)^{2s_1+2b_1-1} },
	\end{align*}
where to the last step, we used the fact that $1-2b_1>0$. Now since $s_1>\frac{3}{4}-b_1$, we have $2s_1>\frac{1}{2}$, the desired inequality follows. Now if $N_2\ll \max(N_0,N_1)$, from \eqref{res-constraint:dyadic}, we must have $N_0\sim N_1\sim \max(N_0,N_1,N_2)$, hence
\begin{align*}
\frac{\max(N_0,N_1)^{1-2b_1}\min(N_0,N_1,N_2) }{\min(N_0,N_1)^{1/2}N_2^{2s_1} }\sim & \frac{\min(N_0,N_1,N_2)}{\max(N_0,N_1,N_2)^{2b_1-1/2}\min(N_0,N_1,N_2)^{2s_1} },
\end{align*}
as desired.
\medskip

	For~\eqref{pres1-2}, we may still assume the constraint  \eqref{res-constraint:dyadic}.  the factor~$N_2^{2s_1}$ is replaced by~$N_1^{s_1}N_2^{s_1}$, which gives the weaker bound
	\[
	\frac{\max(N_0,N_1)^{1-2b_1}\min(N_0,N_1,N_2)}{\min(N_0,N_1)^{1/2}N_1^{s_1}N_2^{s_1}} \lesssim \frac{\min(N_0,N_1,N_2)^{\frac{1}{2}-s_1} }{\max(N_0,N_1,N_2)^{s_1+2b_1-1}}.
	\]
Indeed, if $N_2\gtrsim \max(N_0,N_1)$, we have
\begin{align*}
\frac{\max(N_0,N_1)^{1-2b_1}\min(N_0,N_1,N_2) }{\min(N_0,N_1)^{1/2}N_1^{s_1}N_2^{s_1} }\lesssim & \frac{\max(N_0,N_1,N_2)^{1-2b_1}\min(N_0,N_1,N_2)^{1/2-s_1} }{\max(N_0,N_1,N_2)^{s_1}},
\end{align*}
since $1/2>b_1>\frac{1}{4}$ and $s_1+b_1>3/4$, we have $0<1-2b_1<s_1$, the desired estimate follows.  If $N_2\ll \max(N_0,N_1)$, we must have $N_0\sim N_1, N_2\sim \min(N_0,N_1,N_2)$, we have
\begin{align*}
\frac{\max(N_0,N_1)^{1-2b_1}\min(N_0,N_1,N_2) }{\min(N_0,N_1)^{1/2}N_1^{s_1}N_2^{s_1}}\lesssim & \frac{\max(N_0,N_1,N_2)^{1-2b_1}\min(N_0,N_1,N_2)^{1/2-s_1} }{\max(N_0,N_1,N_2)^{s_1}},
\end{align*}
and the desired estimate follows.
Finally, the condition $s_1>\frac{3}{4}-b_1$ ensures convergence of the dyadic sum in~$N_0,N_1,N_2$. The proof of Corollary \ref{cor:pres1} is now complete. 
\end{proof}
	
	\subsection{Deterministic operator bounds}

We prove in this section the deterministic multilinear estimates needed to close the fixed-point problem for~$\mathcal{T}^{N,\dag}_{n}$, the random averaging operator. Recall that
\[
\mathcal{N}_n^{(2)}
\left(
f^{(2)},f^{(3)}
\right)=2\sum_{n_2\neq n_3}\gamma_{nnn_2n_3}\ov{f_{n_2}^{(2)}}f_{n_3}^{(3)}
+2\sum_{m}\gamma_{nnmm}
\left(
f_{m}^{(2)}(0)\ov{f_{m}^{(3)}}(0)-\frac{1}{m^{2\alpha}}
\right).
\]
We now estimate the first part~$\mathcal{N}_{[23]}^{(2)}$.

\begin{proposition}\label{bilinear}
	Assume $q_j\in [2,\infty)$ and $\gamma_j\in(\frac{1}{q_j'},1)$ for $j=1,2,3$.
	Let $\gamma_0\in(\frac{1}{q_1},1)$, $\varphi\in \mathcal{S}(\R)$, $Z_n\in \widetilde{\mathcal{F}}L_{q_1}^{\gamma_1}$, $n\in\N\setminus\{0\}$, and let $(\beta_{2},\beta_{3})$ satisfy
	\[
	\beta_2,\beta_3>2(1-\alpha),\quad \beta_2+\beta_3>1+2(1-\alpha).
	\]
	Then for any $f^{(j)}\in X_{q_j,q_j}^{2,\gamma_j}\cap X_{q_j,q_j}^{\beta_j,\gamma_j}$ for $j=2,3$,
	\begin{align}\label{bilinear-1}
		\|\varphi(t)Z_n\cdot
		\left(\mathcal{N}_{[23]}^{(2)}\right)_n\left(f^{(2)},f^{(3)}\right)\|_{\widetilde{\mathcal{F}}L_{q_1}^{-\gamma_0}}\lesssim_{\varphi,q_j,\gamma_j}
		\|Z_n\|_{\widetilde{\mathcal{F}}L_{q_1}^{\gamma_1}}
		\prod_{i=2}^{3}\|f^{(i)}\|_{X_{q_i,q_i}^{\beta_i,\gamma_i}},
	\end{align}
	and
	\begin{align}\label{bilinear-2}
		\|\varphi(t)Z_n\cdot
		\left(
		\mathcal{N}_{[23]}^{(2)}\right)_n\left(f^{(2)},f^{(3)}
		\right)\|_{\widetilde{\mathcal{F}}L_{q_1}^{0}}\lesssim_{\varphi,q_j,\gamma_j} \|Z_n\|_{\widetilde{\mathcal{F}}L_{q_1}^{\gamma_1}}
		\prod_{i=2}^{3}\|f^{(i)}\|_{X_{q_i,q_i}^{2,\gamma_i}}.
	\end{align}
\end{proposition}

\begin{proof}
	We first prove~\eqref{bilinear-1} (in the same spirit as~\cite{BCST24}, but much simpler).
	By duality, it suffices to show that for any $G_n\in \widetilde{\mathcal{F}}L_{q_1'}^{\gamma_0}$ with $\|G_n\|_{\widetilde{\mathcal{F}}L^{\gamma_0}_{q_1'}}\leq 1$,
	\begin{equation}\label{pf:bilinear-1}
		\int_{\R} \varphi(t)\ov{G}_n(t)\cdot Z_n(t)\cdot(\mathcal{N}_{[23]}^{(2)})_{n}(f^{(2)}(t),f^{(3)}(t))\, \dd t
		\lesssim
		\|Z_n\|_{\widetilde{\mathcal{F}}L_{q_1}^{\gamma_1}}
		\prod_{i=2}^{3}\|f^{(i)}\|_{X_{q_i,q_i}^{\beta_i,\gamma_i}}.
	\end{equation}
	Denote $\vec{\kappa}=(\kappa_0,\kappa_1,\kappa_2,\kappa_3)\in\R^4$, $\widetilde{\kappa}=\kappa_1-\kappa_2+\kappa_3-\kappa_0$, and
	\begin{align*}
		&\mathbf{a}_n^{(0)}(\kappa_0)=\la\kappa_0\ra^{\gamma_0}|\widetilde{G}_n(\kappa_0)|,\quad \mathbf{a}_n^{(1)}(\kappa_1)=\la\kappa_1\ra^{\gamma_1}|\widetilde{Z}_n(\kappa_1)|,\\
		&\mathbf{a}_{n_2}^{(2)}(\kappa_2)=\la\kappa_2\ra^{\gamma_2}|\widetilde{f}_{n_2}^{(2)}(\kappa_2)|,\quad \mathbf{a}_{n_3}^{(3)}(\kappa_3)=\la\kappa_3\ra^{\gamma_3}|\widetilde{f}_{n_3}^{(3)}(\kappa_3)|.
	\end{align*}
	It suffices to show that
	\begin{multline}
		\label{pf:bilinear-2}
		\int_{\R^4}\dd\vec{\kappa}\;
		\frac{\mathbf{a}_n^{(0)}(\kappa_0)\mathbf{a}_n^{(1)}(\kappa_1)}{\la\kappa_0\ra^{\gamma_0}\la\kappa_1\ra^{\gamma_1}}\cdot
		\sum_{n_2\neq n_3} \gamma_{nnn_2n_3}|\widehat{\varphi}(\widetilde{\kappa}-\Omega(\vec{n}))|\lambda_{n_2}^{-\beta_2}\lambda_{n_3}^{-\beta_3} \frac{\mathbf{a}_{n_2}^{(2)}(\kappa_2)\mathbf{a}_{n_3}^{(3)}(\kappa_3)}{\la\kappa_2\ra^{\gamma_2}\la\kappa_3\ra^{\gamma_3}}
		\\
		\lesssim
		\|\mathbf{a}_n^{(0)}(\kappa_0)\|_{L_{\kappa_0}^{q_1'}}\|\mathbf{a}_n^{(1)}(\kappa_1)\|_{L_{\kappa_1}^{q_1}}\|\mathbf{a}_{n_2}^{(2)}(\kappa_2)\|_{L_{\kappa_2}^{q_2}\ell_{n_2}^{q_2}}
		\|\mathbf{a}_{n_3}^{(3)}(\kappa_3)\|_{L_{\kappa_3}^{q_3}\ell_{n_3}^{q_3}}.
	\end{multline}
To prove this, we claim that :
	\begin{align}\label{claim:sum}
	\sup_{\vec{\kappa},n}\sum_{n_2\neq n_3}\gamma_{nnn_2n_3}|\widehat{\varphi}(\widetilde{\kappa}+\lambda_{n_2}^2-\lambda_{n_3}^2)|\lambda_{n_2}^{-\beta_2}\lambda_{n_3}^{-\beta_3}\lesssim 1.
	\end{align}
Indeed, for any fixed $n,\vec{\kappa}$, we bound the sum as
\begin{align*}
\sum_{N_2,N_3}\frac{\min(N_2,N_3)}{N_2^{\beta_2}N_3^{\beta_3}}\sum_{\substack{n_2\neq n_3\\
n_2\sim N_2, 
n_3\sim N_3 } } |\widehat{\varphi}(\widetilde{\kappa}+\lambda_{n_2}^2-\lambda_{n_3}^2)|.
\end{align*}
For fixed dyadic integers $N_2,N_3$, using the fact that $\widehat{\varphi}$ is rapidly decreasing, we have
\begin{align*}
\sum_{\substack{n_2\neq n_3\\
n_2\sim N_2, n_3\sim N_3
} }|\widehat{\varphi}(\widetilde{\kappa}+\lambda_{n_2}^2-\lambda_{n_3}^3)|\leq & \sum_{m\in \Z}\sup_{a\in (0,1]}|\widehat{\varphi}(\widetilde{\kappa}+m+a)|\sum_{\substack{n_2\neq n_3\\
n_2\sim N_2, n_3\sim N_3
 } }\mathbf{1}_{|\lambda_{n_2}^2-\lambda_{n_3}^2-m|\leq 1}\\
\lesssim & \max(N_2,N_3)^{2(1-\alpha)}\min(N_2,N_3)^{\epsilon},
\end{align*}
where we have used Lemma \ref{lem:db} in the last step. The claim \eqref{claim:sum} follows from the convergent dyadic summation, thanks to the condition for $\beta_2,\beta_3$.

Now, under the numerical conditions $\gamma_j>\frac{1}{q_j'}$ for $j=1,2,3$ and $\gamma_0>\frac{1}{q_1}$, we establish the estimate~\eqref{pf:bilinear-2} follows from H\"older's inequality and the embedding~$\ell^{q_j}\hookrightarrow \ell^{\infty}$.
\medskip

Finally, we prove ~\eqref{bilinear-2}. By Hausdorff--Young, Minkowski, and the embedding $\widetilde{\mathcal{F}}L_q^{\gamma}\hookrightarrow L_t^{\infty}$ (valid when $\gamma>\frac{1}{q'}$),
	\begin{align*}
		\|\varphi(t)Z_n(\mathcal{N}_{[23]}^{(2)})_{n}(f^{(2)},f^{(3)}) \|_{\widetilde{\mathcal{F}}L_{q_1}^0}&\leq \|\varphi(t)Z_n(\mathcal{N}_{[23]}^{(2)})_{n}(f^{(2)},f^{(3)}) \|_{L_t^{q_1'}}\\
		&\lesssim \|\varphi(t)\|_{L_t^{q_1'}}\|Z_n(t)\|_{L_t^{\infty}}\prod_{j=2,3}\|n_j^2 f_{n_j}^{(j)}(t)\|_{\ell_{n_j}^{q_j}L_t^{\infty}}\\
		&\lesssim \|Z_n\|_{\widetilde{\mathcal{F}}L_{q_1}^{\gamma_1}}\|f^{(2)}\|_{X_{q_2,q_2}^{2,\gamma_2}}\|f^{(3)}\|_{X_{q_3,q_3}^{2,\gamma_3}}.
	\end{align*}
	This completes the proof of Proposition \ref{bilinear}.
\end{proof}

The following bounds are consequences of Proposition~\ref{bilinear} and interpolation.

\begin{corollary}\label{co:2-linear:1}
	Assume that $(s,q,\gamma,b,\delta)=(s_{\sigma},q_{\sigma},\gamma_{\sigma},b_{\sigma},\delta_{\sigma})$ satisfy the hierarchy in Definition \ref{hierarchy}.
	Let $(\widetilde{q},\widetilde{\gamma})\in\{(q,\gamma),(2,b)\}$.
	For sufficiently small $\sigma>0$, and
	$$\epsilon_0\in\big(1-\widetilde{q}(1-\widetilde{\gamma}), \frac{2s+2\alpha-3}{3}\big),$$ we have
	\begin{align*}
		\mathrm{(1)} \; &
		\|\chi_T(t)\,\mathcal{I}\big(\mathcal{N}^{(2)}(v_{M_2}^{(\mathrm{D})},v_{M_3}^{(\mathrm{D})})\dotcirc Z\big) \|_{X_{\widetilde{q},\widetilde{q}}^{0,\widetilde{\gamma}}}
		\lesssim T^{\sigma}\max(M_2,M_3)^{10\epsilon_0}(M_2 M_3)^{-(s+\alpha-\frac32)}\|Z\|_{X_{\widetilde{q},\widetilde{q}}^{0,\widetilde{\gamma}}},\\
		\mathrm{(2)} \; &
		\|\chi_T(t)\,\mathcal{I}\big(\mathcal{N}^{(2)}(v_{M_2}^{(\mathrm{C})},v_{M_3}^{(\mathrm{D})})\dotcirc Z\big) \|_{X_{\widetilde{q},\widetilde{q}}^{0,\widetilde{\gamma}}}\lesssim T^{\sigma}\max(M_2,M_3)^{10\epsilon_0}
		M_2^{-\frac1q}
		M_3^{-(s+\alpha-\frac32)} \|Z\|_{X_{\widetilde{q},\widetilde{q}}^{0,\widetilde{\gamma}}},\\
		\mathrm{(3)}  &
		\|\chi_T(t)\,\mathcal{I}\big(\mathcal{N}^{(2)}(v_{M_2}^{(\mathrm{D})},v_{M_3}^{(\mathrm{D})})\dotcirc Z\big) \|_{X_{\widetilde{q},\widetilde{q}}^{0,\widetilde{\gamma}}}
		\lesssim T^{\sigma}
		\max(M_2,M_3)^{10\epsilon_0}(M_2 M_3)^{-\frac{1}{q}}
		\|Z\|_{X_{\widetilde{q},\widetilde{q}}^{0,\widetilde{\gamma}}}.
	\end{align*}
\end{corollary}

\begin{proof}
	For sufficiently small $\sigma>0$, we fix the other small parameter $\epsilon_0$ such that
	\[
	1-\widetilde{q}(1-\widetilde{\gamma}-\sigma)<\epsilon_0<1,
	\]
	and set $\gamma_0=\frac{1-\widetilde{\gamma}-\sigma}{1-\epsilon_0}$ (so that $\gamma_0>\frac{1}{\widetilde{q}}$). Then
	\begin{align*}
		\big\|\mathcal{I}_{\chi_T}(Z\dotcirc \mathcal{N}^{(2)}(f^{(2)},f^{(3)}))\big\|_{X_{\widetilde{q},\widetilde{q}}^{0,\widetilde{\gamma}}} \lesssim T^{\sigma}\big\|\chi(t)Z\dotcirc\mathcal{N}^{(2)}(f^{(2)},f^{(3)})\big\|_{X_{\widetilde{q},\widetilde{q}}^{0,-(1-\epsilon_0)\gamma_0}}.
	\end{align*}
	By interpolation and Proposition~\ref{bilinear}, the right-hand side is bounded by
	\begin{equation}
		\label{eq:ast1}
		T^{\sigma_0}\|Z\|_{X_{\widetilde{q},\widetilde{q}}^{0,\widetilde{\gamma}}} \big(\|f^{(2)}\|_{X_{q_2,q_2}^{\beta_2,\gamma_2}}\|f^{(3)}\|_{X_{q_3,q_3}^{\beta_3,\gamma_3}}\big)^{1-\epsilon_0}\cdot \big(\|f^{(2)}\|_{X_{q_2,q_2}^{2,\gamma_2}} \|f^{(3)}\|_{X_{q_3,q_3}^{2,\gamma_3}}\big)^{\epsilon_0}.
	\end{equation}
	For~(1), we apply~\eqref{eq:ast1} with $\gamma_2=\gamma_3=b$, $q_2=q_3=2$, and $\beta_2=\beta_3=\frac{3}{2}-\alpha+$. Since $s>\frac{3}{2}-\alpha$, we obtain the bound
	\[
	T^{\sigma} (M_2 M_3)^{-(s+\alpha-\frac{3}{2})+10\epsilon_0}
	\|Z\|_{X_{\widetilde{q},\widetilde{q}}^{0,\widetilde{\gamma}}},
	\]
	where the exponent~$10\epsilon_0$ comes from the interpolation.

	For~(2), we apply~\eqref{eq:ast1} with $\gamma_2=\gamma$, $\gamma_3=b$, $q_2=q$, $q_3=2$, $\beta_2=\alpha-\frac{3}{q}+$, $\beta_3=\frac{3}{2}-\alpha+$. We obtain the bound
	\[
	T^{\sigma}M_2^{-\frac{1}{q}}M_3^{-(s+\alpha-\frac{3}{2})}\max(M_2,M_3)^{10\epsilon_0}\|Z\|_{X_{\widetilde{q},\widetilde{q}}^{0,\widetilde{\gamma}}}.
	\]
	A similar computation for~(3) yields
	\[
	T^{\sigma}\max(M_2,M_3)^{10\epsilon_0}(M_2 M_3)^{-\frac{1}{q}}\|Z\|_{X_{\widetilde{q},\widetilde{q}}^{0,\widetilde{\gamma}}}.
	\]
	This completes the proof.
\end{proof}

	
	\section{Multilinear estimates II: quintic interaction}\label{sec:quintilinear}

In this section, we estimate the quintic nonlinearity $\mathcal{N}^{(4)}(u)\dotcirc u$, which arises from the gauge transform and the normal form reduction. The key idea is to exploit the algebraic product structure of this term to reduce the problem to a pointwise-in-time bound for the quartic operator~$\mathcal{N}^{(4)}$.

\subsection{Reduction to a pointwise-in-time bound}

We begin by showing how the special structure of $\mathcal{N}^{(4)}(u)\dotcirc u$ reduces the quintic estimate to a pointwise-in-time bound for~$\mathcal{N}^{(4)}(u)$.
Recall that
\begin{align*}
\mathcal{N}_{n}^{(4)}(f^{(1)},\cdots, f^{(4)})(t):=\sum_{\nu=1}^4\mathcal{N}_{n,\nu}^{(4)}(f^{(1)},\cdots,f^{(4)})(t),
\end{align*}
where
\begin{multline}
	\mathcal{N}_{n,1}^{(4)}(f^{(1)},\cdots,f^{(4)})(t)
	=\sum_{\substack{m_1,m_2,m_3,m_4\\
			\{m_1,m_3\}\cap\{m_2,m_4\}=\emptyset}}
	\gamma_{nnm_4m_4}\gamma_{m_1m_2m_3m_4}
	\int_0^t \prod_{j=1}^4 (f_{m_j}^{(j)})^{\pm_j}(t')\, \dd t',\\
		\mathcal{N}_{n,2}^{(4)}(f^{(1)},\cdots,f^{(4)})(t)
	=\sum_{\substack{m_1,m_2,m_3,m_4\\
m_3\neq m_4,\\
m_1=m_2}}
	\gamma_{nnm_4m_4}\gamma_{m_1m_1m_3m_4}
	\int_0^t \prod_{j=1}^4 (f_{m_j}^{(j)})^{\pm_j}(t')\, \dd t',\\
		\mathcal{N}_{n,3}^{(4)}(f^{(1)},\cdots,f^{(4)})(t)
	=\sum_{\substack{m_1,m_2,m_3,m_4\\ 
	m_1=m_3=m_4	\neq m_2
	}}
	\gamma_{nnm_4m_4}\gamma_{m_4m_2m_4m_4}
	\int_0^t \prod_{j=1}^4 (f_{m_j}^{(j)})^{\pm_j}(t')\, \dd t',\\
		\mathcal{N}_{n,4}^{(4)}(f^{(1)},\cdots,f^{(4)})(t)
	=\sum_{\substack{m_1,m_2,m_3,m_4\\ 
			m_1=m_2=m_3	\neq m_4
	}}
	\gamma_{nnm_4m_4}\gamma_{m_1m_1m_1m_4}
	\int_0^t \prod_{j=1}^4 (f_{m_j}^{(j)})^{\pm_j}(t')\, \dd t'.
\end{multline}
Here $f^{\pm_j}=f$ for $j\in\{1,3\}$ and $f^{\pm_j}=\ov{f}$ for $j\in\{2,4\}$.
Establishing a pointwise-in-time bound for~$\mathcal{N}_{n,\nu}^{(4)}$ is analogous to the trilinear estimate, via a duality argument.

\begin{lemma}\label{lem:N4mapping}
	Let $q_0\in[2,\infty)$, $\gamma_0\in(\frac{1}{q_0'},1)$, $0<\delta_0<\frac{1-\gamma_0}{2}$, and $r_0=\big(\gamma_0-\frac{1}{q_0'}+2\delta_0\big)^{-1}$. Then for any $F\in X_{\infty,r_0}^{0,0}$, $u\in X_{q_0,q_0}^{0,\gamma_0}$, and $0<T<1$,
	\[
	\|\chi_T(t)\,\mathcal{I}(F\dotcirc u)_n(t)\|_{\widetilde{\mathcal{F}}L_{q_0}^{\gamma_0}}
	\lesssim
	T^{\delta_0+\frac{1}{r_0'}}\|\chi_1(t)F_n(t)\|_{L_t^{\infty}}\|u_n\|_{\widetilde{\mathcal{F}}L_{q_0}^{\gamma_0}}.
	\]
	Consequently,
	\[
	\|\chi_T(t)\,\mathcal{I}(F\dotcirc u)\|_{X_{q_0,q_0}^{0,\gamma_0}}
	\lesssim
	T^{\delta_0+\frac{1}{r_0'}} \|u\|_{X_{q_0,q_0}^{0,\gamma_0}}\|F\|_{\ell_n^{\infty}L_t^{\infty}}.
	\]
\end{lemma}

\begin{proof}
	By the inhomogeneous linear estimate (cite \cite{BCST24}, Lemma A.8) , it suffices to show that (implicitly inserting another cutoff $\widetilde{\chi}_T(t)$ such that $F=\widetilde{\chi}_T(t)F$)
	\begin{align}\label{eq:reductionquintic1}
		\|F_n\cdot u_n\|_{\widetilde{\mathcal{F}}L_{q_0}^{\gamma_0+\delta_0-1}}
		\lesssim
		T^{\frac{1}{r_0'}}\|\chi_1(t)F_n(t)\|_{L_t^{\infty}}\|u_n\|_{\widetilde{\mathcal{F}}L_{q_0}^{\gamma_0}}.
	\end{align}
	Set $a=1-\gamma_0-\delta_0\in (0,1)$ and $\widetilde{r}=(1-\gamma_0-2\delta_0)^{-1}$. Then
	\[
	\frac{1}{q_0}=\frac{1}{\widetilde{r}}+\frac{1}{r_0}.
	\]
	Since $0<\delta_0<(1-\gamma_0)/2$, we have $\widetilde{r}\in(1,\infty)$ and $r_0\in[2,\infty)$. By H\"older's and Young's inequalities,
\begin{multline*}
		\|F_n\cdot u_n\|_{\widetilde{\mathcal{F}}L_{q_0}^{-a}}
		\leq
		\|\langle\lambda\rangle^{-a} \widetilde{F}_n\ast \widetilde{u}_n(\lambda)\|_{L_{\lambda}^{q_0}}
		\leq
		\|\langle\lambda\rangle^{-a}\|_{L_{\lambda}^{\widetilde{r}}}\|\widetilde{F}_n\ast\widetilde{u}_n(\lambda)\|_{L_{\lambda}^{r_0}}\\
		\lesssim
		\|\widetilde{F}_n(\lambda)\|_{L_{\lambda}^{r_0}}
		\|\widetilde{u}_n(\lambda)\|_{L_{\lambda}^{1}}.
\end{multline*}
	Since $\gamma_0>1/q_0'$,
	\[
	\|\widetilde{u}_n(\lambda)\|_{L_{\lambda}^1}
	\lesssim
	\|\langle\lambda\rangle^{\gamma_0}\widetilde{u}_n(\lambda)\|_{L_{\lambda}^{q_0}}.
	\]
	Hence
	\[
	\|F_n\cdot u_n\|_{\widetilde{\mathcal{F}}L_{q_0}^{-a}}
	\lesssim
	\|\chi_1(t)F_n\|_{\widetilde{\mathcal{F}}L_{r_0}^{0}}\|u_n\|_{\widetilde{\mathcal{F}}L_{q_0}^{\gamma_0}}.
	\]
	Since $r_0\in[2,\infty)$ and $F=\widetilde{\chi}_T(t)F=\widetilde{\chi}(t/T)\cdot\chi_1(t)F$, by Hausdorff--Young and H\"older's inequality,
	\[
		\|F_n\|_{\widetilde{\mathcal{F}}L_{r_0}^0}
		\leq
		\|\widetilde{\chi}_T(t)\cdot (\chi_1(t) F_n)\|_{L_t^{r_0'}}
		\leq
		\|\widetilde{\chi}_T\|_{L_t^{r_0'}}\|\chi_1(t)F_n\|_{L_t^{\infty}}
		\sim
		T^{\frac{1}{r_0'}}\|\chi_1(t)F_n\|_{L_t^{\infty}},
	\]
	which proves~\eqref{eq:reductionquintic1}.
\end{proof}

\begin{remark}
	We apply this lemma in two situations. The first is to show that $\mathcal{T}_n^{N,\dag}\in \widetilde{\mathcal{F}}L_q^{\gamma}$ in the iteration step for the random averaging operator, where we take $(q_0,\gamma_0)=(q,\gamma)$. The second is in the multilinear estimate for the remainder~$w_N$, where we take $(q_0,\gamma_0)=(2,b)$. These two applications impose the condition
	\[
	\gamma-\frac{1}{q'}=b-\frac{1}{2}.
	\]
\end{remark}

\subsection{Pointwise-in-time estimate for $\mathcal{N}_{n,1}^{(4)}$}

According to Lemma~\ref{lem:N4mapping}, it remains to establish pointwise-in-time bounds for the quartic nonlinearity~$\mathcal{N}^{(4)}$. We introduce some notation. For dyadic integers $(M_j)_{1\leq j\leq 4}$, we denote their non-increasing rearrangement by 
\[
M_{(1)}\geq M_{(2)}\geq M_{(3)}\geq M_{(4)}\,.
\] 
We begin with a weaker bound that follows directly from the Strichartz-type estimates of Section~\ref{sec:trilinear}.
For a Schwartz function $\varphi\in \mathcal{S}(\R)$, define
\begin{multline}\label{def:M4n}
	\mathcal{M}_{n,1}^{(4)}((f^{(j)})_{1\leq j\leq 4})
	:=
	\sum_{\substack{m_1,m_2,m_3,m_4\\
			\{m_1,m_3\}\cap\{m_2,m_4\}=\emptyset}}\gamma_{nnm_4m_4}\gamma_{m_1m_2m_3m_4}
	\\
	\times\int_{\R} \varphi(t)\cdot \prod_{j=1}^4 (f_{m_j}^{(j)})^{\pm_j}(t)\, \dd t.
\end{multline}

\begin{lemma}\label{4-linear}
	Let $f^{(j)}=\mathbf{P}_{M_j}f^{(j)}$, where $q_j\in[2,\infty)$ and $\gamma_j\in\big(\frac1{q_j'},1\big)$ for $j=1,2,3,4$.
	Let $\varphi\in\mathcal{S}(\R)$, and fix $n\in\N$. For any $j_0\in\{1,2,3,4\}$ and $j_1\in\{1,2,3,4\}\setminus\{j_0\}$, there exists a coefficient $K_{M_1,M_2,M_3,M_4}^{(j_0,j_1)}$ such that
	\begin{align}\label{eq1:4-linear}
		|\mathcal{M}_{n,1}^{(4)}((f^{(j)})_{1\leq j\leq 4})|
		\lesssim_{\varphi,q_j,\gamma_j,\varepsilon}
		K_{M_1,M_2,M_3,M_4}^{(j_0,j_1)}
		\prod_{j=1}^4\|f^{(j)}\|_{X_{q_j,q_j}^{0,\gamma_j}},
	\end{align}
	and a coefficient $L_{M_1,M_2,M_3,M_4}$ such that
	\begin{align}\label{eq2:4-linear}
		|\mathcal{M}_{n,1}^{(4)}((f^{(j)})_{1\leq j\leq 4})|
		\lesssim_{\varphi,b,\varepsilon}
		L_{M_1,M_2,M_3,M_4}\prod_{j=1}^4\|f^{(j)}\|_{X^{0,b}}.
	\end{align}
	Denote by $P_1\geq P_2\geq P_3$ the non-increasing rearrangement of $\{M_j:j\neq j_1\}$.
	Then
	\begin{align*}
		K_{M_1,M_2,M_3,M_4}^{(j_0,j_1)}
		&\lesssim_{\varepsilon}
		(M_4\wedge n)\,M_{(4)}\,P_1^{1-\alpha+\varepsilon}
		\Big(1+P_2 P_1^{-1/2}\Big)
		\prod_{j\neq j_0}M_j^{\frac12-\frac1{q_j}},\\
		L_{M_1,M_2,M_3,M_4}
		&\lesssim_{\varepsilon}
		(M_4\wedge n)\,M_{(4)}
		\min\Big(1+M_{(3)}M_{(1)}^{-1/2},\, M_{(1)}^{2(1-\alpha)}M_{(3)}^{\varepsilon}\Big).
	\end{align*}

	Moreover, in the regime $M_{(1)}\gg M_{(2)}$, the following improved bounds hold:
	\begin{align*}
		K_{M_1,M_2,M_3,M_4}^{(j_0,j_1)}
		&\lesssim_{\varepsilon}
		M_{(1)}^{-2+\varepsilon}(M_4\wedge n)\,P_1^{1-\alpha}
		\Big(1+M_{(3)}M_{(1)}^{-1/2}\Big)
		\prod_{j\neq j_0}M_j^{\frac12-\frac1{q_j}},\\
		L_{M_1,M_2,M_3,M_4}
		&\lesssim_{\varepsilon}
		M_{(1)}^{-2+\varepsilon}(M_4\wedge n)
		\min\Big(1+M_{(3)}M_{(1)}^{-1/2},\, M_{(1)}^{2(1-\alpha)}M_{(3)}^{\varepsilon}\Big).
	\end{align*}
\end{lemma}

\begin{proof}
	By Plancherel's theorem,
	\begin{align*}
		\int_{\R} \varphi(t)\cdot \prod_{j=1}^4 (f_{m_j}^{(j)})^{\pm_j}(t)\, \dd t
		=
		\int_{\R^4}\dd\vec{\kappa}\; \widehat{\varphi}(\widetilde{\kappa}-\Omega(\vec{m}))\cdot \prod_{j=1}^4(\widetilde{f}_{m_j}^{(j)})^{\pm_j}(\kappa_j),
	\end{align*}
	where $\vec{\kappa}=(\kappa_1,\kappa_2,\kappa_3,\kappa_4)$ and $\widetilde{\kappa}=\kappa_1-\kappa_2+\kappa_3-\kappa_4$.
	The left-hand side of~\eqref{eq1:4-linear} is therefore bounded by
	\begin{multline}\label{pf4-linear-1}
		\int_{\R^4}\dd\vec{\kappa}\;
		\prod_{j=1}^4\la\kappa_j\ra^{-\gamma_j}
		\sum_{\substack{m_1,m_2,m_3,m_4\\
				m_j\sim M_j\\
				\{m_1,m_3\}\cap\{m_2,m_4\}=\emptyset}} \gamma_{nnm_4m_4}\gamma_{m_1m_2m_3m_4}\,\eta(\widetilde{\kappa}-\Omega(\vec{m}))\prod_{j=1}^4\mathbf{a}^{(j)}_{m_j}(\kappa_j)\\
		\lesssim
		(M_4\wedge n)\,M_{(4)}\int_{\R^4}\dd\vec{\kappa}\;
		\prod_{j=1}^4\la\kappa_j\ra^{-\gamma_j}
		\\\times\sum_{\substack{m_1,m_2,m_3,m_4\\
				\{m_1,m_3\}\cap\{m_2,m_4\}=\emptyset}} \eta(\widetilde{\kappa}-\Omega(\vec{m}))\prod_{j=1}^4\mathbf{a}^{(j)}_{m_j}(\kappa_j)\,\mathbf{1}_{m_j\sim M_j},
	\end{multline}
	where $\eta(\cdot)=|\widehat{\varphi}(\cdot)|$ and $\mathbf{a}_{m_j}^{(j)}(\kappa_j)=\la\kappa_j\ra^{\gamma_j}|(\widetilde{f}^{(j)}_{m_j})^{\pm_j}(\kappa_j)|$. The desired estimate then follows from Lemma~\ref{timeStrichart}, Minkowski's inequality, and H\"older's inequality, together with
	\[
	\|\mathbf{1}_{m_j\sim M_j}\la\kappa_j\ra^{\gamma_j} \widetilde{f}_{m_j}^{(j)}(\kappa_j)\|_{\ell_{m_j}^2 L_{\kappa_j}^{q_j}}
	\lesssim
	\|f^{(j)}\|_{X_{q_j,q_j}^{0,\gamma_j}}M_j^{\frac{1}{2}-\frac{1}{q_j}}.
	\]
	This completes the proof.
\end{proof}

The next proposition upgrades the preceding lemma to a pointwise-in-time estimate by summing over the dyadic modulation blocks, at the cost of a small loss in the highest frequency.

\begin{proposition}\label{4linear'}
	Let $q_j\in[2,\infty)$ and $\gamma_j\in\big(\frac{1}{q_j'},1\big)$, and assume $f^{(j)}=\mathbf{P}_{M_j}f^{(j)}$. Let $j_0\in\{1,2,3,4\}$ be such that $M_{j_0}=M_{(1)}$, and let $j_1\in\{1,2,3,4\}\setminus\{j_0\}$. Denote by $K_{M_1,M_2,M_3,M_4}$ and $L_{M_1,M_2,M_3,M_4}$ the coefficients from Lemma~\ref{4-linear} associated to $j_0$ and~$j_1$. Then for any $|t|\leq 1$,
	\begin{multline}\label{quintic:type1}
		|\mathcal{N}_{n,1}^{(4)}(f^{(1)},f^{(2)},f^{(3)},f^{(4)})(t)|
		\\
		\lesssim_{\varepsilon}
		M_{(1)}^{2\theta_{j_0}+\varepsilon}
		K_{M_1,M_2,M_3,M_4}
		\|\mathbf{1}_{[0,t]}(\cdot)f^{(j_0)}\|_{X_{q_{j_0},q_{j_0}}^{0,\widetilde{\gamma}_{j_0}}}
		\prod_{j\neq j_0}\|f^{(j)}\|_{X_{q_j,q_j}^{0,\gamma_j}},
	\end{multline}
	where
	\[
	\theta_{j_0}=\gamma_{j_0}-\frac{1}{q_{j_0}'},
	\qquad
	\widetilde{\gamma}_{j_0}=\gamma_{j_0}-\theta_{j_0}-\frac{\varepsilon}{2}<\frac{1}{q_{j_0}'},
	\]
	and $(q_j,\gamma_j)\in\{(2,b),(q,\gamma)\}$. Alternatively,
	\begin{multline}\label{quintic:type2}
		|\mathcal{N}_{n,1}^{(4)}(f^{(1)},f^{(2)},f^{(3)},f^{(4)})(t)|
		\\
		\lesssim_{\varepsilon}
		M_{(1)}^{2\theta+\varepsilon}L_{M_1,M_2,M_3,M_4}\,\|\mathbf{1}_{[0,t]}(\cdot)f^{(j_0)}\|_{X^{0,\widetilde{b}}}
		\prod_{j\neq j_0}\|f^{(j)}\|_{X^{0,b}},
	\end{multline}
	where
	\[
	\theta=b-\frac{1}{2},
	\qquad
	\widetilde{b}=b-\theta-\frac{\varepsilon}{2}<\frac{1}{2}.
	\]
\end{proposition}

\begin{proof}
	We decompose each function dyadically in the modulation variable:
	\[
	f^{(j)}=\sum_{L_j} f_{L_j}^{(j)},
	\qquad
	\widetilde{f_{L_j}^{(j)}}(\kappa_j,\cdot)
	=
	\mathbf{1}_{[L_j/2,L_j]}(\langle \kappa_j\rangle)\,\widetilde{f}^{(j)}(\kappa_j,\cdot),
	\]
	and write
	\[
	\mathcal{N}_{n,1}^{(4)}(f^{(1)},\cdots,f^{(4)})(t)
	=
	\sum_{L_1,L_2,L_3,L_4}
	\mathcal{N}_{n,1}^{(4)}(f_{L_1}^{(1)},\cdots,f_{L_4}^{(4)})(t).
	\]
	Denote by $L_{(1)}\geq L_{(2)}\geq L_{(3)}\geq L_{(4)}$ the non-increasing rearrangement of $L_1,L_2,L_3,L_4$.
	Since each $f^{(j)}$ is localized at spatial frequency~$M_j$, the dyadic summation is restricted to $L_{(1)}\lesssim M_{(1)}^2$.

	Recall that $j_0$ is chosen so that $M_{j_0}=M_{(1)}$. Set
	\[
	\theta_{j_0}=\gamma_{j_0}-\frac{1}{q_{j_0}'},
	\qquad
	\widetilde{\gamma}_{j_0}
	=
	\gamma_{j_0}-\theta_{j_0}-\frac{\varepsilon}{2}
	=
	\frac{1}{q_{j_0}'}-\frac{\varepsilon}{2}
	<
	\frac{1}{q_{j_0}'}.
	\]
	For each dyadic modulation block,
	\[
	\|f_{L_{j_0}}^{(j_0)}\|_{X_{q_{j_0},q_{j_0}}^{0,\gamma_{j_0}}}
	\lesssim
	L_{j_0}^{\theta_{j_0}+\frac{\varepsilon}{2}}
	\|f_{L_{j_0}}^{(j_0)}\|_{X_{q_{j_0},q_{j_0}}^{0,\widetilde{\gamma}_{j_0}}}.
	\]
	Applying Lemma~\ref{4-linear} to each dyadic block, we obtain
	\begin{multline*}
		|\mathcal{N}_{n,1}^{(4)}(f^{(1)},f^{(2)},f^{(3)},f^{(4)})(t)|
		\\
		\lesssim
		\sum_{L_1,L_2,L_3,L_4}
		L_{j_0}^{\theta_{j_0}+\frac{\varepsilon}{2}}
		K_{M_1,M_2,M_3,M_4}
		\|f_{L_{j_0}}^{(j_0)}\|_{X_{q_{j_0},q_{j_0}}^{0,\widetilde{\gamma}_{j_0}}}
		\prod_{j\neq j_0}\|f_{L_j}^{(j)}\|_{X_{q_j,q_j}^{0,\gamma_j}}.
	\end{multline*}
	Since $L_{j_0}\leq L_{(1)}\lesssim M_{(1)}^2$, we have $L_{j_0}^{\theta_{j_0}+\varepsilon/2} \lesssim M_{(1)}^{2\theta_{j_0}+\varepsilon}$.
	Summing over the dyadic modulations and applying Cauchy--Schwarz in the $L_j$ variables yields
	\[
	|\mathcal{N}_{n,1}^{(4)}(f^{(1)},f^{(2)},f^{(3)},f^{(4)})(t)|
	\lesssim
	M_{(1)}^{2\theta_{j_0}+\varepsilon}
	K_{M_1,M_2,M_3,M_4}
	\|f^{(j_0)}\|_{X_{q_{j_0},q_{j_0}}^{0,\widetilde{\gamma}_{j_0}}}
	\prod_{j\neq j_0}\|f^{(j)}\|_{X_{q_j,q_j}^{0,\gamma_j}}.
	\]
	This proves~\eqref{quintic:type1}. The proof of~\eqref{quintic:type2} is identical, with $(q_j,\gamma_j)$ replaced by $(2,b)$.
\end{proof}
	
	The following product rule allows us to insert sharp time cutoffs without losing regularity, provided the Fourier--Lebesgue exponent is subcritical.

\begin{lemma}\label{productrule}
	Assume $2\le q_0<\infty$ and $0<\widetilde{\gamma}_0<\frac{1}{q_0'}$.
	Let $F\in \widetilde{\mathcal{F}}L_{q_0}^{\widetilde{\gamma}_0}$.
	Then, uniformly in $|t|\le 1$,
	\[
	\|\mathbf{1}_{[0,t]}\cdot F\|_{\widetilde{\mathcal{F}}L_{q_0}^{\widetilde{\gamma}_0}}
	\lesssim
	\|F\|_{\widetilde{\mathcal{F}}L_{q_0}^{\widetilde{\gamma}_0}}.
	\]
\end{lemma}

\begin{proof}
	This follows from the boundedness of the Hilbert transform on weighted~$L^{q_0}$,
	with the $A_{q_0}$ weight $w(\tau)=\langle \tau\rangle^{q_0\widetilde{\gamma}_0}$.
	Since $0<\widetilde{\gamma}_0<\frac{1}{q_0'}$, the weight~$w$ belongs to~$A_{q_0}$.
	We refer to Lemma~2.6 in~\cite{OTW20} for an identical argument in the case $q_0=2$ and $\widetilde{\gamma}_0=b<\frac12$.
\end{proof}

We now combine the pointwise-in-time bounds from Proposition~\ref{4linear'} with the input norms of type~(C) and~(D) to obtain the deterministic quintic estimates. These handle all configurations with at least two type~(D) inputs.

\begin{corollary}\label{cor:4-linear-1}
	Assume $\alpha>15/16$ and $(s,q,\gamma,b,\delta)=(s_{\sigma},q_{\sigma},\gamma_{\sigma},b_{\sigma},\delta_{\sigma})$ satisfies the hierarchy in Definition~\ref{hierarchy}, with $\theta=b-\frac{1}{2}=\gamma-\frac{1}{q'}$. Assume moreover $M_{(1)}\sim M_{(2)}$ and $\mathbf{P}_{M_j}v_{M_j}^{(\ast_j)}=v_{M_j}^{(\ast_j)}$ for $j=1,2,3,4$.
	Then the following bounds hold:
	\begin{align*}
		\mathrm{(1)} \quad &
		\|\chi(t)\,\mathcal{N}_{n,1}^{(4)}(v^{(\mathrm{D})}_{M_1},v^{(\mathrm{D})}_{M_2},v^{(\mathrm{D})}_{M_3},v^{(\mathrm{D})}_{M_4})\|_{L_t^{\infty}}
		\lesssim
	 M_{(1)}^{10\theta+\delta}\cdot
		\frac{M_{(1)}\wedge n}{M_{(1)}}\cdot M_{(1)}^{-(3\alpha+2s-\frac{7}{2})},\\
		\mathrm{(2)} \quad &
		\|\chi(t)\,\mathcal{N}_{n,1}^{(4)}(v^{(\mathrm{D})}_{M_1},v^{(\mathrm{D})}_{M_2},v^{(\mathrm{D})}_{M_3},v^{(\mathrm{C})}_{M_4})\|_{L_t^{\infty}}
		\lesssim
		RT^{-(\gamma-\frac{1}{q'})} M_{(1)}^{10\theta+\delta}\cdot
		\frac{M_{(1)}\wedge n}{M_{(1)}}\cdot M_{(1)}^{-(3\alpha+2s-\frac72)},\\
		\mathrm{(3)} \quad &
		\|\chi(t)\,\mathcal{N}_{n,1}^{(4)}(v^{(\mathrm{D})}_{M_1},v^{(\mathrm{D})}_{M_2},v^{(\mathrm{C})}_{M_3},v^{(\mathrm{D})}_{M_4})\|_{L_t^{\infty}}
		\lesssim
		RT^{-(\gamma-\frac{1}{q'})} M_{(1)}^{10\theta+\delta}\cdot
		\frac{M_{(1)}\wedge n}{M_{(1)}}\cdot M_{(1)}^{-(3\alpha+2s-\frac72)},\\
		\mathrm{(4)} \quad &
		\|\chi(t)\,\mathcal{N}_{n,1}^{(4)}(v^{(\mathrm{D})}_{M_1},v^{(\mathrm{D})}_{M_2},v^{(\mathrm{C})}_{M_3},v^{(\mathrm{C})}_{M_4})\|_{L_t^{\infty}}
		\lesssim
		R^2T^{-2(\gamma-\frac{1}{q'})} M_{(1)}^{10\theta+\delta}\cdot
		\frac{M_{(1)}\wedge n}{M_{(1)}}\cdot M_{(1)}^{-(3\alpha+2s-4)},\\
		\mathrm{(5)} \quad &
		\|\chi(t)\,\mathcal{N}_{n,1}^{(4)}(v^{(\mathrm{C})}_{M_1},v^{(\mathrm{C})}_{M_2},v^{(\mathrm{D})}_{M_3},v^{(\mathrm{D})}_{M_4})\|_{L_t^{\infty}}
		\lesssim
		R^2T^{-2(\gamma-\frac{1}{q'})} M_{(1)}^{10\theta+\delta}\cdot
		\frac{M_{(1)}\wedge n}{M_{(1)}}\cdot M_{(1)}^{-(3\alpha+2s-4)}.
	\end{align*}
	In particular, for all interactions with at least two type~(D) inputs,
	\begin{align}\label{cor:4-linear-1-uniform}
		\|\chi(t)\,\mathcal{N}_{n,1}^{(4)}(v^{(\ast_1)}_{M_1},v^{(\ast_2)}_{M_2},v^{(\ast_3)}_{M_3},v^{(\ast_4)}_{M_4})\|_{L_t^{\infty}}
		\lesssim
		R^2T^{-2(\gamma-\frac{1}{q'})} M_{(1)}^{10\theta+\delta}\cdot
		\frac{M_{(1)}\wedge n}{M_{(1)}}\cdot M_{(1)}^{-(3\alpha+2s-4)}.
	\end{align}
\end{corollary}

\begin{proof}
	To simplify notation, write $v^{(j)}=\mathbf{P}_{M_j}v^{(j)}$ for $j=1,2,3,4$, where each~$v^{(j)}$ is of type~(C) or~(D). The factors $R\, T^{-(\gamma-\frac{1}{q'})}$ arising from the $X_{q,q}^{0,\gamma}$-norms of the type~(C) inputs are suppressed in the proof and restored in the final estimates.

	We prove the estimates only in the regime $M_4\lesssim n$. In the complementary regime $n\lesssim M_4$, the same argument applies using the sharper factor $(M_4\wedge n)$ in Lemma~\ref{4-linear} and Proposition~\ref{4linear'}: the distinguished factor~$M_4$ from $\gamma_{nnm_4m_4}$ is replaced by~$n$, while all other occurrences of~$M_4$ are bounded by~$M_{(1)}$. This yields the extra factor $\frac{M_{(1)}\wedge n}{M_{(1)}}$.

	For fixed $n\in\N$, we estimate
	\begin{align}\label{quartic:Linfty-new}
		\|\mathcal{N}_{n,1}^{(4)}(v^{(1)},v^{(2)},v^{(3)},v^{(4)})\|_{L_t^{\infty}}.
	\end{align}
	Applying Proposition~\ref{4linear'} with~$j_0$ chosen so that $M_{j_0}=M_{(1)}$, we obtain
	\begin{align}
		\label{numerics1-new}
		\eqref{quartic:Linfty-new}
		\lesssim
		M_{(1)}^{2\theta+\varepsilon}
		K_{M_1,M_2,M_3,M_4}
		\prod_{j=1}^4|a_{M_j}^{(j)}|,
	\end{align}
	if $v^{(j_0)}$ is of type~(C), and
	\begin{align}
		\label{numerics2-new}
		\eqref{quartic:Linfty-new}
		\lesssim
		M_{(1)}^{2\theta+\varepsilon}
		L_{M_1,M_2,M_3,M_4}
		\prod_{j=1}^4|a_{M_j}^{(j)}|,
	\end{align}
	if all $v^{(j)}$ are of type~(D).

	If $M_{(1)}\gg M_{(2)}$, then bounds~(1)--(5) follow immediately from the improved estimates in Lemma~\ref{4-linear} and the hierarchy of parameters. Hence we may assume $M_{(1)}\sim M_{(2)}$.
    Denote $s_C=\alpha-\frac1q-\delta$ in the computation below:

	\emph{(DDDD) interactions.}
	If $M_{(3)}\lesssim M_{(1)}^{1/2}$, then by~\eqref{numerics2-new},
	\begin{align*}
		M_{(1)}^{2\theta+\varepsilon}
		L_{M_1,M_2,M_3,M_4}
		\cdot (M_1 M_2 M_3 M_4)^{-s}
		\lesssim
		M_{(1)}^{10\theta}
		M_{(1)}^{1-2s}\cdot M_{(3)}^{1-2s}.
	\end{align*}
	If $M_{(3)}\gg M_{(1)}^{1/2}$, then~\eqref{numerics2-new} yields
	\begin{align*}
		M_{(1)}^{2\theta+\varepsilon}
		L_{M_1,M_2,M_3,M_4}
		\cdot (M_1 M_2 M_3 M_4)^{-s}
		&\lesssim
		M_{(1)}^{10\theta}\cdot M_{(1)}^{1-2s+2(1-\alpha)}\cdot M_{(3)}^{1-2s}\\
		&\lesssim
		M_{(1)}^{10\theta}\cdot M_{(1)}^{\frac72-3s-2\alpha},
	\end{align*}
	where we used $M_{(3)}^{-1}\ll M_{(1)}^{-1/2}$ in the last step.

	\emph{(DDDC) interactions.}
	For~(2), we apply~\eqref{numerics1-new} with $\mathbf{a}^{(4)}$ a scale-$M_4$ $h^{s_C}$ envelope (see Definition \ref{def:scale-seq}) and the remaining $\mathbf{a}^{(j)}$ scale-$M_j$ $h^s$ envelopes:
	\begin{align*}
		M_{(1)}^{2\theta+\varepsilon}
		K_{M_1,M_2,M_3,M_4}\cdot(M_1 M_2 M_3)^{-s}M_4^{-s_C}
		\lesssim
		M_{(1)}^{10\theta+\frac{2}{q}}\cdot M_{(1)}^{\frac72-3s-2\alpha}.
	\end{align*}
	For~(3), we view $\mathbf{a}^{(3)}$ as the scale-$M_3$ $h^{s_C}$ envelope and the remaining as $h^s$ envelopes:
	\begin{align*}
		M_{(1)}^{2\theta+\varepsilon}
		K_{M_1,M_2,M_3,M_4}\cdot(M_1 M_2 M_4)^{-s}M_3^{-s_C}
		\lesssim
		M_{(1)}^{10\theta+\frac{2}{q}}\cdot M_{(1)}^{\frac72-3s-2\alpha}.
	\end{align*}

	\emph{(DDCC) interactions.}
	For~(4), we view $\mathbf{a}^{(3)}$ and $\mathbf{a}^{(4)}$ as $h^{s_C}$ envelopes and $\mathbf{a}^{(1)},\mathbf{a}^{(2)}$ as $h^s$ envelopes.
	Recall that $M_{(1)}\sim M_{(2)}$ and $s_C=\alpha-\frac1q-\delta$, we have
	\begin{align*}
		M_{(1)}^{2\theta+\varepsilon}
		K_{M_1,M_2,M_3,M_4}\cdot M_1^{-s}M_2^{-s}M_3^{-s_C}M_4^{-s_C}
		\lesssim
		M_{(1)}^{10\theta+\frac{2}{q}}\cdot M_{(1)}^{4-3\alpha-2s}.
	\end{align*}
	For~(5), the same argument gives
	\[
	M_{(1)}^{2\theta+\varepsilon}
	K_{M_1,M_2,M_3,M_4}\cdot M_1^{-s_C}M_2^{-s_C}M_3^{-s}M_4^{-s}
	\lesssim
	M_{(1)}^{10\theta+\frac{2}{q}}\cdot M_{(1)}^{4-3\alpha-2s}.
	\]
	Restoring the omitted factors yields bounds~(1)--(5). The uniform estimate~\eqref{cor:4-linear-1-uniform} follows from these five cases and the hierarchy of parameters.
\end{proof}

For the probabilistic estimates, we need the following variant of the moment bound from Lemma~\ref{momentbd'}, adapted to handle the sharp time cutoff arising from the pointwise-in-time reduction. For a type (C) term $v_M^{(C)}$ at scale $M$ from Definition \ref{def:CD} associated to operators $\mathcal{T}_m^M(t)$, we denote
$$ g_m^{M}(t):=\mathcal{T}_m^M(t)\cdot g_m(\omega).
$$

\begin{lemma}\label{lem:momentnew}
	Let $(L_j,M_j)_{1\leq j\leq 4}$  be dyadic integers, and let $N=\max(M_1,M_2,M_3,M_4)$.
Let $\mathfrak{h}$ be a deterministic tensor.
Then, $(M^{\delta}R,\delta;\mathcal{B}_{\leq M})$-certainly,
for any $\gamma_0<\frac{1}{q'}$ and any $0<|t|\leq 1$,
\begin{multline*}
		\left|\int_{\R^4}\dd\vec{\kappa}\sum_{\substack{m_1,m_2,m_3,m_4\\
				\{m_1,m_3\}\cap\{m_2,m_4\}=\emptyset}}\!\!\!\!\!\!
		\frac{\mathfrak{h}(\vec{m})\,\widehat{\chi}(\widetilde{\kappa}-\Omega(\vec{m}))}{(m_1 m_2 m_3 m_4)^{\alpha}}\right.
		\left.
		\widetilde{\mathbf{1}_{[0,t]}g_{m_1}^{M_1}}(\kappa_1)\,\mathbf{1}_{|\kappa_1|\sim L_1}\prod_{j=2}^4\widetilde{g}^{M_j}_{m_j}(\kappa_j)\,\mathbf{1}_{|\kappa_j|\sim L_j}
		\right|\\
		\lesssim
		M^{\delta+2(\gamma-\gamma_0)}R^4
		(M_{1}M_{2}M_{3}M_{4})^{-\alpha}\sup_{\kappa\in\R}\|\mathfrak{h}(\vec{m})\,\widehat{\chi}(\kappa-\Omega(\vec{m}))\|_{\ell_{m_1,m_2,m_3,m_4}^2}.
	\end{multline*}
\end{lemma}

\begin{proof}
	Note that
	\[
	\widetilde{\mathbf{1}_{[0,t]}g_{m_1}^{M_1}}(\kappa_1)=\widetilde{\mathbf{1}_{[0,t]}\mathcal{T}_{m_1}^{M_1}}(\kappa_1)\, g_{m_1}.
	\]
	The proof follows the same argument as Lemma~\ref{momentbd'}. The only difference is that $\mathbf{1}_{[0,t]}\mathcal{T}_{m_1}^{M_1}\in \widetilde{\mathcal{F}}L_q^{\gamma_0}$ for $\gamma_0<\frac{1}{q'}$, by Lemma~\ref{productrule}. Since the modulation variables satisfy $|\kappa_j|\sim L_j$ and $|\kappa_j|\lesssim M^2$, we pick up an extra loss of $|\kappa_1|^{\gamma-\gamma_0}\lesssim M^{2(\gamma-\gamma_0)}$.
\end{proof}

We now turn to the probabilistic quintic estimates, which handle configurations with at least three type~(C) inputs.

\begin{corollary}\label{cor:4-linearprob}
	Assume $\alpha>17/18$ and $(s,q,\gamma,b,\delta)=(s_{\sigma},q_{\sigma},\gamma_{\sigma},b_{\sigma},\delta_{\sigma})$ satisfies the hierarchy in Definition~\ref{hierarchy}, with $\theta=b-\frac{1}{2}=\gamma-\frac{1}{q'}$. Assume moreover $M_{(1)}\sim M_{(2)}$, and let $M=M_{(1)}$.
Then $(M^{\delta}R,\delta;\mathcal{B}_{\leq M})$-certainly,
	\begin{align*}
		\mathrm{(1)} \quad &
		\|\chi(t)\,\mathcal{N}_{n,1}^{(4)}(v^{(\mathrm{C})}_{M_1},v^{(\mathrm{C})}_{M_2},v^{(\mathrm{C})}_{M_3},v^{(\mathrm{D})}_{M_4})\|_{L_t^{\infty}}
		\lesssim
		R^3 M_{(1)}^{10\theta+\delta}\cdot
		\frac{M_{(1)}\wedge n}{M_{(1)}}\cdot M_{(1)}^{-(3\alpha+s-3)},\\
		\mathrm{(2)} \quad &
		\|\chi(t)\,\mathcal{N}_{n,1}^{(4)}(v^{(\mathrm{C})}_{M_1},v^{(\mathrm{C})}_{M_2},v^{(\mathrm{D})}_{M_3},v^{(\mathrm{C})}_{M_4})\|_{L_t^{\infty}}
		\lesssim
		R^3 M_{(1)}^{10\theta+\delta}\cdot
		\frac{M_{(1)}\wedge n}{M_{(1)}}\cdot
		M_{(1)}^{-(4\alpha+s-4)},\\
		\mathrm{(3)}\quad &
		\|\chi(t)\,\mathcal{N}_{n,1}^{(4)}(v^{(\mathrm{C})}_{M_1},v^{(\mathrm{C})}_{M_2},v^{(\mathrm{C})}_{M_3},v^{(\mathrm{C})}_{M_4})\|_{L_t^{\infty}}
		\lesssim
		R^4
		M_{(1)}^{10\theta+\delta}\cdot
		\frac{M_{(1)}\wedge n}{M_{(1)}}\cdot
		M_{(1)}^{-(5\alpha-4)}.
	\end{align*}
\end{corollary}

\begin{remark}\label{rmk:numerologyquintic}
	The hierarchy of parameters gives
	\[
	3\alpha+2s-4<s+4\alpha-4< \min\{s+3\alpha-3,\, 5\alpha-4\}.
	\]
	Combining this with Corollary~\ref{cor:4-linear-1}, we see that $(M^{\delta}R,\delta;\mathcal{B}_{\leq M})$-certainly, bounds~(1)--(5) in Corollary~\ref{cor:4-linear-1} together with~(1)--(3) of Corollary~\ref{cor:4-linearprob} are all controlled by the uniform estimate
	\begin{align}\label{finalbd:Nn4-1}
\|\chi(t)\mathcal{N}_{n,1}^{(4)}(v_{M_1}^{(\ast_1)},\cdots, v_{M_4}^{(\ast_4)}) \|_{L_t^{\infty}}
\lesssim		R^4T^{-4\theta} M_{(1)}^{10\theta+\delta}\cdot\frac{M_{(1)}\wedge n}{M_{(1)}}\cdot M_{(1)}^{-(3\alpha+2s-4)}.
	\end{align}
This estimate is also consistency with Corollary \ref{cor;N4n2} later, under the constraint of $s<\frac{\alpha}{4}+\frac{1}{4}$. 
\end{remark}
\vspace{0.3cm}

\begin{proof}[Proof of Corollary~\ref{cor:4-linearprob}]
	As in the proof of Corollary~\ref{cor:4-linear-1}, we work in the regime $M_4\lesssim n$; the complementary regime yields the extra factor $\frac{M_{(1)}\wedge n}{M_{(1)}}$ by the same argument.

	\emph{Case~1: $v_{M_j}^{(\ast_j)}$ for $j=1,2,3$ are of type~(C), and $v_{M_4}^{(\ast_4)}$ is of type~(D).}
	Set $h_{m_4}(t'):=\gamma_{nnm_4m_4}\, v_{m_4}^{(\mathrm{D})}(t')$.
	For fixed $t\in\R$ and $n\in\N$,
	\[
	\mathcal{N}_{n,1}^{(4)}(v_{M_1}^{(\mathrm{C})},v_{M_2}^{(\mathrm{C})},v_{M_3}^{(\mathrm{C})},v_{M_4}^{(\mathrm{D})})(t)
	=
	2\Im \int_{\R}\mathbf{1}_{[0,t]}(t')\, h(t')\,V(t')\,\dd t',
	\]
	where $V(t)=\sum_m V_m(t)$ and
	\[
	V_m(t)=\chi_1(t)\sum_{\substack{m_1,m_2,m_3\\
			\{m,m_2\}\cap\{m_1,m_3\}=\emptyset}}
	\gamma_{mm_1m_2m_3}
	\frac{g_{m_1}^{M_1,\dag}(t)\,\ov{g}_{m_2}^{M_2,\dag}(t)\,g_{m_3}^{M_3,\dag}(t)}
	{m_1^\alpha m_2^\alpha m_3^\alpha}.
	\]
	For any $b_0=(1-\vartheta)b<\frac12$,
	\[
	\big|\mathcal{N}_{n,1}^{(4)}(v_{M_1}^{(\mathrm{C})},v_{M_2}^{(\mathrm{C})},v_{M_3}^{(\mathrm{C})},v_{M_4}^{(\mathrm{D})})(t)\big|
	\lesssim
	\|\mathbf{1}_{[0,t]}h\|_{X^{0,b_0}}\|V\|_{X^{0,-b_0}}.
	\]
	By Lemma~\ref{productrule},
	\[
	\|\mathbf{1}_{[0,t]}h\|_{X^{0,b_0}}
	\lesssim
	\|h\|_{X^{0,b_0}}
	\lesssim
	M_4^{1-s}\log M_4.
	\]
	By Lemma~\ref{momentbd'}, $(M^\delta R,\delta;\mathcal{B}_{\le M})$-certainly,
	\begin{align*}
		\|V\|_{X^{0,-b}}
		&\lesssim
		M^\delta R^3 (M_1 M_2 M_3)^{-\alpha}
		\sup_{\kappa}\|\gamma_{mm_1m_2m_3}\,\widehat{\chi}_1(\kappa-\Omega(\vec m))\|_{\ell_{m_1,m_2,m_3,m}^2}\\
		&\lesssim_{\varepsilon}
		M^\delta R^3\cdot M_{(4)}\cdot(M_1 M_2 M_3)^{-\alpha}
		\cdot M_{(3)}^{1/2}M_{(4)}^{1/2}M_{(1)}^{1-\alpha+\varepsilon}.
	\end{align*}
	We also have the crude bound
	\[
	\|V\|_{X^{0,0}}
	\lesssim
	R^3 M_{(4)}M_4^{1/2}\cdot (M_1 M_2 M_3)^{1-\alpha+\varepsilon}
	\lesssim
	R^3 M_{(1)}^3.
	\]
	Interpolating between $X^{0,0}$ and $X^{0,-b}$ gives
	\[
	\|V\|_{X^{0,-b_0}}
	\lesssim
	\|V\|_{X^{0,0}}^{\vartheta}\|V\|_{X^{0,-b}}^{1-\vartheta},
	\]
	and therefore
	\[
	\big|\mathcal{N}_{n,1}^{(4)}(v_{M_1}^{(\mathrm{C})},v_{M_2}^{(\mathrm{C})},v_{M_3}^{(\mathrm{C})},v_{M_4}^{(\mathrm{D})})(t)\big|
	\lesssim
	R^3 M_{(1)}^{10\theta+\delta}\, M_{(1)}^{-(3\alpha+s-3)}.
	\]

	\emph{Case~2: $v_{M_4}^{(\ast_4)}$ is of type~(C), and one of $v_{M_j}^{(\ast_j)}$ for $j=1,2,3$ is of type~(D).}
	By symmetry, it suffices to consider $v_{M_3}^{(\ast_3)}$ of type~(D).
	Set $h_{m_3}(t'):=v_{m_3}^{(\mathrm{D})}(t')$ and
	\[
	V_{m_3}(t)
	=
	\chi_1(t)
	\sum_{\substack{m_1,m_2,m_4\\ \text{non\;\!\! degenerate}}}
	\gamma_{nnm_4m_4}\gamma_{m_1m_2m_3m_4}
	\frac{g_{m_1}^{M_1,\dag}(t)\,\ov{g}_{m_2}^{M_2,\dag}(t)\,\ov{g}_{m_4}^{M_4,\dag}(t)}
	{m_1^\alpha m_2^\alpha m_4^\alpha}.
	\]
	Then $\|\mathbf{1}_{[0,t]}h\|_{X^{0,b_0}}\lesssim M_3^{-s}$,
	and by Lemma~\ref{momentbd'}, $(M^\delta R,\delta;\mathcal{B}_{\le M})$-certainly,
	\begin{align*}
		\|V\|_{X^{0,-b}}
		&\lesssim
		M^\delta R^3 (M_1 M_2 M_4)^{-\alpha}
		\sup_{\kappa}\|\gamma_{nnm_4m_4}\gamma_{m_1m_2m_3m_4}\,\widehat{\chi_1}(\kappa-\Omega(\vec m))\|_{\ell_{m_1,m_2,m_3,m_4}^2}\\
		&\lesssim_{\varepsilon}
		M^\delta R^3\cdot M_4 M_{(4)}\cdot (M_1 M_2 M_4)^{-\alpha}
		\cdot M_{(3)}^{1/2}M_{(4)}^{1/2}M_{(1)}^{1-\alpha+\varepsilon}.
	\end{align*}
	The same interpolation argument yields
	\[
	\big|\mathcal{N}_{n,1}^{(4)}(v_{M_1}^{(\mathrm{C})},v_{M_2}^{(\mathrm{C})},v_{M_3}^{(\mathrm{D})},v_{M_4}^{(\mathrm{C})})(t)\big|
	\lesssim
	R^3 M_{(1)}^{10\theta+\delta}\cdot M_{(1)}^{-(4\alpha+s-4)}.
	\]

	\emph{Case~3: all $v_{M_j}^{(\ast_j)}$ for $j=1,2,3,4$ are of type~(C).}
	We dyadically decompose the modulation variables:
	\[
	\widetilde{v}_{M_j,L_j}^{(\mathrm{C})}(\kappa_j)
	:=
	\mathbf{1}_{L_j/2<|\kappa_j|\leq L_j}\,\widetilde{v}_{M_j}^{(\mathrm{C})}(\kappa_j),
	\qquad j=1,2,3,4,
	\]
	so that
	\[
	\mathcal{N}_{n,1}^{(4)}(v_{M_1}^{(\mathrm{C})},\cdots,v_{M_4}^{(\mathrm{C})})
	=
	\sum_{\substack{L_1,L_2,L_3,L_4\\ L_j\lesssim M_{(1)}^2}}
	\mathcal{N}_{n,1}^{(4)}(v_{M_1,L_1}^{(\mathrm{C})},\cdots,v_{M_4,L_4}^{(\mathrm{C})}).
	\]
	Absorbing $\mathbf{1}_{[0,t]}$ into the first colored Gaussian and applying Lemma~\ref{lem:momentnew}, we deduce that $(M^\delta R,\delta;\mathcal{B}_{\le M})$-certainly,
	\begin{align*}
		\big\|\mathcal{N}_{n,1}^{(4)}(v_{M_1,L_1}^{(\mathrm{C})},\cdots,v_{M_4,L_4}^{(\mathrm{C})})\big\|_{L_t^\infty}
		&\lesssim
		R^4(M_1 M_2 M_3 M_4)^{-\alpha}M^{4\theta+\delta}\\
		&\quad\times\sup_{\kappa}
		\|\gamma_{nnm_4m_4}\gamma_{m_1m_2m_3m_4}\,\widehat{\chi_1}(\kappa-\Omega(\vec m))\|_{\ell_{m_1,m_2,m_3,m_4}^2}\\
		&\lesssim_{\varepsilon}
		R^4 M^{4\theta+\delta}\left(\prod_{i=1}^{4}M_{i}^{-\alpha}\right)M_4 M_{(4)}
		\cdot M_{(1)}^{1-\alpha+\varepsilon}(M_{(3)}M_{(4)})^{\frac12}.
	\end{align*}
	Summing over $L_j\lesssim M_{(1)}^2$ costs at most $M_{(1)}^\varepsilon$, which is absorbed into $M^\delta$. Therefore,
	\[
	\|\mathcal{N}_{n,1}^{(4)}(v_{M_1}^{(\mathrm{C})},\cdots,v_{M_4}^{(\mathrm{C})})\|_{L_t^\infty}
	\lesssim
	R^4 M_{(1)}^{10\theta+\delta}\cdot M_{(1)}^{-(5\alpha-4)}.
	\]
	This proves~(3) and completes the proof.
\end{proof}

\subsection{Pointwise-in-time estimate for $\mathcal{N}_{n,2}^{(4)}$.}

In this subsection, we prove:
\begin{corollary}\label{cor;N4n2}
Assume that $\alpha>\frac{17}{18}$ and $(s,q,\gamma,b,\delta)=(s_{\sigma},q_{\sigma},\gamma_{\sigma},b_{\sigma},\delta_{\sigma})$ satisfies the hierarchy in Definition \eqref{hierarchy} with $\theta=b-\frac{1}{2}=\gamma-\frac{1}{q'}$. Assume moreover $$\mathbf{P}_{M_j}v_{M_j}^{(\ast_j)}=v_{M_j}^{(\ast_j)},\quad \forall j=1,2,3,4, \ast_j\in\{\mathrm{(C)},\mathrm{(D)} \}.$$
Then for $\nu\in\{2,3,4\}$,
\begin{align*}
\big\|\mathcal{N}_{n,\nu}^{(4)}(v_{M_1}^{(\ast_1)},v_{M_2}^{(\ast_2)},v_{M_3}^{(\ast_3)},v_{M_4}^{(\ast_4)}) \big\|_{L_t^{\infty}}\lesssim R^4T^{-4\theta}M_{(1)}^{10\theta}\frac{M_4\wedge n}{M_4}\cdot M_{(1)}^{-(4\alpha-\frac{7}{2})},
\end{align*}
where $M_{(1)}=\max(M_4,M_3,M_1)$.
\end{corollary}

\begin{remark}
Note that, under the constraint
		$$ s<\frac{\alpha}{4}+\frac{1}{4},
		$$
we have $3\alpha+2s-4<4\alpha-\frac{7}{2}$. Hence
	the same uniform estimate as  \eqref{finalbd:Nn4-1} holds:
\begin{align}\label{finalbd:Nn4}
\|\chi(t)\mathcal{N}_{n,1}^{(4)}(v_{M_1}^{(\ast_1)},\cdots, v_{M_4}^{(\ast_4)}) \|_{L_t^{\infty}}
\lesssim		R^4T^{-4\theta} M_{(1)}^{10\theta+\delta}\cdot\frac{M_{(1)}\wedge n}{M_{(1)}}\cdot M_{(1)}^{-(3\alpha+2s-4)}.
\end{align}
Hence \eqref{finalbd:Nn4} is a final uniform bound for dyadic pieces to be summed. See Section \ref{sec:mainproof}.
\end{remark}
\vspace{0.3cm}

\begin{proof}
We will prove the estimates for $\mathcal{N}_{n,2}^{(4)}$, since the estimate for $\mathcal{N}_{n,3}^{(4)}, \mathcal{N}_{n,4}^{(4)}$ are similar (and simpler) to $\mathcal{N}_{n,2}^{(4)}$ (see the last paragraph of the proof).

By similar  argument as in the beginning of the proof of Corollary \ref{cor:4-linear-1}, we only consider the regime where $M_4\lesssim n$. 

By definition, for fixed $n$ and $t$,
\begin{align*}
\mathcal{N}_{n,2}^{(4)}(v_{M_1}^{(\ast_1)},\cdots, v_{M_4}^{(\ast_4)} )(t)=\sum_{\substack{m_1,m_3\\
m_1\neq m_3
 } }&\gamma_{nnm_4m_4}\gamma_{m_1m_1m_3m_4}\\ \times &\int_{\R}\mathbf{1}_{[0,t]}(t')v_{m_1}^{(\ast_1)}(t')\ov{v_{m_1}^{(\ast_2)}}(t')v_{m_3}^{(\ast_3)}(t')\ov{v_{m_4}^{(\ast_4)}}(t')\dd t'. 
\end{align*}
Without loss of generality, in the sequel of the proof, we assume that 
$$ M_4\lesssim \max(M_1,M_3),
$$
since otherwise, from \eqref{eq:gamma3}, we gain directly $\max(M_4,M_3,M_1)^{-2}$ from the correlation coefficient $\gamma_{m_1m_1m_3m_4}$ for which the desired bound follows easily.

Applying \eqref{resonant:1}  in Proposition \ref{linear:type1} in the dual form to $N_0=M_4, N_1=M_3, N_2=N_3=M_1$, and $F=v_{M_1}^{(\ast_1)}, G=v_{M_1}^{(\ast_2)},$
$ Z=v_{M_3}^{(\ast_3)},
$
and the dual function
$$ H(t',x)=\mathbf{1}_{[0,t]}(t')\sum_{m_4\sim M_4}\gamma_{nnm_4m_4}v_{m_4}^{(\ast_4)}(t',x)\;\mathbf{e}_{m_4}(x).
$$
We obtain that (for some $b_1<1/2$)
\begin{align*}
\big|\mathcal{N}_{n,2}^{(4)}(v_{M_1}^{(\ast_1)},\cdots, v_{M_4}^{(\ast_4)}) (t)\big|\lesssim & A(M_4,M_3,M_1)\|v_{M_1}^{(\ast_1)}\|_{X^{0,b}}
\|v_{M_1}^{(\ast_2)}\|_{X^{0,b}}
\|v_{M_3}^{(\ast_3)}\|_{X^{0,b}}
\|H\|_{X^{0,b_1}},
\end{align*}
where we recall that
$$ A(M_4,M_3,M_1)=\min(M_4,M_3,M_1)\frac{\max(M_3,M_4)^{1-2b_1}}{\min(M_3,M_4)^{\frac{1}{2}}}.
$$
By Lemma \ref{productrule} and the correlation coefficient bound $\gamma_{nnm_4m_4}\lesssim \min(M_4,n)$, we have
$$ \|H\|_{X^{0,b_1}}\lesssim M_4\|v_{M_4}^{(\ast_4)}\|_{X^{0,b}}.
$$
As we only use the $X^{0,b}$ norm for terms of any possible type, the worst estimate appears when $\ast_1=\ast_2=\ast_3=\ast_4=\mathrm{C}$. Hence,
\begin{align*}
|\mathcal{N}_{n,2}^{(4)}(v_{M_1}^{(\ast_1)},\cdots, &v_{M_4}^{(\ast_4)})(t)|\\ \lesssim & (RT^{-\theta})^4 A(M_4,M_3,M_1)\cdot M_1^{-2s_C}(M_3M_4)^{-s_D}\cdot M_4\\
\lesssim & (RT^{-\theta})^4\max(M_4,M_3,M_1)^{\frac{5}{2}-2b_1-4s_D}\cdot.
\end{align*}
where
$$ s_D:=\alpha-\frac{1}{2}-\delta,
$$
and to the last line of the right hand side, we used the fact that $M_4\lesssim \max(M_1,M_3)$. Choosing $b_1=\frac{1}{2}-\delta$, we finally obtain the bound
$$ R^4T^{-4\theta}\max(M_4,M_3,M_1)^{-(4\alpha-\frac{7}{2}-6\delta)}.
$$

Finally, we explain why similar argument applies to $\mathcal{N}_{n,3}^{(4)}$ and $\mathcal{N}_{n,4}^{(4)}$. For instance, in order to estimate $\mathcal{N}_{n,3}^{(4)}$, since $m_1=m_3=m_4\neq m_2$, we write $\mathcal{N}_{n,3}^{(4)}$ in a similar for as $\mathcal{N}_{n,2}^{(4)}$ by modifying the correlation coefficient $\gamma_{m_1m_2m_3m_4}$ by $\gamma_{m_1m_2m_3m_4}\mathbf{1}_{m_1=m_3=m_4}$. Since in the proof of Proposition \ref{linear:type1}, the correlation coefficient is bounded by its canonical bound \eqref{eq:gamma'} at the beginning (see the inequality after \eqref{linear:type1-1}), one could simply forget the indicator $\mathbf{1}_{m_1=m_3=m_4}$ at this step to conclude with the same estimate as \eqref{resonant:1}. All the steps in the sequel remain the same. 

In summary, the proof of Corollary \ref{cor;N4n2} is now complete.
\end{proof}

	
	\section{Proof of Proposition~\ref{prop:trilinear} and Proposition~\ref{prop:quintilinear}}\label{sec:mainproof}

In this section, we carry out the remaining dyadic summation arguments needed to close the proof of Propositions~\ref{prop:trilinear} and~\ref{prop:quintilinear}.
Since most of the analysis reduces to controlling dyadic sums, we introduce the following notion of envelope.

\begin{definition}[Scale-$N$ $h^{s}$-envelope]\label{def:scale-seq}
	Let $\mathbf{a}=(a_M)_{M\in 2^{\mathbb N}}$ be a dyadic sequence, and for $\sigma\in\mathbb R$ define
	\[
	\|\mathbf{a}\|_{h^{\sigma}} := \bigl\| M^{\sigma} a_M \bigr\|_{\ell^2_{M\in 2^{\mathbb N}}}.
	\]
	(So $h^0=\ell^2$.) Fix $0<s\leq 1$. For a dyadic number $N\in 2^{\mathbb N}$, we say that $\mathbf{a}$ is a \emph{scale-$N$ $h^{s}$-envelope} if:
	\begin{align*}
		\text{(1)}\quad & \|\mathbf{a}\|_{h^0} \lesssim N^{-s},\\
		\text{(2)}\quad & \|\mathbf{a}\|_{h^{s_0}} \lesssim_{s_0,s} N^{-(s-s_0)},\qquad \forall\, s_0\in(0,s),\\
		\text{(3)}\quad & |a_M|\lesssim \Big(\frac{N}{M}\Big) N^{-s},\qquad \forall\, M\ge N.
	\end{align*}
\end{definition}

\subsection{Proof of Proposition~\ref{prop:trilinear}}

We start with some reductions by interpolation.
We first apply the linear inhomogeneous estimate
\begin{align*}
	\|\mathcal{I}_{\chi_T}\mathcal{N}(\cdots)\|_{X^{0,b}}\lesssim T^{\sigma}\|\mathcal{N}(\cdots)\|_{X^{0,b+\sigma-1}}=T^{\sigma}\|\mathcal{N}(\cdots)\|_{X^{0,-\frac{1}{2}+\sigma+\sigma^{10}}}.
\end{align*}
No matter the types of nonlinearity and inputs, we will establish the crude estimate
\begin{align}\label{algorithm:A_0}
	\|\mathcal{N}(\cdots)\|_{X^{0,0}}\leq R^3 A_0(\vec{N}).
\end{align}
We then interpolate with the estimate
\begin{align}\label{algorithm:A_1}
	\|\mathcal{N}(\cdots)\|_{X^{0,-b_1}}\leq R^3 A_1(\vec{N}),
\end{align}
where $b_1=\frac{1}{2}+\sigma^{10}$.
Since $b=b_{\sigma}=\frac{1}{2}+\sigma^{10}$, this gives
\begin{align}\label{algorithm:output}
	\|\mathcal{I}_{\chi_T}\mathcal{N}(\cdots)\|_{X^{0,b}}\leq R^3 T^{\sigma}(A_0(\vec{N}))^{\frac{2\sigma+4\sigma^{10}}{1+2\sigma^{10}}} \cdot (A_1(\vec{N}))^{\frac{1-2\sigma-2\sigma^{10}}{1+2\sigma^{10}}}.
\end{align}

As a first step, we establish the following lemma, which allows us to reduce to the regime where the two largest frequencies are comparable.

\begin{lemma}[High-low-low to low bound]\label{lem:highlowlowlow}
	For $v_{N_1}^{(\ast_1)},v_{N_2}^{(\ast_2)},v_{N_3}^{(\ast_3)}$, where $\ast_j\in\{(\mathrm{C}),(\mathrm{D})\}$, we have
	\begin{align*}
		\Big\|\mathcal{I}_{\chi_T}\sum_{\substack{K_1,K_2,K_3,K_4 \\
				K_{(1)}\gg K_{(2)}
		}}\mathbf{P}_{K_4}\mathcal{N}\big(\mathbf{P}_{K_1} v_{N_1}^{(\ast_1)},
		\mathbf{P}_{K_2}v_{N_2}^{(\ast_2)},
		\mathbf{P}_{K_3}v_{N_3}^{(\ast_3)}
		\big)\Big\|_{X^{0,b}}\lesssim R^3 T^{-3(b-\frac{1}{2})}N_{(1)}^{-s}N_{(2)}^{-\frac{\theta}{2}}
	\end{align*}
	for some $\theta>2^{100}\sigma$.
\end{lemma}

\begin{proof}
	It suffices to estimate the $X^{0,0}$-norm of the nonlinearity in this regime.
	We distinguish two cases according to the type of the highest-scale input.

	\noi
	$\bullet${\bf Case 1: $\exists\, j\in\{1,2,3\}$ with $(\ast_j)=(\mathrm{C})$ and $N_j=N_{(1)}$.}

	We first establish the crude estimate~\eqref{algorithm:A_0}.
	Denote $a_{n_j}^{(j)}(t)=|\mathbf{P}_{K_j}v_{N_j,n_j}^{(\ast_j)}(t)|$ for $j=1,2,3$. By duality, for any $(d_{n}(t))_{n\in \N}\in \ell^2(\N)$ with compact support in~$t$ and $\|d_{n}(t)\|_{L_t^2\ell_n^2}\leq 1$,
	\begin{align*}
		&\int_{\R}\dd t \sum_{\substack{n_1,n_2,n_3,n_4\\
				n_j\sim K_j
		}} \gamma_{n_1 n_2 n_3 n_4}\cdot d_{n_4}(t)\cdot\prod_{j=1}^3 a_{n_j}^{(j)}(t)\\
		&\qquad\lesssim K_{(1)}^{-2}\cdot (K_1 K_2 K_3 K_4)^{\frac{1}{2}} \|d_{n_4}(t)\|_{L_t^2\ell_{n_4}^2}\prod_{j=1}^3\|a_{n_j}^{(j)}\|_{L_t^{\infty}\ell_{n_j}^2}.
	\end{align*}
	Since $K_{(1)}\gg K_{(2)}$, we have $\gamma_{n_1 n_2 n_3 n_4}\lesssim K_{(1)}^{-2}$ for $n_j\sim K_j$.
	Thus
	\begin{multline*}
		\big\|\chi(t)\,\mathbf{P}_{K_4}\mathcal{N}\big(\mathbf{P}_{K_1}v_{N_1}^{(\ast_1)},
		\mathbf{P}_{K_2}v_{N_2}^{(\ast_2)},
		\mathbf{P}_{K_3}v_{N_3}^{(\ast_3)}
		\big)\big\|_{X^{0,0}}
		\\
		\lesssim K_{(1)}^{-2}\cdot (K_1 K_2 K_3 K_4)^{\frac{1}{2}}\prod_{j=1}^3\|\mathbf{P}_{K_j}v_{N_j}^{(\ast_j)}\|_{X^{0,b}}.
	\end{multline*}
	Since $(\ast_j)=(\mathrm{C})$ and $N_j=N_{(1)}$, we must have $K_j=N_j=N_{(1)}\leq K_{(1)}$. To fix ideas, assume $j=1$, so that $K_1=N_1=N_{(1)}$. For $i=2,3$, if $(\ast_i)=(\mathrm{C})$, then $K_i=N_i$ and
	\[
	\big\| K_i^{1/2}\|\mathbf{P}_{K_i}v_{K_i}^{(\ast_i)}\|_{X^{0,b}}\big\|_{\ell_{K_i}^1}\lesssim R\, T^{-(b-\frac{1}{2})}N_i^{1-\alpha+\frac{1}{q}+\delta}.
	\]
	If $(\ast_i)=(\mathrm{D})$, we have
	\[
	\big\|K_i^{\frac{1}{2}}\|\mathbf{P}_{K_i}v_{N_i}^{(\ast_i)}\|_{X^{0,b}}\big\|_{\ell_{K_i}^2}\lesssim_{\epsilon} N_i^{-(s-\frac{1}{2}-\epsilon)}.
	\]
	Therefore
	\begin{multline*}
		\Big(
		\sum_{K_4}
		\Big(
		\sum_{\substack{K_1,K_2,K_3\\
				K_{(1)}\gg K_{(2)}
		}}
		\Big\|
		\mathbf{P}_{K_4}\mathcal{N}
		\big(
		\mathbf{P}_{K_1}v_{N_1}^{(\ast_1)},
		\mathbf{P}_{K_2}v_{N_2}^{(\ast_2)},
		\mathbf{P}_{K_3}v_{N_3}^{(\ast_3)}
		\big)
		\Big\|_{X^{0,0}}
		\Big)^2
		\Big)^{1/2}\\
		\lesssim_{\epsilon}
		R^3 T^{-(b-\frac{1}{2})}
		N_{(1)}^{-\frac{3}{2}+3(1-\alpha)+\epsilon+3(\frac{1}{q}+\delta)},
	\end{multline*}
	which is sufficient since $\frac{3}{2}-3(1-\alpha)-3(\frac{1}{q}+\delta)>s$.

	\noi
	$\bullet${\bf Case 2: $\exists\, j\in\{1,2,3\}$ with $(\ast_j)=(\mathrm{D})$ and $N_j=N_{(1)}$.}

	The argument is essentially the same as in Case~1, except that we may have $K_{(1)}\ll N_{(1)}$. Assume $N_1=N_{(1)}$ and $v_{N_1}^{(\ast_1)}$ is of type~(D). We have
	\begin{align*}
		&\Big(\sum_{K_4}\Big(\sum_{\substack{K_1,K_2,K_3\\
				K_{(1)}\gg K_{(2)}
		}} \Big\|\mathbf{P}_{K_4}\mathcal{N}
		\big(
		\mathbf{P}_{K_1}v_{N_1}^{(\ast_1)},
		\mathbf{P}_{K_2}v_{N_2}^{(\ast_2)},
		\mathbf{P}_{K_3}v_{N_3}^{(\ast_3)}
		\big)\Big\|_{X^{0,0}}\Big)^2\Big)^{1/2}\\
		&\lesssim
		\Big(\sum_{K_4}
		\Big(\sum_{\substack{K_1,K_2,K_3\\
				K_{(1)}\gg K_{(2)}
		}} K_{(1)}^{-2}(K_2 K_3)^{1-\alpha+\theta}(K_1 K_4)^{1/2}\\
		&\qquad\qquad\times
		\|\mathbf{P}_{K_1}v_{N_1}^{(\mathrm{D})}\|_{X^{0,b}}
		\prod_{j=2}^3\|K_j^{\alpha-\frac{1}{2}-\theta}\mathbf{P}_{K_j}v_{N_j}^{(\ast_j)}\|_{X^{0,b}}
		\Big)^2
		\Big)^{1/2}\\
		&\lesssim
		\|v_{N_1}^{(\mathrm{D})}\|_{X^{0,b}}\prod_{j=2}^3\|v_{N_j}^{(\ast_j)}\|_{X^{\alpha-\frac{1}{2}-\theta,b}},
	\end{align*}
	for some $\theta>2^{100}\sigma$. For the last inequality, we used $2(1-\alpha+\theta)<1$, which gives
	\begin{align*}
		\sum_{K_4}\Big(\sum_{K_1,K_2,K_3}
		K_{(1)}^{-1}\cdot (K_2 K_3)^{1-\alpha+\theta}
		\Big)^2\lesssim 1.
	\end{align*}
	Since $\|v_{N_1}^{(\mathrm{D})}\|_{X^{0,b}}\lesssim N_1^{-s}$ and, by interpolation,
	\[
	\|v_{N_j}^{(\ast_j)}\|_{X^{\alpha-\frac{1}{2}-\theta,b}}\lesssim R\, T^{-(b-\frac{1}{2})}N_j^{-\frac{\theta}{2}},
	\]
	thanks to $\theta>2^{100}\sigma\gg \frac{1}{q}+\delta$, the proof is complete.
\end{proof}

\noi
$\bullet${\bf Proof of (A) in Proposition~\ref{prop:trilinear}:} Thanks to Lemma~\ref{lem:highlowlowlow}, in the analysis below we always assume $K_{(1)}\sim K_{(2)}$ without displaying this constraint explicitly.

For parameters $(s,q,\gamma,b,\delta)=(s_{\sigma},q_{\sigma},\gamma_{\sigma},b_{\sigma},\delta_{\sigma})$ satisfying the hierarchy of small parameters~\eqref{hierarchy}, set
\[
s_0=s-\theta,
\]
for some $\theta>2^{50}\sigma$ to be fixed in the proof. The reason for introducing this parameter is that when an input $v_{N_j}^{(\ast_j)}$ is of type~(D) and $N_j$ is not the dominant scale (i.e.\ $N_j\ll N_{(1)}$), we use the $X^{s_0,b}$ norm and interpolation to get
\[
\|v_{N_j}^{(\ast_j)}\|_{X^{s_0,b}}\leq \|v_{N_j}^{(\ast_j)}\|_{X^{0,b}}^{\frac{\theta}{s}}\cdot \|v_{N_j}^{(\ast_j)}\|_{X^{s,b}}^{1-\frac{\theta}{s}}\lesssim N_j^{-\frac{\theta}{2}}.
\]

\underline{\textbf{1. (D)(D)(D) interactions}:}
Denote
\[
w^{(j)}:=v_{N_j}^{(\mathrm{D})},\quad j=1,2,3.
\]
We first have the trivial estimate in $X^{0,0}$:
\[
\|\mathcal{N}_{[123]}^{(3)}(w^{(1)},w^{(2)},w^{(3)})\|_{X^{0,0}}\lesssim N_{(1)}^{-s}N_{(2)}^{10}.
\]
Thus we choose $A_0(\vec{N})=N_{(1)}^{-s}N_{(2)}^{10}$ in~\eqref{algorithm:A_0}.

To estimate~\eqref{algorithm:A_1}, we have
\begin{align*}
	\|\mathcal{N}_{[123]}^{(3)}(w^{(1)},w^{(2)},w^{(3)})\|_{X^{0,-b_1}}
	&\lesssim \Big(\sum_{K_4}\Big\|
	\sum_{K_1,K_2,K_3}\mathbf{P}_{K_4}\mathcal{N}_{[123]}^{(3)}(\mathbf{P}_{K_1}w^{(1)},
	\mathbf{P}_{K_2}w^{(2)},\\
	&\phantom{\lesssim\Big(\sum_{K_4}\Big\|\sum_{K_1,K_2,K_3}\mathbf{P}_{K_4}\mathcal{N}_{[123]}^{(3)}(}
	\mathbf{P}_{K_3}w^{(3)})
	\Big\|_{X^{0,-b_1}}^2\Big)^{1/2}\\
	&\lesssim
	\Big(\sum_{K_4}\Big(\sum_{K_1,K_2,K_3}\|
	\mathbf{P}_{K_4}\mathcal{N}_{[123]}^{(3)}(\mathbf{P}_{K_1}w^{(1)},
	\mathbf{P}_{K_2}w^{(2)},\\
	&\phantom{\lesssim\Big(\sum_{K_4}\Big(\sum_{K_1,K_2,K_3}\|\mathbf{P}_{K_4}\mathcal{N}_{[123]}^{(3)}(}
	\mathbf{P}_{K_3}w^{(3)})
	\|_{X^{0,-b_1}}\Big)^2\Big)^{1/2}.
\end{align*}
Applying Proposition~\ref{prop:non-resonant} (combining~\eqref{eq:tri1'} and~\eqref{eq:tri1})\footnote{Since the bound exhibits no derivative loss when $K_{(1)}\gg K_{(3)}^2$, we may assume $K_{(1)}\lesssim K_{(3)}^2$. This allows us to transfer the loss to the lower frequency, bounding $K_{(1)}^{2(1-\alpha)}$ by $K_{(3)}^{4(1-\alpha)}$.}, the right-hand side is bounded by
\begin{align*}
	\Big(\sum_{K_4}\Big(\sum_{K_1,K_2,K_3} c_{K_1 K_2 K_3 K_4}\prod_{j=1}^3 a_{K_j}^{(j)}\Big)^2\Big)^{1/2},
\end{align*}
where
\[
a_{K_j}^{(j)}=\|\mathbf{P}_{K_j}w^{(j)}\|_{X^{0,b}},\quad j=1,2,3,
\]
and the coefficient satisfies
\begin{align}\label{cK1K2K3K4_bd1}
	c_{K_1 K_2 K_3 K_4}\lesssim_{\epsilon}
	\begin{cases}
		K_{(4)}K_{(3)}^{4(1-\alpha)+\epsilon}, &\text{if } K_{4}\in\{K_{(1)},K_{(2)}\} \text{ and } K_{(1)}\sim K_{(2)},\\
		K_{(4)}K_{(1)}^{2(1-\alpha)+\epsilon}, &\text{if } K_4\in\{K_{(3)},K_{(4)}\} \text{ and } K_{(1)}\sim K_{(2)}.
	\end{cases}
\end{align}

To carry out the dyadic summation, we establish the following lemma.

\begin{lemma}\label{lem:dyadicsum1}
	Assume $\frac{1}{2}<\alpha\leq 1$ and $s>\frac{1}{2}$ with $\alpha+s>\frac{3}{2}$. Assume that $c_{K_1 K_2 K_3 K_4}$ satisfies~\eqref{cK1K2K3K4_bd1},
	where $K_{(1)}\geq K_{(2)}\geq K_{(3)}\geq K_{(4)}$ is the non-increasing rearrangement of $K_1,K_2,K_3,K_4$. Then for any $0<\theta<\min\{2\alpha+s-5/2,\, s-1/2\}$,
	\begin{align*}
		&\Big(\sum_{K_4}\Big(\sum_{\substack{K_1,K_2,K_3\\ K_{(1)}\sim K_{(2)}}} c_{K_1 K_2 K_3 K_4}\prod_{j=1}^3 a_{K_j}^{(j)}\Big)^2\Big)^{1/2}\\
		&\qquad\lesssim_{\theta} \min_{\sigma\in \mathfrak{S}_3}\|a_{K_{\sigma_1}}^{(\sigma_1)}\|_{\ell_{K_{\sigma_1}}^2}
		\|K_{\sigma_2}^{s-\theta}a_{K_{\sigma_2}}^{(\sigma_2)}\|_{\ell_{K_{\sigma_2}}^2}
		\|K_{\sigma_3}^{s-\theta}a_{K_{\sigma_3}}^{(\sigma_3)}\|_{\ell_{K_{\sigma_3}}^2}.
	\end{align*}
\end{lemma}

Before proving Lemma~\ref{lem:dyadicsum1}, we apply it to conclude part~(1) of Proposition~\ref{prop:trilinear}. Choose $\sigma\in\mathfrak{S}_3$ such that $w^{(\sigma_1)}=v^{(\mathrm{D})}_{N_{\sigma_1}}$ with $N_{\sigma_1}=N_{(1)}$. For the other two type~(D) terms $w^{(j)}$, we have
\[
\|w^{(j)}\|_{X^{s-\theta,b}}\lesssim N_j^{-\theta}\log N_j,
\]
with $\theta=2^{100}\sigma$. The hierarchy $s+3\alpha-7/2>2^{100}\sigma$ ensures that $\theta$ satisfies the smallness assumption in Lemma~\ref{lem:dyadicsum1}.
Applying Lemma~\ref{lem:dyadicsum1}, we obtain the bound
\[
N_{(1)}^{-s}N_{(2)}^{-\frac{\theta}{2}}N_{(3)}^{-\frac{\theta}{2}}.
\]
We choose $A_1(\vec{N})=N_{(1)}^{-s}N_{(2)}^{-\frac{\theta}{2}}$ in~\eqref{algorithm:A_1}. Plugging into~\eqref{algorithm:output} and using the hierarchy of small parameters, we conclude the proof of part~(A) in Proposition~\ref{prop:trilinear} for the (D)(D)(D) interaction.

\begin{proof}[Proof of Lemma~\ref{lem:dyadicsum1}]
	Set $s_0:=s-\theta$, where $\theta$ is small enough that
	\[
	s_0>\frac{1}{2},\quad 2\alpha+s_0>\frac{5}{2}.
	\]
	Since the hypothesis on $c_{K_1 K_2 K_3 K_4}$ is symmetric in $K_1,K_2,K_3$, it suffices to prove the estimate with the fixed choice $(\sigma_1,\sigma_2,\sigma_3)=(1,2,3)$.
	Set
	\[
	A_{K_1}:=|a_{K_1}^{(1)}|,\quad B_{K_2}^{(2)}=K_2^{s_0}|a_{K_2}^{(2)}|,\quad B_{K_3}^{(3)}=K_3^{s_0}|a_{K_3}^{(3)}|.
	\]
	By duality, it suffices to show that for every dyadic sequence $d_{K_4}\geq 0$ with $\|d_{K_4}\|_{\ell_{K_4}^2}=1$,
	\begin{align*}
		\Lambda:=\sum_{K_1,K_2,K_3,K_4}c_{K_1 K_2 K_3 K_4}\, K_2^{-s_0}K_3^{-s_0}\cdot A_{K_1}B_{K_2}^{(2)}B_{K_3}^{(3)}d_{K_4}\lesssim \|A\|_{\ell^2}\|B^{(2)}\|_{\ell^2}\|B^{(3)}\|_{\ell^2}.
	\end{align*}
	We decompose $\Lambda\lesssim_{\epsilon} \Lambda_{1}+\Lambda_2$, where
	\begin{align*}
		\Lambda_1&:=\sum_{\substack{K_1,K_2,K_3,K_4\\
				K_{(1)}\sim K_{(2)}\\
				K_4\in \{K_{(1)},K_{(2)}\}
		}}
		K_{(4)}K_{(3)}^{4(1-\alpha)+\epsilon}
		\cdot
		K_2^{-s_0}K_3^{-s_0}\cdot A_{K_1}B_{K_2}^{(2)}B_{K_3}^{(3)}d_{K_4},\\
		\Lambda_2&:=\sum_{\substack{K_1,K_2,K_3,K_4\\
				K_{(1)}\sim K_{(2)}\\
				K_4\in \{K_{(3)},K_{(4)}\}
		}}
		K_{(4)}K_{(1)}^{2(1-\alpha)+\epsilon}\, K_2^{-s_0}K_3^{-s_0}\cdot A_{K_1}B_{K_2}^{(2)}B_{K_3}^{(3)}d_{K_4}.
	\end{align*}
	For $\Lambda_1$, observe that
	\[
	K_{(4)}
	K_{(3)}^{4(1-\alpha)+\epsilon}
	K_{2}^{-s_0}K_3^{-s_0}\lesssim \max(K_2,K_3)^{-s_0+4(1-\alpha)+\epsilon}\min(K_2,K_3)^{-s_0+1}.
	\]
	The condition $2\alpha+s_0>5/2$ ensures that the right-hand side gains a negative power in~$K_{(3)}$, which controls the sum over $K_2,K_3$.
	Since there is no extra power in $K_1,K_4$, by Cauchy--Schwarz in the regime $K_1\sim K_4$ we obtain
	\[
	\Lambda_1\lesssim_{\epsilon}\|A\|_{\ell^2}\|B^{(2)}\|_{\ell^2}\|B^{(3)}\|_{\ell^2}\|d\|_{\ell^2}.
	\]
	For $\Lambda_2$, since $K_4\in\{K_{(3)},K_{(4)}\}$, we must have $\max\{K_2,K_3\}\sim K_{(2)}\sim K_{(1)}$. To fix ideas, assume $K_2\sim K_{(1)}$. Then
	\begin{align*}
		\Lambda_2&\lesssim \sum_{\substack{K_1,K_2,K_3,K_4\\
				K_{2}\sim K_{(1)}
		}} K_{(4)}K_{(1)}^{-s_0+2(1-\alpha)+\epsilon}K_{3}^{-s_0}\cdot A_{K_1}B_{K_2}^{(2)}B_{K_3}^{(3)}d_{K_4}\\
		&\lesssim \sum_{K_1,K_2,K_3,K_4}
		K_{(1)}^{3-2(\alpha+s_0)+\epsilon}\cdot A_{K_1}B_{K_2}^{(2)}B_{K_3}^{(3)}d_{K_4}.
	\end{align*}
	Since $2\alpha+s_0>5/2$ (and thus $\alpha+s_0>3/2$), the exponent on the top frequency~$K_{(1)}$ is negative. The dyadic sum therefore converges, and is bounded by $\|A\|_{\ell^2}\|B^{(2)}\|_{\ell^2}\|B^{(3)}\|_{\ell^2}\|d\|_{\ell^2}$.
\end{proof}

	
	\underline{\textbf{2. (C)(D)(D) interactions:}} We only estimate $\mathcal{N}_{[123]}^{(3)}(v_{N_1}^{(\mathrm{C})},v_{N_2}^{(\mathrm{D})},v_{N_3}^{(\mathrm{D})})$; the other terms follow from the same argument. Denote
\[
w^{(j)}:=v_{N_j}^{(\mathrm{D})}, \quad j=2,3,\qquad \psi^{(1)}=v_{N_1}^{(\mathrm{C})}.
\]

\noi
$\bullet${\bf Case 1: $N_1= N_{(1)}\geq N_2, N_3$.}

First, we have the trivial estimate
\[
\|\mathcal{N}_{[123]}^{(3)}(\psi^{(1)},w^{(2)},w^{(3)})\|_{X^{0,0}}\lesssim R\, T^{-(b-\frac{1}{2})}N_{(1)}^{-\alpha+\frac{1}{2}+\delta}N_{(2)}^{10}.
\]
Thus we choose $A_0(\vec{N})=N_{(1)}^{-\alpha+\frac{1}{2}+\delta}N_{(2)}^{10}$ in~\eqref{algorithm:A_0}. From~\eqref{algorithm:output} and the parameter hierarchy, it suffices to prove that
\begin{align}\label{pf(2)A1bound}
	\|\mathcal{N}_{[123]}^{(3)}(\psi^{(1)},w^{(2)},w^{(3)})\|_{X^{0,-b_1}}\lesssim R\, T^{-(\gamma-\frac{1}{q'})}N_{(1)}^{-s-\theta}N_{(2)}^{-\frac{\theta}{2}},
\end{align}
for some $\theta>2^{50}\sigma$.

Since $\psi^{(1)}=\mathbf{P}_{N_1}\psi^{(1)}$, following the previous argument we have
\begin{align}
	&\|\mathcal{N}_{[123]}^{(3)}(\psi^{(1)},w^{(2)},w^{(3)})\|_{X^{0,-b_1}}\notag  \\ &\qquad\lesssim
	\Big(\sum_{K_4}\Big(\sum_{\substack{K_1,K_2,K_3\\ K_1=N_1}}\|
	\mathbf{P}_{K_4}\mathcal{N}_{[123]}^{(3)}(\mathbf{P}_{N_1}\psi^{(1)},
	\mathbf{P}_{K_2}w^{(2)},
	\mathbf{P}_{K_3}w^{(3)})
	\|_{X^{0,-b_1}}\Big)^2\Big)^{1/2}. \label{N[123](3)dyadicsum}
\end{align}
We apply~\eqref{eq:tri2} with $j_0=1$ and $j_1$ chosen so that $K_{j_1}\in\{K_{(1)},K_{(2)}\}$ (thus $M_1\sim K_{(1)}$, $M_2=K_{(3)}$, $M_3=K_{(4)}$) to bound this by
\begin{align}\label{control(2):case1}
	\Big(\sum_{K_4}
	\Big(
	\sum_{\substack{K_2,K_3\\
			K_{(1)}\sim K_{(2)}
	}} c_{K_1 K_2 K_3 K_4}\cdot \mathbf{1}_{K_1=N_1}
	a_{K_1}b_{K_2}^{(2)}b_{K_3}^{(3)}
	\Big)^2
	\Big)^{1/2},
\end{align}
where
\[
c_{K_1 K_2 K_3 K_4}\leq C_{\epsilon}\, K_{(4)}K_{(2)}^{1-\alpha+\epsilon}\big(1+K_{(3)}K_{(2)}^{-1/2}\big),
\]
and
\[
a_{K_1}=\|\mathbf{P}_{K_1}\psi^{(1)}\|_{X_{q,q}^{0,\gamma}}\leq R\, T^{-(\gamma-\frac{1}{q'})}K_{1}^{-\alpha+\frac{1}{q}+\delta},\quad b_{K_j}^{(j)}=\|\mathbf{P}_{K_j}w^{(j)}\|_{X^{0,b}},\quad j=2,3.
\]

We claim that~\eqref{control(2):case1} is bounded by
\begin{align}\label{control(2):case1output}
	N_1^{\alpha-s-2\theta}a_{N_1} \prod_{j=2,3}\|K^{s_0}_j b_{K_j}^{(j)}\|_{\ell_{K_j}^2},
\end{align}
where $s_0=s-\theta>1/2$ for some $\theta>2^{50}\sigma$.
Once~\eqref{control(2):case1output} is established, interpolating between $X^{0,b}$ and $X^{s,b}$ gives the bound for $A_1(\vec{N})$ in~\eqref{algorithm:A_1}:
\[
A_1(\vec{N})\leq R\, T^{-(\gamma-\frac{1}{q'})}N_1^{-s-\theta}N_2^{-\frac{\theta}{2}}N_3^{-\frac{\theta}{2}},
\]
thanks to the hierarchy so that $\theta>2^{50}\sigma\gg \frac{1}{q}+\delta$.

\medskip
We now prove~\eqref{control(2):case1output}. Set
\[
B_{K_j}^{(j)}=K_j^{s_0}b_{K_j}^{(j)}, \quad j=2,3.
\]
By duality, for $d_{K_4}\in\ell^2$ with $\|d_{K_4}\|_{\ell_{K_4}^2}=1$,
we need to estimate $\Lambda_1+\Lambda_2$, where
\begin{align*}
	\Lambda_1&:=\sum_{\substack{K_1=N_1, K_2,K_3,K_4\\
			K_{(1)}\sim K_{(2)}\\
			K_1\in\{K_{(1)},K_{(2)}\}
	}} K_{(4)}K_{(2)}^{1-\alpha+\epsilon}K_2^{-s_0}K_3^{-s_0}(1+K_{(3)}K_{(2)}^{-1/2})\cdot a_{N_1}B_{K_2}^{(2)}B_{K_3}^{(3)}d_{K_4},\\
	\Lambda_2&:=\sum_{\substack{K_1=N_1, K_2,K_3,K_4\\
			K_{(1)}\sim K_{(2)}\\
			K_1\in\{K_{(3)},K_{(4)}\}
	}} K_{(4)}K_{(2)}^{1-\alpha+\epsilon}K_2^{-s_0}K_3^{-s_0}(1+K_{(3)}K_{(2)}^{-1/2})\cdot a_{N_1}B_{K_2}^{(2)}B_{K_3}^{(3)}d_{K_4}.
\end{align*}
For $\Lambda_1$, we control the factor
\begin{align*}
	K_{(4)}K_{(2)}^{1-\alpha+\epsilon}K_{2}^{-s_0}K_3^{-s_0}(1+K_{(3)}K_{(2)}^{-1/2})&\lesssim N_1^{1-\alpha+\epsilon}\cdot (K_2 K_3)^{-(s_0-\frac{1}{2})}\\
	&\quad+N_1^{\frac{1}{2}-\alpha+\epsilon}\cdot (K_2 K_3)^{1-s_0}.
\end{align*}
Summing over $K_2,K_3,K_4\leq N_1$, we obtain
\[
\Lambda_1\lesssim_{\epsilon} a_{N_1}\cdot N_1^{\frac{5}{2}-\alpha-2s_0+2\epsilon}\cdot \|B_{K_2}^{(2)}\|_{\ell_{K_2}^2}
\|B_{K_3}^{(3)}\|_{\ell_{K_3}^2}.
\]
The hierarchy gives in particular $s+2\alpha-\frac{5}{2}>2^{100}\sigma$, so there exists $\theta>2^{50}\sigma$ such that with $s_0=s-\theta$,
\[
\frac{5}{2}-\alpha-2s_0+2\epsilon\leq \alpha-s-2\theta,
\]
hence $\Lambda_1$ satisfies the desired bound~\eqref{control(2):case1output}.

For $\Lambda_2$, without loss of generality assume $K_2\geq K_3$. We control the factor
\begin{align*}
	K_{(4)}K_{(2)}^{1-\alpha+\epsilon}K_2^{-s_0}K_3^{-s_0}&(1+K_{(3)}K_{(2)}^{-1/2})\\
	&\lesssim N_1^{\alpha-s-2\theta}\cdot K_{(4)}^{1+s_0-\alpha+3\theta} K_2^{1-\alpha-s_0+\epsilon}K_3^{-s_0}\\
	&\quad+ N_1^{\alpha-s-2\theta}\cdot K_{(4)}^{1+s_0-\alpha+3\theta} K_{(3)}K_2^{\frac{1}{2}-\alpha-s_0+\epsilon}K_3^{-s_0}\\
	&\lesssim N_1^{\alpha-s-2\theta}\cdot K_{(2)}^{\frac{5}{2}-2\alpha-s_0+3\theta+\epsilon}.
\end{align*}
The hierarchy ensures that the total power of $K_{(2)}$ is negative, so $\Lambda_2$ satisfies~\eqref{control(2):case1output}.

\medskip

\noi
$\bullet${\bf Case 2: $N_2= N_{(1)}$ or $N_3=N_{(1)}$.}

Without loss of generality, assume $N_2=N_{(1)}\geq N_1, N_3$.

The trivial $X^{0,0}$ estimate gives
\[
\|\mathcal{N}_{[123]}^{(3)}(\psi^{(1)},w^{(2)},w^{(3)})\|_{X^{0,0}}\lesssim N_{(1)}^{-s}N_{(2)}^{10}.
\]
Thus we choose $A_0(\vec{N})=N_{(1)}^{-s}N_{(2)}^{10}$ in~\eqref{algorithm:A_0}. We will prove the $A_1(\vec{N})$ bound
\begin{align}\label{pf(2)A1boundcase2}
	\|\mathcal{N}^{(3)}_{[123]}(\psi^{(1)},w^{(2)},w^{(3)})\|_{X^{0,-b_1}}\lesssim R^3 T^{-(\gamma-\frac{1}{q'})}N_{(1)}^{-s}N_{(2)}^{-\frac{\theta}{3}},
\end{align}
for some $1>\theta>2^{50}\sigma$.

As discussed before Case~1, we always sum over $K_{(1)}\sim K_{(2)}$.

Set
\[
a_{K_1}=\|\mathbf{P}_{K_1}\psi^{(1)}\|_{X_{q,q}^{0,\gamma}}\leq R\, T^{-(\gamma-\frac{1}{q'})}K_{1}^{-\alpha+\frac{1}{q}+\delta},\quad b_{K_j}^{(j)}=\|\mathbf{P}_{K_j}w^{(j)}\|_{X^{0,b}},\quad j=2,3,
\]
and additionally
\[
\widetilde{a}_{K_1}:=\|\mathbf{P}_{K_1}\psi^{(1)}\|_{X^{0,b}}\leq R\, T^{-(b-\frac{1}{2})}K_1^{-\alpha+\frac{1}{2}+\delta}.
\]
We split the dyadic sum~\eqref{N[123](3)dyadicsum} into two regimes: $K_{(1)}\sim K_{(2)}\gg K_{(3)}^2$ and $K_{(1)}\sim K_{(2)}\lesssim K_{(3)}^2$.
Applying~\eqref{eq:tri1'} in the first regime and~\eqref{eq:tri2} in the second, by duality we need to estimate $\Lambda_1+\Lambda_2$ for any $d_{K_4}\in \ell^2$ with $\|d_{K_4}\|_{\ell_{K_4}^2}=1$, where
\begin{align*}
	\Lambda_1&:=\sum_{\substack{K_1=N_1,K_2,K_3,K_4\\
			K_{(1)}\sim K_{(2)}\gg K_{(3)}^2
	}} K_{(4)}\cdot a_{N_1}b_{K_2}^{(2)}b_{K_3}^{(3)}d_{K_4},\\
	\Lambda_2&:=\sum_{\substack{K_1=N_1,K_2,K_3,K_4\\
			K_{(1)}\sim K_{(2)}\lesssim K_{(3)}^2
	}} K_{(4)}K_{(2)}^{1-\alpha+\epsilon}K_{(3)}K_{(2)}^{-1/2} \cdot a_{N_1}b_{K_2}^{(2)}b_{K_3}^{(3)}d_{K_4}.
\end{align*}
For $\Lambda_1$, we control the factor
\[
K_{(4)}\leq N_1^{\alpha-\frac{1}{2}-\delta-\theta}K_3^{\frac{3}{2}-\alpha+\frac{1}{q}+\delta+\theta}=N_1^{\alpha-\frac{1}{2}-\delta-\theta} K_3^{\frac{3}{2}-\alpha-s_0+\frac{1}{q}+\delta+\theta}\cdot K_3^{s_0}
\]
for some small $\theta>2^{50}\sigma$. The hierarchy gives $\frac{3}{2}-\alpha-s>0$, which implies $\frac{3}{2}-\alpha-s_0+\frac{1}{q}+\delta+\theta<-\frac{\theta}{2}$ for sufficiently small~$\sigma$ and $\theta>2^{50}\sigma$.
Setting $B_{K_3}^{(3)}:=K_3^{s_0}b_{K_3}^{(3)}$ and noting
\[
N_1^{\alpha-\frac{1}{2}-\delta-\theta}a_{N_1}\leq R\, T^{-(\gamma-\frac{1}{q'})}N_1^{-\theta},
\]
we bound $\Lambda_1$ by
\[
\sum_{\substack{K_1=N_1,K_2,K_3,K_4 \\ K_{(1)}\sim K_{(2)}}} R\, T^{-(\gamma-\frac{1}{q'})}(N_1 K_3)^{-\frac{\theta}{2}} B_{K_3}^{(3)}\cdot b_{K_2}^{(2)}d_{K_4}.
\]
Since there is no weight on $b_{K_2}^{(2)}$ and $d_{K_4}$, the dyadic sum gives
\[
\Lambda_1\lesssim R\, T^{-(\gamma-\frac{1}{q'})}N_2^{-s}(N_1 N_3)^{-\frac{\theta}{2}}.
\]

\medskip

For $\Lambda_2$, using $K_{(3)}^2\gtrsim K_{(2)}$, we first control the factor
\[
K_{(4)}K_{(2)}^{1-\alpha+\epsilon}K_{(3)}K_{(2)}^{-1/2}\lesssim K_{(4)}K_{(3)}^{3-2\alpha+2\epsilon}K_{(2)}^{-1/2}.
\]
Set
\[
A_{N_1}=N_1^{\alpha-\frac{1}{q}-\delta-\theta}a_{N_1},\quad B^{(3)}_{K_3}=K_3^{s_0}b_{K_3}^{(3)}.
\]
Then
\begin{align*}
	\Lambda_2\lesssim \sum_{\substack{K_1=N_1,K_2,K_3,K_4\\
			K_{(1)}\sim K_{(2)}}}
	N_1^{-\alpha+\frac{1}{q}+\delta+\theta}K_3^{-s_0}K_{(4)}K_{(3)}^{3-2\alpha+2\epsilon}K_{(2)}^{-1/2}\cdot
	A_{N_1}B_{K_3}^{(3)}b_{K_2}^{(2)}d_{K_4}.
\end{align*}
Under the hierarchy, there exists $\theta>2^{50}\sigma$ such that
\[
\frac{7}{2}-2\alpha-s_0+\frac{1}{q}+\delta+\theta+2\epsilon<-\theta.
\]
Therefore
\[
\Lambda_2\lesssim\sum_{\substack{K_1=N_1,K_2,K_3,K_4\\
		K_{(1)}\sim K_{(2)}
}} N_1^{-\frac{\theta}{2}}K_{3}^{-\frac{\theta}{2}}\cdot A_{N_1}B_{K_3}^{(3)}\cdot b_{K_2}^{(2)}d_{K_4},
\]
which is bounded by $R\, T^{-(\gamma-\frac{1}{q'})}N_2^{-s}N_1^{-\frac{\theta}{2}}N_3^{-\frac{\theta}{2}}$.

\medskip

\underline{\textbf{3. (C)(C)(D) interactions:}} This is proved in Corollary~\ref{cor:randomop}.

\medskip

\underline{\textbf{4. (C)(C)(C) interactions:}} This is proved in Proposition~\ref{prop:CCC}, combined with the interpolation argument described above.

\medskip

\noi
$\bullet${\bf Proof of (B) in Proposition~\ref{prop:trilinear}:} We apply Proposition~\ref{linear:type1} and
Corollary~\ref{cor:pres1}. From the linear inhomogeneous estimate
\[
\|\mathcal{I}_{\chi_T}\mathcal{N}(\cdots)\|_{X^{0,b}}\lesssim T^{\sigma}\|\mathcal{N}(\cdots)\|_{X^{0,-b_2}},
\]
where $b_2=\frac{1}{2}-\sigma-\sigma^{10}<\frac{1}{2}$, so no interpolation is needed. We always assume $K_{(1)}\sim K_{(2)}$ in the dyadic decomposition.

First, assume there exists $j_0\in\{1,2,3\}$ such that $v_{N_{j_0}}=w^{(j_0)}_{N_{j_0}}$ is of type~(D) with $N_{j_0}\sim N_{(1)}$. With $s_1=(\alpha-\frac{1}{2}-\delta)(1-\theta)$ and $b_1=b_2$, we
apply Corollary~\ref{cor:pres1} to $z=w_{N_{j_0}}^{(j_0)}$ and interpolate between $X^{0,b}$ and $X^{\alpha-\frac{1}{2}-\delta,b}$ for the other two inputs to obtain
\[
\|\mathcal{N}_{(23)}^{(3)}(v_{N_1}^{(\ast_1)},v_{N_2}^{(\ast_2)},v_{N_3}^{(\ast_3)})\|_{X^{0,-b_2}}\lesssim R^2 T^{-2(b-\frac{1}{2})}N_{(1)}^{-s}N_{(2)}^{-\theta}N_{(3)}^{-\theta},
\]
for some $\theta>2^{50}\sigma$.

It remains to handle the case where the type~(C) input has the top scale~$N_{(1)}$.

\noi
$\bullet${\bf Case 1: $N_1\gg N_2, N_3$, and $v_{N_1}^{(\ast_1)}=\psi^{(1)}$ is of type~(C).}

\underline{Case 1.1:}
At least one of $v_{N_2}^{(\ast_2)},v_{N_3}^{(\ast_3)}$ is of type~(C).

To fix ideas, assume $v_{N_2}^{(\ast_2)}=\psi_{N_2}^{(2)}$. Then
\begin{multline*}
	\|\mathcal{N}_{(23)}^{(3)}(\mathbf{P}_{N_1}\psi^{(1)},\mathbf{P}_{N_2}\psi^{(2)},v^{(\ast_3)})\|_{X^{0,-b_2}}\\
	\sim \Big(
	\sum_{K_4}\big\|\mathbf{P}_{K_4}\mathcal{N}_{(23)}^{(3)}(\mathbf{P}_{N_1}\psi^{(1)},\mathbf{P}_{N_2}\psi^{(2)},\mathbf{P}_{N_2}v^{(\ast_3)})\big\|_{X^{0,-b_2}}^2
	\Big)^{1/2}.
\end{multline*}
By Proposition~\ref{linear:type1}, the right-hand side is bounded by
\begin{align*}
	\Big(\sum_{K_4}
	c_{K_4 N_1 N_2}^2
	\Big)^{1/2}\cdot a^{(1)}_{N_1} a_{N_2}^{(2)} b_{N_2}^{(3)},
\end{align*}
where
\begin{align}\label{cK4K1K2}
	c_{K_4 K_1 K_2}\lesssim
	\frac{\max(K_4,K_1)^{1-2b_2}
		\min(K_4,K_1,K_2)
	}{\min(K_4,K_1)^{1/2}},
\end{align}
and
\[
a_{N_j}^{(j)}=\|\psi^{(j)}\|_{X^{0,b}},\quad j=1,2,\qquad b_{N_2}^{(3)}=\|\mathbf{P}_{N_2}v^{(\ast_3)}\|_{X^{0,b}}.
\]
Under $K_{(1)}\sim K_{(2)}$, we always have $\max(K_4,K_1)\lesssim K_2+\min(K_4,K_1)$.
From the proof of Corollary~\ref{cor:pres1}, we recall that
\begin{align}\label{controlcK4K1K2}
	c_{K_4 K_1 K_2}\lesssim \begin{cases}
		K_2^{2s_1}\dfrac{\min(K_4,K_1,K_2)^{1-2s_1}}{\max(K_4,K_1,K_2)^{2b_2-\frac{1}{2}}},\\[6pt]
		K_1^{s_1}K_2^{s_1}\cdot \dfrac{\min(K_4,K_1)^{1-2s_1}}{\max(K_4,K_1)^{2b_2-\frac{1}{2}}}
	\end{cases}
	\lesssim K_2^{2s_1}\frac{\min(K_4,K_1)^{1-2s_1}}{\max(K_4,K_1)^{2b_2-\frac{1}{2}}},
\end{align}
for $\frac{3}{4}-b_2<s_1<\alpha-\frac{1}{2}$ and $b_1>\frac{1}{4}+(1-\alpha)$.

Therefore
\begin{align*}
	&\Big(\sum_{K_4}
	c_{K_4 N_1 N_2}^2
	\Big)^{1/2}\cdot a^{(1)}_{N_1} a_{N_2}^{(2)} b_{N_2}^{(3)}\\
	&\qquad\lesssim N_1^{\frac{1}{2}-2b_2}N_2\cdot a_{N_1}^{(1)}a_{N_2}^{(2)}b_{N_2}^{(3)}\lesssim R^3 T^{-3(b-\frac{1}{2})}N_1^{-(\alpha+2b_2-1-\delta)}\cdot N_2^{1-2(\alpha-\frac{1}{2}-\delta)},
\end{align*}
since the worst bound comes from the case where $v^{(\ast_3)}$ is of type~(C).
The hierarchy gives
\[
R^3 T^{-3(b-\frac{1}{2})}\cdot N_1^{-(3\alpha+2b_2-3-3\delta)}\leq R^3 T^{-3(b-\frac{1}{2})}N_1^{-s-2\theta}
\]
for some $\theta>2^{50}\sigma$.

\medskip

\underline{Case 1.2:}
$v^{(\ast_2)}=w^{(2)},v^{(\ast_3)}=w^{(3)}$ are both of type~(D).

By Proposition~\ref{linear:type1},
\begin{align*}
	\|\mathcal{N}_{(23)}^{(3)}(\mathbf{P}_{N_1}\psi^{(1)}, w^{(2)},w^{(3)})\|_{X^{0,-b_2}}\lesssim a_{N_1}^{(1)}\cdot \Big(
	\sum_{K_4}
	\Big(
	\sum_{K_2} c_{K_4 N_1 K_2}\cdot b_{K_2}^{(2)}b_{K_2}^{(3)}
	\Big)^2
	\Big)^{1/2},
\end{align*}
where $c_{K_4 K_1 K_2}$ satisfies~\eqref{cK4K1K2} and
\[
a_{N_1}^{(1)}=\|\psi^{(1)}\|_{X^{0,b}},\quad b_{K_2}^{(j)}=\|\mathbf{P}_{K_2}w^{(j)}\|_{X^{0,b}},\quad j=2,3.
\]
Summing over $K_2$ using~\eqref{controlcK4K1K2}, we obtain the bound
\begin{align*}
	&N_1^{-(2b_2-\frac{1}{2})+(1-2s_1)}a_{N_1}^{(1)}\|K_2^{s_1}b_{K_2}^{(2)}\|_{\ell_{K_2}^2}\|K_2^{s_1}b_{K_2}^{(3)}\|_{\ell_{K_2}^2}\\
	&\qquad\lesssim R\, T^{-(b-\frac{1}{2})}N_1^{-(2b_2-\frac{1}{2})-(\alpha-\frac{1}{2})+(1-2s_1)+\frac{1}{q}+\delta}\cdot N_2^{-\theta}N_3^{-\theta},
\end{align*}
for some $\theta>2^{50}\sigma$, by interpolating between $\|w^{(j)}\|_{X^{0,b}}$ and $\|w^{(j)}\|_{X^{s_0,b}}$ with $s_0=s-\theta>\frac{1}{2}$. For the choice $b_2=\frac{1}{2}-\sigma-\sigma^{10}$ and $s_1=(\alpha-\frac{1}{2}-\delta)(1-\theta)$, the total power of $N_1$ is $-3\alpha+2+o_{\sigma\to 0}(1)<-s$, which is sufficient.

\medskip

\noi
$\bullet${\bf Case 2: $N_2\geq N_1, N_3$ and $v_{N_2}^{(\ast_2)}=\psi^{(2)}$ is of type~(C), or $N_3\geq N_1, N_2$ and $v_{N_3}^{(\ast_3)}=\psi^{(3)}$ is of type~(C).}

By symmetry of the second and third inputs, we only treat $N_2\gg N_1, N_3$ with $v_{N_2}^{(\ast_2)}=\psi^{(2)}$ of type~(C). Hence
\begin{align*}
	&\|\mathcal{N}_{(23)}^{(3)}(v_{N_1}^{(\ast_1)},\mathbf{P}_{N_2}\psi^{(2)},\mathbf{P}_{N_2}v^{(\ast_3)})\|_{X^{0,-b_2}}\\ &\qquad\lesssim \Big(
	\sum_{K_4}\Big(\sum_{
		K_1
	}
	\|\mathbf{P}_{K_4}\mathcal{N}_{(23)}^{(3)}(\mathbf{P}_{K_1}v_{N_1}^{(\ast_1)},\mathbf{P}_{N_2}\psi_{N_2}^{(2)},\mathbf{P}_{N_2}v_{N_3}^{(\ast_3)})\|_{X^{0,-b_2}}
	\Big)^2
	\Big)^{1/2}.
\end{align*}

\underline{Case 2.1:} $v_{N_3}^{(\ast_3)}=\psi^{(3)}$ is of type~(C).

In this case, $N_2=N_3=N_{(1)}$.
By Proposition~\ref{linear:type1} (with $b_1=b_2=\frac{1}{2}-\sigma-\sigma^{10}$),
\begin{align*}
	&\|\mathcal{N}_{(23)}^{(3)}(v_{N_1}^{(\ast_1)},\psi_{N_2}^{(2)},\psi_{N_3}^{(3)})\|_{X^{0,-b_2}}\\
	&\qquad\lesssim
	\Big(\sum_{K_4}
	\Big(\sum_{\substack{K_1\\
			K_{(1)}\sim K_{(2)}
	}}
	A_0(K_4,K_1,N_2)\|\mathbf{P}_{K_1}v_{N_1}^{(\ast_1)}\|_{X^{0,b}}\|\psi_{N_2}^{(2)}\|_{X^{0,b}}\|\psi_{N_3}^{(3)}\|_{X^{0,b}}
	\Big)^2
	\Big)^{1/2}.
\end{align*}
Note that
\[
A_0(K_4,K_1,N_2)\lesssim \max(K_1,K_4)^{1-2b_2}\min(K_1,K_4)^{\frac{1}{2}},
\]
and $\|\psi_{N_j}^{(j)}\|_{X^{0,b}}\leq R\, T^{-(b-\frac{1}{2})}N_j^{-\alpha+\frac{1}{2}+\delta}$ for $j=2,3$.
We distinguish the regimes $K_1\sim K_4\gg N_2$ and $K_1,K_4\lesssim N_2$. If $K_1\sim K_4\gg N_2$, since $N_2\geq N_1$, the input $v_{N_1}^{(\ast_1)}$ must be of type~(D). In this regime, we obtain
\begin{align*}
	&\|\psi_{N_2}^{(2)}\|_{X^{0,b}}^2\cdot \Big(\sum_{K_4}
	\Big(
	\sum_{K_1\sim K_4} K_1^{\frac{3}{2}-2b_2-s_0}\|K_1^{s_0}\mathbf{P}_{K_1}w_{N_1}^{(\mathrm{D})}\|_{X^{0,b}}
	\Big)^2
	\Big)^{1/2}\\
	&\qquad\lesssim R^2 T^{-2(b-\frac{1}{2})}N_2^{-2\alpha+1+2\delta}\cdot N_1^{-\frac{\theta}{2}},
\end{align*}
where we used $\frac{3}{2}-2b_2-s_0=\frac{1}{2}+2(\sigma+\sigma^{10})+\theta-s<0$ for some $\theta>2^{50}\sigma$.

If $K_1,K_4\lesssim N_2$, the worst case is when $v_{N_1}^{(\ast_1)}$ is of type~(C), giving
\begin{align*}
	&\|\psi_{N_2}^{(2)}\|_{X^{0,b}}^2\cdot \Big(\sum_{K_4\lesssim N_2}
	\Big(
	\sum_{K_1\lesssim N_2}
	N_2^{2-\alpha-2b_2+\delta+\theta} \|K_1^{\alpha-\frac{1}{2}-\delta-\theta}\mathbf{P}_{K_1}v_{N_1}^{(\ast_1)}\|_{X^{0,b}}
	\Big)^{2}
	\Big)^{1/2}\\
	&\qquad\lesssim R^2 T^{-2(b-\frac{1}{2})}N_2^{-3\alpha-2b_2+3+\delta+\theta}\cdot N_1^{-\theta}.
\end{align*}
Since $s>\max\{2\alpha-1+2\delta,\, 3\alpha+2b_2-3-\delta-\theta\}$, we obtain
\[
\|\mathcal{N}_{(23)}^{(3)}(v_{N_1}^{(\ast_1)},\psi_{N_2}^{(2)},\psi_{N_3}^{(3)})\|_{X^{0,-b_2}}\lesssim R^2 T^{-2(b-\frac{1}{2})}N_{(1)}^{-s}N_{(2)}^{-\frac{\theta}{2}}.
\]

\underline{Case 2.2:} $v_{N_3}^{(\ast_3)}$ is of type~(D).

We need to control
\begin{multline*}
	\Big(
	\sum_{K_4}
	\Big(
	\sum_{\substack{K_1\\
			K_{(1)}\sim K_{(2)}
	}}
	\max(K_1,K_4)^{1-2b_2}\min(K_1,K_4)^{\frac{1}{2}}
	\\
\times	\|\mathbf{P}_{K_1}v_{N_1}^{(\ast_1)}\|_{X^{0,b}}\|\psi_{N_2}^{(2)}\|_{X^{0,b}}\|\mathbf{P}_{N_2}w_{N_3}^{(3)}\|_{X^{0,b}}
	\Big)^{2}
	\Big)^{1/2}.
\end{multline*}
The same analysis as in Case~2.1 applies to the sum involving
\[
\max(K_1,K_4)^{1-2b_2}\min(K_1,K_4)^{\frac{1}{2}}\|\mathbf{P}_{K_1}v_{N_1}^{(\ast_1)}\|_{X^{0,b}}.
\]
The final output is
\[
\|\psi_{N_2}^{(2)}\|_{X^{0,b}}\|\mathbf{P}_{N_2}w_{N_3}^{(3)}\|_{X^{0,b}}\cdot \big(N_1^{-\frac{\theta}{2}}+N_2^{2-\alpha-2b_2+\delta+\theta}\cdot N_1^{-\theta}\big),
\]
where the first term $N_1^{-\frac{\theta}{2}}$ comes from the regime $N_2\sim K_1\sim K_4$ (and $v_{N_1}^{(\ast_1)}$ must be of type~(D)), and the second from $K_1,K_4\lesssim N_2$. Since $N_2\geq N_3$,
\[
\|\mathbf{P}_{N_2}w_{N_3}^{(3)}\|_{X^{0,b}}\lesssim \big(\frac{N_2}{N_3}\big)^{-1}\cdot N_3^{-s}=\big(\frac{N_2}{N_3}\big)^{-1+s}\cdot N_2^{-s},
\]
so the final estimate is strictly better than in Case~2.1. This concludes the proof of part~(B) in Proposition~\ref{prop:trilinear}.

\medskip

\noi
$\bullet${\bf Proof of (C) and (D) of Proposition~\ref{prop:trilinear}:} Part~(D) is a direct consequence of Corollary~\ref{co:2-linear:1}, with the choice
\[
\epsilon_0=\frac{\theta}{10},\quad \sigma_0=\sigma
\]
for some $\theta>2^{100}\sigma$.

For part~(C), if $v_{N_1}^{(\ast_1)}$ is of type~(D), the desired inequality follows directly from the first estimate of part~(D). It remains to prove the following cases:
\begin{itemize}
	\item $N_1= N_2=N_{(2)}$ or $N_1=N_3=N_{(3)}$, and $v_{N_1}^{(\ast_1)}=\psi_{N_1}^{(1)}$ is of type~(C).
	\item $N_1\leq N_2/2$ or $N_1\leq N_3/2$, and $v_{N_1}^{(\ast_1)}$ is of type~(C) or~(D).
\end{itemize}
By the linear estimate,
\[
\|\chi_T(t)\,\mathcal{I}\mathcal{N}(\cdots)\|_{X^{0,b}}\lesssim T^{\sigma}\|\chi(t)\,\mathcal{N}(\cdots)\|_{X^{0,-b_2}},
\]
for $b_2=\frac{1}{2}-\sigma-\sigma^{10}$. Below we only estimate the $X^{0,-b_2}$ norm of the nonlinearity.

We distinguish the following cases.

\noi
$\bullet${\bf Case 1: $N_1=N_{(1)}=\max(N_2,N_3)$ and $v_{N_1}^{(\ast_1)}=\psi_{N_1}^{(\mathrm{C})}$.}

We apply~\eqref{resonant:2} of Proposition~\ref{linear:type1} (with $b_1=b_2$ and $N_1=\max(M_2,M_3)$, $N_2=M_2$, $N_3=M_3$).
By Littlewood--Paley decomposition, we need to estimate
\begin{align}\label{LPdecomposition(C)}
	\sum_{\substack{K_1,K_2,K_3\\
			K_{(1)}\sim K_{(2)}
	}}\|\mathbf{P}_{K_1}\mathcal{N}^{(2)}(\mathbf{P}_{K_1}v_{N_1}^{(\ast_1)},\mathbf{P}_{K_2}v_{N_2}^{(\ast_2)},\mathbf{P}_{K_3}v_{N_3}^{(\ast_3)})\|_{X^{0,-b_2}}.
\end{align}

To fix ideas (and treat the worst case), assume $v_{N_2}^{(\ast_2)},v_{N_3}^{(\ast_3)}$ are both of type~(C) and $N_1=N_2$. Then $K_1=N_1$, $K_2=N_2$, $K_3=N_3$ and there is no dyadic sum. We get
\begin{align*}
	&\max(N_2,N_3)^{1-2b_2}\min(N_2,N_3)^{\frac{1}{2}}\|\psi_{N_1}^{(1)}\|_{X^{0,b}}\|\psi_{N_2}^{(2)}\|_{X^{0,b}}\|\psi^{(3)}_{N_3}\|_{X^{0,b}}\\
	&\qquad\lesssim R^3 T^{-3(\gamma-\frac{1}{q'})}N_1^{1-2b_2}
	N_3^{\frac{1}{2}}
	\cdot (N_1 N_2 N_3)^{-\alpha+\frac{1}{2}+\delta} \sim R^3 T^{-3(\gamma-\frac{1}{q'})} N_1^{2-2\alpha-2b_2}N_3^{1-\alpha+\delta}\\
	&\qquad\lesssim R^3 T^{-3(\gamma-\frac{1}{q'})}N_1^{-s-\theta},
\end{align*}
for some $\theta>2^{100}\sigma$. The case where at least one of $v_{N_2}^{(\ast_2)},v_{N_3}^{(\ast_3)}$ is of type~(D) is similar, since the dyadic sum only costs a $\log N_1$ factor. We omit the details.

\medskip

\noi
$\bullet${\bf Case 2: $N_1\leq \frac{N_2}{2}$ or $N_1\leq \frac{N_3}{2}$.}

Without loss of generality, assume $N_2=N_{(1)}\geq N_3$ and $N_1\leq \frac{N_2}{2}$.

If $N_1\sim N_{(1)}$, the argument from Case~1 remains valid when $v_{N_1}^{(\ast_1)}$ is of type~(C), and when $v_{N_1}^{(\ast_1)}$ is of type~(D) we can apply the operator bound directly. Hence we assume $N_1\ll N_{2}$.

Applying~\eqref{resonant:2} of Proposition~\ref{linear:type1}, the expression~\eqref{LPdecomposition(C)} is bounded by
\[
	\sum_{\substack{K_1,K_2,K_3,K_4=K_1\\
			K_{(1)}\sim K_{(2)}
	}} \max(K_2,K_3)^{1-2b_2}\min(K_2,K_3)^{\frac{1}{2}}
	\prod_{j=1}^3\|\mathbf{P}_{K_j}v_{N_j}^{(\ast_j)}\|_{X^{0,b}},
\]
where $K_{(1)}\geq K_{(2)}\geq K_{(3)}\geq K_{(4)}$ is the non-increasing rearrangement of $K_1,K_2,K_3,K_4$.

\underline{Case 2.1:} $v_{N_2}^{(\ast_2)}$ is of type~(C).

In this case, $K_2=N_2$ is fixed. If $N_2\sim N_3$, the argument is identical to Case~1. Hence assume $N_2\gg N_3$ (and recall $N_2\gg N_1$). Then at least one of $v_{N_1}^{(\ast_1)},v_{N_3}^{(\ast_3)}$ is of type~(D). The constraint $K_{(1)}\sim K_{(2)}$ forces $\max(K_1,K_3)\gtrsim K_2=N_2$. To fix ideas, assume $K_1\gtrsim N_2$, so $v_{N_1}^{(\ast_1)}$ must be of type~(D). We obtain
\begin{align*}
	&\sum_{K_1\gtrsim N_2,K_3} K_3 N_2^{1-2b_2}K_3^{\frac{1}{2}}N_2^{1-2b_2} \|\mathbf{P}_{K_1}v_{N_1}^{(\ast_1)}\|_{X^{0,b}}\|\psi_{N_2}^{(2)}\|_{X^{0,b}}\|\mathbf{P}_{K_3}v_{N_3}^{(\ast_3)}\|_{X^{0,b}}\\
	&\qquad\lesssim R\, T^{-(\gamma-\frac{1}{q'})}
	N_2^{-\alpha+\frac{1}{2}+\delta}
	\sum_{K_1\gtrsim N_2,K_3} K_3^{\frac{1}{2}}N_2^{1-2b_2}\cdot (N_1/K_1)^{1-s}\cdot K_1^{-s}\|\mathbf{P}_{K_3}v_{N_3}^{(\ast_3)}\|_{X^{0,b}}.
\end{align*}
The worst bound occurs when $v_{N_3}^{(\ast_3)}$ is of type~(C), giving
\[
R^2 T^{-2(\gamma-\frac{1}{q'})}N_2^{\frac{3}{2}-\alpha-2b_2+\delta}\cdot N_2^{-s}\cdot N_2^{1-\alpha+\delta}\lesssim R^2 T^{-2(\gamma-\frac{1}{q'})}\cdot N_2^{-s-\theta},
\]
for some $\theta>2^{100}\sigma$. The analysis when $v_{N_3}^{(\ast_3)}$ is of type~(D) is identical.

\underline{Case 2.2:} $v_{N_2}^{(\ast_2)}$ is of type~(D). Since $K_2\lesssim \max(K_1,K_3)$ in the sum~\eqref{LPdecomposition(C)}, we first sum in $K_2$ to bound the expression by
\begin{align*}
	\|v_{N_2}^{(\mathrm{D})}\|_{X^{0,b}}\cdot \sum_{K_1,K_3}\max(K_1,K_3)^{1-2b_2}\min(K_1,K_3)^{\frac{1}{2}}	\|\mathbf{P}_{K_1}v_{N_1}^{(\ast_1)}\|_{X^{0,b}}\|\mathbf{P}_{K_3}v_{N_3}^{(\ast_3)}\|_{X^{0,b}}.
\end{align*}
The worst bound comes from the case where $v_{N_1}^{(\ast_1)},v_{N_3}^{(\ast_3)}$ are both of type~(C), giving
\[
R^2 T^{-2(\gamma-\frac{1}{q'})}N_2^{-s}\max(N_1,N_3)^{\frac{5}{2}-2\alpha-2b_2+2\delta}.
\]
Under the hierarchy, $\frac{5}{2}-2\alpha-2b_2+2\delta<-\theta$ for some $\theta>2^{100}\sigma$, which concludes the proof in this case.

The proof of Proposition~\ref{prop:trilinear} is now complete.
	
	
	\subsection{Proof of Proposition~\ref{prop:quintilinear}}

In addition to Corollaries~\ref{cor:4-linear-1} and~\ref{cor:4-linearprob}, the proof of Proposition~\ref{prop:quintilinear} reduces to summing over all dyadic scales. We achieve this by establishing an abstract dyadic summation lemma.

\medskip

Recall Definition~\ref{def:scale-seq} of scale-$N$ $h^s$-envelopes.
Let $\mathcal{L}^{(4)}$ be a 4-linear form acting on dyadic sequences, and assume that there exist
nonnegative coefficients $c_{M_1,M_2,M_3,M_4}$ such that for all inputs $\mathbf{a}^{(j)}$,
\[
\big|\mathcal{L}^{(4)}((\mathbf{a}^{(j)})_{j=1}^4)\big|\leq
\sum_{M_1,M_2,M_3,M_4\in 2^{\mathbb N}}
c_{M_1,M_2,M_3,M_4}\prod_{j=1}^4 |a^{(j)}_{M_j}|.
\]
For a quadruple $(M_1,M_2,M_3,M_4)$, denote by
$M_{(1)}\ge M_{(2)}\ge M_{(3)}\ge M_{(4)}$
the non-increasing rearrangement.

\begin{lemma}[Dyadic summation test]\label{lem:sum-abstract}
	Let $0<s_1,s_2,s_3,s_4\leq 1$. Assume that there exists
	\[
	\nu_0\in \Big(0,\min\{s_1,s_2,s_3,s_4\}\Big]
	\]
	such that for every dyadic quadruple $(M_1,M_2,M_3,M_4)\in (2^{\N})^4$ and every choice of scale-$M_j$ $h^{s_j}$-envelopes $b^{(j)}$, the typical-atom estimate
	\begin{equation}\label{eq:typical-atom-output-gain}
		c_{M_1,M_2,M_3,M_4}\prod_{j=1}^4 |b^{(j)}_{M_j}|
		\leq
		\Lambda\cdot M_{(1)}^{-\nu_0}
	\end{equation}
	holds for some $\Lambda>0$ independent of $M_1,M_2,M_3,M_4$.

	Then for any $(N_1,N_2,N_3,N_4)\in(2^{\N})^4$ and any scale-$N_j$ $h^{s_j}$-envelopes~$a^{(j)}$,
	\begin{equation}\label{eq:dyadic-sum-output-gain}
		\big|\mathcal{L}^{(4)}((a^{(j)})_{j=1}^4)\big|
		\lesssim
		\Lambda\cdot N_{(1)}^{-\nu_0}(\log N_{(1)})^4,
	\end{equation}
	where the implicit constant is independent of $N_j$.
\end{lemma}

\begin{proof}
	Fix $(M_1,M_2,M_3,M_4)\in(2^{\N})^4$. For each $j\in\{1,2,3,4\}$, define the one-atom
	envelope $b^{(j)}$ by
	\[
	b_M^{(j)}:=
	\begin{cases}
		M_j^{-s_j},& M=M_j,\\
		0,& M\neq M_j.
	\end{cases}
	\]
	Each $b^{(j)}$ is a scale-$M_j$ $h^{s_j}$-envelope in the sense of
	Definition~\ref{def:scale-seq}. Applying~\eqref{eq:typical-atom-output-gain} to these choices, we obtain the coefficient bound
	\begin{equation}\label{eq:coeff-bound-output-gain}
		c_{M_1,M_2,M_3,M_4}
		\leq
		\Lambda \cdot M_{(1)}^{-\nu_0}
		\prod_{j=1}^4 M_j^{s_j}.
	\end{equation}

	Given the scale-$N_j$ envelopes $a^{(j)}$, define
	\[
	A_M^{(j)}:=M^{s_j}|a_M^{(j)}|,\qquad M\in 2^{\N}.
	\]
	To fix ideas, suppose $N_1=N_{(1)}$.
	Using~\eqref{eq:coeff-bound-output-gain},
	\begin{align*}
		\big|\mathcal{L}^{(4)}((a^{(j)})_{j=1}^4)\big|
		&\leq
		\sum_{M_1,M_2,M_3,M_4}
		c_{M_1,M_2,M_3,M_4}\prod_{j=1}^4 |a^{(j)}_{M_j}|\\
		&\leq
		\sum_{M_1,M_2,M_3,M_4}
		\Lambda\cdot M_{(1)}^{-\nu_0}
		\prod_{j=1}^4 A^{(j)}_{M_j}.
	\end{align*}
	By Definition~\ref{def:scale-seq}, for $j\in\{1,2,3,4\}$ and any $N\in 2^{\N}$,
	\begin{equation}\label{eq:l1-claim-output-gain}
		\sum_{\substack{M\in 2^{\N}\\ M\leq N}} A_M^{(j)}\lesssim (\log N)^{1/2}.
	\end{equation}

	We group the sum according to the largest dyadic scale $L:=M_{(1)}$:
	\begin{align*}
		\big|\mathcal{L}^{(4)}((a^{(j)})_{j=1}^4)\big|
		&\leq
		\sum_{L}
		\Lambda\cdot L^{-\nu_0}
		\sum_{\substack{M_1,M_2,M_3,M_4\\ M_j\leq L}}
		\prod_{j=1}^4 A_{M_j}^{(j)}\\
		&\leq
		\sum_{L}
		\Lambda\cdot L^{-\nu_0}
		\prod_{j=1}^4\Big(\sum_{M_j\leq L} A_{M_j}^{(j)}\Big).
	\end{align*}
	We split the sum into $L\geq N_{(1)}$ and $L<N_{(1)}$. For $L\geq N_{(1)}$, using~\eqref{eq:l1-claim-output-gain},
	\[
	\sum_{L\geq N_{(1)}}
	\Lambda\cdot L^{-\nu_0}(\log L)^2\lesssim \Lambda N_{(1)}^{-\nu_0}(\log N_{(1)})^{2}.
	\]

	For $L<N_{(1)}$, since $N_1=N_{(1)}$, we have $M_1\leq L<N_1$, and
	\begin{align*}
		&\sum_{L<N_{(1)}}\Lambda\cdot L^{-\nu_0}\sum_{\substack{M_1,M_2,M_3,M_4\\ M_j\leq L}}M_1^{s_1}|a_{M_1}^{(1)}|\prod_{j=2}^4 A_{M_j}^{(j)}\\
		&\qquad\lesssim \sum_{L<N_{(1)}}\Lambda \cdot L^{-\nu_0+s_1}(\log L)^{3/2}\|a^{(1)}\|_{\ell_{M_1}^2}.
	\end{align*}
	Since $\|a^{(1)}\|_{\ell_{M_1}^2}\lesssim N_{(1)}^{-s_1}$ and $\nu_0\leq s_1$, this is bounded by
	\[
	\Lambda\cdot N_{(1)}^{-\nu_0}(\log N_{(1)})^4.
	\]
	This proves~\eqref{eq:dyadic-sum-output-gain}.
\end{proof}

We are now ready to prove Proposition~\ref{prop:quintilinear}.

\begin{proof}[Proof of Proposition~\ref{prop:quintilinear}]

For $v_{N_j}^{(\ast_j)}$ with $\ast_j\in\{\mathrm{C},\mathrm{D}\}$ and $j\in\{1,2,3,4,5\}$, denote
\[
F_n(t):=\mathcal{N}_n^{(4)}(v_{N_1}^{(\ast_1)},v_{N_2}^{(\ast_2)},v_{N_3}^{(\ast_3)},v_{N_4}^{(\ast_4)})(t).
\]

By Lemma~\ref{lem:N4mapping},
\begin{align*}
	\big\|\chi_T(t)\,\mathcal{I}(F\dotcirc v_{N_5}^{(\ast_5)})_n\big\|_{\widetilde{\mathcal{F}}L_2^{b}}\lesssim T^{\delta_0+\frac{1}{r_0'}}\|\chi_1(t)F_n(t)\|_{L_t^{\infty}}\|v_{N_5,n}^{(\ast_5)}\|_{\widetilde{\mathcal{F}}L_2^b}.
\end{align*}
Assume that at least three of $v_{N_j}^{(\ast_j)}$, $j\in\{1,2,3,4\}$, are of type~(D). We apply Lemma~\ref{lem:sum-abstract} together with Corollary~\ref{cor:4-linear-1}.
In the proof of Corollary~\ref{cor:4-linear-1}, if $v_{N_j}^{(\ast_j)}$ is of type~(D), we regard $\mathbf{a}^{(j)}=(\|\mathbf{P}_{M}v^{(\ast_j)}\|_{X^{0,b}})_{M\in 2^{\N}}$ as an $h^s$-envelope, while
for $v_{N_j}^{(\ast_j)}$ of type~(C), we use the $h^{s_C}$-envelope $\mathbf{b}^{(j)}=(\|\mathbf{P}_M v^{(\ast_j)}\|_{X_{q,q}^{0,\gamma}})_{M\in 2^{\N}}$, where
\[
s_C=\alpha-\frac{1}{q}-\delta,
\]
according to the different types of interactions in~(1)--(5) of Corollary~\ref{cor:4-linear-1}. In particular, the uniform estimate~\eqref{cor:4-linear-1-uniform} implies that for any given scales $M_1,M_2,M_3,M_4$, the typical-atom estimate holds with the bound $\Lambda M_{(1)}^{-\nu_0}$, where
\begin{align}\label{eq:Lambdanu_0}
	\Lambda:=R^4 T^{-3(\gamma-\frac{1}{q'})},\quad \nu_0=3\alpha+2s-4-10\theta-\frac{2}{q},
\end{align}
and one verifies from the hierarchy in Definition~\ref{hierarchy} that $\nu_0\leq \min\{s_C,s\}$. Applying Lemma~\ref{lem:sum-abstract}, we obtain
\begin{align}\label{bd:finalN_n^4}
	\|\chi_1(t)F_n(t)\|_{L_t^{\infty}}\lesssim_{\epsilon} \Lambda N^{-\nu_0+\epsilon},
\end{align}
where $N=\max(N_1,N_2,N_3,N_4)$.

By Lemma~\ref{lem:N4mapping}, we deduce
\begin{align}
	\big\|\cdot\mapsto
	\mathcal{I}_{\chi_T}\big(
	\cdot\dotcirc\mathcal{N}^{(4)}(v_{N_1}^{(\ast_1)},v_{N_2}^{(\ast_2)},v_{N_3}^{(\ast_3)},v_{N_4}^{(\ast_4)})
	\big)\big\|_{X^{0,b}\to X^{0,b}}
	&\lesssim R^4 T^{\delta+\frac{1}{r_0'}-3(\gamma-\frac{1}{q'})}N^{-\nu_0},\label{opquintic:X0b}\\
	\big\|\cdot\mapsto
	\mathcal{I}_{\chi_T}\big(
	\cdot\dotcirc\mathcal{N}^{(4)}(v_{N_1}^{(\ast_1)},v_{N_2}^{(\ast_2)},v_{N_3}^{(\ast_3)},v_{N_4}^{(\ast_4)})
	\big)\big\|_{X_{q,q}^{0,\gamma}\to X_{q,q}^{0,\gamma}}
	&\lesssim R^4 T^{\delta+\frac{1}{r_0'}-3(\gamma-\frac{1}{q'})}N^{-\nu_0}, \label{opquintic:X0q}
\end{align}
where $r_0^{-1}=\gamma-\frac{1}{q'}+2\delta=b-\frac{1}{2}+2\delta$. Under the hierarchy~\eqref{hierarchy},
\[
\delta+\frac{1}{r_0'}-3\Big(\gamma-\frac{1}{q'}\Big)\geq \sigma,
\quad\nu_0\geq \delta,
\]
which proves the last two inequalities of Proposition~\ref{prop:quintilinear}.

For the first inequality, we distinguish two cases.

\noi
$\bullet${\bf Case 1: $N_5=2N$ and $\max(N_1,N_2,N_3,N_4)\leq N$.}

We conclude by the operator norm estimate~\eqref{opquintic:X0b} and $\|v_{N_5}^{(\mathrm{D})}\|_{X^{0,b}}\leq (2N)^{-s}$.

\medskip

\noi
$\bullet${\bf Case 2: $N_5 \leq \max(N_1,N_2,N_3,N_4)=2N$.}

By Lemma~\ref{lem:N4mapping} and~\eqref{finalbd:Nn4},
\begin{align*}
	&\big\|
	\mathcal{I}_{\chi_T}\big(\mathcal{N}^{(4)}(v_{N_1}^{(\ast_1)},\cdots, v_{N_4}^{(\ast_4)})\dotcirc v_{N_5}^{(\ast_5)}\big)
	\big\|_{X^{0,b}}\\ &\qquad\leq
	\Big(\sum_{M}
	\big\|\mathcal{I}_{\chi_T}\big(\mathbf{P}_{M}\mathcal{N}^{(4)}(v_{N_1}^{(\ast_1)},\cdots, v_{N_4}^{(\ast_4)})\dotcirc \mathbf{P}_{M}v_{N_5}^{(\ast_5)}\big)
	\big\|_{X^{0,b}}^2
	\Big)^{1/2}\\
	&\qquad\lesssim T^{\delta+\frac{1}{r_0'}}\mathcal{Q},
\end{align*}
where
\begin{align}\label{eq:Q}
	\mathcal{Q}:=\Big(
	\sum_{M}\|\chi(t)\,\mathbf{P}_M\mathcal{N}_n^{(4)}(v_{N_1}^{(\ast_1)},\cdots, v_{N_4}^{(\ast_4)})\|_{\ell_n^{\infty}L_t^{\infty}}^2 \|\mathbf{P}_M v_{N_5}^{(\ast_5)}\|_{X^{0,b}}^2
	\Big)^{1/2}.
\end{align}

For $j\in\{1,2,3,4\}$, write
\[
v_{N_j}^{(\ast_j)}
=
\sum_{K_j\in 2^{\N}}\mathbf{P}_{K_j}v_{N_j}^{(\ast_j)}.
\]
For each fixed dyadic quadruple $(K_1,K_2,K_3,K_4)$, set $K:=\max(K_1,K_2,K_3,K_4)$.
For the output block, define
\[
b_M:=\|\mathbf{P}_M v_{N_5}^{(\ast_5)}\|_{X^{0,b}}.
\]
For the four input blocks, define
\[
a_{K_j}^{(j)}
:=
\begin{cases}
	\|\mathbf{P}_{K_j}v_{N_j}^{(\ast_j)}\|_{X^{0,b}},
	& v_{N_j}^{(\ast_j)} \text{ is of type (D)},\\[1mm]
	\|\mathbf{P}_{K_j}v_{N_j}^{(\ast_j)}\|_{X_{q,q}^{0,\gamma}},
	& v_{N_j}^{(\ast_j)} \text{ is of type (C)}.
\end{cases}
\]
By Definition~\ref{def:scale-seq}, $(a_K^{(j)})_{K\in 2^{\N}}$ is a scale-$N_j$ $h^{s_j}$-envelope, with
\[
s_j=
\begin{cases}
	s,& v_{N_j}^{(\ast_j)} \text{ is of type (D)},\\
	s_C:=\alpha-\frac1q-\delta,& v_{N_j}^{(\ast_j)} \text{ is of type (C)}.
\end{cases}
\]
Moreover, since $v_{N_5}^{(\ast_5)}$ is of type~(D) in the first inequality of
Proposition~\ref{prop:quintilinear}, the sequence $(b_M)_{M\in 2^{\N}}$ is a scale-$N_5$ $h^{s_5}$-envelope, where
\[
s_5=\begin{cases}
	s,& v_{N_5}^{(\ast_5)} \text{ is of type (D)},\\
	s_D:=\alpha-\frac12-\delta,& v_{N_5}^{(\ast_5)} \text{ is of type (C)}.
\end{cases}
\]
By the uniform estimate~\eqref{cor:4-linear-1-uniform} or Corollary~\ref{cor:4-linearprob} (see Remark~\ref{rmk:numerologyquintic}),
for every $n\sim M$,
\begin{align}\label{eq:blockwise-pointwise-case2}
	\|\chi_1(t) 
	\mathbf{P}_M\mathcal{N}_n^{(4)}
&	\big(
	\mathbf{P}_{K_1}v_{N_1}^{(\ast_1)},
	\mathbf{P}_{K_2}v_{N_2}^{(\ast_2)},
	\mathbf{P}_{K_3}v_{N_3}^{(\ast_3)},
	\mathbf{P}_{K_4}v_{N_4}^{(\ast_4)}
	\big)(t)
	\|_{L_t^\infty}  \notag \\
	\lesssim
	&\Lambda\,\frac{K\wedge M}{K}\,
	K^{-\nu_0}\,
	\prod_{j=1}^4 K_j^{s_j}a_{K_j}^{(j)},
\end{align}
with the same $\Lambda,\nu_0$ as in~\eqref{eq:Lambdanu_0} and  $K=\max(K_1,K_2,K_3,K_4)$.

Plugging into~\eqref{eq:Q}, we obtain
\begin{align*}
	\mathcal{Q}&\lesssim \Lambda\Big(\sum_{M}b_M^2\Big(\sum_{\vec{K}}\frac{K\wedge M}{K}\, K^{-\nu_0}\cdot\prod_{j=1}^4 K_j^{s_j}a_{K_j}^{(j)}
	\Big)^2\Big)^{1/2}\\
	&\lesssim \Lambda
	\Big(\sum_{M}b_M^2\,\mathrm{I}_M^2\Big)^{1/2}+\Lambda
	\Big(\sum_{M}b_M^2\,\mathrm{II}_M^2\Big)^{1/2},
\end{align*}
where
\[
\mathrm{I}_M:=\sum_{\substack{K_1,K_2,K_3,K_4\\
		K\leq M
}} K^{-\nu_0}\prod_{j=1}^4 K_j^{s_j}a_{K_j}^{(j)},\qquad \mathrm{II}_M:=\sum_{\substack{K_1,K_2,K_3,K_4\\
		K> M
}} M\, K^{-\nu_0-1}\prod_{j=1}^4 K_j^{s_j}a_{K_j}^{(j)}.
\]

Our goal is to show
\begin{equation}\label{eq:target-Q}
	\mathcal{Q}\lesssim N_{(1)}^{-s}N_{(2)}^{-\delta}.
\end{equation}
The key numerology is
\begin{equation}\label{eq:numerology}
	\nu_0 + s_D - s = \nu_0 + \Big(\alpha - \frac{1}{2} - \delta\Big) - s > 10\delta.
\end{equation}

Without loss of generality, assume $N_1 = N_{(1)} = 2N \ge N_5$ and $N_2$ is the second largest among $N_1,N_2,N_3,N_4$, so that $\max(N_2,N_5)=N_{(2)}$. We bound the input sequences using their envelope properties from Definition~\ref{def:scale-seq}:
\begin{align}
	\|a^{(1)}\|_{\ell^2} &\lesssim N_1^{-s_1}, \label{eq:env1} \\
	\|K_2^{s_2-\delta} a^{(2)}\|_{\ell^2} &\lesssim N_2^{-\delta}, \label{eq:env2} \\
	\|K_j^{s_j} a^{(j)}\|_{\ell^2} &\lesssim 1, \quad j=3,4. \label{eq:env34}
\end{align}
For fixed $K=\max(K_1,K_2,K_3,K_4)$, 
by isolating the index $i \in \{1,2,3,4\}$ that attains $K_i = K$,
\begin{equation}\label{eq:HK-decomp}
	\sum_{\max_j K_j=K} \prod_{j=1}^4 K_j^{s_j} a_{K_j}^{(j)}. \le \sum_{i=1}^4 K^{s_i} a_K^{(i)} \prod_{j \neq i} \Big(\sum_{K_j \le K} K_j^{s_j} a_{K_j}^{(j)}\Big).
\end{equation}
By Cauchy--Schwarz and~\eqref{eq:env1}--\eqref{eq:env34},
\begin{align}
	\sum_{K_1 \le K} K_1^{s_1} a_{K_1}^{(1)} &\lesssim K^{s_1} N_1^{-s_1}, \label{eq:sumK1} \\
	\sum_{K_2 \le K} K_2^{s_2} a_{K_2}^{(2)} &\lesssim K^{\delta} N_2^{-\delta}, \label{eq:sumK2} \\
	\sum_{K_j \le K} K_j^{s_j} a_{K_j}^{(j)} &\lesssim_{\epsilon} K^{\epsilon},\quad j=3,4, \label{eq:sumK34}
\end{align}
for any $\epsilon > 0$. Substituting~\eqref{eq:sumK1}--\eqref{eq:sumK34} into~\eqref{eq:HK-decomp} for each choice of the maximum index~$i$, we obtain
\begin{align}
		\sum_{\max_j K_j=K} \prod_{j=1}^4 K_j^{s_j} a_{K_j}^{(j)} &\lesssim K^{s_1+\delta+2\epsilon} N_2^{-\delta} a_K^{(1)} + K^{s_1+\delta+2\epsilon} N_1^{-s_1} \big(K^{s_2-\delta} a_K^{(2)}\big) \notag \\
	&\quad + K^{s_1+\delta+\epsilon} N_1^{-s_1} N_2^{-\delta} \big(K^{s_3} a_K^{(3)} + K^{s_4} a_K^{(4)}\big). \label{eq:HK-bound}.
\end{align}
Therefore,
\begin{equation}\label{eq:HK-final}
	\sum_{\max_j K_j=K} \prod_{j=1}^4 K_j^{s_j} a_{K_j}^{(j)} \lesssim K^{s_1+\delta+2\epsilon} \widetilde{A}_K,
\end{equation}
where
\[
\widetilde{A}_K := N_2^{-\delta} a_K^{(1)} + N_1^{-s_1} \big(K^{s_2-\delta} a_K^{(2)}\big) + N_1^{-s_1} N_2^{-\delta} \big(K^{s_3} a_K^{(3)} + K^{s_4} a_K^{(4)}\big)
\]
satisfies
\begin{equation}\label{eq:Atilde-norm}
	\|\widetilde{A}\|_{\ell_K^2} \lesssim N_1^{-s_1} N_2^{-\delta}.
\end{equation}

We now bound $\mathrm{I}_M$ and $\mathrm{II}_M$. For $\mathrm{I}_M$, by Cauchy--Schwarz in~$K$:
\begin{align}
	\mathrm{I}_M &= \sum_{K \le M} K^{-\nu_0} 	\sum_{\max_j K_j=K} \prod_{j=1}^4 K_j^{s_j} a_{K_j}^{(j)} \lesssim \sum_{K \le M} K^{s_1 - \nu_0 + \delta + 2\epsilon} \widetilde{A}_K \notag \\
	&\le \Big(\sum_{K \le M} K^{2(s_1 - \nu_0 + \delta + 2\epsilon)}\Big)^{1/2} \|\widetilde{A}\|_{\ell_{K}^2} \notag \\
	&\lesssim M^{s_1 - \nu_0 + \delta + 2\epsilon} N_1^{-s_1} N_2^{-\delta}, \label{eq:IM-bound}
\end{align}
where the geometric sum is dominated by $K=M$ since $s_1 - \nu_0 \geq 0$.

Similarly, for $\mathrm{II}_M$:
\[
\mathrm{II}_M = M \sum_{K > M} K^{-\nu_0-1} 	\sum_{\max_j K_j=K} \prod_{j=1}^4 K_j^{s_j} a_{K_j}^{(j)} \lesssim M \sum_{K > M} K^{s_1 - \nu_0 - 1 + \delta + 2\epsilon} \widetilde{A}_K.
\]
Since $s_1 \le 1$ and $\nu_0 > 10\delta$, the exponent $s_1 - \nu_0 - 1 + \delta + 2\epsilon$ is strictly negative for small $\epsilon$. By Cauchy--Schwarz,
\begin{align}
	\mathrm{II}_M &\le M \Big(\sum_{K > M} K^{2(s_1 - \nu_0 - 1 + \delta + 2\epsilon)}\Big)^{1/2} \|\widetilde{A}\|_{\ell_{K}^2} \notag \\
	&\lesssim M^{s_1 - \nu_0 + \delta + 2\epsilon} N_1^{-s_1} N_2^{-\delta}. \label{eq:IIM-bound}
\end{align}
The bounds for $\mathrm{I}_M$ and $\mathrm{II}_M$ coincide. Setting $\rho := s_1 - \nu_0 + \delta + 2\epsilon$, we obtain
\begin{equation}\label{eq:Q-intermediate}
	\mathcal{Q} \lesssim \Lambda\, N_1^{-s_1} N_2^{-\delta} \Big(\sum_M b_M^2 M^{2\rho}\Big)^{1/2}.
\end{equation}

We perform the summation over the output scale~$M$ by distinguishing two subcases.

\smallskip\noindent
\textbf{Subcase 2.1: $v_{N_5}^{(\ast_5)}$ is of type~(C).}
The scale is fixed at $M = N_5$, so there is no dyadic sum over~$M$. We have $b_M = \|\mathbf{P}_{N_5} v_{N_5}^{(\mathrm{C})}\|_{X^{0,b}} \lesssim N_5^{-s_D}$. From~\eqref{eq:Q-intermediate},
\[
\mathcal{Q} \lesssim \Lambda\, N_5^{-s_D+\rho} N_1^{-s_1} N_2^{-\delta} = \Lambda\, N_5^{-s_D + s_1 - \nu_0 + \delta + 2\epsilon} N_1^{-s} N_1^{-(s_1 - s)} N_2^{-\delta}.
\]
Since $M = N_5 \le N_1$ and $s_1 \ge s$, we have $N_1^{-(s_1 - s)} \le N_5^{-(s_1 - s)}$. The total power of~$N_5$ is then $-s_D - \nu_0 + s + \delta + 2\epsilon$.
By~\eqref{eq:numerology}, $\nu_0 + s_D - s > 10\delta$. Choosing $\epsilon$ small enough ensures this exponent is strictly negative, yielding $\mathcal{Q} \lesssim \Lambda\, N_{(1)}^{-s} N_{(2)}^{-\delta}$.

\smallskip\noindent
\textbf{Subcase 2.2: $v_{N_5}^{(\ast_5)}$ is of type~(D).}
In this case, $(b_M)_{M \le N_5}$ is a scale-$N_5$ $h^s$-envelope, and we must sum over all dyadic~$M$. We bound
\begin{equation}\label{eq:sum-bM}
	\Big(\sum_{M} M^{2\rho} b_M^2\Big)^{1/2}.
\end{equation}
We further distinguish the type of $v_{N_1}^{(\ast_1)}$. If $v_{N_1}^{(\ast_1)}$ is of type~(D), then $s_1=s$ and $\rho=s-\nu_0+\delta+2\epsilon<s$. By property~(2) of Definition~\ref{def:scale-seq}, the expression~\eqref{eq:sum-bM} is bounded by $N_5^{-(s-\rho)}$.

If $v_{N_1}^{(\ast_1)}$ is of type~(C), then $s_1=s_C$ and $\rho=s_C-\nu_0+\delta+2\epsilon$. If $\rho\leq s$, we apply property~(2) of Definition~\ref{def:scale-seq} with $s_0=\rho$ to bound~\eqref{eq:sum-bM} by $N_5^{-(s-\rho)}\leq 1$. Inserting this into~\eqref{eq:Q-intermediate} yields
\[
\mathcal{Q}\lesssim \Lambda\, N_1^{-s_C}N_2^{-\delta}\leq \Lambda\, N_1^{-s}N_{(2)}^{-\delta},
\]
since $s_C>s+\delta$.

If $\rho>s$, we bound~\eqref{eq:sum-bM} by
\[
\Big(\sum_{M\leq N_5}M^{2\rho}b_M^2\Big)^{1/2}+\Big(\sum_{M>N_5}M^{2\rho}b_M^2\Big)^{1/2}.
\]
The first term is bounded by $N_5^{\rho-s}$. For the second, property~(3) of Definition~\ref{def:scale-seq} gives $|b_M|\leq (M/N_5)^{-1}\cdot N_5^{-s}$, hence
\[
\Big(\sum_{M>N_5}M^{2\rho} b_M^2\Big)^{1/2}\lesssim N_5^{-s}\cdot N_5^{\rho-s},
\]
where we used $\rho<s_C<1$. Inserting into~\eqref{eq:Q-intermediate} gives
\[
\mathcal{Q}\lesssim \Lambda\, N_1^{-s_C}N_2^{-\delta}N_5^{\rho-s}\leq \Lambda\, N_1^{-s}N_2^{-\delta}N_5^{\rho-s}\cdot N_1^{-(s_C-s)}\leq \Lambda\, N_1^{-s}N_2^{-\delta}N_5^{\rho-s_C}.
\]
Since $\rho-s_C<-\delta$, we finally obtain
\[
\mathcal{Q}\lesssim \Lambda\, N_{(1)}^{-s}N_{(2)}^{-\delta}.
\]
This completes the proof of Case~2 and of Proposition~\ref{prop:quintilinear}.
\end{proof}
	
	\appendix

	
	\appendix

\section{Some elementary materials}
\label{sub:gamma}
\subsection{The key cancellation}
Here we justify \eqref{eq:cancellation}:
\begin{lemma}\label{lem:cancellation}
Assume that $(\gamma_{n_1n_2n_3n_4})_{n_j\in\N}$ be a real-valued sequence which is invariant under the permutation of indices. Then, for sequence of complex numbers $(z_{n})_{n\in\N}$, for fixed $n_4\in\N$, 
\begin{align*}
&\mathrm{Im}\sum_{\substack{n_1,n_2,n_3}}\gamma_{n_1n_2n_3n_4}z_{n_1}\ov{z_{n_2}}z_{n_3}\ov{z_{n_4}}\\=&\mathrm{Im}\sum_{\substack{n_1,n_2,n_3\\
\{n_1,n_3 \}\cap \{n_2,n_4\}=\emptyset
 } }\gamma_{n_1n_2n_3n_4}z_{n_1}\ov{z_{n_2}}z_{n_3}\ov{z_{n_4}}
+2\mathrm{Im}\sum_{\substack{n_1,n_2,n_3\\
n_3\neq n_4, n_1=n_2}
 }\gamma_{n_1n_2n_3n_4}z_{n_1}\ov{z_{n_2}}z_{n_3}\ov{z_{n_4}}\\
-&\mathrm{Im}\sum_{n_2:n_2\neq n_4 }\gamma_{n_4n_4n_2n_4}|z_{n_4}|^2z_{n_4}\ov{z_{n_2}}-\mathrm{Im}\sum_{\substack{n_1: n_1\neq n_4 } }\gamma_{n_1n_1n_1n_4}|z_{n_1}|^2z_{n_1}\ov{z_{n_4}},
\end{align*}
provided that each series on the left and on the right hand side converge.
\end{lemma}
\begin{proof}
Thanks to the imaginary part we have a key cancellation of some real valued contributions. Since the contribution of $\{n_1=n_2, n_3=n_4\}$ and $\{n_1=n_4, n_2=n_3 \}$ is real, the contribution from $\{n_1,n_3\}\cap\{n_2,n_4\}\neq \emptyset$ must contain one exact paired frequencies or one over-paired frequencies (where three frequencies are the same that distinct from the rest one).

Now, we distinguish cases according to the asymmetry of the index $4$, comparing to other three indices. When adding contributions from
$$ n_1=n_2, n_3\neq n_4,\quad \text{ or }\quad  n_2=n_3, n_1\neq n_4,
$$
we could use the symmetry with respect to $n_1,n_3$, and subtract an over-paired contribution for
$$ n_1=n_2=n_3\neq n_4
$$
represented by the last term on the right hand side.

The sum of the two other two contributions from $$ n_1=n_4,\; n_3\neq n_2 \quad \text{ and }\quad n_3=n_4,\; n_1\neq n_2$$
are real-valued, hence they both cancel, with a repeated over-paired contribution from
$$ n_1=n_3=n_4\neq n_2,
$$
which gives rise to the first term of the second line on the right hand side. The proof is complete. 
\end{proof}
\subsection{Estimate of the correlation coefficient}
Recall that
\[
\gamma_{nn_1n_2n_3}
=\frac{2}{\pi^{3}}
\int_{0}^{\pi}
\frac{\sin(nr)\sin(n_{1}r)\sin(n_{2}r)\sin(n_{3}r)}{r^{2}}
\mathrm{d}r\,.
\]
In this appendix we complete the proof of Lemma~\ref{lem:gamma} by proving~\eqref{eq:gamma2} and~\eqref{eq:gamma3}.

\begin{proof}[Proof of~\eqref{eq:gamma2}]
	Let~$n,m$ be such that~$n/2\leq m\leq 2n$. A change of variable gives
	\[
	\gamma_{nnmm}
	=
	\frac{2}{\pi^{3}}\int_{0}^{\pi}
	\frac{\sin(nr)^{2}\sin(mr)^{2}}{r^{2}}
	\mathrm{d}r
	=
	\frac{2n}{\pi^{3}}\int_{0}^{n\pi}
	\frac{\sin(r)^{2}\sin(\frac{m}{n}r)^{2}}{r^{2}}
	\mathrm{d}r\,.
	\]
	Since
	\[
	\frac{2n}{\pi^{3}}\int_{n\pi}^{+\infty}
	\frac{\sin(r)^{2}\sin(\frac{m}{n}r)^{2}}{r^{2}}
	\mathrm{d}r=\mathcal{O}(1)\,,
	\]
	for~$n$ sufficiently large and~$n/2\leq m\leq 2n$,
	\[
	\gamma_{nnmm} = \frac{2n}{\pi^{3}}\int_{0}^{+\infty}
	\frac{\sin(r)^{2}\sin(\frac{m}{n}r)^{2}}{r^{2}}
	\mathrm{d}r + \mathcal{O}(1)
	\geq \frac{c}{2}\,n-\mathcal{O}(1)\,,
	\]
	where
	\[
	c= \frac{2}{\pi^{3}}\inf_{\rho\in[\frac{1}{2},2]}
	\int_{0}^{+\infty}
	\frac{\sin(r)^{2}\sin(\rho r)^{2}}{r^{2}}
	\mathrm{d}r.
	\]
	Since the above integral defines a nonnegative continuous function of~$\rho$ that does not vanish on the compact set~$[1/2,2]$, we conclude that~$c>0$. This completes the proof of~\eqref{eq:gamma2}.
\end{proof}

To prove~\eqref{eq:gamma3} we need the following trigonometric identity, which can be found in~\cite[equation~(10.1)]{Tzv08}.
\begin{lemma}\label{trignometric}
	For~$k,m\in\mathbb{N}$ and~$\theta\in[0,\pi]$,
	\begin{align}\label{eq:trigo}
		\frac{\sin(k\theta)}{\sin\theta}\cdot\frac{\sin(m\theta)}{\sin\theta}=\sum_{j=1}^{\min(k,m)}\frac{\sin((|k-m|+2j-1)\theta)}{\sin\theta}.
	\end{align}
\end{lemma}

\begin{proof}[Proof of~\eqref{eq:gamma3}]
	To compute the integral defining~$\gamma_{n_0n_1n_2n_3}$, we introduce an even function~$\chi_0\in C_c^{\infty}((-\frac{\pi}{2},\frac{\pi}{2}))$ such that~$\chi_0(z)\equiv 1$ for~$|z|\leq \frac{\pi}{4}$, and set~$\chi_1(z):=\frac{1-\chi_0(z)}{z^2}$. Then
	\begin{align*}
		4\gamma_{n_0n_1n_2n_3}=
		&\int_0^\pi \frac{\chi_0(\theta)}{\theta^2}\prod_{j=0}^3\sin(n_j \theta)\,\mathrm{d}\theta+\int_0^\pi \chi_1(\theta)\prod_{j=0}^3 \sin(n_j \theta) \,\mathrm{d}\theta\\
		=&:\mathrm{I}+\mathrm{II}.
	\end{align*}

	We first show that the interior contribution satisfies~$\mathrm{I} = \mathcal{O}(n_0^{-\infty})$ (this is contained in~\cite[Lemma~2.6]{BGT05}).
	We write
	\[
	\frac{1}{\theta^2}=\frac{1}{\sin^2(\theta)}-\Big(\frac{1}{\sin^2(\theta)}-\frac{1}{\theta^2}\Big).
	\]
	By Mittag-Leffler's expansion,
	\[
	\frac{1}{\sin^2(\theta)}-\frac{1}{\theta^2}=\sum_{\ell \neq 0}\frac{1}{(\theta-\pi\ell)^2}.
	\]
	Using symmetry,
	\begin{align*}
		\mathrm{I}=\frac{1}{2}\int_{-\pi}^{\pi} \frac{\chi_0(z)}{\sin^2(z)}\prod_{j=0}^3 \sin(n_jz)\, \mathrm{d}z +\int_{\R}f(z)\prod_{j=0}^3\sin(n_jz)\, \mathrm{d}z=:\mathrm{I}_1+\mathrm{I}_2,
	\end{align*}
	where~$f(z)=-\chi_0(z)\sum_{\ell\neq 0}\frac{1}{(z-\pi\ell)^2}$ belongs to~$C_c^{\infty}(\R)$. Since~$\prod_{j=0}^3\sin(n_jz)$ is a finite linear combination of exponentials~$\e^{i(\pm n_0\pm n_1\pm n_2\pm n_3)z}$ and~$n_0\gg n_1+n_2+n_3$, the term~$\mathrm{I}_2$ is~$O(n_0^{-\infty})$. For~$\mathrm{I}_1$, set
	\[
	P_n(z):=\frac{\sin(nz)}{\sin(z)},
	\]
	so that
	\[
	\mathrm{I}_1=\frac{1}{2}\int_{-\pi}^{\pi}\chi_0(z)\sin^2(z)\prod_{j=0}^3 P_{n_j}(z)\, \mathrm{d}z.
	\]
	By Lemma~\ref{trignometric} and the assumption~$n_0\gg n_1+n_2+n_3$,
	\[
	\prod_{j=0}^3 P_{n_j}(z)=\sum_{j_1=1}^{n_1}\sum_{j_2=1}^{n_2}\sum_{j_3=1}^{n_3}P_{n_0-n_1-n_2-n_3+2j_1+2j_2+2j_3-3}(z)\,.
	\]
	Hence
	\begin{align*}
		\mathrm{I}_1=\frac{1}{2}\sum_{\substack{1\leq j_1\leq n_1\\
				1\leq j_2\leq n_2\\
				1\leq j_3\leq n_3}}\int_{\R}\chi_0(z)\sin(z) \sin(Nz) \,\mathrm{d}z,
	\end{align*}
	where~$N=n_0-n_1-n_2-n_3+2j_1+2j_2+2j_3-3 \geq \frac{n_0}{2}$. Each term in the sum is non-stationary, so
	\[
	\mathrm{I}_1 = \mathcal{O}(n_1n_2n_3)\cdot \mathcal{O}(n_0^{-\infty})=\mathcal{O}(n_0^{-\infty}).
	\]
It remains to show that~$\mathrm{II}=\mathcal{O}(n_0^{-2})$. We write
	\[
	\mathrm{II}=\int_0^{\pi}\chi_1(\theta)F(\theta)\,\mathrm{d}\theta,
	\]
	where
	\begin{align*}
		F(\theta)&:=\frac{1}{16}\sum_{\substack{\iota_j\in\{+1,-1\}\\j=0,1,2,3}}\prod_{j=0}^3\iota_j\;\e^{i\iota_jn_j\theta}\\
		&=\frac{1}{16}\frac{\mathrm{d}}{\mathrm{d}\theta}
		\Big(\sum_{\substack{\iota_j\in\{+1,-1\}\\j=0,1,2,3}}\frac{\iota_0\iota_1\iota_2\iota_3}{\iota_0n_0+\iota_1n_1+\iota_2n_2+\iota_3n_3}\e^{i(\iota_0n_0+\iota_1n_1+\iota_2n_2+\iota_3n_3)\theta}\Big).
	\end{align*}
	Since~$|\sigma(\vec\iota)| := |\iota_0 n_0 + \iota_1 n_1 + \iota_2 n_2 + \iota_3 n_3| \gtrsim n_0$ for all~$\vec\iota\in\{+1,-1\}^4$ (because~$n_0\gg n_1+n_2+n_3$), we may integrate by parts twice. The boundary terms at~$\theta=0$ vanish because~$\chi_1$ is supported away from~$0$. At~$\theta=\pi$, the boundary contribution after the first integration by parts is proportional to
	\[
	\sum_{\vec\iota}\frac{\iota_0\iota_1\iota_2\iota_3}{\sigma(\vec\iota)}(-1)^{\sigma(\vec\iota)}.
	\]
	The substitution~$\vec\iota\mapsto -\vec\iota$ sends~$\sigma\mapsto -\sigma$ while leaving the product~$\iota_0\iota_1\iota_2\iota_3$ and~$(-1)^{\sigma}$ unchanged. Hence each term pairs with its negative, and the sum vanishes. The same argument applies to the boundary term after the second integration. Each remaining integral is~$\mathcal{O}(|\sigma|^{-2}) = \mathcal{O}(n_0^{-2})$, so~$\mathrm{II}=\mathcal{O}(n_0^{-2})$. This completes the proof of~\eqref{eq:gamma3}.
\end{proof}

\section{On the probabilistic scaling}
\label{sec:scaling}

The notion of \emph{probabilistic scaling}, introduced in~\cite{DNY24,DNY22} and further explained in~\cite{DNYvietnam}, identifies the exponent~$\alpha$ for which the high$\times$high interactions arising in the second Picard iterate fail to exhibit any gain of regularity relative to the typical regularity of the initial data.
In this appendix, we show that for the radial cubic Schrödinger equation~\eqref{eq:NLS} on~$\mathbf{B}_{3}$, the corresponding critical exponent is~$\alpha=1$.

Let~$N\in2^{\N}$, $t\geq0$, and~$\alpha\geq0$. We define
\[
\mathcal{I}(t,\mathbf{P}_{N}\phi_{\alpha}):=\int_{0}^{t}\e^{i(t-\tau)\Delta}|\e^{i\tau\Delta}\mathbf{P}_{N}\phi_{\alpha}|^{2}\e^{i\tau\Delta}
\mathbf{P}_{N}
\phi_{\alpha}\,
\mathrm{d}\tau\,,
\]
which corresponds to the second Picard iterate applied to the wave packet~$\mathbf{P}_{N}\phi_{\alpha}$, localized at frequency scale~$N$. Recall that the typical regularity of~$\phi_{\alpha}$ is~$H^{\alpha-1/2-}$.

\begin{lemma}[Probabilistic scaling]\label{lem:ps}
	Let~$\alpha\geq0$. There exists a constant~$C>0$ such that for all~$t\geq0$ and~$N\in2^{\N}$,
	\begin{equation}
		\label{eq:scaling}
		\left\|\mathcal{I}(t,\mathbf{P}_{N}\phi_{\alpha})\right\|_{L^{2}(\Omega;H^{\alpha-1/2})}\geq CtN^{2(1-\alpha)}\,.
	\end{equation}
\end{lemma}

\begin{remark}[Comparison with the case of~$\mathbb{S}^{2}$]
	The divergence of the second Picard iterate described in Lemma~\ref{lem:ps} for~$\alpha\leq1$ is of a different nature from that established in~\cite{BCLST25}, as it arises purely from high$\times$high interactions. By contrast, the high$\times$low interactions isolated in~\cite{BCLST25} were treated in~\cite{BCST24} through suitable random averaging operators.
\end{remark}

\begin{proof}
	We begin by isolating the resonant contribution responsible for the growth in~$N$. Using the decomposition~\eqref{eq:non} of the nonlinearity, probabilistic cancellations, and following the argument of~\cite[Lemma~4.1]{BCLST25}, we obtain the lower bound:
	\begin{equation}
		\label{eq:apC}
		\|\mathcal{I}(t,\mathbf{P}_{N}\phi)\|_{L^{2}(\Omega;H^{s})}
		\geq
		\left|
		\|\mathrm{II}_{N}(t)\|_{L^{2}(\Omega;H^{s})}
		-
		\|\mathrm{III}_{N}(t)\|_{L^{2}(\Omega;H^{s})}
		\right|\,,
	\end{equation}
	where
	\begin{align*}
		\mathrm{II}_{N}(t)&:=t \sum_{n\sim N}
		\left(
		\sum_{m\sim N}\gamma_{nnmm}\frac{|g_{m}(\omega)|^{2}}{m^{2\alpha}}
		\right)
		\frac{g_{n}(\omega)}{n^{\alpha}}\mathbf{e}_{n}\,,
		\\
		\mathrm{III}_{N}(t)&:= t \sum_{n\sim N}\gamma_{nnnn}\frac{|g_{n}(\omega)|^{2}g_{n}(\omega)}{n^{3\alpha}}\mathbf{e}_{n}\,,
	\end{align*}
	correspond to the contributions arising from~\eqref{res:hhh} and~$\gamma_{nnnn}|u_n|^2u_n$, respectively. The key observation underlying~\eqref{eq:apC} is that, when~$u$ is distributed according to a Gaussian measure, the resonant contributions~\eqref{res:hhh} and~$\gamma_{nnnn}|u_n|^2u_n$ are uncorrelated with the non-resonant terms~\eqref{eq:111}--\eqref{eq:112}.

	We first estimate~$\mathrm{III}_{N}(t)$ from above:
	\[
	\|\mathrm{III}_{N}(t)\|_{L^{2}(\Omega;H^{s})}^{2}
	= t^{2}\sum_{n\sim N}n^{2s}\gamma_{nnnn}^{2}\frac{\mathbb{E}[|g_{n}(\omega)|^{6}]}{n^{6\alpha}}
	\lesssim t^{2}N^{2s+3-6\alpha}\,,
	\]
	which yields, for~$s=\alpha-1/2$,
	\[
	\|\mathrm{III}_{N}(t)\|_{L^{2}(\Omega;H^{s})} \lesssim tN^{1-2\alpha}\,.
	\]
	We now turn to the lower bound for~$\mathrm{II}_{N}(t)$. A direct computation gives
	\begin{align*}
		\|\mathrm{II}_{N}(t)\|_{L^{2}(\Omega;H^{s})}^{2}
		&=t^{2}\sum_{n\sim N}n^{2s}
		\mathbb{E}
		\Big|\sum_{m\sim N} \gamma_{nnmm}
		\frac{|g_{m}|^{2}g_{n}}{m^{2\alpha} n^{\alpha}}\Big|^{2}\\
		&=t^{2}\sum_{n\sim N}n^{2s}
		\sum_{m,\ell\sim N} \gamma_{nnmm}\gamma_{nn\ell\ell}
		\frac{\mathbb{E}[|g_{m}|^{2}|g_{\ell}|^{2}|g_{n}|^{2}]}{m^{2\alpha} \ell^{2\alpha} n^{2\alpha}}\\
		&\gtrsim t^{2} N^{2s-6\alpha}\sum_{n,m,\ell\sim N}\gamma_{nnmm}\gamma_{nn\ell\ell}\,.
	\end{align*}
	Invoking the lower bound~\eqref{eq:gamma2} on the correlation coefficients~$\gamma_{nnmm}$ when~$n\sim m$, we obtain
	\[
	\|\mathrm{II}_{N}(t)\|_{L^{2}(\Omega;H^{s})}^{2} \gtrsim t^{2}N^{2s-6\alpha+5}\,.
	\]
	Taking~$s=\alpha-1/2$ yields the lower bound~\eqref{eq:scaling}. Since~$2(1-\alpha)>1-2\alpha$ for all~$\alpha$, the conclusion follows from~\eqref{eq:apC}.
\end{proof}

	\bibliography{biblio}{}

@article{NS15,
	author = {Nahmod, A. R. and Staffilani, G.},
	doi = {10.4171/JEMS/543},
	fjournal = {Journal of the European Mathematical Society (JEMS)},
	issn = {1435-9855},
	journal = {J. Eur. Math. Soc. (JEMS)},
	language = {English},
	number = {7},
	pages = {1687--1759},
	title = {Almost sure well-posedness for the periodic 3D quintic nonlinear {Schr{\"o}dinger} equation below the energy space},
	volume = {17},
	year = {2015},
	zbl = {1326.35353},
	zbmath = {6483979}}

@article{Bri20,
	author = {Bringmann, B.},
	doi = {10.1093/imrn/rnz385},
	fjournal = {IMRN. International Mathematics Research Notices},
	issn = {1073-7928},
	journal = {Int. Math. Res. Not.},
	language = {English},
	number = {11},
	pages = {8657--8697},
	title = {Almost sure local well-posedness for a derivative nonlinear wave equation},
	volume = {2021},
	year = {2021},
	zbl = {1473.35361},
	zbmath = {7398555}}

@article{DNYvietnam,
	author = {Deng, Y. and Nahmod, A. R. and Yue, H.},
	doi = {10.1007/s10013-023-00672-w},
	fjournal = {Vietnam Journal of Mathematics},
	issn = {2305-221X},
	journal = {Vietnam J. Math.},
	language = {English},
	number = {4},
	pages = {1001--1015},
	title = {The probabilistic scaling paradigm},
	volume = {52},
	year = {2024},
	zbl = {1546.35201},
	zbmath = {7903084}}

@article{sterbenz2010energy,
	author = {Sterbenz, J. and Tataru, D.},
	doi = {10.1007/s00220-010-1061-4},
	journal = {Comm. Math. Phys.},
	number = {1},
	pages = {139--230},
	publisher = {Springer},
	title = {{Energy dispersed large data wave maps in 2+ 1 dimensions}},
	volume = {298},
	year = {2010}}

@article{Tzv08,
	author = {Tzvetkov, N.},
	doi = {10.5802/aif.2422},
	fjournal = {Annales de l'Institut Fourier},
	issn = {0373-0956},
	journal = {Ann. Inst. Fourier},
	language = {English},
	number = {7},
	pages = {2543--2604},
	title = {Invariant measures for the defocusing nonlinear {Schr{\"o}dinger} equation},
	volume = {58},
	year = {2008},
	zbl = {1171.35116},
	zbmath = {5505491}}

@article{BGT05,
	author = {Burq, N. and G{\'e}rard, P. and Tzvetkov, N.},
	doi = {10.1007/s00222-004-0388-x},
	fjournal = {Inventiones Mathematicae},
	issn = {0020-9910},
	journal = {Invent. Math.},
	language = {English},
	number = {1},
	pages = {187--223},
	title = {Bilinear eigenfunction estimates and the nonlinear {Schr{\"o}dinger} equation on surfaces},
	volume = {159},
	year = {2005},
	zbl = {1092.35099},
	zbmath = {2132029}}

@article{BCST24-2,
	author = {Burq, N. and Camps, N. and Sun, C. and Tzvetkov, N.},
	journal = {in preparation},
	title = {Probabilistic well-posedeness for the nonlinear {S}chr{\"o}dinger equation on the {2D} sphere {II}: the {G}ibbs measure},
	year = {2026}}

@article{OTW20,
	author = {Oh, T. and Tzvetkov, N. and Wang, Y.},
	doi = {10.1017/fms.2020.51},
	fjournal = {Forum of Mathematics, Sigma},
	issn = {2050-5094},
	journal = {Forum Math. Sigma},
	language = {English},
	note = {Id/No e48},
	pages = {63},
	title = {Solving the 4NLS with white noise initial data},
	volume = {8},
	year = {2020},
	zbl = {1452.35193},
	zbmath = {7276284}}

@article{BCST24,
	author = {Burq, N. and Camps, N. and Sun, C. and Tzvetkov, N.},
	journal = {arXiv:2404.18229},
	title = {Probabilistic well-posedeness for the nonlinear {S}chr{\"o}dinger equation on the {$2d$} sphere I: positive regularities},
	year = {2025}}

@article{BCLST25,
	author = {Burq, N. and Camps, N. and Latocca, M. and Sun, C. and Tzvetkov, N.},
	doi = {10.4171/EMSS/92},
	fjournal = {EMS Surveys in Mathematical Sciences},
	issn = {2308-2151},
	journal = {EMS Surv. Math. Sci.},
	language = {English},
	number = {1},
	pages = {123--154},
	title = {The second {Picard} iteration of {NLS} on the {{\(2d\)}} sphere does not regularize {Gaussian} random initial data},
	volume = {12},
	year = {2025},
	zbl = {1568.35121},
	zbmath = {8057575}}

@article{Bou96,
	author = {Bourgain, J.},
	doi = {10.1007/BF02099556},
	fjournal = {Communications in Mathematical Physics},
	issn = {0010-3616},
	journal = {Comm. Math. Phys.},
	language = {English},
	number = {2},
	pages = {421--445},
	title = {Invariant measures for the 2D-defocusing nonlinear {Schr{\"o}dinger} equation},
	volume = {176},
	year = {1996},
	zbl = {0852.35131},
	zbmath = {906546}}

@article{BB14,
	author = {Bourgain, J. and Bulut, A.},
	doi = {10.4171/JEMS/461},
	fjournal = {Journal of the European Mathematical Society (JEMS)},
	issn = {1435-9855},
	journal = {J. Eur. Math. Soc. (JEMS)},
	language = {English},
	number = {6},
	pages = {1289--1325},
	title = {Almost sure global well-posedness for the radial nonlinear {Schr{\"o}dinger} equation on the unit ball. {II}: the 3D case},
	volume = {16},
	year = {2014},
	zbl = {1301.35145},
	zbmath = {6324432}}

@article{Yue21,
	author = {Yue, H.},
	doi = {10.1016/j.jde.2021.01.031},
	fjournal = {Journal of Differential Equations},
	issn = {0022-0396},
	journal = {J. Differ. Equations},
	language = {English},
	pages = {754--804},
	title = {Global well-posedness for the energy-critical focusing nonlinear {Schr{\"o}dinger} equation on {{\(\mathbb{T}^4\)}}},
	volume = {280},
	year = {2021},
	zbl = {1459.35346},
	zbmath = {7319449}}

@article{DNY22,
	author = {Deng, Y. and Nahmod, A. R. and Yue, H.},
	doi = {10.1007/s00222-021-01084-8},
	fjournal = {Inventiones Mathematicae},
	issn = {0020-9910},
	journal = {Invent. Math.},
	language = {English},
	number = {2},
	pages = {539--686},
	title = {Random tensors, propagation of randomness, and nonlinear dispersive equations},
	volume = {228},
	year = {2022},
	zbl = {1506.35208},
	zbmath = {7514024}}

@article{DNY24,
	author = {Deng, Y. and Nahmod, A. R. and Yue, H.},
	doi = {10.4007/annals.2024.200.2.1},
	fjournal = {Annals of Mathematics. Second Series},
	issn = {0003-486X},
	journal = {Ann. Math. (2)},
	language = {English},
	number = {2},
	pages = {399--486},
	title = {Invariant {Gibbs} measures and global strong solutions for nonlinear {Schr{\"o}dinger} equations in dimension two},
	volume = {200},
	year = {2024},
	zbmath = {7995674},
	zbl = {1559.35324}}

@article{Tao01,
	author = {Tao, T.},
	doi = {10.1007/PL00005588},
	fjournal = {Communications in Mathematical Physics},
	issn = {0010-3616},
	journal = {Comm. Math. Phys.},
	language = {English},
	number = {2},
	pages = {443--544},
	title = {Global regularity of wave maps. {II}: {Small} energy in two dimensions},
	volume = {224},
	year = {2001},
	zbl = {1020.35046},
	zbmath = {1709188}}

@article{Tao04,
	author = {Tao, T.},
	doi = {10.1142/S0219891604000032},
	fjournal = {Journal of Hyperbolic Differential Equations},
	issn = {0219-8916},
	journal = {J. Hyperbolic Differ. Equ.},
	language = {English},
	number = {1},
	pages = {27--49},
	title = {Global well-posedness of the {Benjamin}-{Ono} equation in {{\({H}^ 1({R})\)}}.},
	volume = {1},
	year = {2004},
	zbl = {1055.35104},
	zbmath = {2078689}}

@article{Kan25,
	arxiv = {arXiv:2512.02250},
	author = {Kaneshiro, C.},
	howpublished = {Preprint, {arXiv}:2512.02250 [math.{PR}] (2025)},
	journal = {arXiv:2512.02250},
	title = {A {New} {Proof} of the {Abstract} {Random} {Tensor} {Estimate} by {Deng}, {Nahmod}, and {Yue}},
	year = {2025}}

@article{BR25,
	author = {Bringmann, B. and Rodnianski, I.},
	doi = {10.2140/pmp.2025.6.139},
	fjournal = {Probability and Mathematical Physics},
	issn = {2690-0998},
	journal = {Probab. Math. Phys.},
	language = {English},
	number = {1},
	pages = {139--193},
	title = {Well-posedness of a gauge-covariant wave equation with space-time white noise forcing},
	volume = {6},
	year = {2025},
	zbl = {1558.35323},
	zbmath = {7989460}}

@article{Tsutsumi,
 author = {Takaoka, H. and Tsutsumi, Y.},
 title = {Well-posedness of the {Cauchy} problem for the modified {KdV} equation with periodic boundary condition},
 fjournal = {IMRN. International Mathematics Research Notices},
 journal = {Int. Math. Res. Not.},
 issn = {1073-7928},
 volume = {2004},
 number = {56},
 pages = {3009--3040},
 year = {2004},
 language = {English},
 doi = {10.1155/S1073792804140555},
 zbMATH = {2135032},
 Zbl = {1154.35442}
}

@incollection{vHan17,
	author = {van Handel, R.},
	booktitle = {Convexity and concentration},
	doi = {10.1007/978-1-4939-7005-6_4},
	isbn = {978-1-4939-7004-9; 978-1-4939-7005-6},
	language = {English},
	pages = {107--156},
	publisher = {New York, NY: Springer},
	title = {Structured random matrices},
	year = {2017},
	zbl = {1376.15029},
	zbmath = {6781194}}

@article{HP93,
	author = {Haagerup, U. and Pisier, G.},
	doi = {10.1215/S0012-7094-93-07134-7},
	fjournal = {Duke Mathematical Journal},
	issn = {0012-7094},
	journal = {Duke Math. J.},
	language = {English},
	number = {3},
	pages = {889--925},
	title = {Bounded linear operators between {{\(C^*\)}}-algebras},
	volume = {71},
	year = {1993},
	zbl = {0803.46064},
	zbmath = {558991}}

@article{FW25,
	archiveprefix = {arXiv},
	author = {Forlano, J. and Wang, Y.},
	eprint = {2509.14861},
	journal = {arXiv:2509.14861},
	primaryclass = {math.AP},
	title = {Invariant Gibbs dynamics for the nonlinear Schr\"odinger equations on the disc},
	year = {2025}}

@article{GKO24,
 author = {Gubinelli, M. and Koch, H. and Oh, T.},
 title = {Paracontrolled approach to the three-dimensional stochastic nonlinear wave equation with quadratic nonlinearity},
 fjournal = {Journal of the European Mathematical Society (JEMS)},
 journal = {J. Eur. Math. Soc. (JEMS)},
 issn = {1435-9855},
 volume = {26},
 number = {3},
 pages = {817--874},
 year = {2024},
 language = {English},
 doi = {10.4171/JEMS/1294},
 zbMATH = {7834921},
 Zbl = {1537.35434}
}

@article{GIP15,
 author = {Gubinelli, M. and Imkeller, P. and Perkowski, N.},
 title = {Paracontrolled distributions and singular {PDEs}},
 fjournal = {Forum of Mathematics, Pi},
 journal = {Forum Math. Pi},
 issn = {2050-5086},
 volume = {3},
 pages = {75},
 note = {Id/No e6},
 year = {2015},
 language = {English},
 doi = {10.1017/fmp.2015.2},
 zbMATH = {6476283},
 Zbl = {1333.60149}
}
	\bibliographystyle{alpha} 

\end{document}